\documentclass[a4paper,11pt,twoside]{article}
\usepackage[english]{babel}
\usepackage[utf8x]{inputenc}

\usepackage[margin=2cm]{geometry}

\usepackage{amsmath,amsthm,amssymb}
\usepackage{times}
\usepackage{enumerate}
\usepackage{color}

\makeatletter
\def\author#1{\gdef\autrun{\def\and{\unskip, }#1}\gdef\@author{#1}}
 \makeatother

\usepackage{amssymb,latexsym}
\usepackage{mathrsfs}
\usepackage{amssymb}
\usepackage{amsmath,latexsym}
\usepackage{amscd,latexsym} 
\usepackage[all]{xy}

\newcommand{\N}{{\mathbb N}}

\newcommand{\R}{{\mathbb R}}

\newtheorem{theorem}{Theorem}[section]

\newtheorem{corollary}[theorem]{Corollary}

\newtheorem{remark}[theorem]{Remark}
\newtheorem{hypothesis}[theorem]{Hypothesis}
\newtheorem{lemma}[theorem]{Lemma}

\newtheorem{proposition}[theorem]{Proposition}
\newtheorem{claim}[theorem]{Claim}
\newtheorem{comparison}[theorem]{Comparison}

\numberwithin{equation}{section}

\begin{document}

\title{Parameterized splitting theorems and bifurcations\\ for potential operators, Part I:
Abstract  theory\thanks
{Partially supported by the NNSF  11271044 of China.
\endgraf\hspace{2mm} 2020 {\it Mathematics Subject Classification.}
Primary: 58E07, 58E09; Secondary: 58E05, 47J15.
 \endgraf\hspace{2mm} {\it Key words and phrases.}
 Bifurcation, potential operator, splitting theorem.
 }}
\author{Guangcun Lu}

\date{November 11, 2021}

 \maketitle \vspace{-0.3in}




\begin{abstract}
This is the first part of a series devoting to the generalizations and applications
of common theorems in variational bifurcation theory.
Using parameterized versions of splitting theorems in Morse theory
we generalize  some famous bifurcation theorems for potential operators
by weakening standard assumptions on the differentiability of the involved functionals,
which opens up 
a way of bifurcation studies for quasi-linear elliptic boundary value problems.
\end{abstract}



%



\tableofcontents

\section{Introduction}\label{sec:Intro}
\setcounter{equation}{0}

Let $H$ be a real Hilbert space, $U$ an open neighborhood of $0$ in $H$,
 and $I$  an open interval in $\mathbb{R}$. Suppose that
$\mathcal{F}:I\times U\to\mathbb{R}$  is G\^ateaux differentiable in the second variable
and that
\begin{equation}\label{e:Intro.1}
\mathcal{F}'_u(\lambda,u)=0
\end{equation}
possesses the trivial solution $u=0$ for each $\lambda\in I$.
A point $(\mu,0)\in I\times H$ is called  a {\it bifurcation point} of (\ref{e:Intro.1})
if every neighborhood of it in $I\times H$ contains a solution $(\lambda,u)$ of (\ref{e:Intro.1})
with $u\ne 0$.
(In the literature people also said
there is {\it continuous bifurcation} of (\ref{e:Intro.1}) at $\mu$
 provided that $\mathcal{S}=\{(\lambda,u)\in I\times U\,|\, \mathcal{F}'_u(\lambda,u)=0\;\hbox{and}\;u\ne 0\}$
has a connected subset $\mathcal{C}$ such that $\bar{\mathcal{C}}\cap(\R\times\{0\})=\{(\mu,0)\}$.)
 Bifurcation theory is concerned with the structure of the solutions of (\ref{e:Intro.1}) near
 a bifurcation point.
 H. Poincare  \cite{Poi} initiated the mathematical study of this subject in 1885.
 A. Liapunov \cite{Lia} and  E. Schmidt \cite{Sch08} reduced bifurcation theory to a finite-dimensional problem.
 M. A. Krasnosel'skii \cite{Kra} and J. Cronin \cite{Cro}
 studied bifurcation problems with  the method of topological degrees.
 The former also first used variational methods to study such problems,
which leaded to tremendous  progress  since the late 1960s,
 see \cite{Ba1, BaCl, Ben, Ber, Ber72, Boh, Can, Ch2, ChWa, EvSt07, FaRa1, FaRa2, IoSch, Kie, Liu,  Mar,  MaWi, Pro, Rab, Stu14B, Wa}
  etc and references therein. Nevertheless, the key first step in all of these studies is to reduce the problem
  to a finite dimensional situation via either the Lyapunov-Schmidt  reduction or the center
 manifold theorems. It was because of this point that the functional $\mathcal{F}$ was often assumed to be $C^2$, except
  \cite{Can, EvSt07, IoSch, Kie, Liu, McTu, Stu14B},
  where functionals of class $C^{1,1}$ or even $C^1$ are considered.
Using  topological degree C. A. Stuart also  studied bifurcation equations of some
non-potential operators with Hadamard differentiability in \cite{Stu14A,  Stu15}.

In this paper we  also use variational methods, precisely Morse theory, to study bifurcation theory
of (\ref{e:Intro.1}) for the functional $\mathcal{F}$ with lower smoothness.
The reference \cite{Pro} by G. Prodi  seems to be the first paper studying bifurcation problems with Morse theory. See also
  M. Berger \cite{Ber72}, B\"ohme \cite{Boh} and  Marino \cite{Mar}.
  L. Nirenberg \cite{Ni} used the Morse lemma and the splitting theorem to study
  such problems from a different view.
 K. C. Chang \cite{Ch2} gave a Morse theory proof of Rabinowitz bifurcation theorem
\cite{Rab}. The basic idea is to study changes of
critical groups of $\mathcal{F}(\lambda,\cdot)$ at $0$ as $\lambda$ varying.
As we know, main tools  to compute critical groups are the splitting theorem and its corollary (the shifting theorem), which
are stated for $C^2$ functionals on Hilbert spaces (\cite[Theorem~I.5.1]{Ch} and \cite[Theorem~8.3]{MaWi}.
 In fact, the first step in the proof of the splitting theorem
(cf. \cite{Ch, MaWi}) is also the Lyapunov-Schmidt finite-dimensional reduction, which requires
 $C^2$-smoothness of functionals. Thus the Rabinowitz bifurcation theorem
    are inapplicable for studying bifurcations
  of quasilinear elliptic equations of potential type
since the corresponding  variational functionals on natural chosen Sobolev spaces
 cannot be of class $C^2$  in general.
Recently, the author developed  Morse theory methods for a class of
quasilinear elliptic equations  by proving some  splitting theorems for some classes of non-$C^2$ functionals \cite{Lu1}--\cite{Lu7}.
In their proofs our finite-dimensional reductions either required weaker differentiability for potential operators or
 were completed on a smaller space.
The ideas of this paper come from these studies.

 ({\bf Notations}. Let $X, Y$ be Banach spaces, and $H$ a Hilbert space. Throughout the paper,
 we denote by $\mathscr{L}(X,Y)$ the space of all bounded linear  operators from $X$ to $Y$,
 by $\mathscr{L}(X)=\mathscr{L}(X,X)$, and by $\mathscr{L}_s(H)$ the space of all bounded linear self-adjoint
 operators on $H$, by  $B_X(x,r)$ (resp. $\bar{B}_X(x,r)$) the open (resp. closed) ball in $X$ with
 center $x\in X$ and radius $r>0$. For a map $f$ from $X$ to $Y$,
 $Df(x)$ (resp. $df(x)$ or $f'(x)$) denotes the G\^ateaux (resp. Fr\'echet) derivative of $f$ at $x\in X$,
 which is an element in $\mathscr{L}(X,Y)$. Of course, we also use $f'(x)$ to denote $Df(x)$
 without occurring of confusions. When $Y=\mathbb{R}$, $f'(x)\in \mathscr{L}(X,\mathbb{R})=X^\ast$,
 and if $X=H$ we call the Riesz representation of $f'(x)$ in $H$
 gradient of $f$ at $x$, denoted by $\nabla f(x)$. The the Fr\'echet (or G\^ateaux) derivative of $\nabla f$ at $x\in H$
 is denoted by $f''(x)$, which is an element in $\mathscr{L}_s(H)$.
($f''(x)$ can be seen as a symmetric bilinear form on $H$ without confusion occurring.)
 Moreover, we always use $0$ to
 denote the origin in different vector spaces without special statements.
 As usual $\mathbb{N}$ denotes the set of all positive integers, and $\mathbb{N}_0=\mathbb{N}\cup\{0\}$.
 The unit sphere in a Banach space $X$ is denoted by $SX$. If $G$ is a compact Lie group, and
 $M$ is a $G$-space or $G$-module, we use $M^G$ to denote the fixed point set of the $G$-action,
 i.e., $M^G=\{x\in M\,|\, gx=x\;\forall g\in G\}$.)

Our main aim is to study generalizations of some previous famous bifurcation theorems
such as those by Krasnosel'skii \cite{Kra}, by Rabinowitz  \cite{Rab}, and by Fadelll and Rabinowitz  \cite{FaRa1, FaRa2}
with parameterized versions of two new splitting theorems established by the author \cite{Lu7} and \cite{Lu1,Lu2, Lu4}
 (see Appendix~\ref{app:A}) and the parameterized version of a splitting  lemma  by Bobylev and Burman \cite{BoBu} (see
Appendix~\ref{app:B}). The following are our basic assumptions related to them.

\begin{hypothesis}\label{hyp:1.1}
{\rm Let $H$ be a Hilbert space with inner product $(\cdot,\cdot)_H$
and the induced norm $\|\cdot\|$, and let $X$ be a dense linear subspace in $H$.
Let  $U$ be an open neighborhood of $0$ in $H$,
and let $\mathcal{L}\in C^1(U,\mathbb{R})$ satisfy $\mathcal{L}'(0)=0$.
Assume that the gradient $\nabla\mathcal{L}$ has a G\^ateaux derivative $B(u)\in \mathscr{L}_s(H)$ at every point
$u\in U\cap X$, and that the map $B: U\cap X\to
\mathscr{L}_s(H)$  has a decomposition
$B=P+Q$, where for each $x\in U\cap X$,  $P(x)\in\mathscr{L}_s(H)$ is  positive definitive and
$Q(x)\in\mathscr{L}_s(H)$ is compact. $B$, $P$ and $Q$ are also assumed to satisfy the following
properties:
\begin{enumerate}
\item[(D1)]  $\{u\in H\,|\, B(0)u=\mu u,\;\mu\le 0\}\subset X$.
\item[(D2)] For any sequence $(x_k)\subset
U\cap X$ with $\|x_k\|\to 0$, $\|P(x_k)u-P(0)u\|\to 0$ for any $u\in H$.
\item[(D3)] The  map $Q: U\cap X\to \mathscr{L}_s(H)$ is continuous at $0$ with respect to the topology
on $H$.
\item[(D4)] For any sequence $(x_k)\subset U\cap X$ with $\|x_k\|\to 0$, there exist
 constants $C_0>0$ and $k_0\in\N$ such that
$(P(x_k)u, u)_H\ge C_0\|u\|^2$ for all $u\in H$ and for all $k\ge k_0$.
\end{enumerate}}
\end{hypothesis}

By \cite[Lemma~2.7]{Lu7} the condition ({\rm D4}) is equivalent to the following
\begin{enumerate}
\item[(D4*)] There exist positive constants $\eta_0>0$ and  $C'_0>0$ such that $\bar{B}_H(0,\eta_0)\subset U$ and
$$
(P(x)u, u)\ge C'_0\|u\|^2\quad\forall u\in H,\;\forall x\in
\bar{B}_H(0,\eta_0)\cap X.
$$
\end{enumerate}

\begin{hypothesis}\label{hyp:1.2}
{\rm Let $U\subset H$ be as in Hypothesis~\ref{hyp:1.1},  $\mathcal{L}\in C^1(U,\mathbb{R})$ satisfy
$\mathcal{L}'(0)=0$  and the gradient $\nabla\mathcal{L}$ have the G\^ateaux derivative
$\mathcal{L}''(u)\in\mathscr{L}_s(H)$ at any $u\in U$, which is a compact operator
 and approaches to $\mathcal{L}''(0)$ in $\mathscr{L}_s(H)$ as $u\to 0$ in $H$.
}
\end{hypothesis}

\begin{hypothesis}\label{hyp:1.3}
{\rm Let $H$ be a Hilbert space with inner product $(\cdot,\cdot)_H$
and the induced norm $\|\cdot\|$, and let $X$ be a Banach space with
norm $\|\cdot\|_X$, such that $X\subset H$ is dense in $H$ and $\|x\|\le\|x\|_X\;\forall x\in X$. 
For an open neighborhood $U$ of $0$ in $H$, $U\cap X$
is also an open neighborhood of $0$ in $X$, denoted by $U^X$.
 Let $\mathcal{L}:U\to\mathbb{R}$ be a continuous functional  satisfying  the following
conditions:
\begin{enumerate}
\item[(F1)] $\mathcal{L}$ is continuously directional
differentiable and $D\mathcal{L}(0)=0$.
\item[(F2)] There exists a continuous and continuously directional differentiable
 map $A: U^X\to X$, which is also  strictly Fr\'{e}chet differentiable
 at $0$,   such that
$D\mathcal{ L}(x)[u]=(A(x), u)_H$ for all $x\in U\cap X$ and $u\in X$.
\item[(F3)] There exists a map $B: U\cap X\to \mathscr{L}_s(H)$ such that
$(DA(x)[u], v)_H=(B(x)u, v)_H$ for all $x\in U\cap X$  and
$u, v\in X$. (So $B(x)$ induces an element in $\mathscr{L}(X)$, denoted by $B(x)|_X$,
and $B(x)|_X=DA(x)\in\mathscr{L}(X),\;\forall x\in U\cap X$.)
\item[(C)]   $\{u\in H\,|\, B(0)(u)\in X\}\subset X$,
in particular  ${\rm Ker}(B(0))\subset X$.
\item[(D)] 
    $B$ satisfies the conditions (D1)--(D4) as in Hypothesis~\ref{hyp:1.1} for the present $X$.
\end{enumerate}}
\end{hypothesis}

({\it Note}: (F2) in \cite[\S8]{Lu2} does not have the requirement that $A: U^X\to X$ is continuous since
we consider that it is contained in the assumption of the strict Fr\'{e}chet differentiability of $A$ at $0$.
Actually, the latter may only imply that $A$ is continuous near $0\in X$, which is not sufficient
for us using the mean value theorem in the third line of \cite[page 2958]{Lu2}.
Actually, in order to assure effectiveness of such an argument we only need the assumption  that
$A: U^X\to X$ is continuous and G\^ateaux differentiable, which is  weaker since
the continuously directional differentiability of $A$ implies that $A$ is G\^ateaux differentiable.)

\begin{hypothesis}\label{hyp:1.4}
{\rm Let $H, X$ and $U\subset H$ be as in Hypothesis~\ref{hyp:1.3},
 $\mathcal{L}\in C^1(U,\mathbb{R})$ has the critical point $0\in U$, and there exist
maps $A\in C^1(U^X, X)$ and $B: U\cap X\to \mathscr{L}_s(H)$ satisfying (F2)-(F3) in
Hypothesis~\ref{hyp:1.3} and
\begin{enumerate}
\item[(D*)] For each $x\in U\cap X$,  $B(x)\in\mathscr{L}_s(H)$ is a compact linear  operator
and the  map $B: U\cap X\to \mathscr{L}_s(H)$ is continuous at $0$ with respect to the topology on $H$.
\end{enumerate}}
\end{hypothesis}

Clearly, Hypothesis~\ref{hyp:1.2} and Hypothesis~\ref{hyp:1.4} are stronger than
Hypothesis~\ref{hyp:1.1} and Hypothesis~\ref{hyp:1.3}, respectively.
In applications, bifurcation theorems obtained under Hypothesises~\ref{hyp:1.1},\ref{hyp:1.2}
are mainly applied to quasi-linear elliptic systems and Lagrange systems;
those obtained under Hypothesises~\ref{hyp:1.3}, \ref{hyp:1.4} or the following are
more powerful for us studying bifurcations in Lagrange systems and geodesics on Finsler manifolds.

An important case of the bifurcation problem (\ref{e:Intro.1})
is that of eigenvalues of nonlinear problems. The following are two related hypothesises.

\begin{hypothesis}\label{hyp:Bif.2.2.0}
{\rm Let tuples $(H,X,U,\mathcal{L}, A, B)$ and
 $(H,X,U, \widehat{\mathcal{L}}, \widehat{A}, \widehat{B})$ satisfy  (F1)-(F3)
  in Hypothesis~\ref{hyp:1.3}, $\mathcal{L},\widehat{\mathcal{L}}\in C^1(U,\mathbb{R})$ and $A, \widehat{A}\in C^1(U^X,X)$.
   Suppose that $\lambda^\ast$ is an eigenvalue  of
\begin{equation}\label{e:Spl.2.1}
B(0)v-\lambda\widehat{B}(0)v=0,\quad v\in H,
\end{equation}
and that  for each $\lambda$ near $\lambda^\ast$ the operator $\mathfrak{B}_{\lambda}:=B(0)-\lambda \widehat{B}(0)$
satisfies (C) in Hypothesis~\ref{hyp:1.3} and (D1) in Hypothesis~\ref{hyp:1.1}, i.e.,
$$
\{u\in H\,|\,\mathfrak{B}_{\lambda}u\in X\}\cup\{u\in H\,|\, \mathfrak{B}_{\lambda}u=\mu u,\;\mu\le 0\}\subset X.
$$
 Assume also that one of the following two conditions holds:
 \begin{enumerate}
 \item[(I)] $\lambda^\ast\ne 0$,
 and for each $x\in U\cap X$,  $B(x)$  has a decomposition
$B(x)=P(x)+ Q(x)$,
where $P(x)\in\mathscr{L}_s(H)$ is  positive definitive,
$Q(x)\in\mathscr{L}_s(H)$ is  compact, and both satisfy the properties (D2)--(D4) in Hypothesis~\ref{hyp:1.1};
 $\widehat{B}$ satisfies (D*) in Hypothesis~\ref{hyp:1.4}, i.e., for each $x\in U\cap X$,
  $\widehat{B}(x)\in\mathscr{L}_s(H)$ is compact and  $\widehat{B}(x)\to\widehat{B}(0)$ in $\mathscr{L}_s(H)$
  as $x\to 0$ along $X\cap U$ in $H$.
    \item[(II)]  $\lambda^\ast=0$, $B$ is as in (I), and   $\widehat{B}(x)\to\widehat{B}(0)$ in $\mathscr{L}_s(H)$
  as $x\to 0$ along $X\cap U$ in $H$.
\end{enumerate}}
\end{hypothesis}

Under Hypothesis~\ref{hyp:Bif.2.2.0}, for each $\lambda$ near $\lambda^\ast$ it is easily checked that  the functional
${\mathcal{L}}_\lambda:={\mathcal{L}}-\lambda\widehat{\mathcal{L}}$ satisfies Hypothesis~\ref{hyp:1.3}.
In fact, this also holds under the following weaker one.

\begin{hypothesis}\label{hyp:Bif.2.2.0+}
{\rm ``$\mathcal{L},\widehat{\mathcal{L}}\in C^1(U,\mathbb{R})$ and $A, \widehat{A}\in C^1(U^X,X)$"
in Hypothesis~\ref{hyp:Bif.2.2.0} is replaced by
``$\widehat{\mathcal{L}}\in C^1(U,\mathbb{R})$ and $\widehat{A}\in C^1(U^X,X)$".
}
\end{hypothesis}

Let us outline  our main ideas to study bifurcation problem of (\ref{e:Intro.1})
when $\mathcal{F}(\lambda,u)=\mathcal{L}(u)-\lambda\widehat{\mathcal{L}}(u)$, where $\mathcal{L}\in C^1(U,\mathbb{R})$ is as in Hypothesis~\ref{hyp:1.1} with $X=H$ and  $\widehat{\mathcal{L}}\in C^1(U,\mathbb{R})$ satisfies Hypothesis~\ref{hyp:1.2}.
Note that for each $\lambda\in\mathbb{R}$, $\mathcal{L}_{\lambda}:=\mathcal{L}-\lambda\widehat{\mathcal{L}}$
also satisfies Hypothesis~\ref{hyp:1.1} with $X=H$. Let $\lambda^\ast\in\mathbb{R}$
be an eigenvalue of $\mathcal{L}''(0)u=\lambda\widehat{\mathcal{L}}''(0)u$.
New finite-dimensional reduction in the proofs of \cite[Theorems~2.12,2.16]{Lu7}
shows that solving
$$
\mathcal{L}'(u)-\lambda\widehat{\mathcal{L}}'(u)=0
$$
for $(\lambda,u)$ near $(\lambda^\ast, 0)$ in $\mathbb{R}\times H$
is equivalent to solving
 $$
 d\mathcal{L}^\circ_\lambda(z)=0
 $$
 for $(\lambda,z)$  near $(\lambda^\ast, 0)$ in $\mathbb{R}\times H^0$,
 where $H^0={\rm Ker}(B(0)-\lambda^\ast\widehat{B}(0))$ and  $(\lambda^\ast-\delta, \lambda^\ast+\delta)\ni\lambda\mapsto \mathcal{L}^\circ_\lambda\in
 C^1(\bar{B}_{H^0}(0,\epsilon))$ is continuous.
    Hence under suitable additional conditions we may carry out other
 arguments along \cite{Rab, FaRa1, FaRa2, Ch2, Wa, BaCl, Ba1} etc, and obtain many bifurcation theorems.
For example, suppose that the eigenvalue $\lambda^\ast$ is also isolated. Then for each $\lambda\ne\lambda^\ast$
 close to $\lambda^\ast$ the origin $0\in H$ is a nondegenerate critical point
 of $\mathcal{L}_{\lambda}$ in the sense of \cite{Lu7}, in particular
  an isolated critical point of $\mathcal{L}_{\lambda}$ and thus $0\in H^\circ$
 is such a critical point of $\mathcal{L}^\circ_{\lambda}$  as well. Under some additional conditions we
 can use the parameterized shifting theorem in  \cite{Lu7}
  to compute critical groups of $\mathcal{L}^\circ_{\lambda}$ at $0\in H^0$
 and conclude that  $\mathcal{L}^\circ_{\lambda}$ takes a local maximum (resp.
 minimum) at $0\in H^0$ as $\lambda$ varies in one (resp. other) side of $\lambda^\ast$.
 Thus  a generalization of Rabinowitz bifurcation theorem \cite{Rab} can follow from these and
 Canino's finite dimension version \cite[Theorem~5.1]{Can}
  for an extension of the Rabinowitz's theorem by  Ioffe and Schwartzman \cite{IoSch}.
If $H$ is equipped with an orthogonal action of  a compact Lie group $G$
 for which  $U$, $\mathcal{L}$ and $\widehat{\mathcal{L}}$ are $G$-invariant,
we may apply methods in Fadelll and Rabinowitz  \cite{FaRa1, FaRa2} to
$G$-invariant $\mathcal{L}^\circ_\lambda$ to generalize their results as $G=\mathbb{Z}_2$ or $S^1$.
Sometime, (for example, when $\mathcal{L}$ and $\widehat{\mathcal{L}}$ satisfy  Hypothesis~\ref{hyp:1.3} and
 Hypothesis~\ref{hyp:1.4}, respectively) the reduced functional  $\mathcal{L}^\circ_\lambda$ is of class $C^2$,
 we may apply a result by  Bartsch and Clapp \cite[\S4]{BaCl} to
 $G$-invariant $\mathcal{L}^\circ_\lambda$ to obtain corresponding generalizations for any compact Lie group
 because we may explicitly compute the number $d$ in \cite[\S4]{BaCl} in our situation.

Once parameterized versions of other splitting theorems on Banach spaces are established,
our above methods are applicable. For example, we  may also write the parameterized versions of the splitting
lemmas at infinity in \cite{Lu3}, and then use them to derive some theorems of bifurcations at infinity.
These will be given otherwise.

 The bifurcation theorems  by Krasnosel'skii \cite{Kra}, by Rabinowitz  \cite{Rab}, and by Fadelll and Rabinowitz  \cite{FaRa1, FaRa2},
and their previous some generalizations, often require that  functionals $\mathcal{F}(\lambda,x)$
depend on $(\lambda,x)$ in $C^2$ way and that corresponding potential operators
$\nabla_x\mathcal{F}(\lambda,x)$ have forms either $A_\lambda x+ o(\|x\|)$ or
$(A-\lambda) x+ o(\|x\|)$ for $x\to 0$. Compared these with the above six hypothesises
it is easily seen that our functionals require lower smoothness and corresponding potential operators
 might have higher nonlinearity.


The rest of the paper is organized as follows.
Section~\ref{sec:CGS} contains some notions, simple recall of necessary results about stability of critical groups,
and  a proposition about the uniform (PS) condition for a family of functionals.
In Section~\ref{sec:B.2.1}, for $\mathcal{F}\in C^0(I\times U,\mathbb{R})$ with  nonlinear dependence on parameters,
 we first discuss some necessary conditions for a point $(\mu,0)\in I\times H$ to be a bifurcation point of (\ref{e:Intro.1}),
 and then  as applications of splitting theorems in Appendix~\ref{app:A}
we prove two sufficient conditions for bifurcations occurring,
which are part generalizations of results by Chow and Lauterbach \cite{ChowLa} and by Kielh\"ofer \cite{Kie}.
Some comparisons of our results with previous work are also given in Section~\ref{app:C}.
In Sections~\ref{sec:B.2}, \ref{sec:B.3} we shall use the parameterized versions of  splitting theorems
\cite{Lu7, Lu1,Lu2, Lu4} to generalize some older bifurcation theorems.
Section~\ref{sec:B.2}  is dedicated to some generalizations of  Rabinowitz bifurcation theorem  \cite{Rab}.
  Section~\ref{sec:B.3} deals with the equivariant case;
  some previous bifurcation theorems, such as those by Fadelll and Rabinowitz  \cite{FaRa1, FaRa2}
  by Bartsch and Clapp  \cite{BaCl} are generalized so that they can be used to study
variational bifurcation  for the integral functionals as in \cite[(1.3)]{Lu7}.
In order to compare our methods with previous those we state corresponding results
for $C^2$ functionals which can be obtained with our methods
at the end of each of Sections~\ref{sec:B.2.1},\ref{sec:B.2},\ref{sec:B.3}.
In Section~\ref{sec:BBH},  for potential operators of Banach-Hilbert regular functionals
(cf. Appendix~\ref{app:B}) we prove parallel results to some  bifurcation theorems in last two sections.
In Appendixes~\ref{app:A},~\ref{app:B} we give more general parameterized versions of splitting  lemmas by author \cite{Lu1, Lu2, Lu7}, and
by Bobylev-Burman \cite{BoBu}, respectively.


\section{Preliminary materials}\label{sec:CGS}
\setcounter{equation}{0}

In this section, for the reader's reference we list necessary notions and
results on stability of critical groups,
 and then give related  propositions.

Recall that a $C^1$ functional $f$ on a Banach space $X$ is said to satisfy
the {\bf Palais-Smale condition} ((PS) {\bf condition}, for short)  in a closed subset $S\subset X$ if
every sequence $(u_n)\subset S$ with $f'(u_n)\to 0$ and $(f(u_n))$  bounded
 has a subsequence converging to $u\in X$. For some $c\in\mathbb{R}$, if
 ``$(f(u_n))$  bounded" is replaced by ``$f(u_n)\to c$", we say that $f$ satisfies
 the {\bf Palais-Smale condition at level $c$} ($(PS)_c$ {\bf condition}, for short) in $S$.
 A family of $C^1$ functionals on $X$  parameterized by
  a metric space $\Lambda$,  $\{f_\lambda\}_{\lambda\in\Lambda}$, is called to satisfy
 the {\bf uniform Palais-Smale condition} ({\bf uniform} (PS) {\bf condition}, for short)  in a closed subset $S\subset X$ if
every sequence $(\lambda_n, u_n)\subset \Lambda\times S$ such that $f'_{\lambda_n}(u_n)\to 0$ and $(f_{\lambda_n}(u_n))$  bounded
 has a subsequence converging to $(\lambda,u)\in \Lambda\times S$
 with $f'_{\lambda}(u)=0$  (cf. \cite[Definition~2.4]{CorH}).
 (If $\Lambda$ is compact we can assume $\lambda_n\to\lambda\in\Lambda$ in this definition.)

\begin{lemma}[\hbox{\cite[Remark 2.1(b)]{CorH}}]\label{lem:pre.1}
For a $C^1$ functional $f$ on a Banach space (or $C^1$ Finsler manifold) $X$ with norm $\|\cdot\|$,
 the {\bf weak slope} $|df|(u)$  of $f$ (as a continuous functional on the metric space $X$)
  at any $u\in X$ (cf. \cite[Definition (2.1)]{DegMa}) is equal to $\|f'(u)\|$.
\end{lemma}

This lemma and \cite[Theorem 5.1]{CorH} directly lead to:

\begin{theorem}\label{th:stablity1}
For a family of functionals on a Banach space $X$, $\{f_\sigma\in C^1(X,\mathbb{R})\,|\, \sigma\in [0,1]\}$,
suppose that there exists an open set $U$ such that
\begin{enumerate}
\item[\rm (i)]  $z_\sigma\in U$ is a unique critical point of $f_\sigma$ in $\overline{U}$, $\forall\sigma\in [0,1]$;
\item[\rm (ii)] $\sigma\to f_\sigma$ is continuous in $C^1(\overline{U})$ topology;
\item[\rm (iii)]  $f_\sigma$ satisfies the (PS) condition in $\overline{U}$,  $\forall\sigma\in [0,1]$.
\end{enumerate}
Then $C_\ast(f_\sigma,z_\sigma;{\bf K})$ is independent of $\sigma$ for any Abel group ${\bf K}$.
\end{theorem}

Actually, \cite[Theorem 5.1]{CorH} was stated for ${\bf K}=\mathbb{R}$. However, the proof therein is still effective for any
any Abel group ${\bf K}$ because the arguments involving in ${\bf K}$ is to use \cite[Proposition 5.2]{CorH} and the latter
also clearly holds true if $\mathbb{R}$ is replaced by ${\bf K}$ by its proof.
Moreover, by \cite[Proposition 3.7]{Cor} the definition of critical groups used in \cite{CorH} is equivalent to the usual one
as in \cite{Ch,Ch1,MaWi, McTu}, i.e.,
$$
C_q(f,z;{\bf K})=H_q(\{f\le f(z)\}\cap U, \{f\le f(z)\}\cap(U\setminus\{z\});{\bf K})
$$
if $z\in X$ is an isolated lower critical point of $f$, where $U$ is a neighborhood of $z$.

When $X$ is a Hilbert space, Theorem~\ref{th:stablity1} was proved in  \cite[Theorem~8.8]{MaWi} (if $f_\sigma\in C^{2-0}(X,\mathbb{R})$),
in Theorem~5.6 of \cite[Chapter I]{Ch} (if $f_\sigma\in C^2(X,\mathbb{R})$), and in \cite[Corollary~5.1.25]{Ch1} (if $f_\sigma\in C^1(X,\mathbb{R})$).

The conditions (ii)-(iii) in Theorem~\ref{th:stablity1} can also be replaced by others.

\begin{theorem}[\hbox{\cite[Theorem 3.6]{CiDe}}]\label{th:stablity2}
For a family of functionals on a Banach space $X$, $\{f_\sigma\in C^1(X,\mathbb{R})\,|\, \sigma\in [0,1]\}$,
suppose that there exists an open set $U$ such that
\begin{enumerate}
\item[\rm (i)] for each $\sigma\in [0,1]$, $z_\sigma\in U$ is a unique critical point of $f_\sigma$ in $\overline{U}$, and
            $[0,1]\ni \sigma\to z_\sigma\in U$ is also continuous;
\item[\rm (ii)] $\sigma\to f_\sigma$ is continuous in $C^0(\overline{U})$ topology;
\item[\rm (iii)] $\{f_\sigma\}_{\sigma\in[0,1]}$ satisfies the uniform (PS) condition on $\overline{U}$, that is,
for every sequence $\sigma_k\to \sigma$ in $[0,1]$ and $(u_k)$ in $\overline{U}$ with $f'_{\sigma_k}(u_k)\to 0$ and
$(f_{\sigma_k}(u_k))$ bounded, there exists a subsequence $(u_{k_j})$ convergent to some $u$ with $f'_\sigma(u)=0$.
\end{enumerate}
Then $C_\ast(f_\sigma,z_\sigma;{\bf K})$ is independent of $\sigma$ for any Abel group ${\bf K}$.
\end{theorem}

 Let $X$ be a Banach space with dual space $X^\ast$.
 A map $T$ from a subset $D$ of $X$ to $X^\ast$ is said to be of {\bf class $(S)_+$ }
 if  for any sequence $(u_j)\subset D$  with  converging weakly to $u$ in $X$
  for which $\varlimsup_{j\to\infty}\langle T(u_j), u_j-u\rangle\le 0$, it
follows that $(u_j)$ converges strongly to $u$ in $X$.
 The definition was introduced by Browder \cite{Bro} and  Skrypnik \cite{Skr} (the condition of belonging to class $(S)_+$ is called condition $\alpha$ in the latter paper).

\begin{proposition}\label{prop:stablity}
Under Hypothesis~\ref{hyp:1.1} or Hypothesis~\ref{hyp:1.3} the restriction of the gradient $\nabla\mathcal{L}$
to a small neighborhood of $0\in H$ is  of the class $(S)_+$.
Thus $\mathcal{L}$ satisfies the (PS) condition in a closed neighborhood of  $0\in H$.
\end{proposition}

The case that Hypothesis~\ref{hyp:1.3} is satisfied  was proved in the proof of Theorem~2.12 of \cite[pages 2966-2967]{Lu2}.
 If Hypothesis~\ref{hyp:1.1} holds the proposition may be proved
by almost same arguments.  It may also be contained in the proof of the following more general claim.

\begin{proposition}\label{prop:stablity1}
Let $H$, $X$ and $U$ be as in Hypothesis~\ref{hyp:1.1},  $I$  an  interval  in $\mathbb{R}$.
  Suppose that $\{\mathcal{F}_\lambda\in C^1(U,\mathbb{R})\,|\,\lambda\in I\}$  satisfies:
$\mathcal{F}'_\lambda(0)=0$,  the gradient $\nabla\mathcal{F}_\lambda$ has a G\^ateaux derivative $B_\lambda(u)\in \mathscr{L}_s(H)$ at every point
$u\in U\cap X$, and that the map $B_\lambda: U\cap X\to
\mathscr{L}_s(H)$  has a decomposition
$B_\lambda=P_\lambda+Q_\lambda$, where for each $x\in U\cap X$,
$Q_\lambda(x)\in\mathscr{L}_s(H)$ is compact. For
some small $\delta>0$ with $\bar{B}_H(0, \delta)\subset U$ operators $P_\lambda$, $Q_\lambda$ and $\nabla\mathcal{F}_\lambda$ are also assumed to satisfy the following
properties:
 \begin{enumerate}
 \item[\rm (i)]   There exist positive constants $c_0>0$ such that
$$
(P_\lambda(x)u, u)\ge c_0\|u\|^2\quad\forall u\in H,\;\forall x\in
\bar{B}_H(0,\delta)\cap X,\quad\forall\lambda\in I.
$$
 \item[\rm (ii)]  As $x\in U\cap X$ approaches $0$ in $H$,
$Q_\lambda(x)\to Q_\lambda(0)$ in $\mathscr{L}_s(H)$ uniformly  with respect to $\lambda\in I$.

  \item[\rm (iii)]  If $(\lambda_n)\subset I$ converges to $\lambda\in I$ then
  $$
  \|Q_{\lambda_n}(0)-Q_\lambda(0)\|\to 0\quad\hbox{and}\quad \|\nabla\mathcal{F}_{\lambda_n}(x)-
  \nabla\mathcal{F}_\lambda(x)\|\to 0\quad\forall x\in \bar{B}_H(0, \delta)\cap X.
  $$
\end{enumerate}
 Then for any sequences $\lambda_n\to\lambda_0$ in $I$ and  $(u_n)\subset\bar{B}_H(0, \delta)$ such that
  $\mathcal{F}'_{\lambda_n}(u_n)\to 0$ and $(\mathcal{F}_{\lambda_n}(u_n))$ is bounded, there exists  a subsequence $u_{n_k}\to u_0\in \bar{B}_H(0, \delta)$
with $\mathcal{F}'_{\lambda_0}(u_0)=0$.

Moreover, if $X$ itself is a normed line space with norm $\|\cdot\|$ such that the inclusion $X\hookrightarrow H$ is
continuous (and so $U\cap X$ is an open neighborhood of $0$ in $X$, denoted by $U^X$ as before),
then the sentence ``\textsf{the gradient $\nabla\mathcal{F}_\lambda$ has a G\^ateaux derivative $B_\lambda(u)\in \mathscr{L}_s(H)$ at every point
$u\in U\cap X$}" in the above statement  can be replaced by
``\textsf{there exists a map $B_\lambda: U^X\to \mathscr{L}_s(H)$ and a continuous and continuously directional differentiable
 map $A_\lambda: U^X\to X$ such that
 \begin{eqnarray}\label{e:pre.1}
D\mathcal{F}_\lambda(x)[u]=(A_\lambda(x), u)_H\quad\hbox{and}\quad
(DA_\lambda(x)[u], v)_H=(B_\lambda(x)u, v)_H
\end{eqnarray}
 for all $x\in U^X$  and $u, v\in X$}".
\end{proposition}
\begin{proof}
By (ii) we may shrink $\delta>0$ so that
$$
\|Q_\lambda(x)-Q_\lambda(0)\|<\frac{c_0}{2},\quad\forall x\in \bar{B}_H(0, \delta)\cap X,\quad\forall\lambda\in I.
$$
It follows from this  and (i) that for all $x\in \bar{B}_H(0, \delta)\cap X$ and $u\in H$,
\begin{eqnarray}\label{e:pre.2}
\bigl(B_\lambda(x)u,u\bigr)_H&=&\bigl(P_\lambda(x)u,u\bigr)_H+
\bigl([Q_\lambda(x)-Q_\lambda(0)]u,u\bigr)_H+ \bigl(Q_\lambda(0)u, u\bigr)_H\nonumber\\
&\ge& \frac{c_0}{2}\|u\|^2+ \bigl(Q_\lambda(0)u, u\bigr)_H.
\end{eqnarray}
Clearly, $(u_n)$ has a subsequence $(u_{n_k})$ weakly converging to $u_0\in \bar{B}_H(0, \delta)$.
Since $\bar{B}_H(0, \delta)\cap X$ is dense in $\bar{B}_H(0, \delta)$, and $\mathcal{F}_\lambda\in C^1(U)$, we have sequences
$(v_{n_k})\subset\bar{B}_H(0, \delta)\cap X$ and $(u_{0m})\subset\bar{B}_H(0, \delta)\cap X$ such that
for all $k,m=1,2,\cdots$,
\begin{equation}\label{e:pre.3}
\left.\begin{array}{ll}
\|u_{0m}-u_0\|<\frac{1}{m},\quad
\|v_{n_k}-u_{n_k}\|<\frac{1}{k}\quad\hbox{and}\quad\\
\|\nabla\mathcal{F}_{\lambda_{n_k}}(v_{n_k})-\nabla\mathcal{F}_{\lambda_{n_k}}(u_{n_k})\|
<\frac{1}{k}.
\end{array}\right\}
\end{equation}
Since $\nabla\mathcal{F}_\lambda$ is continuous and has a G\^ateaux derivative $B_\lambda(u)\in \mathscr{L}_s(H)$ at every point
$u\in U\cap X$, for each fixed $k$ using the mean value theorem we have $\tau\in (0, 1)$ such that
\begin{eqnarray*}
&&(\nabla\mathcal{F}_{\lambda_{n_k}}(v_{n_k}), v_{n_k}-u_{0m})_H\\
&=&(\nabla\mathcal{F}_{\lambda_{n_k}}(v_{n_k})-\nabla\mathcal{F}_{\lambda_{n_k}}(u_{0m}),
v_{n_k}-u_{0m})_H-(\nabla\mathcal{F}_{\lambda_{n_k}}(u_{0m}), v_{n_k}-u_{0m})_H\\
&=&\bigl(D(\nabla\mathcal{F}_{\lambda_{n_k}})(\tau v_{n_k}+ (1-\tau)u_{0m})[v_{n_k}-u_{0m}], v_{n_k}-u_{0m}\bigr)_H\\&&-(\nabla\mathcal{F}_{\lambda_{n_k}}(u_{0m}), v_{n_k}-u_{0m})_H\\
&=&\bigl(B_{\lambda_{n_k}}(\tau v_{n_k}+ (1-\tau)u_{0m})(v_{n_k}-u_{0m}), v_{n_k}-u_{0m}\bigr)_H\\&&-(\nabla\mathcal{F}_{\lambda_{n_k}}(u_{0m}), v_{n_k}-u_{0m})_H\\
&\ge& \frac{c_0}{2}\|v_{n_k}-u_{0m}\|^2-(\nabla\mathcal{F}_{\lambda_{n_k}}(u_{0m}), v_{n_k}-u_{0m})_H\\&&+
(Q_{\lambda_{n_k}}(0)(v_{n_k}-u_{0m}), v_{n_k}-u_{0m})_H,
\end{eqnarray*}
where the final inequality is because of (\ref{e:pre.2}).
Since $\mathcal{F}_\lambda\in C^1(U)$ letting $m\to\infty$ we get
\begin{eqnarray*}
(\nabla\mathcal{F}_{\lambda_{n_k}}(v_{n_k}), v_{n_k}-u_{0})_H&\ge& \frac{c_0}{2}\|v_{n_k}-u_{0}\|^2-(\nabla\mathcal{F}_{\lambda_{n_k}}(u_{0}), v_{n_k}-u_{0})_H\\
&&+(Q_{\lambda_{n_k}}(0)(v_{n_k}-u_{0}), v_{n_k}-u_{0})_H,\quad\forall k=1,2,\cdots.
\end{eqnarray*}
Note that (\ref{e:pre.3}) implies $\lim_{k\to\infty}\nabla\mathcal{F}_{\lambda_{n_k}}(v_{n_k})=\lim_{k\to\infty}\nabla\mathcal{F}_{\lambda_{n_k}}(u_{n_k})=0$ and $v_{n_k}\rightharpoonup u_0$.
It easily follows these  and (iii) that
$$
\frac{c_0}{2}\lim_{k\to\infty}\|v_{n_k}-u_0\|^2\le
\lim_{k\to\infty}(\nabla\mathcal{F}_{\lambda_{n_k}}(v_{n_k}), v_{n_k}-u_0)_H=0
$$
and thus $u_{n_k}\to u_0$ in $H$. The desired claim is proved.

In order to prove the second part, we only need to change the seven lines below (\ref{e:pre.3}) into:

``Since $A_{\lambda}: U^X\to X$  is continuous, and continuously directional differentiable,
 and satisfies (\ref{e:pre.1}),  for each fixed $k$ using the mean value theorem we have $\tau\in (0, 1)$ such that
\begin{eqnarray*}
&&(\nabla\mathcal{F}_{\lambda_{n_k}}(v_{n_k}), v_{n_k}-u_{0m})_H\\
&=&(\nabla\mathcal{F}_{\lambda_{n_k}}(v_{n_k})-\nabla\mathcal{F}_{\lambda_{n_k}}(u_{0m}),
v_{n_k}-u_{0m})_H-(\nabla\mathcal{F}_{\lambda_{n_k}}(u_{0m}), v_{n_k}-u_{0m})_H\\
&=&(A_{\lambda_{n_k}}(v_{n_k})-A_{\lambda_{n_k}}(u_{0m}), v_{n_k}-u_{0m})_H-(\nabla\mathcal{F}_{\lambda_{n_k}}(u_{0m}), v_{n_k}-u_{0m})_H\\
&=&\bigl(DA_{\lambda_{n_k}}(\tau v_{n_k}+ (1-\tau)u_{0m})[v_{n_k}-u_{0m}], v_{n_k}-u_{0m}\bigr)_H\\&&-(\nabla\mathcal{F}_{\lambda_{n_k}}(u_{0m}), v_{n_k}-u_{0m})_H
\end{eqnarray*}"the same conclusion is obtained.
 \end{proof}

\begin{corollary}\label{cor:stablity2}
Suppose that one of the following two conditions holds:
 \begin{enumerate}
\item[\rm (I)]   $\mathcal{L}\in C^1(U,\mathbb{R})$ and $\widehat{\mathcal{L}}\in C^1(U,\mathbb{R})$ satisfy
 Hypothesis~\ref{hyp:1.1} and  Hypothesis~\ref{hyp:1.2}, respectively.
 \item[\rm (II)]  $\mathcal{L}\in C^1(U,\mathbb{R})$ and $\widehat{\mathcal{L}}\in C^1(U,\mathbb{R})$ satisfy
 Hypothesis~\ref{hyp:1.3} and  Hypothesis~\ref{hyp:1.4}, respectively.
  \end{enumerate}
  Then for a bounded interval $I$ in $\mathbb{R}$ and for
   some small $\delta>0$ with $\bar{B}_H(0, \delta)\subset U$
  the conclusion of Proposition~\ref{prop:stablity1} holds true  on $\bar{B}_H(0, \delta)$
for the family $\{\mathcal{F}_\lambda:=\mathcal{L}-\lambda\widehat{\mathcal{L}}\,|\, \lambda\in I\}$,
 that is, for any sequences $\lambda_n\to\lambda_0$ in $I$ and  $(u_n)\subset\bar{B}_H(0, \delta)$ such that
 $\mathcal{F}'_{\lambda_n}(u_n)\to 0$ and $(\mathcal{F}_{\lambda_n}(u_n))$ is bounded, there exists
   a subsequence $u_{n_k}\to u_0\in \bar{B}_H(0, \delta)$
with $\mathcal{F}'_{\lambda_0}(u_0)=0$.
\end{corollary}

\begin{proof}
{\bf Case (I)}.
The gradient $\nabla\mathcal{F}_\lambda=\nabla\mathcal{L}-\lambda\nabla\mathcal{G}$ has a G\^ateaux derivative $B_\lambda(u)=B(u)-\lambda\widehat{\mathcal{L}}''(u)\in \mathscr{L}_s(H)$ at every point $u\in U\cap X$, and that $B_\lambda(x)=P_\lambda(x)+Q_\lambda(x)=P(x)+Q_\lambda(x)$ for each $x\in U\cap X$, where
$Q_\lambda(x)=Q(x)-\lambda\widehat{\mathcal{L}}''(x)\in \mathscr{L}_s(H)$ is compact.
By (D4) or (D4*) in Hypothesis~\ref{hyp:1.1}, $P_\lambda\equiv P$ satisfies Proposition~\ref{prop:stablity1}(i) for some small $\delta>0$ with $\bar{B}_H(0, \delta)\subset U$.
Since the internal $I\subset\mathbb{R}$ is bounded, it is easily seen that
the conditions (ii) and (iii) in Proposition~\ref{prop:stablity1} are satisfied.

\vspace{4pt}\noindent
{\bf Case (II)}. For each $x\in U\cap X$, let $A_\lambda(x)=A(x)-\lambda\widehat{\mathcal{L}}'(x)$,
$B_\lambda(x)=B(x)-\lambda\widehat{\mathcal{L}}''(x)$, $P_\lambda\equiv P$  and $Q_\lambda(x)=Q(x)-\lambda\widehat{\mathcal{L}}''(x)$.
Similar arguments show that the conditions of Proposition~\ref{prop:stablity1} may be satisfied.
\end{proof}

\section{New necessary conditions and sufficient criteria for  bifurcations of gradient mappings }\label{sec:B.2.1}

For a completely continuous operator $A:U\to H$,
 if  $A$ is the gradient of a weakly continuous and uniformly
differentiable functional $f:U\to\mathbb{R}$ and has the Fr\'echet derivative $A'(0)$,
(which must be compact and self-adjoint,) Krasnosel'ski \cite{Kra}
proved that  each nonzero eigenvalue $\mu$ of $A'(0)$ gives a bifurcation point
$(\mu,0)$ of the equation $A(u)=\lambda u$.
More precisely, for any sufficiently small $r>0$ there exists $\lambda_r\in\mathbb{R}$, $u_r\in SH$
such that $A(u_r)=\lambda_r u_r$ and $\lambda_r\to\mu$ as $r\to 0$.
Since then several extensions and improvements of this  work have been made.
For example, if $f\in C^2(U, \mathbb{R})$ satisfies $f'(0)=0$
and $\mu$ is an isolated eigenvalue of $f^{\prime\prime}(0)$ of finite multiplicity,
Rabinowitz \cite{Rab74} proved that  $(\mu,0)$
is a bifurcation point of the equation $\nabla f(u)=\lambda u$.
See \cite{Rab74,To} and references therein for further details.
In this section we give  necessary conditions and
sufficient criteria for a point $(\mu,0)\in I\times H$ to be a bifurcation point of (\ref{e:Intro.1})
under some new assumptions of $\mathcal{F}$. More sufficient criteria
will be given in next three sections.

\subsection{Necessary conditions}\label{sec:B.2.1N}

\begin{theorem}\label{th:Ka1}
Let $H$, $X$ and $U$ be as in Hypothesis~\ref{hyp:1.1},
 and $\Lambda$ a topological space.
  For each $\lambda\in\Lambda$, let $\{\mathcal{F}_\lambda\in C^1(U, \mathbb{R})\,|\,\lambda\in\Lambda\}$
    satisfy $\mathcal{F}'_\lambda(0)=0$, and let the gradient $\nabla\mathcal{F}_\lambda$ have
 a G\^ateaux derivative $B_\lambda(u)\in \mathscr{L}_s(H)$ at every point
$u\in U\cap X$. Suppose that the map $B_\lambda: U\cap X\to
\mathscr{L}_s(H)$  has a decomposition $B_\lambda=P_\lambda+Q_\lambda$, where for each $x\in U\cap X$,
 $P_\lambda(x)\in\mathscr{L}_s(H)$ is  positive definitive and
$Q_\lambda(x)\in\mathscr{L}_s(H)$ is compact, and
   that $P_\lambda$ and $Q_\lambda$ satisfy the following conditions:
    \begin{enumerate}
\item[\rm (i)]  If $(x_k)\subset U\cap X$ approaches to $0$ in $H$ and $(\lambda_k)\subset \Lambda$ converges to $\lambda^\ast$ then
$\|P_{\lambda_k}(x_k)h-P_{\lambda^\ast}(0)h\|\to 0$ for each $h\in H$.
 \item[\rm (ii)]  For some small $\delta>0$, there exists a positive constant $c_0>0$ such that
$$
(P_\lambda(x)u, u)\ge c_0\|u\|^2\quad\forall u\in H,\;\forall x\in
\bar{B}_H(0,\delta)\cap X,\quad\forall\lambda\in \Lambda.
$$
 \item[\rm (iii)]  $Q_\lambda: U\cap X\to \mathscr{L}_s(H)$ is uniformly continuous at $0$  with respect to $\lambda\in \Lambda$.
  \item[\rm (iv)]  If $(\lambda_n)\subset \Lambda$ converges to $\lambda^\ast$ then
  $\|Q_{\lambda_n}(0)-Q_{\lambda^\ast}(0)\|\to 0$.
  \end{enumerate}
Then $0\in H$ is a degenerate critical point of   $\mathcal{F}_{\lambda^\ast}$ if
  $({\lambda}^\ast,0)\in \Lambda\times U$ is a bifurcation point of
    \begin{equation}\label{e:Ka0}
 \mathcal{F}'_{{\lambda}}(u)=0,\quad (\lambda,u)\in\Lambda\times U.
\end{equation}
  \end{theorem}

\begin{proof}
 Since $({\lambda}^\ast, 0)\in\Lambda\times U$ is a bifurcation point  of the equation  (\ref{e:Ka0}),
 there exists a sequence $({\lambda}_k, \bar{u}_k)\in\Lambda\times(U\setminus\{0\})$
such that ${\lambda}_k\to{\lambda}^\ast$, $\bar{u}_k\to 0$ and
 \begin{equation}\label{e:KBi.2.2+}
 \mathcal{F}'_{{\lambda}_k}(\bar{u}_k)=0,\quad\forall k\in\mathbb{N}.
\end{equation}
  Passing to a subsequence, if necessary, we can assume $\bar{v}_k=\bar{u}_k/\|\bar{u}_k\|\rightharpoonup v^\ast$.
By (ii)  there exist positive constants $\eta_0>0$ and  $c_0>0$ such that $B_H(0,\eta_0)\subset U$ and
 \begin{equation}\label{e:KBi.2.3}
(P_\lambda(u)h, h)_H\ge c_0\|h\|^2,\quad\forall h\in H,\;\forall u\in
B_H(0,\eta_0)\cap X,\;\forall\lambda\in I.
\end{equation}
 Clearly, we can assume that $(\bar{u}_k)\subset B_H(0,\eta_0)\setminus\{0\}$.
Since $B_H(0,\eta_0)\cap X$ is dense in $B_H(0,\eta_0)$ we may choose
 $({u}_k)\subset (B_H(0,\eta_0)\cap X)\setminus\{0\}$ such that
  \begin{eqnarray}\label{e:KBi.2.3.0}
&&  \|u_k-\bar{u}_k\|\to 0\quad\hbox{and so}\quad u_k\to 0,\;v_k:={u}_k/\|{u}_k\|\rightharpoonup v^\ast,\\
&&\left\|\frac{1}{\|u_k\|}\nabla\mathcal{F}_{{\lambda}_k}(u_k)-
  \frac{1}{\|\bar{u}_k\|}\nabla\mathcal{F}_{{\lambda}_k}(\bar{u}_k)\right\|<\frac{1}{2^k},\;\forall k,\label{e:KBi.2.3.1}\\
&&\biggl|\frac{1}{\|u_k\|^2}(\nabla\mathcal{F}_{{\lambda}_k}(u_k), u_k)_H
-\frac{1}{\|\bar{u}_k\|^2}(\nabla\mathcal{F}_{{\lambda}_k}(\bar{u}_k), \bar{u}_k)_H\biggr|<\frac{1}{2^k},\;\forall k.\label{e:KBi.2.3.1+}
\end{eqnarray}
  For each fixed $k$, since $\nabla\mathcal{F}_{\lambda_k}(0)=0$
  using the Mean Value Theorem  we get $t_k\in (0, 1)$ such that
\begin{eqnarray}\label{e:KBi.2.4.1}
\frac{1}{\|u_k\|^2}(\nabla\mathcal{F}_{{\lambda}_k}(u_k), u_k)_H&=&
(D(\nabla\mathcal{F}_{{\lambda}_k})(t_ku_k)v_k, v_k)_H\nonumber\\
&=&(P_{{\lambda}_k}(t_ku_k)v_k, v_k)_H +(Q_{{\lambda}_k}(t_ku_k)v_k, v_k)_H\nonumber\\
&\ge& c_0+(Q_{{\lambda}_k}(t_ku_k)v_k, v_k)_H\nonumber\\
&=&c_0+([Q_{{\lambda}_k}(t_ku_k)-Q_{{\lambda}_k}(0)]v_k, v_k)_H\nonumber\\
&&+ (Q_{{\lambda}_k}(0)v_k, v_k)_H.
\end{eqnarray}
Since $t_ku_k\to 0$ by (\ref{e:KBi.2.3.0}),  using (iii) and (iv), respectively,   we deduce that for sufficiently $k$,
\[\begin{split}
|([Q_{{\lambda}_k}(t_ku_k)-Q_{{\lambda}_k}(0)]v_k, v_k)_H|&<\frac{c_0}{4}\quad\hbox{and}\quad\\
 |(Q_{{\lambda}_k}(0)v_k, v_k)_H-(Q_{{\lambda}^\ast}(0)v^\ast, v^\ast)_H|&< \frac{c_0}{2}
 \end{split}
 \]
and thus (\ref{e:KBi.2.4.1}) leads to
\begin{eqnarray*}\label{e:KBi.2.4.1+}
\frac{1}{\|u_k\|^2}(\nabla\mathcal{F}_{{\lambda}_k}(u_k), u_k)_H>
\frac{c_0}{2}+ (Q_{{\lambda}^\ast}(0)v^\ast, v^\ast)_H.
\end{eqnarray*}
 By (\ref{e:KBi.2.2+}) and (\ref{e:KBi.2.3.1+}) the left side approaches to zero.
Hence $v^\ast\ne 0$.

Obverse that (\ref{e:KBi.2.2+}) and (\ref{e:KBi.2.3.1}) imply
  \begin{equation}\label{e:KBi.2.6}
\Big|\frac{1}{\|u_k\|}(\nabla\mathcal{F}_{{\lambda}_k}(u_k), h)_H\Bigr|\le\frac{1}{2^k}\|h\|,\quad
\forall h\in H,\quad\forall k\in\mathbb{N}.
 \end{equation}
  Fixing  $h\ne 0$, as in (\ref{e:KBi.2.4.1}), for some $\tau_k\in (0, 1)$, depending on $u_k$ and $h$,
 \begin{eqnarray}\label{e:KBi.2.7+}
\frac{1}{\|u_k\|}(\nabla\mathcal{F}_{{\lambda}_k}(u_k), h)_H&=&(D(\nabla\mathcal{F}_{{\lambda}_k})(\tau_ku_k)v_k, h)_H\nonumber\\
&=&(P_{{\lambda}_k}(\tau_ku_k)v_k, h)_H +(Q_{{\lambda}_k}(\tau_ku_k)v_k, h)_H\nonumber\\
&=&(v_k, P_{{\lambda}_k}(\tau_ku_k)h)_H +(v_k, Q_{{\lambda}_k}(\tau_ku_k)h)_H.
\end{eqnarray}
As above, by (iii)-(iv) we deduce that $(v_k, Q_{{\lambda}_k}(\tau_ku_k)h)_H\to (v^\ast, Q_{{\lambda}^\ast}(0)h)_H$.
Moreover, (i) implies that $P_{{\lambda}_k}(\tau_ku_k)h\to P_{{\lambda}^\ast}(0)h$. It follows from these
and (\ref{e:KBi.2.3.0}), (\ref{e:KBi.2.6}) and (\ref{e:KBi.2.7+}) that
$$
(v^\ast, P_{{\lambda}^\ast}(0)h)_H +(v^\ast, Q_{{\lambda}^\ast}(0)h)_H=0,\quad\forall h\in H,
$$
and thus $D(\nabla\mathcal{F}_{{\lambda}^\ast})(0)v^\ast=0$.
\end{proof}

Similarly, we have
\begin{theorem}\label{th:Ka2}
In Theorem~\ref{th:Ka1}, if we replace ``Hypothesis~\ref{hyp:1.1}"  by
``Hypothesis~\ref{hyp:1.3}", and  ``the gradient $\nabla\mathcal{F}_\lambda$ has
 a G\^ateaux derivative $B_\lambda(u)\in \mathscr{L}_s(H)$ at every point
$u\in U\cap X$" by ``there exists a map $B_\lambda: U\cap X\to \mathscr{L}_s(H)$ and a continuous and continuously directional differentiable
 map $A_\lambda: U^X\to X$ such that $D\mathcal{ L}(x)[u]=(A_\lambda(x), u)_H$ and
$(DA_\lambda(x)[u], v)_H=(B_\lambda(x)u, v)_H$ for all $x\in U\cap X$  and
$u, v\in X$", then $0\in H$ is a degenerate critical point of   $\mathcal{F}_{\lambda^\ast}$ provided that
  $({\lambda}^\ast,0)\in \Lambda\times U^X$ is a bifurcation point of
   $$
 A_{{\lambda}}(u)=0,\quad (\lambda,u)\in\Lambda\times U^X.
$$
\end{theorem}

Indeed, in the proof of Theorem~\ref{th:Ka1} we only need  to replace
$D(\nabla\mathcal{F}_{{\lambda}_k})(t_ku_k)$ and $D(\nabla\mathcal{F}_{{\lambda}_k})(\tau_ku_k)$
by $A(t_ku_k)$ and $A(\tau_ku_k)$, respectively, and complete the proof of Theorem~\ref{th:Ka2}.

\begin{corollary}\label{cor:Bi.2.2}
 Let $\mathcal{L}\in C^1(U,\mathbb{R})$ satisfy
 Hypothesis~\ref{hyp:1.1} with $X=H$, and let
  $\widehat{\mathcal{L}}_j\in C^1(U,\mathbb{R})$, $j=1,\cdots,n$, satisfy
 Hypothesis~\ref{hyp:1.2}.
Suppose that $(\vec{\lambda}^\ast, 0)\in\mathbb{R}^n\times U$ is a (multiparameter) bifurcation point  for the equation
\begin{equation}\label{e:Bi.2.1}
\mathcal{L}'(u)=\sum^n_{j=1}\lambda_j\widehat{\mathcal{L}}'_j(u),\quad u\in U.
\end{equation}
Then $\vec{\lambda}^\ast=(\lambda^\ast_1,\cdots,\lambda^\ast_n)$ is an  eigenvalue  of
\begin{equation}\label{e:Bi.2.2}
\mathcal{L}''(0)v-\sum^n_{j=1}\lambda_j\widehat{\mathcal{L}}''_j(0)v=0,\quad v\in H,
\end{equation}
 that is, $0$ is
a degenerate critical point of the functional $\mathcal{L}-\sum^n_{j=1}\lambda^\ast_j\widehat{\mathcal{L}}_j$
in the sense stated above Theorem~\ref{th:A.1}.
Moreover, if $\vec{\lambda}^\ast=0$, we only need  that
each $\widehat{\mathcal{L}}_j\in C^1(U,\mathbb{R})$ has properties:
$\widehat{\mathcal{L}}_j'(0)=0$  and the gradient $\nabla\widehat{\mathcal{L}}_j$ has the G\^ateaux derivative
$\mathcal{L}''(u)\in\mathscr{L}_s(H)$ at any $u\in U$, which approaches to
 $\mathcal{L}''(0)$ in $\mathscr{L}_s(H)$ as $u\to 0$ in $H$.
\end{corollary}

This result generalizes  the  necessity part of Theorem~12 in \cite[Chapter~4, \S4.3]{Skr}
(including the classical Krasnoselsi potential bifurcation theorem \cite{Kra}).
The  sufficiency  part of Theorem~12 in \cite[Chapter~4, \S4.3]{Skr} is contained in the case that the
condition (a) in Corollary~\ref{cor:Bi.2.4.2} holds.

 Denoted by  $H(\vec{\lambda})$ the solution space  of (\ref{e:Bi.2.2}).
 It is of finite dimension as the kernel of a linear Fredholm operator.

When $n=1$, comparing with Theorem~12 in \cite[Chapter~4, \S4.3]{Skr1}, the latter also required:
\begin{enumerate}
\item[(a)]  $\widehat{\mathcal{L}}$ is
weakly continuous and uniformly differentiable in $U$;
\item[(b)] $\mathcal{L}'$ has uniformly positive definite Frech\`et derivatives and satisfies the condition
$\alpha)$ in \cite[Chapter~3, \S2.2]{Skr1}.
\end{enumerate}

\begin{proof}[Proof of Corollary~\ref{cor:Bi.2.2}]
First, we consider the case that $\widehat{\mathcal{L}}_j\in C^1(U,\mathbb{R})$, $j=1,\cdots,n$, satisfy
 Hypothesis~\ref{hyp:1.2}.
Let $B=P+Q$ and $\widehat{B}_j=\widehat{P}_j+\widehat{Q}_j$ be the
 corresponding operators with $\mathcal{L}$ and
  $\widehat{\mathcal{L}}_j$, $j=1,\cdots,n$, respectively.
 Then $\widehat{P}_j=0$ and $\widehat{Q}_j=\widehat{\mathcal{L}}''_j$ for $j=1,\cdots,n$.
 Let $\mathcal{F}_{\vec{\lambda}}(u)=\mathcal{L}(u)-\sum^n_{j=1}\lambda_j\widehat{\mathcal{L}}'_j(u)$.
 Denote by $B_{\vec{\lambda}}=P_{\vec{\lambda}}+Q_{\vec{\lambda}}$
   the  corresponding operators.
   Then $P_{\vec{\lambda}}=P$ and $Q_{\vec{\lambda}}=Q-\sum^n_{j=1}\lambda_j\widehat{\mathcal{L}}''_j(u)$.
  Take $\Lambda$ to be any compact neighborhood of $\vec{\lambda}^\ast$ in $\mathbb{R}^n$.
  It is easily checked that  $\{\mathcal{F}_{\vec{\lambda}}\,|\,\vec{\lambda}\in\Lambda\}$
   satisfies the conditions of Theorem~\ref{th:Ka1}.

    Next, we prove the part of ``Moreover".
 We can assume that there exist positive constants $\eta_0>0$ and  $c_0>0$ such that $B_H(0,\eta_0)\subset U$ and
 \begin{equation}\label{e:KBi.2.31}
(P(u)h, h)_H\ge 2c_0\|h\|^2,\quad\forall h\in H,\;\forall u\in
B_H(0,\eta_0)\cap X.
\end{equation}
Since each $\widehat{\mathcal{L}}''_j(u)$ is continuous at $0\in H$, we can shrink $\eta_0>0$ and choose a
small compact neighborhood  $\Lambda$ of $0\in\mathbb{R}^n$ such that
\begin{eqnarray}\label{e:KBi.2.32}
&&\|\widehat{\mathcal{L}}''_j(u)-\widehat{\mathcal{L}}''_j(0)\|<1,\quad\forall u\in
B_H(0,\eta_0),\;j=1,\cdots,n,\\
&&\sum^n_{j=1}|\lambda_j(\widehat{\mathcal{L}}''_j(u)h,h)_H|<c_0\|h\|^2, \;\forall u\in
B_H(0,\eta_0),\;\forall \vec{\lambda}\in\Lambda,\;\forall h\in H.\label{e:KBi.2.33}
\end{eqnarray}
Let us write $Q_{\vec{\lambda}}\equiv Q$ and $P_{\vec{\lambda}}(u)=P(u)-\sum^n_{j=1}\lambda_j\widehat{\mathcal{L}}''_j(u)$.
Then (\ref{e:KBi.2.33}) and (\ref{e:KBi.2.31}) lead to
 \begin{equation}\label{e:KBi.2.34}
(P_{\vec{\lambda}}(u)h, h)_H\ge c_0\|h\|^2,\quad\forall h\in H,\;\forall u\in
B_H(0,\eta_0)\cap X.
\end{equation}
Moreover, if $(x_k)\subset B_H(0,\eta_0)\cap X$ approaches to $0$ in $H$ and $(\vec{\lambda}_k)\subset \Lambda$ converges to $0$ then
for any fixed $h\in H$ we deduce from (\ref{e:KBi.2.33}) and the definition of $P_{\vec{\lambda}}$ that
\begin{eqnarray*}
\|P_{\vec{\lambda}_k}(x_k)h-P_{0}(0)h\|\le \sum^n_{j=1}|\lambda_{k,j}|\|\widehat{\mathcal{L}}''_j(x_k)h\|\le
\sum^n_{j=1}|\lambda_{k,j}|(\|\widehat{\mathcal{L}}''_j(0)\|+1)\|h\|\to 0.
\end{eqnarray*}
 These show that $Q_{\vec{\lambda}}$ and $P_{\vec{\lambda}}$ for $\vec{\lambda}\in\Lambda$
satisfy the assumptions of Theorem~\ref{th:Ka1}.
\end{proof}

Using Theorem~\ref{th:Ka2} and similar reasoning we can obtain

\begin{corollary}\label{cor:Bi.2.2*}
 Let $\mathcal{L}\in C^1(U,\mathbb{R})$ satisfy
 Hypothesis~\ref{hyp:1.3}, and let
  $\widehat{\mathcal{L}}_j\in C^1(U,\mathbb{R})$, $j=1,\cdots,n$, satisfy
 Hypothesis~\ref{hyp:1.4}.
Suppose that $(\vec{\lambda}^\ast, 0)\in\mathbb{R}^n\times U^X$ is a (multiparameter) bifurcation point  for the equation
\begin{equation}\label{e:Bi.2.1*}
A(u)=\sum^n_{j=1}\lambda_j\widehat{A}_j(u),\quad u\in U^X.
\end{equation}
Then $\vec{\lambda}^\ast=(\lambda^\ast_1,\cdots,\lambda^\ast_n)$ is an  eigenvalue  of
\begin{equation}\label{e:Bi.2.2*}
B(0)v-\sum^n_{j=1}\lambda_j\widehat{B}_j(0)v=0,\quad v\in H,
\end{equation}
 that is, $0$ is
a degenerate critical point of the functional $\mathcal{L}-\sum^n_{j=1}\lambda^\ast_j\widehat{\mathcal{L}}_j$
in the sense stated above Theorem~\ref{th:A.1}.
 Moreover, if $\vec{\lambda}^\ast=0$, we only need
that each $\widehat{\mathcal{L}}_j\in C^1(U,\mathbb{R})$ satisfies
 Hypothesis~\ref{hyp:1.4} without requirement that $B(x)\in\mathscr{L}_s(H)$ is compact.
\end{corollary}

Conversely, if $\vec{\lambda}^\ast$ is an isolated eigenvalue  of (\ref{e:Bi.2.2*}), under some additional conditions
we shall show in Theorem~\ref{th:Bi.2.3} that $(\vec{\lambda}^\ast, 0)\in\mathbb{R}^n\times U^X$ is a  bifurcation point
 of (\ref{e:Bi.2.1*}).

\subsection{Sufficient criteria}\label{sec:B.2.1S}

Changes of Morse type numbers imply existence of bifurcation instants \cite{Boh, Ber72}.
Different generalizations are given in \cite{ChowLa,Kie, MaWi, SmoWa, PRS,Ryb}.
For example,  \cite[Theorem~8.8]{MaWi} showed that changes of critical groups lead to bifurcations.
However, it is difficult to compute critical groups for non-twice continuously differentiable functionals.
In this section, with helps of splitting theorems in \cite{Lu2,Lu3, Lu7}
we give some general bifurcation results for potential operator families of a class of non-twice continuously differentiable functionals.

Let $H$ be a real Hilbert space,  $I$  a bounded open interval containing $0$ in $\mathbb{R}$,
and $\{B_\lambda\}_{\lambda\in I}\subset \mathscr{L}_s(H)$
such that $\|B_\lambda-B_0\|\to 0$ as $\lambda\to 0$.
Suppose that  $0$ is an isolated point of the spectrum $\sigma(B_0)$
with $n=\dim{\rm Ker}(B_0)\in (0, \infty)$, and that
${\rm Ker}(B_\lambda)=\{0\}\;\forall\pm\lambda\in (0,\varepsilon_0)$ for some
positive number $\varepsilon_0\ll 1$.
By \cite[Remark I.21.1]{Kie1}  the generalized eigenspace $E_0$ of $B_0$ with eigenvalue $0$
is equal to $N(B_0)={\rm Ker}(B_0)$ (i.e., the algebraic and geometric multiplicities of $0$ are same),
and $H=E_0\oplus R(B_0)$, where $R(B_0)={\rm Im}(B_0)$.
It was shown in  Sections II.5.1 and III.6.4 of \cite{Ka} that
the generalized eigenspace $E_0$
is perturbed to an invariant space $E_\lambda$  of $B_\lambda$ of dimension $n$,
and all perturbed eigenvalues near $0$ (the so-called $0$-group  in Kato's terminology in \cite{Ka}),
denoted by ${\rm eig}_0(B_\lambda)$,
are eigenvalues of the finite-dimensional operator $B_\lambda$ restricted to the $n$-dimensional
invariant space $E_\lambda$ (see \cite[Remark II.4]{Kie1}). In other words,
${\rm eig}_0(B_\lambda)$ is the set of eigenvalues of $B_\lambda$ which approach $0$ as $\lambda\to 0$.
Hence shrinking $\varepsilon_0$ if necessary,  for each $\lambda\in (-\varepsilon_0,\varepsilon_0)\setminus\{0\}$, $B_\lambda$ has exactly
$n$ eigenvalues (depending  continuously on $\lambda$) near zero, and none of them is zero.
 Let $r(B_\lambda)$ be the number of elements in ${\rm eig}_0(B_\lambda)\cap\mathbb{R}^-$ and
 \begin{equation}\label{e:Bi.1.1}
 r^+_{B_\lambda}=\lim_{\lambda\to 0+}r(B_\lambda),\qquad
 r^-_{B_\lambda}=\lim_{\lambda\to 0-}r(B_\lambda).
 \end{equation}
Then if $|\lambda|>0$ is small enough we have $\mu_\lambda-\mu_0=r^+_{B_\lambda}$ for $\lambda>0$,
and $\mu_\lambda-\mu_0=r^-_{B_\lambda}$ for $\lambda<0$, where $\mu_\lambda$
is the dimension of negative definite space of $B_\lambda$.

In some sense the following may be viewed as a converse of Theorem~\ref{th:Ka1}.

\begin{theorem}\label{th:Bi.1.1}
Let $H$ and $I$ be as above,  $U$ an open neighborhood of $0$ in $H$,
 and let $\mathcal{F}:I\times U\to\mathbb{R}$ be such that  each
    $\mathcal{F}_\lambda:=\mathcal{F}(\lambda,\cdot)$ satisfies  Hypothesis~\ref{hyp:1.1} on $U$
    with  corresponding operators  $B_\lambda$, $P_\lambda$ and $Q_\lambda$.
Suppose that the following eight conditions are satisfied:
    \begin{enumerate}
\item[\rm (i)]  For some small $\delta>0$, $\lambda\mapsto \mathcal{F}_\lambda$
    is continuous at $\lambda=0$ in $C^0(\bar{B}_H(0, \delta))$ topology.
 \item[\rm (ii)]  For some small $\delta>0$, there exist positive constants $c_0>0$ such that
$$
(P_\lambda(x)u, u)\ge c_0\|u\|^2\quad\forall u\in H,\;\forall x\in
\bar{B}_H(0,\delta)\cap X,\quad\forall\lambda\in I.
$$
 \item[\rm (iii)]  $Q_\lambda: U\cap X\to \mathscr{L}_s(H)$ is uniformly continuous at $0$  with respect to $\lambda\in I$.

  \item[\rm (iv)]  If $(\lambda_n)\subset I$ converges to $\lambda\in I$ then
  $$
  \|Q_{\lambda_n}(0)-Q_\lambda(0)\|\to 0\quad\hbox{and}\quad \|\nabla\mathcal{F}_{\lambda_n}(x)-\nabla\mathcal{F}_\lambda(x)\|\to 0\quad\forall x\in U\cap X.
  $$
\item[\rm (v)]  ${\rm Ker}(B_\lambda(0))=\{0\}$
 for small $|\lambda|\ne 0$.
 \item[\rm (vi)] $B_\lambda(0)\to B_0(0)$  as $\lambda\to 0$;
\item[\rm (vii)] $0\in\sigma(B_0(0))$.
\item[\rm (viii)]  $r^+_{B_\lambda(0)}\ne
r^-_{B_\lambda(0)}$.
 \end{enumerate}
   Then $(0,0)\in I\times U$ is  a bifurcation point of the equation (\ref{e:Intro.1}).
 Moreover, the same conclusion still holds if the above conditions (i)-(iv) are replaced by the following two
\begin{enumerate}
\item[\rm (a)]  for some small $\delta>0$, $\lambda\mapsto \mathcal{F}_\lambda$
    is continuous at $\lambda=0$ in $C^1(\bar{B}_H(0, \delta))$ topology;
 \item[\rm (b)]  for some small $\delta>0$, each $\mathcal{F}_\lambda$ satisfies the (PS) condition in $\bar{B}_H(0, \delta)$.
\end{enumerate}
\end{theorem}

\begin{proof}
By a contradiction, suppose that $(0,0)\in I\times U$ is not a bifurcation point of the equation
(\ref{e:Intro.1}). Then by shrinking $\delta>0$  we can find $0<\varepsilon_0\ll 1$  such that for each
$\lambda\in [-\varepsilon_0,\varepsilon_0]$ the functional $\mathcal{F}_\lambda$ has a unique
critical point $0$ sitting in $\bar{B}_H(0, \delta)$.

By the first part of Proposition~\ref{prop:stablity1} we see that $\{\mathcal{F}_\lambda\}_{|\lambda|\le\varepsilon_0}$ satisfies
 the  uniform (PS) condition on  $\bar{B}_H(0, \delta)$. It follows from this, (i) and  Theorem~\ref{th:stablity2} that
 \begin{eqnarray}\label{e:Bi.1.2}
 C_\ast(\mathcal{F}_\lambda, 0;{\bf K})=C_\ast(\mathcal{F}_0, 0;{\bf K}),\quad\forall \lambda\in [-\varepsilon_0,\varepsilon_0].
 \end{eqnarray}

It remains to prove  that the assumptions (v)-(viii) insure that (\ref{e:Bi.1.2}) cannot occur.

By (v), we can assume that $0$ is a nondegenerate critical point of $\mathcal{F}_\lambda$
for $0<|\lambda|\le\varepsilon_0$ by shrinking $\varepsilon_0>0$ if necessary.
It follows from this, (\ref{e:Bi.1.2}) and
Theorem~\ref{th:A.1} with $\lambda=0$ (or \cite[Theorem~2.1]{Lu7})
that all $\mathcal{F}_\lambda$, $0<|\lambda|\le\varepsilon_0$,
have the same Morse index $\mu_\lambda$ at $0\in H$, that is,
\begin{eqnarray}\label{e:Bi.1.3}
 [-\varepsilon_0, \varepsilon_0]\setminus\{0\}\ni\lambda\mapsto \mu_\lambda\quad\hbox{is constant.}
 \end{eqnarray}

On the other hand, by \cite[Proposition~B.2]{Lu2}, each $\varrho\in\sigma(B_0(0))\cap\{t\in\mathbb{R}^-\,|\,
t\le 0\}$ is an isolated point in $\sigma(B_0(0))$, which is also an
eigenvalue of finite multiplicity. (This can also be derived from \cite[Lemma~2.2]{BoBu}.)
Since $0\in\sigma(B_0(0))$ by (vii),
$0$ is an isolated point of the spectrum $\sigma(B_0(0))$ and an eigenvalue of $B_0(0)$ of the
 finite multiplicity $s_0$ by \cite[Lemma~2.2]{BoBu}.
Thus we can assume
$$
\sigma(B_0(0))\cap\{t\in\mathbb{R}^-\,|\,
t\le 0\}=\{0,\varrho_1,\cdots,\varrho_k\},
$$
 where $\varrho_i<0$ and has multiplicity $s_i$ for each $i=1,\cdots,k$.
As above (\ref{e:Bi.1.1}), we may use this and (vi) to prove:
if $0<|\lambda|$ is small enough, $B_\lambda(0)$
 has exactly $s_i$ (possible same) eigenvalues near $\varrho_i$, but total dimension
 of corresponding eigensubspaces is equal to that of eigensubspace of $\varrho_i$.
Hence if $\lambda\in (0, \varepsilon_0]$ (resp. $-\lambda\in (0,\varepsilon_0]$) is small enough
we obtain
$\mu_\lambda=\mu_0+ r^+_{B_\lambda(0)}$
(resp. $\mu_{-\lambda}=\mu_0+ r^-_{B_\lambda(0)}$).
These and (viii) imply
$$
\mu_\lambda-\mu_{-\lambda}=r^+_{B_\lambda(0)}-
r^-_{B_\lambda(0)}\ne 0\quad\hbox{for small  $\lambda\in (0, \varepsilon_0]$},
$$
which contradicts the claim in (\ref{e:Bi.1.3}).

In order to prove the final part, note that under the assumptions of the first paragraph
we may use (a)-(b) and Theorem~\ref{th:stablity1} to derive (\ref{e:Bi.1.2}).
The remained arguments are same.
\end{proof}

Correspondingly, we have the following converse of Theorem~\ref{th:Ka2}.

\begin{theorem}\label{th:Bif.1.1}
 Let $H$, $X$ and $U$ be as in Hypothesis~\ref{hyp:1.3},
  and $I\subset\mathbb{R}$  a bounded  open interval containing $0$.
  Let $\mathcal{F}:I\times U\to\mathbb{R}$ be such that  each
    $\mathcal{F}_\lambda:=\mathcal{F}(\lambda,\cdot)$ satisfies  Hypothesis~\ref{hyp:1.3}
    with  corresponding operators $A_\lambda$, $B_\lambda$, $P_\lambda$ and $Q_\lambda$.
    Suppose  that either (i)-(viii) in Theorem~\ref{th:Bi.1.1} or
(a)-(b) and (v)-(viii) in Theorem~\ref{th:Bi.1.1}
 are satisfied.  Then $(0,0)\in I\times U^X$ is  a bifurcation point of the equation
\begin{eqnarray}\label{e:Bi.1.3.1}
 A_\lambda(x)=0.
 \end{eqnarray}
\end{theorem}

The proof is almost repeating that of Theorem~\ref{th:Bi.1.1}. In fact, it suffices to replace
 ``By Proposition~\ref{prop:stablity1}"  and
``Theorem~\ref{th:A.1} with $\lambda=0$ (or \cite[Theorem~2.1]{Lu7})"
with   ``By the latter part of Proposition~\ref{prop:stablity1}" and
``Theorem~\ref{th:A.4} with $\lambda=0$ (or \cite[(2.7)]{Lu2})", respectively.

\begin{corollary}\label{cor:Bi.3}
Let $\lambda^\ast\in\mathbb{R}$.   Suppose that $\mathcal{L}\in C^1(U,\mathbb{R})$ satisfy
 Hypothesis~\ref{hyp:1.1} without (D1),  $\widehat{\mathcal{L}}\in C^1(U,\mathbb{R})$ satisfy  Hypothesis~\ref{hyp:1.2},
 and that for each real $\lambda$ near $\lambda^\ast$ the following holds:
 \begin{enumerate}
  \item[\rm (I.a)] $\{u\in H\,|\, {\mathcal{L}}''(0)u-\lambda\widehat{\mathcal{L}}''(0)u=\mu u,\,\mu\le 0\}\subset X$ for each $\lambda$ near $\lambda^\ast$;
  \item[\rm (I.b)] ${\rm Ker}({\mathcal{L}}''(0)-\lambda^\ast\widehat{\mathcal{L}}''(0))\ne\{0\}$;
 \item[\rm (I.c)] ${\rm Ker}({\mathcal{L}}''(0)-\lambda\widehat{\mathcal{L}}''(0))=\{0\}$ for each $\lambda\ne\lambda^\ast$ near $\lambda^\ast$;
 \item[\rm (I.d)] $\lim_{\lambda\to 0+}r({\mathcal{L}}''(0)-(\lambda+\lambda^\ast)\widehat{\mathcal{L}}''(0))\ne
 \lim_{\lambda\to 0-}r({\mathcal{L}}''(0)-(\lambda+\lambda^\ast)\widehat{\mathcal{L}}''(0))$.
  \end{enumerate}
  Then  $(\lambda^\ast,0)\in \mathbb{R}\times U$ is  a bifurcation point of the equation
 \begin{eqnarray}\label{e:Bi.1.4}
  \mathcal{L}'(u)-\lambda\widehat{\mathcal{L}}'(u)=0.
  \end{eqnarray}
   Moreover, the condition (I.d)  can be replaced by
   \leftmargini=10mm
  \begin{enumerate}
  \item[\rm (I.d')] ${\mathcal{L}}''(0)$ is invertible, ${\mathcal{L}}''(0)\widehat{\mathcal{L}}''(0)=\widehat{\mathcal{L}}''(0){\mathcal{L}}''(0)$
  and the positive and negative indexes of inertia of the restriction of ${\mathcal{L}}''(0)$ to $H^0_{\lambda^\ast}:={\rm Ker}({\mathcal{L}}''(0)-\lambda^\ast\widehat{\mathcal{L}}''(0))$
  are different.
  \end{enumerate}
  \end{corollary}

\begin{corollary}\label{cor:Bi.3.1}
Let $\lambda^\ast\in\mathbb{R}$.  Suppose that $\mathcal{L}\in C^1(U,\mathbb{R})$ satisfy
 Hypothesis~\ref{hyp:1.3} without (C) and (D1), $\widehat{\mathcal{L}}\in C^1(U,\mathbb{R})$ satisfy  Hypothesis~\ref{hyp:1.4}
 with corresponding operators $\widehat{A}$ and $\widehat{B}=\widehat{P}+\widehat{Q}$,
 and that for each real $\lambda$ near $\lambda^\ast$
 the following holds: \leftmargini=10mm
 \begin{enumerate}
  \item[\rm (II.a)] $\{u\in H\,|\, B(0)u-\lambda\widehat{B}(0)u=\mu u,\,\mu\le 0\}\cup
  \{u\in H\,|\, B(0)u-\lambda\widehat{B}(0)u\in X\}\subset X$ for each $\lambda$ near $\lambda^\ast$;
  \item[\rm (II.b)] ${\rm Ker}(B(0)-\lambda^\ast\widehat{B}(0))\ne\{0\}$;
 \item[\rm (II.c)] ${\rm Ker}(B(0)-\lambda\widehat{B}(0))=\{0\}$ for each $\lambda\ne\lambda^\ast$ near $\lambda^\ast$;
 \item[\rm (II.d)] $\lim_{\lambda\to 0+}r(B(0)-(\lambda+\lambda^\ast)\widehat{B}(0))\ne
 \lim_{\lambda\to 0-}r(B(0)-(\lambda+\lambda^\ast)\widehat{B}(0))$.
  \end{enumerate}
  Then  $(\lambda^\ast,0)\in \mathbb{R}\times U^X$ is  a bifurcation point of the equation
 \begin{eqnarray}\label{e:Bi.1.4.1}
  A(u)-\lambda\widehat{A}(u)=0.
  \end{eqnarray}
   Moreover, the condition (II.d)) can be replaced by \leftmargini=12mm
  \begin{enumerate}
   \item[\rm (II.d')] ${\mathcal{L}}''(0)$ and $\widehat{\mathcal{L}}''(0)$ in (I.d') are replaced by $B(0)$ and $\widehat{B}(0)$, respectively.
 \end{enumerate}
 \end{corollary}

If  $\lim_{\lambda\to 0+}r(B(0)-(\lambda+\lambda^\ast)\widehat{B}(0))-
 \lim_{\lambda\to 0-}r(B(0)-(\lambda+\lambda^\ast)\widehat{B}(0))$ is an odd, which is stronger than (II.d),
 then it follows from Theorem~\ref{th:C.2} (\cite[Theorem~6.3]{Stu14A}) that there is  continuous bifurcation at $\lambda^\ast$;
see Comparison~\ref{comp:C.2} in Section~\ref{app:C}.

Consider $\mathcal{F}_\lambda:=\widehat{\mathcal{L}}-(\lambda+\lambda^\ast)\widehat{\mathcal{L}}$ for each real $\lambda$ near $0$.
Then Corollary~\ref{cor:Bi.3} (resp. Corollary~\ref{cor:Bi.3.1}) can follows from  Corollary~\ref{cor:stablity2} and
Theorem~\ref{th:Bi.1.1} (resp. Corollary~\ref{cor:stablity2} and Theorem~\ref{th:Bif.1.1})
if (I) (resp. (II)) holds. As to the ``Moreover" parts in these two corollaries, by (\ref{e:Bi.2.19.1}) in the proof of Corollary~\ref{cor:Bi.2.4.2},
we see that $\mathcal{F}_\lambda$ has different Morse indexes at $0\in H$ as $\lambda$ varies in both sides of $\lambda^\ast$.

Based on the arguments in \cite{Lu1}, we may use
Theorem~\ref{th:Bi.1.1} to get a generalization of \cite[Theorem~5.4.1]{EKBB} immediately. See
\cite[Theorem~3.10]{Lu8}
for a high dimensional analogue (corresponding to
\cite[Theorems~5.4.2 and 5.7.4]{EKBB}).

In Corollary~\ref{cor:Bi.3}, if the condition ``${\mathcal{L}}''(0)$ is invertible" in (I.d')
is strengthened as ``${\mathcal{L}}''(0)$ is positive definite", other conditions are unnecessary.
That is, we have

\begin{theorem}\label{th:Bi.3}
Let $\mathcal{L}\in C^1(U,\mathbb{R})$ (resp.  $\widehat{\mathcal{L}}\in C^1(U,\mathbb{R})$) satisfy
 Hypothesis~\ref{hyp:1.1} with $X=H$ (resp.  Hypothesis~\ref{hyp:1.2}).
Suppose that ${\mathcal{L}}''(0)$ is positive definite.
Then $(\lambda^\ast,0)\in \mathbb{R}\times U$ is  a bifurcation point of the equation
(\ref{e:Bi.1.4}) if and only if  $\lambda^\ast\in\mathbb{R}$ is an  eigenvalue of
  \begin{eqnarray}\label{e:Bi.1.5}
  \mathcal{L}''(0)u-\lambda\widehat{\mathcal{L}}''(0)u=0,\; u\in H.
  \end{eqnarray}
\end{theorem}

\begin{proof}
Necessity is contained in Corollary~\ref{cor:Bi.2.2}. It remains to prove sufficiency.
Let $J$ denote the inverse of $(\mathcal{L}''(0))^{1/2}$, which is in $\mathscr{L}_s(H)$
and commutes with any $S\in\mathscr{L}_s(H)$ that commutes with $\mathcal{L}''(0)$. Define functionals
$$
\mathfrak{L}(u):= \mathcal{L}(Ju),\quad    \widehat{\mathfrak{L}}(u):= \widehat{\mathcal{L}}(Ju),\quad u\in J^{-1}(U).
$$
Then for any $u\in J^{-1}(U)$ and all $v\in H$ we have $\nabla\mathfrak{L}(u)= J\nabla\mathcal{L}(Ju)$,
$\nabla\widehat{\mathfrak{L}}(u)= J\nabla\widehat{\mathcal{L}}(Ju)$ and
$$
D(\nabla\mathfrak{L})(u)[v]= JD(\nabla\mathcal{L})(Ju)[Jv]\quad\hbox{and}\quad
D(\nabla\widehat{\mathfrak{L}})(u)[v]= JD(\nabla\widehat{\mathcal{L}})(Ju)[Jv].
$$
It follows that $\mathfrak{L}\in C^1(J^{-1}(U),\mathbb{R})$ (resp.  $\widehat{\mathfrak{L}}\in C^1(J^{-1}(U),\mathbb{R})$) satisfy
 Hypothesis~\ref{hyp:1.1} with $X=H$ (resp.  Hypothesis~\ref{hyp:1.2}).
Moreover, the bifurcation problem (\ref{e:Bi.1.4}) is equivalent to the following
\begin{eqnarray}\label{e:Bi.1.6}
  \mathfrak{L}'(u)-\lambda\widehat{\mathfrak{L}}'(u)=0,\quad u\in J^{-1}(U)
  \end{eqnarray}
and that $\lambda^\ast\in\mathbb{R}$ is an  eigenvalue of (\ref{e:Bi.1.5}) if and only if it is that of
  \begin{eqnarray}\label{e:Bi.1.7}
  J\mathcal{L}''(0)Ju-\lambda J\widehat{\mathcal{L}}''(0)Ju=0\quad\Longleftrightarrow\quad
  u-\lambda Lu=0,
  \end{eqnarray}
where $L:=J\widehat{\mathcal{L}}''(0)J$. Since $L\in\mathscr{L}_s(H)$ is compact,
The eigenvalues of (\ref{e:Bi.1.7}) consists of a sequence of nonzero  reals
$\{\lambda_k\}_{k=1}^\infty$ diverging  to infinity, and each of them is
 of  finite multiplicity. Hence $\lambda^\ast=\lambda_{k_0}$ for some $k_0\in\mathbb{N}$.
 Let $H_k$ be the eigensubspace corresponding to $\lambda_k$ for $k\in\mathbb{N}$, i.e.,
 \begin{equation}\label{e:Bi.2.8sec3}
H_k={\rm Ker}(I-\lambda_kL)={\rm Ker}(\mathfrak{L}''(0)-\lambda_k \widehat{\mathfrak{L}}''(0)),\quad
\forall k\in\mathbb{N}.
\end{equation}
 Then $H=\oplus^\infty_{k=0}H_k$, where $H_0={\rm Ker}(L)={\rm Ker}(\widehat{\mathcal{L}}''(0))$.

Now the eigenvalue $\lambda^\ast=\lambda_{k_0}$  is isolated and
has finite multiplicity. Let us choose $\delta>0$ such that $(\lambda^\ast-\delta,
\lambda^\ast+ \delta)\setminus\{\lambda^\ast\}$ has no intersection with $\{\lambda_k\}^\infty_{k=0}$,
where $\lambda_0=0$. Since $\mathfrak{L}_\lambda:=\mathfrak{L}-\lambda\widehat{\mathfrak{L}}$
 satisfies Hypothesis~\ref{hyp:1.1} with $X=H$, it has finite Morse index $\mu_\lambda$ and nullity $\nu_\lambda$ at $0$.
Moreover, since each $H_k$ is an invariant subspace of $\mathfrak{L}''(0)-\lambda\widehat{\mathfrak{L}}''(0)=I-\lambda L$, and
$H=\oplus^\infty_{k=0}H_k$ is an orthogonal decomposition, we have
$\mu_\lambda=\sum^\infty_{k=0}\mu_{\lambda,k}$,
where $\mu_{\lambda,k}$ is the dimension of
the maximal negative definite space of the quadratic functional
$H_k\ni u\mapsto f_{\lambda,k}(u):=(\mathfrak{L}''(0)u-\lambda\widehat{\mathfrak{L}}''(0)u,u)_H$.
Note that $f_{\lambda,k}(h)=(1-\lambda/\lambda_k)(h,h)_H,\;\forall h\in H_k$.
We deduce that $\mu_{\lambda,k}$ is equal to
$\dim H_k$ (resp. $0$) if $\lambda>\lambda_k$ (resp. $\lambda<\lambda_k$).
Hence  the Morse index of $\mathfrak{L}_\lambda$ at $0$,
 \begin{eqnarray}\label{e:Bi.2.13sec3}
\mu_\lambda=\sum_{\lambda_k<\lambda}\dim H_k.
\end{eqnarray}
This implies  that
$$
\mu_\lambda=\left\{\begin{array}{ll}
\mu_{\lambda^\ast},&\quad\forall \lambda\in (\lambda^\ast-\delta, \lambda^\ast),\\
\mu_{\lambda^\ast}+ \nu_{\lambda^\ast}, &\quad
\forall\lambda\in (\lambda^\ast, \lambda^\ast+\delta),
\end{array}
\right.
$$
Theorem~\ref{th:A.1} with $\lambda=0$ (or \cite[Theorem~2.1]{Lu7}) we obtain that
$$
C_q(\mathcal{L}_\lambda, 0;{\bf K})=C_q(\mathfrak{L}_\lambda, 0;{\bf K})=\left\{\begin{array}{ll}
\delta_{q\mu_{\lambda^\ast}}{\bf K},&\quad\forall \lambda\in (\lambda^\ast-\delta, \lambda^\ast),\\
\delta_{q(\mu_{\lambda^\ast}+ \nu_{\lambda^\ast})}{\bf K}, &\quad \forall\lambda\in (\lambda^\ast, \lambda^\ast+\delta).
\end{array}
\right.
$$
On the other hand, as in the proof of Theorem~\ref{th:Bi.1.1},
suppose that $(\lambda^\ast,0)\in \mathbb{R}\times U$ is not a bifurcation point of the equation
(\ref{e:Bi.1.4}). Then $C_q(\mathcal{L}_\lambda, 0;{\bf K})$ is independent of
$\lambda\in (\lambda^\ast-\delta,\lambda^\ast+\delta)$ by shrinking $\delta>0$ (if necessary),
which leads to a contradiction.
\end{proof}

 It is easily seen that the functional $\mathcal{L}$ in
\cite[\S4.3, Theorem~4.3]{Skr} or  in \cite[Chap.1, Theorem~3.4]{Skr2}
satisfies the conditions of Theorem~\ref{th:Bi.3}. Hence the latter is a generalization of
\cite[Chap.1, Theorem~3.4]{Skr2}. A stronger generalization will be given in Corollary~\ref{cor:Bi.2.4.2}.

Since  $\mathcal{L}''(0)(X)\subset X$  implies $J(X)\subset X$,
slightly modifying the proof of Theorem~\ref{th:Bi.3} we can obtain:

\begin{theorem}\label{th:Bi.4}
$\mathcal{L}\in C^1(U,\mathbb{R})$ satisfy
 Hypothesis~\ref{hyp:1.3} without (C) and (D1), $\widehat{\mathcal{L}}\in C^1(U,\mathbb{R})$ satisfy  Hypothesis~\ref{hyp:1.4}
 with corresponding operators $\widehat{A}$ and $\widehat{B}=\widehat{P}+\widehat{Q}$.
Suppose that $B(0)$ is positive definite  and that
$$
\{u\in H\,|\, B(0)u-\lambda\widehat{B}(0)u=\mu u,\,\mu\le 0\}\cup
  \{u\in H\,|\, B(0)u-\lambda\widehat{B}(0)u\in X\}\subset X
  $$
   for each $\lambda$.
Then $(\lambda^\ast,0)\in \mathbb{R}\times U^X$ is  a bifurcation point of the equation
(\ref{e:Bi.1.4.1}) if and only if  $\lambda^\ast\in\mathbb{R}$ is an  eigenvalue of
  \begin{eqnarray}\label{e:Bi.1.8}
  B(0)u-\lambda\widehat{B}(0)u=0,\; u\in H.
  \end{eqnarray}
\end{theorem}

If $\lambda^\ast\in\mathbb{R}$ is an  eigenvalue of (\ref{e:Bi.1.8}) of odd multiplicity,
 then it follows from Theorem~\ref{th:C.2} (\cite[Theorem~6.3]{Stu14A}) that there is
 continuous bifurcation at $\lambda^\ast$;
see Comparison~\ref{comp:C.3} in Section~\ref{app:C}.

Note that the proofs of Theorems~\ref{th:Bi.1.1},~\ref{th:Bif.1.1}, ~\ref{th:Bi.3} and \ref{th:Bi.4}
do not need the parameterized versions of the splitting lemmas and Morse-Palais lemmas.
In next three sections we shall use them to get stronger results.

In order to compare our method with previous those  let us consider  the case that
 the classical splitting lemma and Morse-Palais lemma for $C^2$ functionals can be used.
Indeed, it is easily seen that our above proof methods of Theorem~\ref{th:Bi.1.1}
 can lead to:

\begin{theorem}\label{th:Bi.6}
Let $H$, $I$ and $U$ be as in Theorem~\ref{th:Bi.1.1},
 and let $\mathcal{F}:I\times U\to\mathbb{R}$ satisfy the following conditions:
 \begin{enumerate}
 \item[\rm (1)] Each $\mathcal{F}_\lambda\in C^2(U,\mathbb{R})$, $\mathcal{F}'_\lambda(0)=0$ and $\mathcal{F}_0^{\prime\prime}(0)$ is a Fredholm operator.
 \item[\rm (2)] The assumptions (v)-(viii) in Theorem~\ref{th:Bi.1.1} hold with $B_\lambda(0)=\mathcal{F}_\lambda^{\prime\prime}(0)$.
 \item[\rm (3)] Either the assumptions (a)-(b) in Theorem~\ref{th:Bi.1.1} hold, or (i) in Theorem~\ref{th:Bi.1.1} holds and
 $\{\mathcal{F}_\lambda\}_{|\lambda|<\epsilon}$ satisfies the uniform (PS) condition on some closed neighborhood of $0\in H$ for
 small $\epsilon>0$.
 \end{enumerate}
  Then $(0,0)\in I\times U$ is  a bifurcation point of the equation (\ref{e:Intro.1}).
\end{theorem}

By (vi) in Theorem~\ref{th:Bi.1.1} with $B_\lambda(0)=\mathcal{F}_\lambda^{\prime\prime}(0)$,
 $\mathcal{F}_\lambda^{\prime\prime}(0)$ is  Fredholm  for each $\lambda$ near $0\in\mathbb{R}$.

\begin{corollary}\label{cor:Bi.7}
Let $U$ be an open neighborhood of $0$ in a real Hilbert space $H$,
and let $\mathcal{L}, \widehat{\mathcal{L}}\in C^2(U,\mathbb{R})$ satisfy
$\mathcal{L}'(0)=0$ and  $\widehat{\mathcal{L}}'(0)=0$. Suppose that
$\lambda^\ast\in\mathbb{R}$ is an isolated eigenvalue of (\ref{e:Bi.1.5})
 of finite multiplicity, that the following  conditions are satisfied:
 \begin{enumerate}
 \item[\rm (1)] For each $\lambda$ near $\lambda^\ast$,
 $\mathcal{L}-\lambda\widehat{\mathcal{L}}$ satisfies the (PS) condition in a fixed closed neighborhood of $0$.
  \item[\rm (2)] $\lim_{\lambda\to 0+}r(\mathcal{L}''(0)-(\lambda+\lambda^\ast)\widehat{\mathcal{L}}''(0))\ne
 \lim_{\lambda\to 0-}r(\mathcal{L}''(0)-(\lambda+\lambda^\ast)\widehat{\mathcal{L}}''(0))$.
 \end{enumerate}
  Then $(0,0)\in I\times U$ is  a bifurcation point of the equation (\ref{e:Bi.1.4}).
   Moreover, the condition (2)  can be replaced by one of form (I.d') in Corollary~\ref{cor:Bi.3}.
 \end{corollary}

Since $\mathcal{L}''(0)-\lambda^\ast\widehat{\mathcal{L}}''(0)$ is a Fredholm operator, so is
$\mathcal{L}''(0)-\lambda\widehat{\mathcal{L}}''(0)$ for each real $\lambda$ near $\lambda^\ast$.
Compare Corollary~\ref{cor:Bi.7} with Corollaries~\ref{cor:Bi.2.6},\ref{cor:Bi.2.7}.

By the proof of Theorem~\ref{th:Bi.3} we easily get

\begin{theorem}\label{th:Bi.7.1}
Let $U$ be an open neighborhood of $0$ in a real Hilbert space $H$,
and let $\mathcal{L}, \widehat{\mathcal{L}}\in C^2(U,\mathbb{R})$ satisfy
$\mathcal{L}'(0)=0$ and  $\widehat{\mathcal{L}}'(0)=0$.
Suppose that  $\mathcal{L}''(0)$ is positive definite and  $\widehat{\mathcal{L}}''(0)$ is compact.
Then $(\lambda^\ast,0)\in \mathbb{R}\times U$ is  a bifurcation point of the equation
(\ref{e:Bi.1.4}) if and only if  $\lambda^\ast\in\mathbb{R}$ is an  eigenvalue of
(\ref{e:Bi.1.5}).
 \end{theorem}

Since $\mathcal{L}''(0)$ is positive definite we have $\mu_0=0$ and thus each $\mu_\lambda$
(the Morse index of $\mathcal{L}-\lambda\widehat{\mathcal{L}}$  at $0\in H$) is finite by
(\ref{e:Bi.2.13sec3}).

\subsection{Comparison with previous work}\label{app:C}\setcounter{equation}{0}

\subsubsection*{Comparison with  work in \cite{ChowLa, Kie}}
Based on the center manifold theory  Chow and Lauterbach  \cite{ChowLa} proved: if $\mathcal{F}\in C^2(I\times U,\mathbb{R})$
 satisfies the following four conditions:
\begin{enumerate}
\item[1)] $\mathcal{F}'_\lambda(0)=0\;\forall\lambda\in I$,
\item[2)] $0<\dim{\rm Ker}(\mathcal{F}''_0(0))<\infty$,
\item[3)]  $0$ is isolated in $\sigma(\mathcal{F}''_0(0))$,
\item[4)]  $r^+_{\mathcal{F}''_\lambda(0)}\ne
r^-_{\mathcal{F}''_\lambda(0)}$,
\end{enumerate}
then $(0,0)\in I\times U$ must be a bifurcation point of (\ref{e:Intro.1}).

This result was generalized by Kielh\"ofer \cite{Kie}
with the Lyapunov-Schmidt reduction and Conley's theorem on bifurcation of invariant sets.
The main result of \cite{Kie} (see also \cite[Theorem~II.7.3]{Kie1})
can be stated as follows (in our notations):
\textsl{Let $X\subset H$ be a continuously embedded subspace having norm $\|\cdot\|_X$,
$G:I\times X\to H$  continuous, $G(\lambda,0)=0\;\forall\lambda\in I$. Suppose
\begin{enumerate}
\item[\rm (A)] $G$ has a continuous
Fr\'echet derivative with respect to $u$ in a neighborhood of $(0,0)\in I\times X$, $G'_u(\lambda,0)=D_\lambda$,
and $D_0:X\to H$ is a Fredholm operator of index zero having an isolated eigenvalue $0$;
\item[\rm (B)] there is a differentiable potential $\mathcal{F}:I\times X\to\mathbb{R}$
 such that $\mathcal{F}'_u(\lambda,u)[h]= (G(\lambda,u),  h)$  for
all $h\in X$ and for all $(\lambda,u)$ in a neighborhood of $(0,0)\in I\times D$;
\item[\rm (C)]  the crossing number of $D_\lambda$ at $\lambda=0$
(defined in \cite[Definition II.7.1]{Kie1}) is nonzero, (which is equivalent to the condition that
$r^+_{B_\lambda}\ne r^-_{B_\lambda}$ if $D_\lambda=B_\lambda$ as in (\ref{e:Bi.1.1})).
\end{enumerate}
Then  $(0,0)\in I\times X$ is a bifurcation point for the equation $G(\lambda,u)=0$.}

 The major difference between these two results and our Theorems~\ref{th:Bi.1.1},\ref{th:Bif.1.1}
 is that they require the higher smoothness for $G$ or $\mathcal{F}$, while we
require some kinds of (PS) conditions.
In particular, comparing Theorem~\ref{th:Bi.6} with the result by Chow and Lauterbach  \cite{ChowLa}
stated above  it is easily seen that my methods and theirs have both advantages and disadvantages. 
In applications to quasilinear elliptic systems it is impossible to require higher smoothness for potential functionals.

\subsubsection*{Comparison with  work in \cite{EvSt07, Stu14A, Stu14B}}
For convenience we state the main result and related notions of \cite{Stu14A} in closer notation
in this paper.

Let $X$ and $Y$ be real Banach spaces, and let
$I\subset\mathbb{R}$ and $U$ be an open interval and an open neighborhood of the origin of $X$, respectively.
The \textsf{uniform partial Lipschitz modulus} of a function $G:I\times U\to Y$ with respect to $X$ at $(\mu,0)\in I\times U$
is defined by
\begin{equation}\label{e:parLips}
L_X(G,\mu)=\lim_{\delta\to 0}\sup_{\begin{array}{ll}
(\lambda,u),(\lambda,v)\in B((\mu,0),\delta)\\
u\ne v
\end{array}}\frac{\|G(\lambda,u)-G(\lambda,v)\|}{\|u-v\|},
\end{equation}
where $B((\mu,0),\delta)$ is the open ball of radius $\delta$ in $I\times U$ centered at $(\mu,0)$.
For a point $(\mu,0)\in I\times U$ and $\delta>0$ such that $B((\mu,0),\delta)\subset I\times U$ there exists a similar quantity defined by
\begin{equation}\label{e:Delta}
\Delta_\delta(G,\mu)=\sup_{\begin{array}{ll}
(\lambda,u)\in B((\mu,0),\delta)\\
u\ne 0
\end{array}}\frac{\|G(\lambda,u)-G(\mu, u)\|}{\|u\|}.
\end{equation}
Let $\Phi_0(X,Y)$ (resp. ${\rm Iso}(X,Y)$) denote the set of
linear and bounded Fredholm operators of index (resp. Banach space isomorphisms) from $X$ to $Y$, and let
$F(X,Y)=\{L\in\mathscr{L}(X, Y)\,|\,\dim R(L)<\infty\}$. All of them are equipped with  the metric inherited from $\mathscr{L}(X, Y)$
without special statements.
Each $L\in\Phi_0(X,Y)$ has  the well-defined  \textsf{essential conditioning number} $\gamma(L)$ (\cite[page 1042]{Stu14A}).

A \textsf{path of Fredholm operators of index zero} is a continuous function $L$ from an interval $J$
to $\Phi_0(X, Y )$. Call $\lambda_0\in J$  an \textsf{isolated singular point} of $L$ if there
exists $\delta>0$  such that $J_{\lambda_0,\delta}:=(\lambda_0-\delta, \lambda_0+\delta)\subset J$
and $L(\lambda)\in{\rm Iso}(X,Y)$  for $0<|\lambda-\lambda_0|<\delta$.
There always exists $S\in C(J_{\lambda_0,\delta}, {\rm Iso}(Y,X))$ such that $S(t)L(t)=I_X-K(t)\;\forall t\in J_{\lambda_0,\delta}$,
where each $K(t):X\to X$ is a compact linear operator.
For $\lambda\in J_{\lambda_0,\delta}\setminus\{\lambda_0\}$ let ${\rm deg}_{\rm L.S}(I_X-K(\lambda))$ denote
the Leray-Schauder degree of the restriction of $I_X-K(\lambda)$ to a neighborhood of the origin.
The quantity
\begin{equation}\label{e:parity}
\sigma(L,\lambda_0):={\rm deg}_{\rm L.S}(I_X-K(\lambda_l)){\rm deg}_{\rm L.S}(I_X-K(\lambda_r))\in\{-1,1\}
\end{equation}
where $\lambda_0-\delta<\lambda_l<\lambda_0<\lambda_r<\lambda_0+\delta$,  does not depend on the choices
of $S$ and $\lambda_l<\lambda_0<\lambda_r$, and is called the \textsf{local parity} of the path $L$
at an isolated singularity $\lambda_0$. It is invariant under homotopies and reparametrizations.

\begin{lemma}[\hbox{\cite[CRITERION 1, page 1048]{Stu14A}}]\label{lem:C.1}
 Let $K\in\mathscr{L}(X)$ be compact, $L(\lambda)\\=K-\lambda I_X$ and $\lambda_0\in \sigma(K)\setminus\{0\}$. Then
  $\lambda_0\in J$  is an isolated singular point of $L$
and $\sigma(L,\lambda_0)=(-1)^n$, where $n$ is the algebraic multiplicity of $\lambda_0$ as an eigenvalues of $K$.
\end{lemma}

\begin{theorem}[\hbox{\cite[Theorem~5.12]{FitPe91}}]\label{th:C.1}
Let $X$ and $Y$ be real Banach spaces, $X\subset Y$ and $X\hookrightarrow Y$ be continuous.
Suppose that $\lambda_0\in J$ is an isolated singular point of a continuous path $L:J\to \Phi_0(X, Y )$
and that $0$ is isolated in the real spectrum of $L(\lambda_0)$.
Then for any sufficiently small $\varepsilon>0$ there exists a $\delta>0$ such that if $0<|\lambda-\lambda_0|\le\delta$,
then $L(\lambda)$ has only a finite number of eigenvalues in $(-\varepsilon,0)$ each of which is
finite algebraic multiplicity. If, for $0<|\lambda-\lambda_0|\le\delta$, $n(\lambda)$ denotes the sum
of the algebraic multiplicities of the eigenvalues of $L(\lambda)$ in $(-\varepsilon,0)$, then $n(\lambda)$
is constant on $[\lambda_0-\delta, \lambda_0)$ and on $(\lambda_0, \lambda_0+\delta]$, and moreover
$$
\sigma(L,\lambda_0)=(-1)^{n(\lambda_0+\delta)+ n(\lambda_0-\delta)}.
$$
\end{theorem}

We mainly compare our results with the following.

\begin{theorem}[\hbox{\cite[Theorem~6.3]{Stu14A}}]\label{th:C.2}
Let $X$ and $Y$ be real Banach spaces, and let
$I\subset\mathbb{R}$ and $U$ be an open interval and an open neighborhood
of the origin of $X$, respectively.
A map $F:I\times U\to Y$ satisfies the following properties.
\begin{enumerate}
\item[\rm (B1)] $F\in C(I\times U,Y)$ and $F(\lambda,0)=0$ for all $(\lambda,0)\in I\times U$.
\item[\rm (B2)] $F(\lambda,\cdot):X\to Y$ is Hadamard differentiable at $0$ for all $(\lambda,0)\in I\times U$ and
$\lambda\mapsto D_xF(\lambda,0)\in\mathscr{L}(X,Y)$ is continuous.
\item[\rm (B3)] $(\lambda_0,0)\in I\times U$ and $D_xF(\lambda_0,0)\in\Phi_0(X,Y)$.
\item[\rm (B4)] $L_X(R,\lambda_0)<\infty$, where $R(\lambda,x)=F(\lambda,x)-D_xF(\lambda,0)x$ for $(\lambda,x), (\lambda,0)\in I\times U$, and
 \begin{equation}\label{e:C.6}
 \gamma(D_xF(\lambda_0,0)L_X(R,\lambda_0)<1.
 \end{equation}
\end{enumerate}
 \begin{enumerate}
\item[\rm (i)] If $\lambda_0$ is an isolated singular point of $D_xF(\cdot,0)$  with $\sigma(D_xF(\cdot,0),\lambda_0)=-1$,
then $\lambda_0$ is a bifurcation point for $F(\lambda,u)=0$. If, in addition,
 \begin{equation}\label{e:C.7}
\lim_{\delta\to 0}\Delta_\delta(F,\lambda)=0\quad\hbox{for all $\lambda$ in an open neighbourhood of $\lambda_0$},
\end{equation}
there is \textsf{\rm continuous bifurcation} at $\lambda_0$.
\item[\rm (ii)] If ${\rm ker}D_xF(\lambda_0,0)=\{0\}$ and $\lim_{\delta\to 0}\Delta_\delta(F,\lambda_0)=0$, then $\lambda_0$ is not a bifurcation
point for $F(\lambda,u)=0$.
\end{enumerate}
\end{theorem}

We also need  the following lemma.

\begin{lemma}\label{lem:C.2}
Let $H$ be a Hilbert space with inner product $(\cdot,\cdot)_H$
and the induced norm $\|\cdot\|$, and let $X$ be a Banach space with
norm $\|\cdot\|_X$, such that $X\subset H$ is dense in $H$ and $\|x\|\le\|x\|_X\;\forall x\in X$.
Let $\mathfrak{B}\in\mathscr{L}_s(H)$ be a Fredholm operator of index zero.
Suppose that $\hat{\mathfrak{B}}:=\mathfrak{B}|_X\in\mathscr{L}(X)$ and that
  $\{u\in H\,|\, \mathfrak{B}u\in X\}\subset X$. Then
   \begin{enumerate}
\item[\rm i)] $\hat{\mathfrak{B}}$ is also a Fredholm operator of index zero;
 \item[\rm ii)] $0$ sits in the resolvent set $\rho(\mathfrak{B})$ of $\mathfrak{B}$ if and only if $0$ belongs to $\rho(\hat{\mathfrak{B}})$.
\end{enumerate}
  Moreover, if $\{u\in H\,|\, \mathfrak{B}u=\mu u,\,\mu<0\}\subset X$ then $\mu<0$
  is an eigenvalue of $\mathfrak{B}$ if and only if it is that of $\hat{\mathfrak{B}}$ with same (algebraic and geometry) multiplicity.
\end{lemma}

\begin{proof}
Let $N={\rm Ker}(\mathfrak{B})$. Then $N\subset X$ and $N={\rm Ker}(\hat{\mathfrak{B}})$ has finite dimension.
Since $\mathfrak{B}$ is self-adjoint, the range $R(\mathfrak{B})$ of $\mathfrak{B}$ is closed and
is equal to $N^\bot$. We claim that $R(\hat{\mathfrak{B}})$ is closed in $X$. In fact,
let $(x_k)\subset R(\hat{\mathfrak{B}})$ converge to $\bar{x}$ in $X$.
Using $H=R(\mathfrak{B})\oplus N$ we can decompose $\bar{x}$ into $y+z$, where $y\in R(\mathfrak{B})$ and $z\in N$.
Take $u\in \mathfrak{B}^{-1}(y)$. Then $\mathfrak{B}(u)=y=\bar{x}-z\in X$.
By the assumption we get that $u\in X$ and so $y\in R(\hat{\mathfrak{B}})$. Moreover,
since $(x_k, z)_H=0\;\forall k$, letting $k\to\infty$ we get $0=(\bar{x},z)_H=(y, z)_H+ (z, z)_H=(z,z)_H$.
 That is to say, $z=0$ and hence $\bar{x}\in X$. The above arguments also show that
 $X$ has a direct sum decomposition of vector spaces $X=R(\hat{\mathfrak{B}})\dot{+}N$.
Thus $\dim X/R(\hat{\mathfrak{B}})=\dim N<\infty$. i) is proved.
Sine the assumption implies ${\rm Ker}(\mathfrak{B})={\rm Ker}(\hat{\mathfrak{B}})$,
the open mapping theorem leads to ii).

Clearly, it suffices to prove that each negative eigenvalue $\mu$ of $\mathfrak{B}$
is that of $\hat{\mathfrak{B}}$. Let $v\in H\setminus\{0\}$ be such that
$\mu u=\mathfrak{B}u$. By the second assumption we see that $u\in X$ and so an eigenvector
of $\hat{\mathfrak{B}}$ belonging to $\mu$.

Note that each  eigenvalue $\lambda$ of a self-adjoint operator on a Hilbert space has the same algebraic and geometry multiplicities, that is,
${\rm Ker}((\lambda I-\mathfrak{B})^n)={\rm Ker}(\lambda I-\mathfrak{B})$ for any $n\in\N$.
For a negative eigenvalue $\mu$ of $\mathfrak{B}$, the second assumption we deduce that
$$
{\rm Ker}(\lambda I-\mathfrak{B})={\rm Ker}(\lambda I_X-\hat{\mathfrak{B}})\subset
{\rm Ker}((\lambda I_X-\hat{\mathfrak{B}})^n)\subset {\rm Ker}((\lambda I-\mathfrak{B})^n)
$$
and hence ${\rm Ker}(\lambda I_X-\hat{\mathfrak{B}})\subset
{\rm Ker}((\lambda I_X-\hat{\mathfrak{B}})^n)$ for any $n\in\N$.
\end{proof}

Having the above preparations we can make comparisons with
work of \cite{EvSt07, Stu14A, Stu14B}, precisely, Theorem~\ref{th:C.2} and \cite[Theorems~4.2, 5.1]{EvSt07} (or \cite[Theorems~3.2,3.3]{Stu14B}).

Clearly, for sufficiencies of bifurcations
we only need to compare Theorem~\ref{th:C.2}(i) and \cite[Theorem~5.1]{EvSt07} (or \cite[Theorem~3.3]{Stu14B}) with results in Section~\ref{sec:B.2.1S}
because those in Sections~\ref{sec:B.2},\ref{sec:B.3} and \ref{sec:BBH} are either of Rabinowitz type or Fadell--Rabinowitz type.

Theorem~\ref{th:Bi.1.1}, Corollary~\ref{cor:Bi.3} and Theorem~\ref{th:Bi.3} only involve one or two of Hypothesises~\ref{hyp:1.1}, \ref{hyp:1.2}.
Corresponding potential operators  are only G\^ateaux differentiable, and they have no
comparisons with Theorem~\ref{th:C.2}(i) and \cite[Theorem~5.1]{EvSt07} in which
 related operators are required to be either Hadamard differentiable or $w$-Hadamard differentiable at $u=0$ (see \cite{EvSt07,Stu14B} for precise definitions).

Theorems~\ref{th:Bif.1.1}, \ref{th:Bi.4} and Corollary~\ref{cor:Bi.3.1}
have no any relations with \cite[Theorem~5.1]{EvSt07}
(since the latter is stated  in Hilbert spaces).

Of course,  \cite[Theorem~5.1]{EvSt07} for even functionals which are bounded below cannot lead to
Theorems~\ref{th:Bi.6}, \ref{th:Bi.7.1} yet.

Hence it remains to make comparisons between Theorem~\ref{th:C.2}(i) and any one of  Theorem~\ref{th:Bif.1.1},
Corollary~\ref{cor:Bi.3.1} and Theorems~\ref{th:Bi.4},~\ref{th:Bi.6},~\ref{th:Bi.7.1}.

\begin{comparison}[\hbox{\it Comparison with Theorem~\ref{th:Bif.1.1}}]\label{comp:C.1}
{\rm Under the assumptions of Theorem~\ref{th:Bif.1.1} let us define $F:I\times U^X\to X$ by $F(\lambda,x)=A_\lambda(x)$,
where $U^X$ is explained as in  Hypothesis~\ref{hyp:1.3}.
Our assumptions on $\mathcal{F}:I\times U^X\to\mathbb{R}$ cannot give the continuity of $F$.
But (B1) of Theorem~\ref{th:C.2} requires $F$ to be continuous.
Our condition ``each  $\mathcal{F}_\lambda:=\mathcal{F}(\lambda,\cdot)$ satisfies  Hypothesis~\ref{hyp:1.3}"
implies that $F(\lambda,\cdot)$ has the strict Fr\'{e}chet derivative at $0\in X$, $D_xF(\lambda,0)=B_\lambda(0)|_X\in \mathscr{L}(X)$.
But the condition (vi) in Theorem~\ref{th:Bi.1.1} only says that $I\ni\lambda\mapsto B_\lambda(0)\in \mathscr{L}_s(H)$
is continuous at $\lambda=0$, and has no any conclusions for the continuity of the path
$I\ni\lambda\mapsto B_\lambda(0)|_X\in \mathscr{L}(X)$. Hence Theorem~\ref{th:Bif.1.1} cannot be included in
Theorem~\ref{th:C.2}(i) (\cite[Theorem~6.3(i)]{Stu14A}). }
\end{comparison}

\begin{comparison}[\hbox{\it Comparison with Corollary~\ref{cor:Bi.3.1}}]\label{comp:C.2}
{\rm
 Let the assumptions of Corollary~\ref{cor:Bi.3.1} be satisfied.
Define $F:I\times U^X\to X$ by $F(\lambda,x)=A(x)-\lambda\widehat{A}(x)$.
By Hypothesis~\ref{hyp:1.3} this is continuous, and each
$F(\lambda,\cdot)$ has the strict Fr\'{e}chet derivative at $0\in X$, $D_xF(\lambda,0)=B(0)|_X- \lambda\widehat{B}(0)|_X   \in \mathscr{L}(X)$.
Hence $F$ satisfies the assumptions (B1) and (B2) in Theorem~\ref{th:C.2}. By the assumptions,
$B(0)- \lambda\widehat{B}(0)=P(0)+ (Q(0)- \lambda\widehat{B}(0))\in\mathscr{L}_s(H)$, $P(0)$ is positive definite,
and $Q(0)- \lambda\widehat{B}(0)$ is compact. Thus
 $B(0)- \lambda\widehat{B}(0):H\to H$ is a
Fredholm operator of index zero. Using (II.a) and Lemma~\ref{lem:C.2} we deduce that each
$D_xF(\lambda,0)=B(0)|_X- \lambda\widehat{B}(0)|_X   \in \mathscr{L}(X)$ is also
a Fredholm operator of index zero. This means that the assumption (B3) in Theorem~\ref{th:C.2}
holds with $F$ and $\lambda_0=\lambda\in I$. For any fixed $\lambda_0\in I$ we claim that
$L_X(R,\lambda_0)=0$, where for  $(\lambda,x)\in I\times U$,
$$
R(\lambda,x)=F(\lambda,x)-D_xF(\lambda,0)x=A(x)-(B(0)|_X)x-\lambda(\widehat{A}(x)-(\widehat{B}(0)|_X)x).
$$
In fact, since $A(0)=0=\widehat{A}(0)$, and $A$ and $\widehat{A}$ have strict Fr\'{e}chet derivatives
$B(0)|_X$ and $\widehat{B}(0)|_X$  at $0\in X$,  respectively, for any $\epsilon>0$ there exists a small $\nu>0$ such that
$$
\|A(u)-A(v)-B(0)|_X(u-v)\|_X\le \epsilon\|u-v\|_X\quad\hbox{and}\quad$$ $$ \|\widehat{A}(u)-\widehat{A}(v)-\widehat{B}(0)|_X(u-v)\|_X\le \epsilon\|u-v\|_X
$$
for any $u,v\in U$ satisfying $\|u\|_X<\nu$ and $\|v\|_X<\nu$. It follows that
$$
\|R(\lambda,u)-R(\lambda,v)\|_X\le (1+|\lambda|)\epsilon\|u-v\|_X
$$
and hence $L_X(R,\lambda_0)=0$ by (\ref{e:parLips}). Then (B4)
and (\ref{e:C.6}) in Theorem~\ref{th:C.2} hold with any $\lambda_0\in I$.

Moreover, for  any $\lambda,\mu\in I$ and $u\in U$, if $|\mu-\lambda|+\|u\|_X<\nu$ we have
\begin{eqnarray*}
\|F(\lambda,u)-F(\mu,u)\|_X=|\lambda-\mu|\cdot\|\widehat{A}(u)\|_X\le|\lambda-\mu|(\|\widehat{B}(0)|_X\|+\epsilon)\|u\|_X.
\end{eqnarray*}
By (\ref{e:Delta}) this means
$$
\Delta_\nu(F,\lambda)\le|\lambda-\mu|(\|\widehat{B}(0)|_X\|+\epsilon).
$$
It follows that $\lim_{\nu\to 0}\Delta_\nu(F,\lambda)=0$ for any $\lambda\in I$, that is,
(\ref{e:C.7}) holds.

Define $\alpha:I\to \Phi_0(H, H)$ by $\alpha(\lambda)=B(0)-\lambda\widehat{B}(0)$.
It is a path of Fredholm operators of index zero. By (II.b) and (II.c) of Corollary~\ref{cor:Bi.3.1}
we see that $\lambda^\ast$ is  an isolated singular point of $\alpha$.
Using Hypothesis~\ref{hyp:1.3} we have $\alpha^X(\lambda):=B(0)|_X-\lambda\widehat{B}(0)|_X\in\mathscr{L}(X)$.
It follows from Lemma~\ref{lem:C.2} and (II.a) of Corollary~\ref{cor:Bi.3.1} that $\alpha^X:I\to \Phi_0(X, X)$
is also a path of Fredholm operators of index zero and has an isolated singular point $\lambda^\ast$ and that
$\alpha(\lambda)$ and $\alpha^X(\lambda)$ have same negative eigenvalues and
algebraic multiplicities. Let 
\begin{eqnarray*}
&&r_+=\lim_{\lambda\to 0+}r(B(0)-(\lambda+\lambda^\ast)\widehat{B}(0))\quad\hbox{and}\\
&&r_-=\lim_{\lambda\to 0-}r(B(0)-(\lambda+\lambda^\ast)\widehat{B}(0)).
\end{eqnarray*}
By the definitions of $r(B_\lambda)$ above (\ref{e:Bi.1.1}) and $n(\lambda)$ in Theorem~\ref{th:C.1}
we have $r_+=n(\lambda^\ast+\delta)$ and $r_-=n(\lambda^\ast-\delta)$ for $\delta>0$ small enough.
Then Theorem~\ref{th:C.1} gives
$$
\sigma(\alpha,\lambda^\ast)=\sigma(\alpha^X,\lambda^\ast)=(-1)^{r_++r_-}.
$$
Clearly, the condition  that $\sigma(\alpha^X,\lambda^\ast)=-1$ is stronger than the assumption $r_+\ne r_-$. This shows that Corollary~\ref{cor:Bi.3.1} cannot be derived
from Theorem~\ref{th:C.2}(i).

If $\sigma(\alpha^X,\lambda^\ast)=-1$, i.e.,   $r_+-r_-$ is an odd, then by Theorem~\ref{th:C.2}(i) there is  continuous bifurcation at $\lambda^\ast$.}
\end{comparison}

\begin{comparison}[\hbox{\it Comparison with Theorem~\ref{th:Bi.4}}]\label{comp:C.3}
{\rm Under the assumptions of Theorem~\ref{th:Bi.4}, since $B(0)$ is positive definite  and
$\widehat{B}(0)$ is compact, each eigenvalue of (\ref{e:Bi.1.8}) is isolated and has finite algebraic multiplicity.
By arguments in Comparison~\ref{comp:C.2} we see that the map
$F:\R\times U^X\to X$ defined by $F(\lambda,x)=A(x)-\lambda\widehat{A}(x)$ satisfies
(B1)-(B4) and (\ref{e:C.7}) in Theorem~\ref{th:C.2} with $\lambda_0=\lambda^\ast$.
Let $\alpha$ and $\alpha^X$ be as above. Then we have $\sigma(\alpha,\lambda^\ast)=\sigma(\alpha^X,\lambda^\ast)$.
Moreover, by Lemma~\ref{lem:C.1} we get $\sigma(\alpha,\lambda^\ast)=(-1)^n$,
 where $n$ is the algebraic multiplicity of $\lambda^\ast$ as an eigenvalues of (\ref{e:Bi.1.8}).
In order to apply Theorem~\ref{th:C.2}(i) it is required that $n$ is an odd. But our
condition in Theorem~\ref{th:Bi.4} is that $n\ne 0$, i.e., $\lambda^\ast$ is an eigenvalues of (\ref{e:Bi.1.8}).
As above, if  $\sigma(\alpha,\lambda^\ast)=-1$ then it follows from Theorem~\ref{th:C.2}(i) that
there is  continuous bifurcation at $\lambda^\ast$.}
\end{comparison}

\begin{comparison}[\hbox{\it Comparison with Theorems~\ref{th:Bi.6},\ref{th:Bi.7.1}}]\label{comp:C.4}
{\rm
Theorem~\ref{th:Bi.6} is a special case of Theorem~\ref{th:Bi.1.1}.
Though each potential operator $\nabla\mathcal{F}_\lambda$ is $C^1$,
Theorem~\ref{th:Bi.6}  cannot be derived from Theorem~\ref{th:C.2}(i) yet because
$I\times U\ni (\lambda,u)\mapsto \nabla\mathcal{F}_\lambda\in H$ is not assumed to be continuous.

 Theorem~\ref{th:Bi.7.1} is a special case of  Theorem~\ref{th:Bi.3}.
Though $I\times U\ni (\lambda,u)\mapsto \nabla\mathcal{L}-\lambda\nabla\widehat{\mathcal{L}}\in H$
is $C^1$ and hence satisfies (B1)-(B4) and (\ref{e:C.7}) in Theorem~\ref{th:C.2} with $\lambda_0=\lambda^\ast$,
using the same arguments as those at the end of Comparison~\ref{comp:C.2} or \ref{comp:C.3}
 we see that Theorem~\ref{th:Bi.7.1}  cannot follow from Theorem~\ref{th:C.2}(i) yet.
}
\end{comparison}

For necessary conditions of bifurcations  we only need to compare
Theorem~\ref{th:C.2}(ii) and \cite[Theorem~4.2]{EvSt07} (or \cite[Theorem~3.2]{Stu14B})
with results  in Section~\ref{sec:B.2.1N}.

Theorem~\ref{th:Ka1} and  Corollary~\ref{cor:Bi.2.2} cannot be derived from any one of
Theorem~\ref{th:C.2}(ii) and \cite[Theorem~4.2]{EvSt07} because potential operators in the preceding two results
are only G\^ateaux differentiable and related operators in the latter two results
involve either Hadamard differentiability or $w$-Hadamard differentiability.

\cite[Theorem~4.2]{EvSt07} is stated  in Hilbert spaces and hence cannot include
 Theorem~\ref{th:Ka2} and Corollary~\ref{cor:Bi.2.2*}.
As in Comparison~\ref{comp:C.1}, Theorem~\ref{th:Ka2} cannot follow  from Theorem~\ref{th:C.2}(ii).

Finally, let us compare Theorem~\ref{th:C.2}(ii) with Corollary~\ref{cor:Bi.2.2*} with $n=1$.
By the arguments in Comparison~\ref{comp:C.2} we see that
the map $F:I\times U^X\to X$ defined by $F(\lambda,x)=A(x)-\lambda\widehat{A}_1(x)$
satisfies (B1)-(B2), (B4) and (\ref{e:C.7}) in Theorem~\ref{th:C.2} with $\lambda_0=\lambda^\ast$.
But in the present case there is no the condition ``$\{u\in H\,|\, B(0)u-\lambda\widehat{B}(0)u\in X\}\subset X$ for each $\lambda$ near $\lambda^\ast$''
as in (II.a) of Corollary~\ref{cor:Bi.3.1}. Thus we cannot use  it  and Lemma~\ref{lem:C.2} to deduce that each
$D_xF(\lambda,0)=B(0)|_X- \lambda\widehat{B}(0)|_X   \in \mathscr{L}(X)$ is also
a Fredholm operator of index zero, that is, the assumption (B3) in Theorem~\ref{th:C.2}.
Hence Corollary~\ref{cor:Bi.2.2*} with $n=1$ cannot be derived from Theorem~\ref{th:C.2}(ii).

\section{Some new bifurcation results of  Rabinowitz type}\label{sec:B.2}
\setcounter{equation}{0}

Let $H$ be a real Hilbert space, $U$ an open neighborhood of $0$ in $H$,
and let $f\in C^2(U, \mathbb{R})$ satisfy  $f'(0)=0$.
A classical bifurcation theorem by Rabinowitz \cite{Rab} claimed:
If $\lambda^\ast\in\mathbb{R}$ is an isolated eigenvalue of $f^{\prime\prime}(0)$ of finite multiplicity,
then $(\lambda^\ast,0)$ is a bifurcation point of the equation $\nabla f(u)=\lambda u$ and
at least one of the following alternatives occurs:
\begin{enumerate}
\item[(i)] $(\lambda^\ast, 0)$ is not an isolated solution in $\{\lambda^\ast\}\times U$ of the equation $\nabla f(u)=\lambda u$;
\item[(ii)]  for every $\lambda\in\mathbb{R}$ near $\lambda^\ast$ there is a nontrivial solution
$u_\lambda$ of the equation $\nabla f(u)=\lambda u$ converging to $0$ as $\lambda\to\lambda^\ast$;

\item[(iii)] there is an one-sided  neighborhood $\Lambda$ of $\lambda^\ast$ such that
for any $\lambda\in\Lambda\setminus\{\lambda^\ast\}$, the equation $\nabla f(u)=\lambda u$ has at least
two nontrivial solutions converging to zero as $\lambda\to\lambda^\ast$.
\end{enumerate}
Under the same hypotheses B\"ohme \cite{Boh} and Marino \cite{Mar} independently proved that
for each small $r>0$, the equation $\nabla f(u)=\lambda u$ possesses  at least
two distinct one parameter families of solutions $(\lambda(r), u(r))$ having $\|u(r)\|=r$ and $\lambda(r)\to\lambda^\ast$
as $r\to 0$.  McLeod and Turner \cite{McTu} extended the latter result  to functions
 of class $C^{1,1}$ depending linearly on two parameters and such that the Lipschitz constant of the gradient goes to
zero as the parameters and $u$ go to zero. Ioffe and Schwartzman \cite{IoSch}
generalized  the above Rabinowitz's bifurcation theorem  to equations $Lu+ \nabla\varphi_\lambda(u)=\lambda u$
with $\varphi$ continuous jointly in $(\lambda,u)$ and $\varphi_\lambda$  of class $C^{1,1}$, where $L\in\mathscr{L}_s(H)$ and
the Lipschitz constant of $\nabla\varphi_\lambda$ satisfies some additional conditions.

In this section we shall give a few new bifurcation results as the above Rabinowitz bifurcation theorem.
As said in Introduction, we can reduce some of them  to the following
finite-dimensional result of the above Rabinowitz's type,  which may be obtained as a corollary of \cite[Theorem~2]{IoSch}
(or \cite[Theorem~4.2]{CorH}).

\begin{theorem}[\hbox{\cite[Theorem~5.1]{Can}}]\label{th:Bi.2.1}
 Let $X$ be a finite dimensional normed space, let $\delta>0$, $\epsilon>0$, $\lambda^\ast\in\mathbb{R}$ and
for every $\lambda\in [\lambda^\ast-\delta, \lambda^\ast+\delta]$, let
$f_\lambda:B_X(0,\epsilon)\to\mathbb{R}$ be a function of class $C^1$.
Assume that
\begin{enumerate}
\item[\rm a)] the functions $\{(\lambda,u)\to f_\lambda(u)\}$ and
$\{(\lambda,u)\to f'_\lambda(u)\}$  are continuous on
$[\lambda^\ast-\delta, \lambda^\ast+\delta]\times B_X(0,\epsilon)$;
\item[\rm b)] $u=0$ is a critical point of $f_{\lambda^\ast}$;
\item[\rm c)] $f_\lambda$ has an isolated local minimum (maximum) at zero for every
$\lambda\in (\lambda^\ast,\lambda^\ast+\delta]$ and an isolated local maximum (minimum) at
zero for every $\lambda\in [\lambda^\ast-\delta, \lambda^\ast)$.
\end{enumerate}
Then one at least of the following assertions holds:
\begin{enumerate}
\item[\rm i)] $u=0$ is not an isolated critical point of $f_{\lambda^\ast}$;
\item[\rm ii)] for every $\lambda\ne\lambda^\ast$ in a neighborhood of $\lambda^\ast$ there is a nontrivial critical point of
$f_\lambda$ converging to zero as $\lambda\to\lambda^\ast$;
\item[\rm iii)] there is an one-sided (right or left) neighborhood of $\lambda^\ast$ such that for every
$\lambda\ne\lambda^\ast$ in the neighborhood there are two distinct nontrivial critical points of $f_\lambda$
converging to zero as $\lambda\to\lambda^\ast$.
\end{enumerate}
In particular,  $(\lambda^\ast, 0)\in [\lambda^\ast-\delta, \lambda^\ast+\delta]\times B_X(0,\epsilon)$
is a bifurcation point of $f'_\lambda(u)=0$.
\end{theorem}

Note that the local minimum (maximum) in the assumption c) must  be strict!

The following is the main result in this section.

\begin{theorem}\label{th:Bi.2.4}
Let $\mathcal{L}\in C^1(U,\mathbb{R})$ (resp.  $\widehat{\mathcal{L}}\in C^1(U,\mathbb{R})$) satisfy
 Hypothesis~\ref{hyp:1.1} with $X=H$ (resp.  Hypothesis~\ref{hyp:1.2}),
and let $\lambda^\ast\in\mathbb{R}$ be an isolated eigenvalue of
\begin{eqnarray}\label{e:Bi.2.7.4}
\mathcal{L}''(0)v-\lambda\widehat{\mathcal{L}}''(0)v=0,\quad v\in H.
\end{eqnarray}
(If $\lambda^\ast=0$, it is enough that $\widehat{\mathcal{L}}\in C^1(U,\mathbb{R})$ satisfies Hypothesis~\ref{hyp:1.2}
without requirement that each $\widehat{\mathcal{L}}''(u)\in\mathscr{L}_s(H)$ is compact.)
Suppose that the Morse indexes of $\mathcal{L}_\lambda:=\mathcal{L}-\lambda\widehat{\mathcal{L}}$
at $0\in H$ 
 take, respectively, values $\mu_{\lambda^\ast}$ and $\mu_{\lambda^\ast}+\nu_{\lambda^\ast}$
 as $\lambda\in\mathbb{R}$ varies in
 two deleted half neighborhoods  of $\lambda^\ast$,
where $\mu_{\lambda}$ and $\nu_{\lambda}$ are the Morse index and the nullity of  $\mathcal{L}_{\lambda}$
at $0$, respectively.      Then  one of the following alternatives occurs:
\begin{enumerate}
\item[\rm (i)] $(\lambda^\ast, 0)$ is not an isolated solution in $\{\lambda^\ast\}\times U$ of
\begin{eqnarray}\label{e:Bi.2.7.3}
\mathcal{L}'(u)=\lambda\widehat{\mathcal{L}'}(u);
\end{eqnarray}
 \item[\rm (ii)]  for every $\lambda\in\mathbb{R}$ near $\lambda^\ast$ there is a nontrivial solution
$u_\lambda$ of (\ref{e:Bi.2.7.3}) converging to $0$ as $\lambda\to\lambda^\ast$;

\item[\rm (iii)] there is an one-sided  neighborhood $\Lambda$ of $\lambda^\ast$ such that
for any $\lambda\in\Lambda\setminus\{\lambda^\ast\}$,
(\ref{e:Bi.2.7.3}) has at least two nontrivial solutions converging to
zero as $\lambda\to\lambda^\ast$.
\end{enumerate}
In particular,  $(\lambda^\ast, 0)\in\mathbb{R}\times U$ is a bifurcation point  for the equation
(\ref{e:Bi.2.7.3}).
\end{theorem}

When $X=H$ let us compare the assumptions of Theorem~\ref{th:Bi.2.4} with those of Corollary~\ref{cor:Bi.3}.
The assumption on the Morse indexes of $\mathcal{L}_\lambda$ at $0\in H$ in the former
is stronger than the condition (I.d) in the latter. It means that one of
$\lim_{\lambda\to 0+}r(B(0)-(\lambda+\lambda^\ast)\widehat{\mathcal{L}}''(0))$ and
 $\lim_{\lambda\to 0-}r(B(0)-(\lambda+\lambda^\ast)\widehat{\mathcal{L}}''(0))$
is equal to $0$, and another is equal to $\nu_{\lambda^\ast}$.

Clearly, our assumptions and those by Ioffe and Schwartzman \cite{IoSch} cannot be contained each other.
In Remark~\ref{rm:Bi.2.4.3} we shall compare it with the above Rabinowitz bifurcation theorem.

\begin{proof}[\it {Proof of Theorem~\ref{th:Bi.2.4}}]
Take $\delta>0$ such that $[\lambda^\ast-\delta,\lambda^\ast+\delta]\setminus\{\lambda^\ast\}$
contains no eigenvalues of (\ref{e:Bi.2.7.4}). Then
 $0\in H$ is a nondegenerate (and thus isolated) critical point of $\mathcal{L}_\lambda$
 for each $\lambda\in [\lambda^\ast-\delta,\lambda^\ast+\delta]\setminus\{\lambda^\ast\}$.

\textsf{Firstly, we assume that} the Morse index $\mu_\lambda$ of $\mathcal{L}_\lambda$  satisfies
  \begin{eqnarray}\label{e:Bi.2.7.5}
 \mu_\lambda=\left\{\begin{array}{ll}
 \mu_{\lambda^\ast},&\quad\forall \lambda\in [\lambda^\ast-\delta, \lambda^\ast),\\ \mu_{\lambda^\ast}+\nu_{\lambda^\ast}, &\quad
\forall\lambda\in (\lambda^\ast, \lambda^\ast+\delta].
\end{array}
\right.
 \end{eqnarray}
  It follows from  Theorem~\ref{th:A.1} that  for any $q\in\mathbb{N}_0$,
\begin{eqnarray}\label{e:Bi.2.7.6}
C_q(\mathcal{L}_{\lambda},0;{\bf K})=
\left\{\begin{array}{ll}
\delta_{q\mu_{\lambda^\ast}}{\bf K},&\quad
\forall \lambda\in [\lambda^\ast-\delta, \lambda^\ast),\\
\delta_{q(\mu_{\lambda^\ast}+\nu_{\lambda^\ast})}{\bf K},&\quad
\forall \lambda\in (\lambda^\ast, \lambda^\ast+\delta].
\end{array}\right.
\end{eqnarray}
 Let $H^0_{\lambda^\ast}$ be the eigenspace of (\ref{e:Bi.2.7.4}) associated with $\lambda^\ast$ and $(H^0_{\lambda^\ast})^\bot$
the orthogonal complementary of $H^0_{\lambda^\ast}$ in $H$.
Applying \cite[Theorem~2.16]{Lu7}
to $\mathcal{L}_{\lambda}=\mathcal{L}_{\lambda^\ast}+(\lambda-\lambda^\ast)\widehat{\mathcal{L}}$
with $\lambda\in [\lambda^\ast-\delta, \lambda^\ast+\delta]$,
 we have  $\epsilon>0$ and a unique continuous map
 $$
 \psi:[\lambda^\ast-\delta, \lambda^\ast+\delta]\times (B_H(0,\epsilon)\cap H^0_{\lambda^\ast})\to (H^0_{\lambda^\ast})^\bot
 $$
   such that for each $\lambda\in [\lambda^\ast-\delta, \lambda^\ast+\delta]$,
 $\psi(\lambda, 0)=0$ and
\begin{eqnarray*}
 P^\bot_{\lambda^\ast}\nabla\mathcal{L}(z+ \psi(\lambda, z))-
 \lambda P^\bot_{\lambda^\ast}\nabla\widehat{\mathcal{L}}(z+ \psi(\lambda, z))=0\quad\forall z\in B_{H}(0,\epsilon)\cap H^0_{\lambda^\ast},
 \end{eqnarray*}
 where  $P^\bot_{\lambda^\ast}$ is the orthogonal projection onto $(H^0_{\lambda^\ast})^\bot$; moreover
  the functional
 \begin{equation}\label{e:Bi.2.15}
 \mathcal{L}^\circ_\lambda: B_H(0, \epsilon)\cap H^0_{\lambda^\ast}\to\mathbb{R},\;z\mapsto
 \mathcal{L}(z+\psi(\lambda, z))- \lambda\widehat{\mathcal{L}}(z+\psi(\lambda, z))
  \end{equation}
   is of class $C^1$, and has differential at $z\in B_H(0, \epsilon)\cap H^0_{\lambda^\ast}$ given by
\begin{eqnarray}\label{e:Bi.2.15+}
D\mathcal{L}^\circ_\lambda(z)[h]=D\mathcal{L}(z+\psi(\lambda,z))[h]-
\lambda D\widehat{\mathcal{L}}(z+\psi(\lambda,z))[h],\quad\forall h\in H^0_{\lambda^\ast}.
\end{eqnarray}
Note that for each $\lambda\in[\lambda^\ast-\delta, \lambda^\ast+\delta]$,
the map $z\mapsto z+ \psi({\lambda}, z))$ induces an one-to-one correspondence
 between the critical points of  $\mathcal{L}_{\lambda}^\circ$ near $0\in H^0_{\lambda^\ast}$
and those of $\mathcal{L}_{\lambda}$ near $0\in H$. Hence the problem is reduced to finding
the critical points of $\mathcal{L}^\circ_\lambda$
near $0\in H^0_{\lambda^\ast}$ for $\lambda$ near $\lambda^\ast$.

Now by  Theorem~\ref{th:A.1} and \cite[Theorem~2.18]{Lu7},
 $0\in H^0_{\lambda^\ast}$ is also an
isolated critical point of $\mathcal{L}^\circ_\lambda$, and
 from (\ref{e:Bi.2.7.5}) and (\ref{e:Bi.2.7.6}) we may derive that for any $j\in\mathbb{N}_0$,
\begin{eqnarray}\label{e:Bi.2.16}
C_{j}(\mathcal{L}^\circ_{\lambda},0;{\bf K})=\left\{\begin{array}{ll}
\delta_{j0}{\bf K},&\quad \forall \lambda\in [\lambda^\ast-\delta, \lambda^\ast),\\
\delta_{j\nu_{\lambda^\ast}}{\bf K},&\quad
\forall \lambda\in (\lambda^\ast, \lambda^\ast+\delta].
\end{array}\right.
\end{eqnarray}

For a $C^1$ function $\varphi$ on a neighborhood  of the origin $0\in\mathbb{R}^N$
we may always find  $\tilde\varphi\in C^1(\mathbb{R}^N,
\mathbb{R})$ such that it agrees with $\varphi$ near $0\in\mathbb{R}^N$ and is also
coercive (so satisfies the (PS)-condition).
Suppose that $0$ is an isolated critical point of $\varphi$. By Corollary~6.96, Proposition~6.97 and Example~6.45 in \cite{MoMoPa},
we obtain
\begin{equation}\label{e:Bi.2.17}
\left.
\begin{array}{ll}
&C_k(\varphi,0;{\bf K})=\delta_{k0}\;\Longleftrightarrow\;\hbox{ $0$ is a strict local minimizer of $\varphi$},\\
&C_k(\varphi,0;{\bf K})=\delta_{kN}\;\Longleftrightarrow\;\hbox{ $0$ is a strict local maximizer of $\varphi$}
\end{array}\right\}
\end{equation}
and thus $C_0(\varphi, 0;{\bf K})=0=C_N(\varphi, 0;{\bf K})$
if $0\in  \mathbb{R}^N$ is neither a local  maximizer nor a local
minimizer of $\varphi$.

Because of these,  (\ref{e:Bi.2.16}) and (\ref{e:Bi.2.17}) lead to
\begin{eqnarray}\label{e:Bi.2.18}
0\in H^0_{\lambda^\ast}\;\hbox{is a strict local}\left\{
\begin{array}{ll}
\hbox{minimizer of}\;\mathcal{L}^\circ_{\lambda},&\quad
\forall \lambda\in [\lambda^\ast-\delta, \lambda^\ast),\\
\hbox{maximizer of}\;\mathcal{L}^\circ_{\lambda},&\quad
\forall \lambda\in (\lambda^\ast, \lambda^\ast+\delta].
\end{array}\right.
\end{eqnarray}
By this and Theorem~\ref{th:Bi.2.1}, one of the following possibilities occurs:
\begin{enumerate}
\item[(1)]  $0\in H^0_{\lambda^\ast}$ is not an isolated critical point of $\mathcal{L}^\circ_{\lambda^\ast}$;
\item[(2)]  for every $\lambda\in\mathbb{R}$ near $\lambda^\ast$,
$\mathcal{L}^\circ_{\lambda}$ has a nontrivial critical point  converging to
$0\in H^0_{\lambda^\ast}$ as $\lambda\to\lambda^\ast$;

\item[(3)] there is an one-sided  neighborhood $\Lambda$ of $\lambda^\ast$ such that
for any $\lambda\in\Lambda\setminus\{\lambda^\ast\}$,
$\mathcal{L}^\circ_{\lambda}$ has two nontrivial critical points
 converging to zero as $\lambda\to\lambda^\ast$.
\end{enumerate}
Obviously, they lead to (i), (ii) and (iii), respectively.

\textsf{Next,  assume}
\begin{eqnarray}\label{e:Bi.2.19-}
\mu_\lambda=\left\{\begin{array}{ll}
 \mu_{\lambda^\ast}+\nu_{\lambda^\ast},&\quad\forall \lambda\in [\lambda^\ast-\delta, \lambda^\ast),\\ \mu_{\lambda^\ast}, &\quad \forall\lambda\in (\lambda^\ast, \lambda^\ast+\delta].
\end{array}
\right.
\end{eqnarray}
Then we can obtain
\begin{eqnarray}\label{e:Bi.2.19}
0\in H^0_{\lambda^\ast}\;\hbox{is a strict local}\left\{
\begin{array}{ll}
\hbox{maximizer of}\;\mathcal{L}^\circ_{\lambda},&\quad
\forall \lambda\in [\lambda^\ast-\delta, \lambda^\ast),\\
\hbox{minimizer of}\;\mathcal{L}^\circ_{\lambda},&\quad
\forall \lambda\in (\lambda^\ast, \lambda^\ast+\delta].
\end{array}\right.
\end{eqnarray}
This also leads to either (i) or (ii) or (iii).\end{proof}

From the proof of Theorem~\ref{th:Bi.2.4} it is easily seen that  Theorem~\ref{th:Bi.2.4}(i)
may be replaced by ``$0\in H^0_{\lambda^\ast}$ is not an isolated critical point of
$\mathcal{L}^\circ_{\lambda^\ast}$". In fact, they are equivalent in the present case.

The following two corollaries strengthen Corollary~\ref{cor:Bi.3}.

\begin{corollary}\label{cor:Bi.2.4.1}
Let $\mathcal{L}\in C^1(U,\mathbb{R})$ (resp.  $\widehat{\mathcal{L}}\in C^1(U,\mathbb{R})$) satisfy
 Hypothesis~\ref{hyp:1.1} with $X=H$ (resp.  Hypothesis~\ref{hyp:1.2}),
and let $\lambda^\ast\in\mathbb{R}$ be an isolated eigenvalue of (\ref{e:Bi.2.7.4}).
(If $\lambda^\ast=0$, it is enough that $\widehat{\mathcal{L}}\in C^1(U,\mathbb{R})$ satisfies Hypothesis~\ref{hyp:1.2}
without requirement that each $\widehat{\mathcal{L}}''(u)\in\mathscr{L}_s(H)$ is compact.)
Suppose that $\widehat{\mathcal{L}}''(0)$ is either semi-positive or semi-negative.
Then the conclusions of Theorem~\ref{th:Bi.2.4} hold true.
\end{corollary}

\begin{proof}
It suffices to prove that the assumption about
the Morse indexes of $\mathcal{L}_\lambda=\mathcal{L}-\lambda\widehat{\mathcal{L}}$
at $0\in H$ in Theorem~\ref{th:Bi.2.4} can be satisfied.
Since $\lambda^\ast\in\mathbb{R}$ is an isolated eigenvalue of (\ref{e:Bi.2.7.4}),
by the proof of Theorem~\ref{th:Bi.2.4} we have $\delta>0$ such that
 $0\in H$ is a nondegenerate critical point of $\mathcal{L}_\lambda$ for each $\lambda\in
 [\lambda^\ast-\delta,\lambda^\ast+\delta]\setminus\{\lambda^\ast\}$. This is equivalent to
 the fact that  $0\in H$ is a nondegenerate critical point of the $C^\infty$ functional
 $H\ni u\mapsto\mathfrak{L}_\lambda(u):=([\mathcal{L}''(0)-\lambda\widehat{\mathcal{L}}''(0)]u,u)_H$
 for each $\lambda\in [\lambda^\ast-\delta,\lambda^\ast+\delta]\setminus\{\lambda^\ast\}$. Hence
 $$
 \mathcal{L}''(0)u-\lambda\widehat{\mathcal{L}}''(0)u\ne 0,\quad\forall u\in H\setminus\{0\},\quad\forall \lambda\in [\lambda^\ast-\delta,\lambda^\ast+\delta]\setminus\{\lambda^\ast\}.
 $$
 This shows that $0\in H$ is a unique critical point of each $\mathfrak{L}_\lambda$.
 Since $ \nabla\mathfrak{L}_\lambda(u)=\mathcal{L}''(0)u-\lambda\widehat{\mathcal{L}}''(0)u=
 P(0)u+ Q(0)u-\lambda\widehat{\mathcal{L}}''(0)u$ is a bounded linear Fredholm operator,
 it is easily seen that $\mathfrak{L}_\lambda$ satisfies the (PS) condition on any closed ball.
  Hence for any small $0<\epsilon<\delta$,
   the families
 $\{\mathfrak{L}_\lambda\,|\,\lambda\in [\lambda^\ast-\delta,\lambda^\ast-\epsilon]\}$ and
 $\{\mathfrak{L}_\lambda\,|\,\lambda\in [\lambda^\ast+\epsilon,\lambda^\ast+\delta]\}$
 satisfy the conditions of Theorem~\ref{th:stablity1} on any closed ball $\bar{B}_H(0, r)$.
 Thus the critical groups $C_\ast(\mathfrak{L}_\lambda, 0;{\bf K})$
 are all isomorphic for all 
 $\lambda\in [\lambda^\ast-\delta,\lambda^\ast)$ (resp.
      $\lambda\in (\lambda^\ast,\lambda^\ast+\delta]$) because of arbitrariness of $\epsilon>0$.
 Note that
$C_q(\mathfrak{L}_\lambda, 0;{\bf K})=\delta_{q\mu_{\lambda}}{\bf K}$
for each $\lambda\in [\lambda^\ast-\delta,\lambda^\ast+\delta]\setminus\{\lambda^\ast\}$.
 It follows that
   \begin{equation}\label{e:Bi.2.19.1}
   \mu_\lambda=\left\{\begin{array}{ll}
 \mu_{\lambda^\ast-\delta},&\quad\forall \lambda\in [\lambda^\ast-\delta, \lambda^\ast),\\
 \mu_{\lambda^\ast+\delta}, &\quad \forall\lambda\in (\lambda^\ast, \lambda^\ast+\delta].
\end{array}
\right.
\end{equation}

 We assume now that $\widehat{\mathcal{L}}''(0)$ is semi-positive. Then
 $$
 \mathcal{L}''(0)-\lambda_1\widehat{\mathcal{L}}''(0)\ge \mathcal{L}''(0)-\lambda_2\widehat{\mathcal{L}}''(0),
\quad\forall\; \lambda^\ast-\delta\le\lambda_1\le\lambda_2\le\lambda^\ast+\delta.
$$
Since all $\mu_{\lambda}$ and $\nu_{\lambda}$ are finite,
from  this, \cite[Proposition~2.3.3]{Ab} and (\ref{e:Bi.2.19.1}) we derive
 \begin{eqnarray*}
&&\mu_{\lambda^\ast-\delta}\le\mu_{\lambda^\ast}\le \mu_{\lambda^\ast+\delta},\\
&&\mu_{\lambda^\ast-\delta}\le\mu_{\lambda^\ast}+\nu_{\lambda^\ast}\le \mu_{\lambda^\ast+\delta},\\
&&\mu_{\lambda^\ast}\le\lim_{\lambda\to\lambda^\ast}\inf\mu_\lambda=\mu_{\lambda^\ast-\delta},\\
&&\mu_{\lambda^\ast}+\nu_{\lambda^\ast}\ge\lim_{\lambda\to\lambda^\ast}\sup\mu_\lambda
=\mu_{\lambda^\ast+\delta}.
\end{eqnarray*}
These imply that (\ref{e:Bi.2.7.5}) is satisfied.

 Similarly, if
$\widehat{\mathcal{L}}''(0)$ is semi-negative, then (\ref{e:Bi.2.19-}) holds true.
 The desired conclusions are obtained.
\end{proof}

 Note that (\ref{e:Bi.2.2}) has no isolated eigenvalues if
$\cap^n_{j=1}{\rm Ker}(\widehat{\mathcal{L}}''_j(0))\cap{\rm Ker}(\mathcal{L}''(0))\ne\{0\}$.
It is natural to ask when  $\vec{\lambda}^\ast$ is an isolated  eigenvalue  of (\ref{e:Bi.2.2}).
We also consider the case $n=1$ merely.
Suppose that  $\mathcal{L}''(0)$ is  invertible.
Then $\mathcal{L}''(0)$ cannot be negative definite because $\mathcal{L}\in C^1(U,\mathbb{R})$ satisfies Hypothesis~\ref{hyp:1.1} with $X=H$.
It follows that $0$ is not an eigenvalue  of (\ref{e:Bi.2.7.4}), and that $\lambda\in\mathbb{R}\setminus\{0\}$
is an eigenvalue  of (\ref{e:Bi.2.7.4}) if and only if $1/\lambda$
is an eigenvalue  of the compact linear self-adjoint operator $L:=[\mathcal{L}''(0) ]^{-1}\widehat{\mathcal{L}}''(0)\in\mathscr{L}_s(H)$.
By Riesz-Schauder theory, the spectrum of $L$, $\sigma(L)$,
consists of zero and a sequence of real eigenvalues of  finite multiplicity
converging to zero, $\{1/\lambda_k\}_{k=1}^\infty$. (Hence $|\lambda_k|\to\infty$ as $k\to\infty$.)
Let $H_k$ be the eigenspace corresponding to $1/\lambda_k$ for $k\in\mathbb{N}$.
Then $H=\oplus^\infty_{k=0}H_k$, where $H_0={\rm Ker}(L)={\rm Ker}(\widehat{\mathcal{L}}''(0))$ and
\begin{equation}\label{e:Bi.2.8}
H_k={\rm Ker}(I/\lambda_k-L)={\rm Ker}(\mathcal{L}''(0)-\lambda_k \widehat{\mathcal{L}}''(0)),\quad
\forall k\in\mathbb{N}.
\end{equation}

\begin{corollary}\label{cor:Bi.2.4.2}
Let $\mathcal{L}\in C^1(U,\mathbb{R})$ (resp.  $\widehat{\mathcal{L}}\in C^1(U,\mathbb{R})$) satisfy
 Hypothesis~\ref{hyp:1.1} with $X=H$ (resp.  Hypothesis~\ref{hyp:1.2}). Suppose that the following two conditions
  satisfied:
 \begin{enumerate}
\item[\rm (a)] $\mathcal{L}''(0)$ is invertible
and $\lambda^\ast=\lambda_{k_0}$ is an eigenvalue of (\ref{e:Bi.2.7.4}).
 \item[\rm (b)] $\mathcal{L}''(0)\widehat{\mathcal{L}}''(0)=\widehat{\mathcal{L}}''(0)\mathcal{L}''(0)$ (so
 each $H_k$  is an invariant subspace of $\mathcal{L}''(0)$),   and $\mathcal{L}''(0)$
 is either positive  or negative on $H_{k_0}$.
 \end{enumerate}
 Then the conclusions of Theorem~\ref{th:Bi.2.4} hold true.
 Moreover, if ${\mathcal{L}}''(0)$ is positive definite, the condition (b) is unnecessary.
\end{corollary}

Clearly, the ``Moreover" part strengthens Theorem~\ref{th:Bi.3} and so
\cite[\S4.3, Theorem~4.3]{Skr} or  in \cite[Chap.1, Theorem~3.4]{Skr2}.

\begin{proof}[\it Proof of Corollary~\ref{cor:Bi.2.4.2}]
We only need to prove that the conditions of Theorem~\ref{th:Bi.2.4} are satisfied.
By the assumptions the eigenvalue  $\lambda^\ast\ne 0$ is isolated and
has finite multiplicity. Let us choose $\delta>0$ such that $(\lambda^\ast-\delta,
\lambda^\ast+ \delta)\setminus\{\lambda^\ast\}$ has no intersection with $\{\lambda_k\}^\infty_{k=0}$,
where $\lambda_0=0$. Since each $H_k$ is an invariant subspace of $\mathcal{L}''(0)-\lambda\widehat{\mathcal{L}}''(0)$ by (b), and
$H=\oplus^\infty_{k=0}H_k$ is an orthogonal decomposition, we have
$\mu_\lambda=\sum^\infty_{k=0}\mu_{\lambda,k}$,
where $\mu_{\lambda,k}$ is the dimension of
the maximal negative definite space of the quadratic functional
$H_k\ni u\mapsto f_{\lambda,k}(u):=(\mathcal{L}''(0)u-\lambda\widehat{\mathcal{L}}''(0)u,u)_H$.
Note that $H_k$ has an orthogonal decomposition $H_k^+\oplus H_k^-$, where
$H_k^+$ (resp. $H_k^-$) is the positive (resp. negative) definite subspace
of $\mathcal{L}''(0)|_{H_k}$,
and $f_{\lambda,k}(h)=(1-\lambda/\lambda_k)(\mathcal{L}''(0)h,h)_H,\;\forall h\in H_k$, we deduce
that $\mu_{\lambda,k}$ is equal to
$\dim H^+_k$ (resp. $\dim H^-_k$) if $\lambda>\lambda_k$ (resp. $\lambda<\lambda_k$).
Hence  the Morse index of $\mathcal{L}_\lambda$ at $0$,
 \begin{eqnarray}\label{e:Bi.2.13}
\mu_\lambda=\sum_{\lambda_k<\lambda}\dim H_k^+ + \sum_{\lambda_n>\lambda}\dim H_n^-.
\end{eqnarray}
(Here we do not use the finiteness of $\mu_\lambda$.) Let $\nu^+_{\lambda^\ast}=\dim H^+_{k_0}$ (resp. $\nu^-_{\lambda^\ast}=\dim H^-_{k_0}$).
Using (\ref{e:Bi.2.13}) it is easy to verify that
$$
\mu_\lambda=\left\{\begin{array}{ll}
\mu_{\lambda^\ast}+ \nu^-_{\lambda^\ast},&\quad\forall \lambda\in (\lambda^\ast-\delta, \lambda^\ast),\\
\mu_{\lambda^\ast}+ \nu^+_{\lambda^\ast}, &\quad
\forall\lambda\in (\lambda^\ast, \lambda^\ast+\delta),
\end{array}
\right.
$$
 Since $\mathcal{L}''(0)$
 is either positive  or negative  on $H^0_{\lambda^\ast}$, we get
 either (\ref{e:Bi.2.7.5}) or (\ref{e:Bi.2.19-}).
Note that $\mathcal{L}_\lambda:=\mathcal{L}-\lambda\widehat{\mathcal{L}}$
 satisfies Hypothesis~\ref{hyp:1.1} with $X=H$. It has finite Morse index $\mu_\lambda$ and nullity $\nu_\lambda$ at $0$.
Thus (\ref{e:Bi.2.7.6}) and (\ref{e:Bi.2.18}) can be obtained.

If ${\mathcal{L}}''(0)$ is positive definite, by the proof of Theorem~\ref{th:Bi.3}
we have (\ref{e:Bi.2.7.5}) and hence (\ref{e:Bi.2.7.6}).
\end{proof}

\begin{remark}\label{rm:Bi.2.4.3}
{\rm It is possible to prove a generalization of \cite[Theorem~2.16]{Lu7} 
similar to  Theorem~\ref{th:A.5-},
 which leads to a more general version of  Theorem~\ref{th:Bi.2.4}  as the following Theorem~\ref{th:Bif.2.2.0}.
These will be explored otherwise.}
\end{remark}

Now we study corresponding analogue  of the above results in the setting of \cite{Lu1,Lu2}.
Here is a more general form of the analogue of Theorem~\ref{th:Bi.2.4}.

\begin{theorem}\label{th:Bif.2.2.0}
Let $H$, $X$ and $U$ be as in Hypothesis~\ref{hyp:1.3},
and let $\{\mathcal{L}_\lambda\in C^1(U, \mathbb{R})\,|\,\lambda\in\Lambda\}$ be a continuous family of functionals
parameterized by an open interval $\Lambda\subset\mathbb{R}$ containing  $\lambda^\ast$.
 For each $\lambda\in\Lambda$, assume $\mathcal{L}'_\lambda(0)=0$, and that
  there exists a map $A_\lambda\in C^1(U^X, X)$
such that for all $x\in U\cap X$  and $u, v\in X$,
 $$
 D\mathcal{L}_\lambda(x)[u]=(A_\lambda(x), u)_H\quad\hbox{and}\quad
(DA_\lambda(x)[u], v)_H=(B_\lambda(x)u, v)_H.
$$
 Suppose also that the following conditions hold.
 \begin{enumerate}
 \item[\rm (a)]   $B_\lambda$ has a decomposition
$B_\lambda=P_\lambda+Q_\lambda$, where for each $x\in U\cap X$,
 $P_\lambda(x)\in\mathscr{L}_s(H)$ is  positive definitive and
$Q_\lambda(x)\in\mathscr{L}_s(H)$ is compact, so that
$$
(\mathcal{L}_{\lambda}, H, X, U, A_{\lambda}, B_{\lambda}=P_{\lambda}+ Q_{\lambda})
$$
satisfies Hypothesis~\ref{hyp:1.3}.

\item[\rm (b)]  For each $h\in H$, it holds that $\|P_{\lambda}(x)h-P_{\lambda^\ast}(0)h\|\to 0$
as $x\in U\cap X$ approaches to $0$ in $H$ and $\lambda\in\Lambda$ converges to $\lambda^\ast$.

 \item[\rm (c)]  For some small $\delta>0$, there exist positive constants $c_0>0$ such that
$$
(P_\lambda(x)u, u)\ge c_0\|u\|^2\quad\forall u\in H,\;\forall x\in
\bar{B}_H(0,\delta)\cap X,\quad\forall\lambda\in \Lambda.
$$
 \item[\rm (d)]  $Q_\lambda: U\cap X\to \mathscr{L}_s(H)$ is uniformly continuous at $0$  with respect to $\lambda\in \Lambda$.
  \item[\rm (e)]  If $\lambda\in \Lambda$ converges to $\lambda^\ast$ then
  $\|Q_{\lambda}(0)-Q_{\lambda^\ast}(0)\|\to 0$.
   \item[\rm (f)]  ${\rm Ker}(B_{\lambda^\ast}(0))\ne\{0\}$, ${\rm Ker}(B_{\lambda}(0))=\{0\}$ for any $\lambda\in\Lambda\setminus\{\lambda^\ast\}$,
   and the Morse indexes of $\mathcal{L}_\lambda$
at $0\in H$ 
 take, respectively, values $\mu_{\lambda^\ast}$ and $\mu_{\lambda^\ast}+\nu_{\lambda^\ast}$
 as $\lambda\in\mathbb{R}$ varies in
 two deleted half neighborhoods  of $\lambda^\ast$,
where $\mu_{\lambda}$ and $\nu_{\lambda}$ are the Morse index and the nullity of  $\mathcal{L}_{\lambda}$
at $0$, respectively.
    \end{enumerate}
 Then  one of the following alternatives occurs:
\begin{enumerate}
\item[\rm (i)] $(\lambda^\ast,0)$ is not an isolated solution  in  $\{\lambda^\ast\}\times U^X$ of the equation
\begin{equation}\label{e:Bif.2.2.1*}
A_\lambda(u)=0,\quad (\lambda, u)\in\Lambda\times U^X.
\end{equation}

\item[\rm (ii)]  For every $\lambda\in\Lambda$ near $\lambda^\ast$ there is a nontrivial solution $u_\lambda$ of (\ref{e:Bif.2.2.1*}) in $U^X$,
which  converges to $0$ in $X$ as $\lambda\to\lambda^\ast$.

\item[\rm (iii)] There is an one-sided  neighborhood $\Lambda^\ast$ of $\lambda^\ast$ such that for any $\lambda\in\Lambda^\ast\setminus\{\lambda^\ast\}$,
(\ref{e:Bif.2.2.1*}) has at least two nontrivial solutions in $U^X$,
which  converge to $0$ in $X$  as $\lambda\to\lambda^\ast$.
\end{enumerate}
Therefore $(\lambda^\ast,0)\in\Lambda\times U^X$ is a bifurcation point  for the equation
(\ref{e:Bif.2.2.1*}); in particular $(\lambda^\ast,0)\in\Lambda\times U$ is a bifurcation point  for the equation
\begin{equation}\label{e:Bif.2.2*}
D\mathcal{L}_\lambda(u)=0,
\quad (\lambda,u)\in \Lambda\times U.
\end{equation}
 \end{theorem}

 \begin{proof}
Though the proof is almost repeat of that of Theorem~\ref{th:Bi.2.4}
we give it for completeness.
Let $H^+_\lambda$, $H^-_\lambda$ and $H^0_\lambda$ be the positive definite, negative definite and zero spaces of
${B}_\lambda(0)$.  Denote by $P^0_\lambda$ and $P^\pm_\lambda$ the orthogonal projections onto
$H^0_\lambda$ and $H^\pm_\lambda=H^+_\lambda\oplus H^-_\lambda$,
and by $X^\star_\lambda=X\cap H^\star_\lambda$ for $\star=+,-$, and by  $X^\pm_\lambda=P^\pm_\lambda(X)$.
By Theorem~\ref{th:A.5-}   there exists $\delta>0$ with $[\lambda^\ast-\delta,\lambda^\ast+\delta]\subset\Lambda$,
$\epsilon>0$ and a (unique) $C^1$ map
$$
\psi:[\lambda^\ast-\delta,\lambda^\ast+\delta]\times B_{H^0_{\lambda^\ast}}(0,\epsilon)\to X^\pm_{\lambda^\ast}
$$
which is $C^1$ in the second variable and
satisfies $\psi(\lambda, 0)=0\;\forall\lambda\in [\lambda^\ast-\delta,\lambda^\ast+\delta]$ and
$$
 P^\pm_{\lambda^\ast}A_{\lambda}(z+ \psi(\lambda,z))=0\quad\forall (\lambda,z)\in [\lambda^\ast-\delta,\lambda^\ast+\delta]\times B_{H^0_{\lambda^\ast}}(0,\epsilon),
 $$
an open neighborhood $W$ of $0$ in $H$ and an origin-preserving homeomorphism
\begin{eqnarray*}
&&\hspace{-8mm}[\lambda^\ast-\delta, \lambda^\ast+\delta]\times B_{H^0_{\lambda^\ast}}(0,\epsilon)\times
\left(B_{H^+_{\lambda^\ast}}(0, \epsilon) + B_{H^-_{\lambda^\ast}}(0, \epsilon)\right)\to [\lambda^\ast-\delta, \lambda^\ast+\delta]\times W,\nonumber\\
&&\hspace{20mm}({\lambda}, z, u^++u^-)\mapsto ({\lambda},\Phi_{{\lambda}}(z, u^++u^-)),
\end{eqnarray*}
such that for each $\lambda\in [\lambda^\ast-\delta, \lambda^\ast+\delta]$,
\begin{eqnarray}\label{e:NSpl.2.2}
&&\mathcal{L}_{\lambda}\circ\Phi_{\lambda}(z, u^++ u^-)=\|u^+\|^2-\|u^-\|^2+ \mathcal{
L}_{{\lambda}}(z+ \psi({\lambda}, z))\\
&& \quad\quad \forall (z, u^+ + u^-)\in  B_{H^0_{\lambda^\ast}}(0,\epsilon)\times
\left(B_{H^+_{\lambda^\ast}}(0, \epsilon) + B_{H^-_{\lambda^\ast}}(0, \epsilon)\right).\nonumber
\end{eqnarray}
(Clearly, each $\Phi_\lambda$ is  an origin-preserving homeomorphism.)
 Moreover, the functional
\begin{equation}\label{e:NSpl.2.3}
\mathcal{L}_{\lambda}^\circ: B_{H^0_{\lambda^\ast}}(0,\epsilon)\to \mathbb{R},\;
z\mapsto\mathcal{L}_{\lambda}(z+ \psi({\lambda}, z))
\end{equation}
 is of class $C^{2}$,
and for each $\lambda\in[\lambda^\ast-\delta, \lambda^\ast+\delta]$
the map $z\mapsto z+ \psi({\lambda}, z))$ induces an one-to-one correspondence
 between the critical points of  $\mathcal{L}_{\lambda}^\circ$ near $0\in H^0_{\lambda^\ast}$
and solutions of $A_\lambda(u)=0$ near $0\in U^X$.
 Hence the problem is reduced to finding
the critical points of $\mathcal{L}^\circ_\lambda$
near $0\in H^0_{\lambda^\ast}$ for $\lambda$ near $\lambda^\ast$.

By (f), we can shrink $\delta>0$ so that for each $\lambda\in [\lambda^\ast-\delta, \lambda^\ast+\delta]\setminus\{\lambda^\ast\}$,
$0\in H$  is a nondegenerate critical point  $\mathcal{L}_\lambda$ and the Morse indexes of $\mathcal{L}_\lambda$
at $0\in H$ take constant values in $[\lambda^\ast-\delta, \lambda^\ast)$ and $(\lambda^\ast, \lambda^\ast+\delta]$, respectively.
Assume that (\ref{e:Bi.2.7.5}) is satisfied.
Then Theorem~\ref{th:A.4} with $\lambda=0$ leads to
(\ref{e:Bi.2.7.6}). From these and Corollary~\ref{cor:A.6} we deduce
(\ref{e:Bi.2.16}) and therefore the expected conclusions as in the proof of Theorem~\ref{th:Bi.2.4}.
Another case can be proved similarity.
\end{proof}

 As showed below Theorem~\ref{th:Bi.2.4},
Theorem~\ref{th:Bif.2.2.0} is a strengthened version of  Theorem~\ref{th:Bif.1.1}
under the stronger assumptions on the Morse indexes of $\mathcal{L}_\lambda$ at $0\in H$.
Consequently,  for bifurcations of eigenvalues of nonlinear problems
we obtain the following two corollaries, which strengthen Corollary~\ref{cor:Bi.3.1}
and  correspond to Corollaries~\ref{cor:Bi.2.4.1},~\ref{cor:Bi.2.4.2}, respectively.

\begin{corollary}\label{cor:Bif.2.2.1}
Under Hypothesis~\ref{hyp:Bif.2.2.0+},
suppose that the eigenvalue $\lambda^\ast\in\mathbb{R}$  is isolated and
 that $B(0)$ is either semi-positive or semi-negative.
  Then the conclusions of Theorem~\ref{th:Bif.2.2.0} hold if
  (\ref{e:Bif.2.2.1*}) and (\ref{e:Bif.2.2*}) are, respectively, replaced by the following two equations
  \begin{eqnarray}\label{e:Bif.2.2.1}
&&A(u)=\lambda\widehat{A}(u),\quad (\lambda, u)\in\mathbb{R}\times U^X,\\
&&D\mathcal{L}(u)=\lambda D\widehat{\mathcal{L}}(u),\label{e:Bif.2.2}
\quad (\lambda,u)\in \mathbb{R}\times U.
\end{eqnarray}
   \end{corollary}

  Indeed, by the proof of Corollary~\ref{cor:Bi.2.2} it is easily seen  that
  for sufficiently small $\delta>0$   the family
  $\{\mathcal{L}_\lambda:=\mathcal{L}-\lambda \widehat{\mathcal{L}}\,|\,\lambda\in (\lambda^\ast-\delta, \lambda^\ast+\delta)\}$
  satisfies the assumptions of Theorem~\ref{th:Bif.2.2.0}.
  (Of course, as in the proof Corollary~\ref{cor:Bi.2.4.1}, we need to replace
$\mathcal{L}''(0)$ and $\widehat{\mathcal{L}}''(0)$ by $B(0)$ and $\widehat{B}(0)$, respectively.)

Similarly,  we may modify the proof of Corollary~\ref{cor:Bi.2.4.2} to get:

\begin{corollary}\label{cor:Bif.2.2.2}
Under Hypothesis~\ref{hyp:Bif.2.2.0+},
suppose that the following two conditions are satisfied:
 \begin{enumerate}
\item[\rm (a)] $B(0)$ is invertible
and $\lambda^\ast$ is an eigenvalue of $B(0)u-\lambda\widehat{B}(0)u=0$.
 \item[\rm (b)] $B(0)\widehat{B}(0)=\widehat{B}(0)B(0)$,   and $B(0)$
 is either positive  or negative on $H^0_{\lambda^\ast}={\rm Ker}(B(0)-\lambda^\ast\widehat{B}(0))$.
 \end{enumerate}
 Then the conclusions of Corollary~\ref{cor:Bif.2.2.1} are true.
 Moreover, if $B(0)$ is positive definite, the condition (b) is unnecessary.
\end{corollary}

Since $B(0)$ is invertible, the case (II) of Hypothesis~\ref{hyp:Bif.2.2.0} cannot occur.
The ``Moreover" part strengthens Theorem~\ref{th:Bi.4}.

In applications to Lagrange systems using
Theorem~\ref{th:Bif.2.2.0} and Corollaries~\ref{cor:Bif.2.2.1},~\ref{cor:Bif.2.2.2}
produce stronger results than using Theorem~\ref{th:Bi.2.4} and Corollaries~\ref{cor:Bi.2.4.1},~\ref{cor:Bi.2.4.2}
(see \cite{Lu9}). Moreover, we may give parameterized versions of splitting lemmas for the Finsler energy functional
in \cite{Lu5} as Theorem~\ref{th:A.5-} and then obtain similar results to
Theorem~\ref{th:Bif.2.2.0} and Corollaries~\ref{cor:Bif.2.2.1},~\ref{cor:Bif.2.2.2} for geodesics on Finsler manifolds.

Now we give a multi-parameter bifurcation result,
which is not only a converse of Corollary~\ref{cor:Bi.2.2*} in stronger conditions, but also
a generalizations of Corollaries~\ref{cor:Bif.2.2.1},~\ref{cor:Bif.2.2.2} in some sense.

\begin{theorem}\label{th:Bi.2.3}
Let $\mathcal{L}\in C^1(U,\mathbb{R})$ satisfy
Hypothesis~\ref{hyp:1.3}, except (C) and (D1),   and let
  $\widehat{\mathcal{L}}_j\in C^1(U,\mathbb{R})$, $j=1,\cdots,n$, satisfy
 Hypothesis~\ref{hyp:1.4}.
Suppose also that the following conditions hold:
\begin{enumerate}
\item[\rm (A)]  $\vec{\lambda}^\ast$ is an isolated  eigenvalue  of (\ref{e:Bi.2.2}), (writting $H(\vec{\lambda}^\ast)$
the solution space  of (\ref{e:Bi.2.2}) with $\vec{\lambda}=\vec{\lambda}^\ast$),
 and for each point $\vec{\lambda}$ near $0\in\mathbb{R}^n$, it holds that
\begin{equation}\label{e:Bi.2.2.1}
\{u\in H\,|\,\mathfrak{B}_{\vec{\lambda}}u\in X\}\cup\{u\in H\,|\, \mathfrak{B}_{\vec{\lambda}}u
=\mu u,\;\mu\le 0\}\subset X
\end{equation}
where $\mathfrak{B}_{\vec{\lambda}}:=B(0)-\sum^n_{j=1}\lambda_j\widehat{B}_j(0)$.
 \item[\rm (B)]  Either $H(\vec{\lambda}^\ast)$ has odd dimension, or $\mathfrak{B}_{\vec{\lambda}^\ast}\widehat{B}_j(0)=\widehat{B}_j(0)\mathfrak{B}_{\vec{\lambda}^\ast}$
 for $j=1,\cdots,\\n$, and there exists  $\vec{\lambda}\in\mathbb{R}^n\setminus\{{0}\}$ such that  the  symmetric bilinear form
\begin{equation}\label{e:Bi.2.2.3}
H(\vec{\lambda}^\ast)\times H(\vec{\lambda}^\ast)\ni (z_1,z_2)\mapsto \mathscr{Q}_{\vec{\lambda}}(z_1,z_2):=
 \sum^n_{j=1}(\lambda_j-\lambda^\ast_j)(\widehat{B}_j(0)z_1,z_2)_H
\end{equation}
has different Morse indexes and coindexes.
\end{enumerate}
(If $\vec{\lambda}^\ast=0$, we only need  that
each $\widehat{\mathcal{L}}_j\in C^1(U,\mathbb{R})$ is as in ``Moreover" part of Corollary~\ref{cor:Bi.2.2}.)
Then $(\vec{\lambda}^\ast, 0)\in\mathbb{R}^n\times U^X$ is a  bifurcation point of
(\ref{e:Bi.2.1*}). Moreover, suppose that (B) is replaced by
\begin{enumerate}
\item[\rm (B')] $\mathfrak{B}_{\vec{\lambda}^\ast}\widehat{B}_j(0)=\widehat{B}_j(0)\mathfrak{B}_{\vec{\lambda}^\ast}$
 for $j=1,\cdots,n$, and there exists  $\vec{\mu}\in\mathbb{R}^n\setminus\{{0}\}$ such that
the form $ \mathscr{Q}_{\vec{\mu}}$ on $H(\vec{\lambda}^\ast)$ is either positive definite or negative one.
\end{enumerate}
Then one of the following alternatives occurs:
\begin{enumerate}
\item[\rm (i)] $(\vec{\lambda}^\ast, 0)$ is not an isolated solution of (\ref{e:Bi.2.1*}) in
 $\{\vec{\lambda}^\ast\}\times U^X$.

\item[\rm (ii)]  For every $t$ near $0\in\mathbb{R}$ there is a nontrivial solution $u_t$ of (\ref{e:Bi.2.1*})
with $\vec{\lambda}=t(\vec{\mu}-\vec{\lambda}^\ast)+\vec{\lambda}^\ast$
    converging to $0$ in $X$ as $t\to 0$.

\item[\rm (iii)] There is an one-sided  neighborhood $\mathfrak{T}$ of $0\in\mathbb{R}$ such that
for any $t\in\mathfrak{T}\setminus\{0\}$,
(\ref{e:Bi.2.1*}) with $\vec{\lambda}=t(\vec{\mu}-\vec{\lambda}^\ast)+\vec{\lambda}^\ast$ has at least two nontrivial solutions converging to
$0$ in $X$ as $t\to 0$.
\end{enumerate}
\end{theorem}

\begin{proof} {\bf Step 1} ({\it Prove the first claim}).
Let $H^+_{\vec{\lambda}}$, $H^-_{\vec{\lambda}}$ and $H^0_{\vec{\lambda}}$ be the positive definite, negative definite and zero spaces of
$\mathfrak{B}_{\vec{\lambda}}=B(0)-\sum^n_{j=1}\lambda_j\widehat{B}_j(0)$.
 Denote by $P^0_{\vec{\lambda}}$ and $P^\pm_{\vec{\lambda}}$ the orthogonal projections onto
  $H^0_{\vec{\lambda}}$ and $H^\pm_{\vec{\lambda}}=H^+_{\vec{\lambda}}\oplus H^-_{\vec{\lambda}}$,
and by $X^\star_{\vec{\lambda}}=X\cap H^\star_{\vec{\lambda}}$ for $\star=+,-$, and by  $X^\pm_{\vec{\lambda}}=P^\pm_{\vec{\lambda}}(X)$.
By the assumption, for each $\vec{\lambda}\in\mathbb{R}^n$ near $\vec{\lambda}^\ast$ spaces
$H^-_{\vec{\lambda}}$ and $H^0_{\vec{\lambda}}$ have finite dimensions and are contained in $X$, so
$H^-_{\vec{\lambda}}=X^-_{\vec{\lambda}}$ and $H^0_{\vec{\lambda}}=X^0_{\vec{\lambda}}$.
Because the corresponding operator $A$ with $\mathcal{L}\in C^1(U,\mathbb{R})$
is not assumed to be $C^1$,  Theorem~\ref{th:A.5-}  cannot be used.
But modifying its proof  (as in the proof of \cite[Theorem~2.1]{Lu2})  there exist  $\delta>0$ and
$\epsilon>0$, a (unique) $C^{1-0}$ map
\begin{eqnarray*}
\psi:\vec{\lambda}^\ast+[-\delta,\delta]^n\times B_{H^0}(0,\epsilon)\to X^\pm_{\vec{\lambda}^\ast}
\end{eqnarray*}
which is strictly Fr\'echet differentiable in the second variable at $0\in H^0$ and satisfies $\psi(\vec{\lambda}, 0)=0\;\forall\vec{\lambda}\in \vec{\lambda}^\ast+[-\delta,\delta]^n$, and
an open neighborhood $W$ of $0$ in $H$ and an origin-preserving homeomorphism
\begin{eqnarray*}
&&\hspace{-10pt}(\vec{\lambda}^\ast+[-\delta,\delta]^n)\times B_{H^0_{\vec{\lambda}^\ast}}(0,\epsilon)\times
\left(B_{H^+_{\vec{\lambda}^\ast}}(0, \epsilon) + B_{H^-_{\vec{\lambda}^\ast}}(0, \epsilon)\right)\to
(\vec{\lambda}^\ast+[-\delta,\delta]^n)\times W,\nonumber\\
&&\hspace{20mm}({\vec{\lambda}}, z, u^++u^-)\mapsto ({\vec{\lambda}},\Phi_{\vec{\lambda}}(z, u^++u^-)),
\end{eqnarray*}
such that for each $\vec{\lambda}\in \vec{\lambda}^\ast+[-\delta,\delta]^n$,
the functional $\mathcal{L}_{\vec{\lambda}}$ satisfies
\begin{eqnarray}\label{e:BiSpl.2.2}
&&\mathcal{L}_{\vec{\lambda}}\circ\Phi_{\vec{\lambda}}(z, u^++ u^-)=\|u^+\|^2-\|u^-\|^2+ \mathcal{
L}_{\vec{\lambda}}(z+ \psi(\vec{\lambda}, z))\\
&& \quad\quad \forall (z, u^+ + u^-)\in  B_{H^0_{\vec{\lambda}^\ast}}(0,\epsilon)\times
\left(B_{H^+_{\vec{\lambda}^\ast}}(0, \epsilon) + B_{H^-_{\vec{\lambda}^\ast}}(0, \epsilon)\right).\nonumber
\end{eqnarray}
 Moreover,  the functional
$\mathcal{L}_{\vec{\lambda}}^\circ: B_{H^0_{\vec{\lambda}^\ast}}(0,\epsilon)\to \mathbb{R}$ given by
$\mathcal{L}_{\vec{\lambda}}^\circ(z)=\mathcal{L}_{\vec{\lambda}}(z+ \psi({\vec{\lambda}}, z))$
 is of class $C^{2-0}$, has first-order  differential  at $z_0\in
B_{H^0_{\vec{\lambda}^\ast}}(0, \epsilon)$ are given by
 \begin{eqnarray}\label{e:BiSpl.2.4}
d\mathcal{L}^\circ_{\vec{\lambda}}(z_0)[z]=\bigl(A(z_0+ \psi(\vec{\lambda}, z_0))-\sum^n_{j=1}\lambda_j \widehat{A}_j(z_0+ \psi(\vec{\lambda}, z_0)), z\bigr)_H\quad\forall z\in H^0_{\vec{\lambda}^\ast},
\end{eqnarray}
and $d\mathcal{L}^\circ_{\vec{\lambda}}$ has strict Fr\'echet derivative at $0\in H^0_{\vec{\lambda}^\ast}$,
\begin{eqnarray}\label{e:BiSpl.2.5}
 && d^2\mathcal{L}^\circ_{\vec{\lambda}}(0)[z,z']\notag\\&&=
 \left(P^0_{\vec{\lambda}^\ast}\bigr[\mathfrak{B}_{\vec{\lambda}}-
 \mathfrak{B}_{\vec{\lambda}}(P^\pm_{\vec{\lambda}^\ast}\mathfrak{B}_{\vec{\lambda}}|_{X^\pm_{\vec{\lambda}^\ast}})^{-1}
 (P^\pm_{\vec{\lambda}^\ast}\mathfrak{B}_{\vec{\lambda}})\bigr]z, z'\right)_H,\;
  \forall z,z'\in H^0_{\vec{\lambda}^\ast}.
 \end{eqnarray}
 Clearly, (\ref{e:BiSpl.2.5}) implies $d^2\mathcal{L}^\circ_{\vec{\lambda}^\ast}(0)=0$.
 As the arguments above \cite[Claim~2.17]{Lu7} we have
  \begin{equation}\label{e:Bi.2.2.2}
  d^2\mathcal{L}^\circ_{\vec{\lambda}}(0)[z_1,z_2]=\sum^n_{j=1}(\lambda_j-\lambda^\ast_j)(\widehat{B}_j(0)(z_1+D_z\psi(\vec{\lambda}, 0)[z_1]),z_2)_H,\quad\forall z_1, z_2\in H^0_{\vec{\lambda}^\ast}.
  \end{equation}
 (The negative sign in \cite[(2.59)]{Lu7} should be removed.)

 As before, using the above results   the
  bifurcation problem (\ref{e:Bi.2.1*}) is reduced to the following one
 \begin{equation}\label{e:Bif.2.2.5}
 \nabla\mathcal{L}^\circ_{\vec{\lambda}}(u)=0,\;u\in B_{H^0_{\vec{\lambda}^\ast}}(0,\epsilon),\;\vec{\lambda}\in \vec{\lambda}^\ast+[-\delta,\delta]^n.
 \end{equation}
By ({\bf A}) we may shrink $\delta>0$ so that
for each $\vec{\lambda}\in \vec{\lambda}^\ast+([-\delta,\delta]^n\setminus\{0\})$,
$\mathcal{L}\in C^1(U,\mathbb{R})$ satisfy Hypothesis~\ref{hyp:1.3} and has
$0\in H$ as a nondegenerate critical point.
Using Theorem~\ref{th:A.4} with $\widehat{\mathcal{L}}=0$
(or \cite[Remark~2.2]{Lu2})  and a consequence of (\ref{e:BiSpl.2.2})  as in \cite[Theorem~2.18]{Lu7}
 we obtain that for any Abel group ${\bf K}$ and for each $\vec{\lambda}\in \vec{\lambda}^\ast+([-\delta,\delta]^n\setminus\{0\})$,
\begin{equation}\label{e:Bif.2.2.6}
C_q(\mathcal{L}^\circ_{\vec{\lambda}}, 0;{\bf K})=\delta_{(q+{\mu_{\vec{\lambda}^\ast}})\mu_{\vec{\lambda}}}{\bf K},\quad
\forall q\in\mathbb{N}\cup\{0\},
\end{equation}
where $\mu_{\vec{\lambda}}$ is the Morse index of $\mathcal{L}_{\vec{\lambda}}$  at $0\in H$.

By a contradiction suppose now that $(\vec{\lambda}^\ast, 0)\in\mathbb{R}^n\times U^X$ is not a  bifurcation point of
(\ref{e:Bi.2.1*}). Then $(\vec{\lambda}^\ast, 0)\in\mathbb{R}^n\times  H^0_{\vec{\lambda}^\ast}$ is not a  bifurcation point of
(\ref{e:Bif.2.2.5}). Hence we may shrink $\epsilon>0$ and $\delta>0$ such that
$\mathcal{L}^\circ_{\vec{\lambda}}$ has a unique critical point $0$ in $B_{H^0_{\vec{\lambda}^\ast}}(0,\epsilon)$
for each $\vec{\lambda}\in \vec{\lambda}^\ast+[-\delta,\delta]^n$. As before it follows from the stability of critical groups
(cf.Theorem~\ref{th:stablity1}) that
\begin{equation}\label{e:Bi.2.7.2}
C_\ast(\mathcal{L}^\circ_{\vec{\lambda}}, 0;{\bf K})=C_\ast(\mathcal{L}^\circ_{\vec{\lambda}^\ast}, 0;{\bf K}),\quad
\forall \vec{\lambda}\in \vec{\lambda}^\ast+[-\delta,\delta]^n.
\end{equation}
This and (\ref{e:Bif.2.2.6}) imply  that the Morse index of $\mathcal{L}^\circ_{\vec{\lambda}}$ at $0$
is constant with respect to $\vec{\lambda}\in \vec{\lambda}^\ast+([-\delta,\delta]^n\setminus\{{0}\})$.
On the other hand,  for every $\vec{\lambda}\in \vec{\lambda}^\ast+([-\delta,\delta]^n\setminus\{{0}\})$,
(\ref{e:Bi.2.2.2}) implies
$$
d^2\mathcal{L}^\circ_{2\vec{\lambda}^\ast-\vec{\lambda}}(0)=-d^2\mathcal{L}^\circ_{\vec{\lambda}}(0)
$$
 as the nondegenerate quadratic forms on $H^0_{{\lambda}^\ast}$. Note that $H^0_{\vec{\lambda}^\ast}=H(\vec{\lambda}^\ast)$.
So if $\dim H^0_{\vec{\lambda}^\ast}$ is odd, $d^2\mathcal{L}^\circ_{\vec{\lambda}}(0)$ and $d^2\mathcal{L}^\circ_{2\vec{\lambda}^\ast-\vec{\lambda}}(0)$
must have different Morse indexes. This contradicts (\ref{e:Bi.2.7.2}).

In other situation, since for each $j=1,\cdots,n$, $\mathfrak{B}_{\vec{\lambda}^\ast}\widehat{B}_j(0)=\widehat{B}_j(0)\mathfrak{B}_{\vec{\lambda}^\ast}$
implies that $P^\star_{\vec{\lambda}^\ast}\widehat{B}_j(0)=\widehat{B}_j(0)P^\star_{\vec{\lambda}^\ast}$,
 $\star=0,+,-$,   and $D_z\psi(\vec{\lambda}, 0)[z_1]\in H^\pm_{\vec{\lambda}}$ for all $z_1\in H^0_{\vec{\lambda}}$,
we deduce
\begin{eqnarray*}
(D_z\psi(\vec{\lambda}, 0)[z_1], \widehat{B}_j(0)z_2)_H&=&(D_z\psi(\vec{\lambda}, 0)[z_1], \widehat{B}_j(0)P^0_{\vec{\lambda}^\ast}z_2)_H\\
&=&(D_z\psi(\vec{\lambda}, 0)[z_1], P^0_{\vec{\lambda}^\ast}\widehat{B}_j(0)z_2)_H=0,\quad
\forall z_1, z_2\in H^0_{\vec{\lambda}^\ast}.
\end{eqnarray*}
Thus (\ref{e:Bi.2.2.2}) becomes
$$
d^2\mathcal{L}^\circ_{\vec{\lambda}}(0)[z_1,z_2]=\sum^n_{j=1}(\lambda_j-\lambda^\ast_j)(\widehat{B}_j(0)z_1,z_2)_H=
  \mathscr{Q}_{\vec{\lambda}}(z_1,z_2),\quad\forall z_1, z_2\in H^0_{\vec{\lambda}^\ast}.
$$
Then the same reasoning  leads to a contradiction. The first claim is proved.

\vspace{4pt}\noindent
\noindent{\bf Step 2} ({\it Prove the second claim}).
By (\ref{e:Bi.2.2.3}), $\mathscr{Q}_{t(\vec{\mu}-\vec{\lambda}^\ast)+\vec{\lambda}^\ast}=t\mathscr{Q}_{\vec{\mu}}$
for any $t\in\mathbb{R}$. By replacing $\vec{\mu}$ by $2\vec{\lambda}^\ast-\vec{\mu}$ (if necessary)  we may assume that
the form $d^2\mathcal{L}^\circ_{\vec{\mu}}(0)=\mathscr{Q}_{\vec{\mu}}$ is positive definite.
Then when $t>0$ with small $|t|$ the forms
$d^2\mathcal{L}^\circ_{t(\vec{\mu}-\vec{\lambda}^\ast)+\vec{\lambda}^\ast}$
are positive definite and so $0\in H^0_{\vec{\lambda}^\ast}$ is a strict local
minimizer of $\mathcal{L}^\circ_{t(\vec{\mu}-\vec{\lambda}^\ast)+\vec{\lambda}^\ast}$;
when $t<0$ with small $|t|$ the forms
$d^2\mathcal{L}^\circ_{t(\vec{\mu}-\vec{\lambda}^\ast)+\vec{\lambda}^\ast}$
are negative definite and so $0\in H^0_{\vec{\lambda}^\ast}$ is a strict local
maximizer of $\mathcal{L}^\circ_{t(\vec{\mu}-\vec{\lambda}^\ast)+\vec{\lambda}^\ast}$.
As before the desired  claims follows from  Theorem~\ref{th:Bi.2.1}.
\end{proof}

Some results in next Sections~\ref{sec:B.3},\ref{sec:BBH} can also be generalized
to such more general multi-parameter forms.

Finally, in order to understand advantages and disadvantages of our above methods
let us state results which can be obtained with the parameterized version of the classical splitting lemma
for $C^2$ functionals stated in Remark~\ref{rm:Spl.2.5} and the classical Morse-Palais lemma  for $C^2$ functionals.

\begin{theorem}\label{th:Bi.2.5}
Let $U$ be an open neighborhood of $0$ in a real Hilbert space $H$,
$\Lambda\subset\mathbb{R}$ an open interval  containing  $\lambda^\ast$,
and let $\{\mathcal{L}_\lambda\in C^2(U, \mathbb{R})\,|\,\lambda\in\Lambda\}$ be a  family of functionals
such that $\Lambda\times U\ni (\lambda,u)\mapsto\nabla\mathcal{L}_\lambda(u)\in H$ is  continuous.
Suppose that the following conditions are satisfied:
\begin{enumerate}
\item[\rm (a)] For each $\lambda\in\Lambda$,  $0$ is a critical point of $\mathcal{L}_\lambda$ with
 finite Morse index $\mu_\lambda$ and nullity $\nu_\lambda$.

\item[\rm (b)]  If  $\nu_{\lambda^\ast}\ne 0$, $\nu_{\lambda}=0$ for any $\lambda\in\Lambda\setminus\{\lambda^\ast\}$,
   and $\mu_\lambda$  
take, respectively, values $\mu_{\lambda^\ast}$ and $\mu_{\lambda^\ast}+\nu_{\lambda^\ast}$
 as $\lambda\in\mathbb{R}$ varies in
 two deleted half neighborhoods  of $\lambda^\ast$.
\end{enumerate}
  Then  one of the following alternatives occurs:
\begin{enumerate}
\item[\rm (i)] $(\lambda^\ast,0)$ is not an isolated solution  in  $\{\lambda^\ast\}\times U$ of the equation
\begin{equation}\label{e:Bi.2.7.3*}
D\mathcal{L}_\lambda(u)=0,\quad (\lambda,u)\in \Lambda\times U.
\end{equation}
\item[\rm (ii)]  For every $\lambda\in\Lambda$ near $\lambda^\ast$ there is a nontrivial solution $u_\lambda$ of (\ref{e:Bi.2.7.3*}) in $U$
converging to $0$  as $\lambda\to\lambda^\ast$.

\item[\rm (iii)] There is an one-sided  neighborhood $\Lambda^\ast$ of $\lambda^\ast$ such that
for any $\lambda\in\Lambda^\ast\setminus\{\lambda^\ast\}$, (\ref{e:Bi.2.7.3*}) has at least two nontrivial solutions in $U$
converging to $0$  as $\lambda\to\lambda^\ast$.
\end{enumerate}
In particular, $(\lambda^\ast,0)\in\Lambda\times U$ is a bifurcation point  for the equation
(\ref{e:Bi.2.7.3*}).
\end{theorem}

In fact, we only need to make suitable replacements in the  proof of Theorem~\ref{th:Bif.2.2.0}.
For example, we replace symbols ``$X$" by ``$H$", ``$A_\lambda$" by ``$\nabla \mathcal{L}_\lambda$", and
 ``$B_\lambda$" by ``$\mathcal{L}^{\prime\prime}_{\lambda}$", and phrases
 ``By Theorem~\ref{th:A.5-}" by ``the parameterized version of the classical splitting lemma  for $C^2$ functionals
stated in Remark~\ref{rm:Spl.2.5}",  and ``Theorem~\ref{th:A.4} with $\lambda=0$" by
   ``the classical Morse-Palais lemma  for $C^2$ functionals",
 ``Corollary~\ref{cor:A.6}" by ``a corresponding result with Corollary~\ref{cor:A.6}".

\begin{theorem}\label{th:Bi.2.6}
The conclusions of Theorem~\ref{th:Bi.2.5} hold true if (a) and (b) are replaced by the following conditions:
\begin{enumerate}
\item[\rm (a')] For each $\lambda\in\Lambda$,  $\mathcal{L}'_\lambda(0)=0$, $0<\nu_{\lambda^\ast}:=\dim{\rm Ker}(\mathcal{L}''_{\lambda^\ast}(0))<\infty$.
\item[\rm (b')] For some small $\delta>0$, $\mathcal{L}_\lambda$ has an isolated local minimum (maximum) at zero for every
$\lambda\in (\lambda^\ast,\lambda^\ast+\delta]$ and an isolated local maximum (minimum) at
zero for every $\lambda\in [\lambda^\ast-\delta, \lambda^\ast)$.
\end{enumerate}
 \end{theorem}

\begin{proof}
As in the proof of Theorem~\ref{th:Bif.2.2.0},  the original bifurcation problem  is reduced to that of
 \begin{eqnarray*}
 \nabla\mathcal{L}^\circ_{{\lambda}}(u)=0,\;u\in B_{H^0_{{\lambda}^\ast}}(0,\epsilon),\;
 \lambda\in [\lambda^\ast-\delta,\lambda^\ast+\delta]\setminus\{\lambda^\ast\}.
 \end{eqnarray*}
Here each $\mathcal{L}^\circ_\lambda: B_H(0, \epsilon)\cap H^0_{\lambda^\ast}\to\mathbb{R}$ is
   a $C^2$ functional given by (\ref{e:NSpl.2.3}), and
 has differential at $z\in B_H(0, \epsilon)\cap H^0_{\lambda^\ast}$ given by
 $D\mathcal{L}_{\lambda}^\circ(z)[\zeta]=(\nabla\mathcal{L}_{\lambda}(z+ \psi({\lambda}, z)),\zeta)_H$
 for $\zeta\in H^0_{\lambda^\ast}$. Clearly, the assumption (b') implies that $\mathcal{L}_\lambda^\circ$
 has an isolated local minimum (maximum) at zero for every
$\lambda\in (\lambda^\ast,\lambda^\ast+\delta]$ and an isolated local maximum (minimum) at
zero for every $\lambda\in [\lambda^\ast-\delta, \lambda^\ast)$.
Then Theorem~\ref{th:Bi.2.1} leads to the desired  claims.
\end{proof}

Comparing Theorem~\ref{th:Bi.2.6} with \cite[Theorem~4.2]{CorH} we have no  the (PS) condition,
but require higher smoothness of functionals.

 Corresponding to Corollaries~\ref{cor:Bi.2.4.1},~\ref{cor:Bi.2.4.2} (or
Corollaries~\ref{cor:Bif.2.2.1},~\ref{cor:Bif.2.2.2}) we have

\begin{corollary}\label{cor:Bi.2.6}
Let $U$ be an open neighborhood of $0$ in a real Hilbert space $H$,
and let $\mathcal{L}, \widehat{\mathcal{L}}\in C^2(U,\mathbb{R})$ satisfy
$\mathcal{L}'(0)=0$ and  $\widehat{\mathcal{L}}'(0)=0$. Suppose that
$\lambda^\ast\in\mathbb{R}$ is an  eigenvalue of
$\mathcal{L}''(0)u=\lambda\widehat{\mathcal{L}}''(0)u$ in $H$  of finite multiplicity.
 Then the conclusions of Theorem~\ref{th:Bi.2.5} hold true with $\mathcal{L}_\lambda=\mathcal{L}-\lambda\widehat{\mathcal{L}}$
if  one of the following conditions is satisfied:
\begin{enumerate}
\item[\rm (a)] The eigenvalue $\lambda^\ast$ is isolated,  $\mathcal{L}_\lambda$  has finite Morse index  at $0$ for each $\lambda\in\mathbb{R}$ near $\lambda^\ast$, and  $\widehat{\mathcal{L}}''(0)$ is either semi-positive or semi-negative.
\item[\rm (b)] For some small $\delta>0$, $\mathcal{L}_\lambda$ has an isolated local minimum (maximum) at zero for every
$\lambda\in (\lambda^\ast,\lambda^\ast+\delta]$ and an isolated local maximum (minimum) at
zero for every $\lambda\in [\lambda^\ast-\delta, \lambda^\ast)$.
\end{enumerate}
\end{corollary}

When  $\widehat{\mathcal{L}}(u)=\frac{1}{2}(u,u)_H$, the case (a) shows that
 the  Rabinowitz bifurcation theorem stated at the beginning of this section
 can be obtained under the additional condition that $\mathcal{L}_\lambda$  has finite Morse index  at $0$ for
each $\lambda\in\mathbb{R}$ near $\lambda^\ast$. The latter  condition cannot imply that
$\mathcal{L}$  has finite Morse index  at $0$ if $\lambda^\ast\ne 0$.

\begin{proof}[\it Proof of Corollary~\ref{cor:Bi.2.6}]
The case that (b) holds follows from Theorem~\ref{th:Bi.2.6} directly.

When the case (a) holds, since $\lambda^\ast\in\mathbb{R}$ is an isolated eigenvalue of
$\mathcal{L}''(0)u=\lambda\widehat{\mathcal{L}}''(0)u$ in $H$  of finite multiplicity,
for each $\lambda\in\mathbb{R}$ near $\lambda^\ast$,
 $\mathcal{L}''(0)-\lambda\widehat{\mathcal{L}}''(0)$ is a bounded linear Fredholm operator.
(Hence $\mathcal{L}_\lambda$ has  finite  nullity $\nu_\lambda$  at $0$.)
Repeating the proof of Corollary~\ref{cor:Bi.2.4.1}
we can get (\ref{e:Bi.2.19.1}) and hence the condition (b) of Theorem~\ref{th:Bi.2.5} is satisfied.
\end{proof}

\begin{corollary}\label{cor:Bi.2.7}
Let $U$ be an open neighborhood of $0$ in a real Hilbert space $H$,
and let $\mathcal{L}, \widehat{\mathcal{L}}\in C^2(U,\mathbb{R})$ satisfy
$\mathcal{L}'(0)=0$ and  $\widehat{\mathcal{L}}'(0)=0$.
 Suppose that the following two conditions are satisfied:
 \begin{enumerate}
\item[\rm (a)] $\mathcal{L}''(0)$ is invertible, $\widehat{\mathcal{L}}''(0)$ is compact
and $\lambda^\ast$ is an eigenvalue of $\mathcal{L}''(0)u=\lambda\widehat{\mathcal{L}}''(0)u$ in $H$.
 \item[\rm (b)] For each $\lambda\in\mathbb{R}$ near $\lambda^\ast$,  $\mathcal{L}_\lambda=\mathcal{L}-\lambda\widehat{\mathcal{L}}$ has finite Morse index at $0\in H$, $\mathcal{L}''(0)\widehat{\mathcal{L}}''(0)=\widehat{\mathcal{L}}''(0)\mathcal{L}''(0)$,
    and  $\mathcal{L}''(0)$  is either positive  or negative on $H^0_{\lambda^\ast}:={\rm Ker}(\mathcal{L}''(0)-\lambda^\ast\widehat{\mathcal{L}}''(0))$.
 \end{enumerate}
 Then the conclusions of Theorem~\ref{th:Bi.2.5} hold true.
 Moreover, if ${\mathcal{L}}''(0)$ is positive definite, the condition (b) is unnecessary.
\end{corollary}

Clearly, the assumption (a) implies that $\mathcal{L}_\lambda$  has finite nullity at $0$ for each $\lambda\in\mathbb{R}$.
Repeating the proof of Corollary~\ref{cor:Bi.2.4.1} gives the desired claim.
The ``Moreover" part strengthens Theorem~\ref{th:Bi.7.1}.

Compare Corollaries~\ref{cor:Bi.2.6},\ref{cor:Bi.2.7} with
 \cite{Rab74,Rab}  by Rabinowitz.
Similarly, we can also write a corresponding result with Theorem~\ref{th:Bi.2.3}.

\section{Bifurcation for equivariant problems}\label{sec:B.3}
\setcounter{equation}{0}

There are abundant studies of bifurcations for equivariant maps,
see \cite{Ba1, BetPS1, FaRa1, FaRa2, Fe1, SmoWa} and the references therein.
In this section we shall generalize  main results of \cite{FaRa1, FaRa2}
in the setting of Section~\ref{sec:B.2}.

Let $G$ be a topological group. A (left) {\bf action} of it  over a Banach space $X$ is a continuous
map $G\times X\to X,\;(g,x)\mapsto gx$ satisfying the usual rules: $(g_1g_2)x=g_1(g_2x)$ and $ex=x$ for all $x\in X$
and $g_1,g_2\in G$, where $e\in G$ denotes the unite element. This action is called {\bf linear} (resp. {\bf
isometric}) if for each $g\in G$, $X\ni x\mapsto gx\in X$ is linear (and so a bounded linear operator
from $X$ to itself) (resp. $\|gx\|=\|x\|$ for all $x\in X$). When $X$ is a Hilbert space a linear isometric action
is also called  an orthogonal action (since a linear isometry must be an orthogonal linear transformation).

\subsection{Bifurcation theorems of Fadell--Rabinowitz type}\label{sec:B.3.1}

We first study finite dimensional situations.
The following is a refinement of \cite[Theorem~5.1]{Can}.

\begin{theorem}\label{th:Bi.2.1E}
Under the assumptions of Theorem~\ref{th:Bi.2.1},
let $X$ be equipped with a linear isometric action of a compact Lie group $G$ so that
each $f_\lambda$ is invariant under the $G$-action. Suppose also that
 the local minimums (maximums) in assumption c) of Theorem~\ref{th:Bi.2.1} are strict,
 and that $u=0$ is an isolated critical point of $f_{\lambda^\ast}$.
Then when the Lie group  $G$ is equal to $\mathbb{Z}_2=\{{\rm id}, -{\rm id}\}$ (resp. $S^1$
without fixed points except $0$,
which implies $\dim X$ to be an even more than one),
 there exist left and right  neighborhoods $\Lambda^-$ and $\Lambda^+$ of $\lambda^\ast$ in $\mathbb{R}$
and integers $n^+, n^-\ge 0$, such that $n^++n^-\ge\dim X$ (resp. $\frac{1}{2}\dim X$),
and that for $\lambda\in\Lambda^-\setminus\{\lambda^\ast\}$ (resp. $\lambda\in\Lambda^+\setminus\{\lambda^\ast\}$)
$f_\lambda$ has at least $n^-$ (resp. $n^+$) distinct critical
$G$-orbits different from $0$, which converge to
 $0$ as $\lambda\to\lambda^\ast$.
\end{theorem}

This is essentially contained in the proofs of  \cite{Rab, FaRa1, FaRa2}.
For completeness we will give most of its proof detail.

\begin{proof}[\it Proof of Theorem~\ref{th:Bi.2.1E}]
By assumptions we have either
\begin{eqnarray}\label{e:fBi.1}
0\in X\;\hbox{is a strict local}\left\{
\begin{array}{ll}
\hbox{minimizer of}\;f_{\lambda},&\quad
\forall \lambda\in [\lambda^\ast-\delta, \lambda^\ast),\\
\hbox{maximizer of}\;f_{\lambda},&\quad
\forall \lambda\in (\lambda^\ast, \lambda^\ast+\delta]
\end{array}\right.
\end{eqnarray}
or
\begin{eqnarray}\label{e:fBi.2}
0\in X\;\hbox{is a strict local}\left\{
\begin{array}{ll}
\hbox{maximizer of}\;f_{\lambda},&\quad
\forall \lambda\in [\lambda^\ast-\delta, \lambda^\ast),\\
\hbox{minimizer of}\;f_{\lambda},&\quad
\forall \lambda\in (\lambda^\ast, \lambda^\ast+\delta].
\end{array}\right.
\end{eqnarray}
By the assumption a) of Theorem~\ref{th:Bi.2.1}, replacing
$f_\lambda$ by $f_\lambda-f_\lambda(0)$ if necessary, we now and henceforth assume $f_\lambda(0)=0\;\forall\lambda$.
Since $u=0$ is an isolated critical point of $f_{\lambda^\ast}$, by \cite[page 136]{Ja} there exist mutually disjoint:

\vspace{4pt}\noindent
{\bf Case 1}. $0\in X$ is a local minimizer of $f_{\lambda^\ast}$;

\vspace{4pt}\noindent
{\bf Case 2}. $0\in X$ is a proper local maximizer of $f_{\lambda^\ast}$, i.e., it is a  local maximizer of $f_{\lambda^\ast}$
and $0$ belongs to the closure of $\{f_{\lambda^\ast}<0\}$;

\vspace{4pt}\noindent
{\bf Case 3}. $0\in X$ is a saddle point of $f_{\lambda^\ast}$, i.e., $f_{\lambda^\ast}$  takes both positive and negative values in every neighborhood of $0$.

Though our proof ideas are following   \cite[Theorem~2.2]{Rab} and \cite[Theorem~1.2]{FaRa1, FaRa2},
some technical improvements are necessary since each $f_{\lambda}$ is only of class $C^1$.

By the assumption b) of Theorem~\ref{th:Bi.2.1}, the function $(\lambda,z)\mapsto
Df_\lambda(z)$ is continuous on
$[\lambda^\ast-\delta, \lambda^\ast+\delta]\times B_X(0,\epsilon)$.
It follows that
\begin{eqnarray*}
R_{\delta,\epsilon}:=\{(\lambda, z)\in (\lambda^\ast-\delta, \lambda^\ast+\delta)\times B_X(0,\epsilon)\,|\,z\in B_X(0,\epsilon)\setminus
K(f_\lambda)\}
\end{eqnarray*}
is an open subset in $[\lambda^\ast-\delta, \lambda^\ast+\delta]\times B_X(0,\epsilon)$, where
$K(f_\lambda)$ denotes the critical set of $f_\lambda$.

Though the local flow
 of $\nabla f_{\lambda}$ must not exist,
fortunately, by suitably modifying the standard constructions in \cite[\S4]{Pa1} (see also
 \cite[Propsitions~5.29, 5.57]{MoMoPa}) we have:

\begin{lemma}\label{lem:Fpseudogradient}
There exists a smooth map $R_{\delta,\epsilon}\to X,\;(\lambda,z)\mapsto\mathscr{V}_\lambda(z)$,
such that for each $\lambda\in(\lambda^\ast-\delta, \lambda^\ast+ \delta)$ the map
$\mathscr{V}_\lambda: B_X(0,\epsilon)\setminus
K(f_\lambda)\to X$
is  $G$-equivariant and satisfies
$$
\|\mathscr{V}_\lambda(z)\|\le 2\|Df_{\lambda}(z)\|\quad\hbox{and}\quad
\langle Df_{\lambda}(z), \mathscr{V}_\lambda(z)\rangle\ge
\|Df_{\lambda}(z)\|^2
$$
for all $z\in B_X(0,\epsilon)\setminus K(f_\lambda)$, i.e., it is pseudo-gradient vector field of
$f_\lambda$ in the sense of \cite{Pa1}.
\end{lemma}

Notice that this is actually true for any compact Lie group $G$.
There exist some variants and generalizations for the notion of the pseudo-gradient vector field in \cite{Pa1},
which are more easily constructed and are still effective in applications. For example,
according to \cite{Guo} $\mathscr{V}_\lambda$ is also called a pseudo-gradient vector field of
$f_\lambda$ even if two inequalities in Lemma~\ref{lem:Fpseudogradient} are replaced by
$$
\|\mathscr{V}_\lambda(z)\|\le \alpha\|Df_{\lambda}(z)\|\quad\hbox{and}\quad
\langle Df_{\lambda}(z), \mathscr{V}_\lambda(z)\rangle\ge
\beta\|Df_{\lambda}(z)\|^2,
$$
where $\alpha$ and $\beta$ are two positive constants.

\begin{proof}[\it Proof of Lemma~\ref{lem:Fpseudogradient}]
Since $(\lambda^\ast-\delta, \lambda^\ast+\delta)\times B_{X}(0, \epsilon)\ni (\lambda,z)\mapsto
Df_\lambda(z)\in X^\ast$ is continuous, for any given
$(\lambda,z)\in R_{\delta,\epsilon}$ we have an open neighborhood
$O_{(z,\lambda)}$ of $z$ in $B_{X}(0, \epsilon)\setminus K(f_\lambda)$,
a positive number $r_{(z,\lambda)}$
with $(\lambda-r_{(z,\lambda)}, \lambda+r_{(z,\lambda)})\subset
(\lambda^\ast-\delta, \lambda^\ast+ \delta)$ and
a  vector $v_{(z,\lambda)}\in X$ such that
for all $(z',\lambda')\in O_{(z,\lambda)}\times (\lambda-r_{(z,\lambda)}, \lambda+r_{(z,\lambda)})$,
$$
\|v_{(z,\lambda)}\|< 2\|Df_{\lambda'}(z')\|\quad\hbox{and}\quad
\langle Df_{\lambda'}(z'), v_{(z,\lambda)}\rangle>
\|Df_{\lambda'}(z')\|^2.
$$
Now all above $O_{(z,\lambda)}\times (\lambda-r_{(z,\lambda)}, \lambda+r_{(z,\lambda)})$
form an open cover $\mathscr{C}$ of $R_{\delta,\epsilon}$,
and the latter admits a $C^\infty$-unit decomposition $\{\eta_\alpha\}_{\alpha\in\Xi}$
subordinate to a locally finite
refinement $\{W_\alpha\}_{\alpha\in\Xi}$ of $\mathscr{C}$.
Since each $W_\alpha$ can be contained in some open subset of form
$O_{(z,\lambda)}\times (\lambda-r_{(z,\lambda)}, \lambda+r_{(z,\lambda)})$, we have
a  vector $v_\alpha\in X$ such that for all $(z',\lambda')\in W_\alpha$,
$$
\|v_\alpha\|< 2\|Df_{\lambda'}(z')\|\quad\hbox{and}\quad
\langle Df_{\lambda'}(z'), v_\alpha\rangle>
\|Df_{\lambda'}(z')\|^2.
$$
Set $\chi(\lambda,z)=\sum_{\alpha\in\Xi}\eta_\alpha(\lambda,z) v_\alpha$. Then $\chi$ is a smooth map
from $R_{\delta,\epsilon}$ to $X$,  and satisfies
$$
\|\chi(z,\lambda)\|<2\|Df_{\lambda}(z)\|\quad\hbox{and}\quad
\langle Df_{\lambda}(z), \chi(z,\lambda)\rangle>
\|Df_{\lambda}(z)\|^2
$$
for all $(z,\lambda)\in R_{\delta,\epsilon}$. Let $d\mu$ denote the right invariant Haar measure on $G$. Define
$$
R_{\delta,\epsilon}\ni (z,\lambda)\mapsto\mathscr{V}_\lambda(z)=\int_G g^{-1}\chi(gz,\lambda)d\mu\in X.
$$
It is easily checked that $\mathscr{V}_\lambda$ satisfies requirements (see the proof \cite[Propsition~5.57]{MoMoPa}).
\end{proof}

Consider the ordinary differential equation
\begin{equation}\label{e:fBi.3}
\frac{d\varphi}{ds}=-\mathscr{V}_{\lambda^\ast}(\varphi),\quad\varphi(0,z)=z
\end{equation}
in $B_X(0,\epsilon)$. For each $z\in B_X(0,\epsilon)$ this equation possesses a unique solution $\varphi(t, z)$  whose
 the maximal existence interval is $(\omega_-(z),\omega_+(z))$ with
$-\infty\le\omega_-(z)<0<\omega_+(z)\le+\infty$.
Note that $0\in X$ is an isolated critical point of $f_{\lambda^\ast}$.
There exists  a small neighborhood $\mathcal{B}$ of $0$ in $B_X(0,\epsilon)$ such that
$u=0$ is a unique critical point of $f_{\lambda^\ast}$ siting in it. Put
\begin{eqnarray*}
&&S^+=\{z\in \mathcal{B}\,|\, \omega_+(z)=+\infty\;\&\;\varphi(t,z)\in \mathcal{B},\;\forall  t>0\},\\
&&S^-=\{z\in \mathcal{B}\,|\, \omega_-(z)=-\infty\;\&\;\varphi(t,z)\in \mathcal{B},\;\forall  t<0\}.
\end{eqnarray*}
Then $S^+\ne\emptyset$ (resp. $S^-\ne\emptyset$) if there are points near $0\in \mathcal{B}$ where $f_{\lambda^\ast}$
 is positive (resp. negative).
Repeating the constructions in \cite[Lemmas~1.15,1.16]{Rab} and \cite[Lemma~8.28]{FaRa2}
(and replacing $\psi$ therein by $\varphi$) we can obtain:

\begin{lemma}\label{lem:FBi.3.3-}
There is a constant $\gamma>0$ and a $G$-invariant open neighborhood $\mathcal{Q}$ of $0$ in $X$
with compact closure  $\overline{\mathcal{Q}}$ contained in $\mathcal{B}$ such that
\begin{enumerate}
\item[\rm (i)] if $z\in \mathcal{Q}$ then
$|f_{\lambda^\ast}(z)|<\gamma$;
\item[\rm (ii)] if $(z,t)\in \mathcal{Q}\times\mathbb{R}$ satisfies
$|f_{\lambda^\ast}(\varphi(t,z))|<\gamma$, then $\varphi(t,z)\in\mathcal{Q}$;
\item[\rm (iii)] if $z\in \partial\mathcal{Q}$, then either
$|f_{\lambda^\ast}(z)|=\gamma$ or
$\varphi(t,z)\in \partial\mathcal{Q}$ for all $t$ satisfying
$|f_{\lambda^\ast}(\varphi(t,z))|\le\gamma$.
\end{enumerate}
\end{lemma}

Actually, for sufficiently small $\rho>0$ with $\bar{B}_X(0,\rho)\subset\mathcal{B}$,
 $\mathcal{Q}$ may be chosen as
$$
\{z\in B_X(0,\epsilon)\,|\,-\gamma<f_{\lambda^\ast}(z)<\gamma\}\cap\{\varphi(t,z)\,|\,
z\in B_X(0,\rho),\;\omega_-(z)<t<\omega_+(z)\}.
$$
Clearly, when $0\in X$ is a strict local maximizer (resp. minimizer) of $f_{\lambda^\ast}$,
$\overline{\mathcal{Q}}\cap(f_{\lambda^\ast})^{-1}(\gamma)$ (resp. $\overline{\mathcal{Q}}\cap(f_{\lambda^\ast})^{-1}(-\gamma)$)
may be empty. Moreover, for any $z\in\overline{\mathcal{Q}}$, if $\varphi(t,z)\notin \overline{\mathcal{Q}}$
for some $t=t(z)>0$, then $f_{\lambda^\ast}(\varphi(t,z))<-\gamma$.

In the following, if $G=\mathbb{Z}_2==\{{\rm id}, -{\rm id}\}$  let $i_G$ denote the genus in \cite{Rab},
and if $G=S^1$ without fixed points except $0$
let $i_G$ denote the index ${\rm Index}^\ast_{\mathbb{C}}$ in \cite[\S7]{FaRa2}.
Let $T^\ast=S^\ast\cap\partial\overline{\mathcal{Q}}$, $\ast=+,-$.

\begin{lemma}\label{lem:FBi.3.4}
Both $T^+$ and $T^-$ are $G$-invariant compact subsets of $\partial\overline{\mathcal{Q}}$, and also satisfy
\begin{enumerate}
\item[\rm (i)] $\min\{f_\lambda(z)\,|\,z\in T^+\}>0$ and
$\max\{f_\lambda(z)\,|\,z\in T^-\}<0$;
\item[\rm (ii)] $i_{\mathbb{Z}_2}(T^+)+ i_{\mathbb{Z}_2}(T^-)\ge\dim X$ and
$i_{S^1}(T^+)+ i_{S^1}(T^-)\ge\frac{1}{2}\dim X$.
\end{enumerate}
\end{lemma}
By the definitions of $T^+$ and $T^-$, Lemma~\ref{lem:FBi.3.4}(i) is apparent.
Two inequalities in Lemma~\ref{lem:FBi.3.4}(ii) are \cite[Lemma~2.11]{FaRa1} and \cite[Theorem~8.30]{FaRa2}, respectively.

Define a piecewise-linear  function $\alpha:[0, \infty)\to [0, \infty)$
by $\alpha(t)=0$ for $t\le 1$, $\alpha(t)=-t+2$ for $t\in [1,2]$, and $\alpha(t)=0$ for $t\ge 2$.
For $\rho>0$ let $\alpha_\rho:[0, \infty)\to [0, \infty)$ be defined by
$\alpha_\rho(t)=\alpha(t/\rho)$.
Let $d(z, \partial\overline{\mathcal{Q}})=\inf\{\|z-u\|\,|\, u\in \partial\overline{\mathcal{Q}}\}$ for $z\in \overline{\mathcal{Q}}$.
For each $\lambda\in (\lambda^\ast-\delta, \lambda^\ast+ \delta)$, define
$$
\Xi_{\rho,\lambda}:\overline{\mathcal{Q}}\to H^0,\;z\mapsto\alpha_\rho(d(z,\partial\overline{\mathcal{Q}}))\mathscr{V}_{\lambda^\ast}(z)+(1-\alpha_\rho(d(z, \partial\overline{\mathcal{Q}})))
\mathscr{V}_{\lambda}(z),
$$
which is $G$-equivariant and Lipschitz continuous.
Since $\inf\{\|Df_{\lambda^\ast}(z)\|\,|\, z\in\partial\overline{\mathcal{Q}}\}>0$ we may choose
a small $\rho>0$ so that
\begin{eqnarray}\label{e:fBi.4}
&&\inf\{\|Df_{\lambda^\ast}(z)\|\,|\, z\in\overline{\mathcal{Q}}\;\&\; d(z, \partial\overline{\mathcal{Q}}))\le 2\rho\}>0,\\
&&\inf\{\langle Df_{\lambda^\ast}(z),\mathscr{V}_{\lambda^\ast}(z)\rangle\,|\, z\in\overline{\mathcal{Q}}\;\&\; d(z, \partial\overline{\mathcal{Q}}))\le 2\rho\}>0.\nonumber
\end{eqnarray}
Moreover, the function $(\lambda,z)\mapsto Df_\lambda(z)$ is uniform continuous on
$[\lambda^\ast-\delta, \lambda^\ast+\delta]\times \overline{\mathcal{Q}}$ by the assumption b) of Theorem~\ref{th:Bi.2.1}.
Then this and Lemma~\ref{lem:Fpseudogradient} imply that there exists a small $0<\delta_0<\delta$ such that
\begin{eqnarray}\label{e:fBi.5}
&&\frac{1}{2}\|Df_{\lambda^\ast}(z)\|\le \|Df_{\lambda}(z)\|\le 2\|Df_{\lambda^\ast}(z)\|,\\
&&\frac{1}{2}\langle Df_{\lambda^\ast}(z),\mathscr{V}_{\lambda^\ast}(z)\rangle\le \langle Df_{\lambda}(z),\mathscr{V}_{\lambda^\ast}(z)\rangle\le 2
\langle Df_{\lambda^\ast}(z),\mathscr{V}_{\lambda^\ast}(z)\rangle\label{e:fBi.6}
\end{eqnarray}
for all $z\in\overline{\mathcal{Q}}$ satisfying $d(z, \partial\overline{\mathcal{Q}}))\le 2\rho$ and for
all $\lambda\in [\lambda^\ast-\delta_0, \lambda^\ast+\delta_0]$.
(Clearly, (\ref{e:fBi.4}) and (\ref{e:fBi.5}) show that any critical point of $f_\lambda$ sitting in
$\overline{\mathcal{Q}}$ must be in
$\{ z\in\overline{\mathcal{Q}}\,|\, d(z, \partial\overline{\mathcal{Q}}))>2\rho\}$.)
By (\ref{e:fBi.5}) and Lemma~\ref{lem:Fpseudogradient} we deduce
\begin{eqnarray}\label{e:fBi.7}
\|\Xi_{\rho,\lambda}(z)\|&\le& \alpha_\rho(d(z,\partial\overline{\mathcal{Q}}))\|\mathscr{V}_{\lambda^\ast}(z)\|+
(1-\alpha_\rho(d(z, \partial\overline{\mathcal{Q}})))
\|\mathscr{V}_{\lambda}(z)\|\nonumber\\
&\le&2\alpha_\rho(d(z,\partial\overline{\mathcal{Q}}))\|Df_{\lambda^\ast}(z)\|+
2(1-\alpha_\rho(d(z, \partial\overline{\mathcal{Q}})))
\|Df_{\lambda}(z)\|\nonumber\\
&\le&  4\|Df_{\lambda}(z)\|,\quad\forall z\in \overline{\mathcal{Q}}\setminus K(f_\lambda),\quad
\forall \lambda\in [\lambda^\ast-\delta_0, \lambda^\ast+\delta_0].
\end{eqnarray}
Similarly, (\ref{e:fBi.5})-(\ref{e:fBi.6}) and Lemma~\ref{lem:Fpseudogradient} produce
\begin{eqnarray*}\label{e:fBi.8}
&&\langle Df_{\lambda}(z), \Xi_{\rho,\lambda}(z)\rangle\\
&=& \alpha_\rho(d(z,\partial\overline{\mathcal{Q}}))
\langle Df_{\lambda}(z),\mathscr{V}_{\lambda^\ast}(z)\rangle
+(1-\alpha_\rho(d(z, \partial\overline{\mathcal{Q}})))
\langle Df_{\lambda}(z), \mathscr{V}_{\lambda}(z)\rangle\\
&\ge&\frac{1}{2}\alpha_\rho(d(z,\partial\overline{\mathcal{Q}}))\|Df_{\lambda^\ast}(z)\|^2+
(1-\alpha_\rho(d(z, \partial\overline{\mathcal{Q}})))
\|Df_{\lambda}(z)\|^2\\
&\ge&\frac{1}{8}\alpha_\rho(d(z,\partial\overline{\mathcal{Q}}))\|Df_{\lambda}(z)\|^2+
(1-\alpha_\rho(d(z, \partial\overline{\mathcal{Q}})))
\|Df_{\lambda}(z)\|^2\\
&\ge&\frac{1}{8}\|Df_{\lambda}(z)\|^2,\quad\forall z\in \overline{\mathcal{Q}}\setminus K(f_\lambda),\quad
\forall \lambda\in [\lambda^\ast-\delta_0, \lambda^\ast+\delta_0].
\end{eqnarray*}
Hence  $\Xi_{\rho,\lambda}$ is a locally Lipschitz continuous
pseudo-gradient vector field of $f_{\lambda}(z)$ in the sense of
\cite{Guo}.

Replacing $\hat{V}$ in the proofs of \cite[Lemma~1.19]{Rab}  and \cite[Lemma~8.65]{FaRa2} by
$\Xi_{\rho,\lambda}$ and suitably modifying the proof of  \cite[Theorem~4]{Cl74} or \cite[Theorem~1.9]{Rab74}
we may obtain:

\begin{lemma}\label{lem:FBi.3.3}
Let $\gamma$ and $\mathcal{Q}$ be as in Lemma~\ref{lem:FBi.3.3-}.
For each $\lambda\in [\lambda^\ast-\delta_0, \lambda^\ast+ \delta_0]$
the restriction of $f_{\lambda}$ to $\overline{\mathcal{Q}}$
has the $G$-deformation property, i.e.,
for every $c\in\mathbb{R}$
and every $\hat{\tau}>0$, every $G$-neighborhood $U$ of
$K_{\lambda,c}:=K(f_\lambda)\cap\{z\in \overline{\mathcal{Q}}\,|\,
f_\lambda(z)\le c\}$ there exists an $\tau\in (0,\hat{\tau})$ and
a $G$ equivariant homotopy $\eta:[0, 1]\times \overline{\mathcal{Q}}\to \overline{\mathcal{Q}}$
with the following properties:
\begin{enumerate}
\item[\rm (i)] if $A\subset \overline{\mathcal{Q}}$ is closed and $G$-invariant, so is $\eta(s,A)$ for any $s\in [0,1]$;
\item[\rm (ii)]  $\eta(s,z)=z$ if $z\in \overline{\mathcal{Q}}\setminus (f_\lambda)^{-1}[c-\hat{\tau}, c+\hat{\tau}]$;
\item[\rm (iii)] $\eta(s,\cdot)$ is homeomorphism of $\overline{\mathcal{Q}}$
to $\overline{\mathcal{Q}}$ for each $s\in [0,1]$;
\item[\rm (iv)] $\eta(1, A_{\lambda,c+\tau}\setminus U)\subset A_{\lambda,c-\tau}$,
 where $A_{\lambda, d}:=\{z\in \overline{\mathcal{Q}}\,|\, f_\lambda(z)\le d\}$ for $d\in\mathbb{R}$;
\item[\rm (v)] if $K_{\lambda,c}=\emptyset$, $\eta(1, A_{\lambda,c+\tau})\subset A_{\lambda,c-\tau}$.
\end{enumerate}
\end{lemma}

Let
\begin{eqnarray*}
&&\mathscr{E}(\overline{\mathcal{Q}})=\{A\subset \overline{\mathcal{Q}}\,|\,A\;\hbox{is closed and $G$-invariant}\},\\
&&\mathscr{E}(\overline{\mathcal{Q}})_j=\{A\in \mathcal{E}(\overline{\mathcal{Q}})\,|\, i_G(A)\le j\},\;j=1,2,\cdots,\\
&&\mathscr{F}=\{\chi\in C(\overline{\mathcal{Q}}, \overline{\mathcal{Q}})\,|\,\hbox{$\chi$
is $G$-equivariant, one to one, and $\chi(x)=x\,\forall x\in T^-$}\},\\
&&\mathbb{R}^-\cdot K=\{\varphi(s,x)\,|\, -\infty<s<0,\;x\in K\}\quad\hbox{for}\;K\subset T^-.
\end{eqnarray*}

Suppose $i_G(T^-)=k>0$.
For $1\le j\le k$ define
\begin{eqnarray*}
&&G_{j}=\{\chi(\mathbb{R}^-\cdot K )\,|\,\chi\in\mathscr{F}, K\subset T^-
\;\hbox{belongs to $\mathscr{E}(\overline{\mathcal{Q}})$ and $i_G(K)\ge j$}\},\\
&&\Gamma_{j}=\{\overline{A\setminus B}\,|\,\hbox{for some integer $q\in [j, i_G(T^-)]$},\;A\in G_{q}, B\in \mathscr{E}(\overline{\mathcal{Q}})_{q-j}\}.
\end{eqnarray*}
As in the proofs of \cite[Lemma~2.12]{FaRa1} and \cite[Lemma~8.55]{FaRa2} we have:
\begin{enumerate}
\item[1)] $\Gamma_{j+1}\subset\Gamma_{j},\;1\le j\le i_G(T^-)-1$.
\item[2)] If $\chi\in\mathscr{F}$ and $A\in\Gamma_{j}$, then $\chi(A)\in\Gamma_{j}$.
\item[3)] If $A\in\Gamma_{j}$ and $B\in \mathscr{E}(\overline{\mathcal{Q}})_s$ with $s<j$, then $\overline{A\setminus B}\in \Gamma_{j-s}$.
\end{enumerate}

Let us define
$$
c_{\lambda,j}=\inf_{A\in\Gamma_{j}}\max_{z\in A}f_{\lambda}(z),\quad j=1,\cdots,k.
$$
Then the above property 1) implies $c_{\lambda,1}\le c_{\lambda,2}\le\cdots\le c_{\lambda,k}$.

\textsf{Firstly, let us assume that (\ref{e:fBi.1}) is true.}
 For each $\lambda\in [\lambda^\ast-\delta, \lambda^\ast)$, $0\in X$ is always a local (strict) minimizer  of $f_{\lambda}$.
 Therefore for  sufficiently small $\rho_\lambda>0$, we have
$f_{\lambda}(z)>0$ for any $0<\|z\|\le\rho_\lambda$
and so
$$
c_{\lambda,1}\ge \min_{\|z\|=\rho_\lambda}f_{\lambda}(z)>0.
$$
Using Lemma~\ref{lem:FBi.3.3} and
repeating other arguments in \cite[Theorem~2.9]{FaRa1} and \cite[Lemmas~8.67, 8.72]{FaRa2} we obtain:
\begin{enumerate}
\item[(a)] each $c_{\lambda,j}$
is a critical value of $f_{\lambda}$ with a corresponding critical point in $\mathscr{Q}$;
\item[(b)] if $c_{\lambda,j+1}=\cdots=c_{\lambda,j+p}=c$ then $i_G(K_{\lambda,c})\ge p$,
where $K_{\lambda,c}$ is as in Lemma~\ref{lem:FBi.3.3};
\item[(c)] all critical points corresponding to $c_{\lambda,j}$, $1\le j\le i_G(T^-)$, converge to $z=0$ as $\lambda\to\lambda^\ast$.
\end{enumerate}
$G$ is equal to $\mathbb{Z}_2=\{{\rm id}, -{\rm id}\}$ (resp. $S^1$

Moreover, when $p>1$,  $K_{\lambda,c}$ contains infinitely many distinct $G$-orbits. (In fact, if $G=\mathbb{Z}_2=\{{\rm id}, -{\rm id}\}$
 this follows from  \cite[Lemma~2.8]{FaRa1}.
If $G=S^1$ without fixed points except $0$,
since  $0\notin K_{\lambda,c}$ and  the isotropy group $S^1_x$ of the $S^1$-action at any $x\in X$ is either finite or $S^1$,
we deduce that the induced $S^1$-action on $K_{\lambda,c}$ has only finite isotropy groups and thus that
 $K_{\lambda,c}/S^1$ is an infinite set by \cite[Remark~6.16]{FaRa2}.)
These lead to

\begin{claim}\label{cl:F4.5}
If $\lambda\in [\lambda^\ast-\delta, \lambda^\ast)$ is close to $\lambda^\ast$,
$f_{\lambda}$ has at least $k$ distinct
 nontrivial critical $G$-orbits, which also converge to $0$ as $\lambda\to\lambda^\ast$.
\end{claim}

 For every $\lambda\in (\lambda^\ast, \lambda^\ast+\delta]$,
$0\in X$ is a local maximizer of $f_{\lambda}$, by considering
$-f_{\lambda}$ the same reason yields:

\begin{claim}\label{cl:F4.6}
If $i_{G}(T^+)=l>0$, for every $\lambda\in (\lambda^\ast, \lambda^\ast+\delta]$  close to $\lambda^\ast$,
$f_{\lambda}$ has at least $l$ distinct  nontrivial
critical $G$-orbits converging to $0$ as $\lambda\to\lambda^\ast$.
\end{claim}
These two claims together yield the  result of Theorem~\ref{th:Bi.2.1E}.

\textsf{Next, if (\ref{e:fBi.2}) holds true}, the similar arguments lead to:

\begin{claim}\label{cl:F4.7}
if $\lambda\in (\lambda^\ast, \lambda^\ast+\delta]$ (resp.
$\lambda\in [\lambda^\ast-\delta, \lambda^\ast)$) is close to $\lambda^\ast$,
$f_{\lambda}$ has at least $l$ (resp. $k$) distinct nontrivial
critical $G$-orbits, which also converge to $0$ as $\lambda\to\lambda^\ast$.
\end{claim}
So the expected result is still obtained.
\end{proof}

The following  may be viewed as a generalization of
Fadell--Rabinowitz theorems in \cite{FaRa1, FaRa2}.

\begin{theorem}\label{th:Bi.3.2}
Under the assumptions of Theorem~\ref{th:Bi.2.4}, let $H^0_{\lambda^\ast}$ be the eigenspace of
(\ref{e:Bi.2.7.4}) associated with $\lambda^\ast$. Let $H$ be equipped with
an orthogonal action of  a compact Lie group $G$ under which
$U$, $\mathcal{L}$ and $\widehat{\mathcal{L}}$  are $G$-invariant.
  If the Lie group  $G$ is equal to $\mathbb{Z}_2=\{{\rm id}, -{\rm id}\}$ (resp. $S^1$
     and the fixed point set of the induced $S^1$-action on $H^0_{\lambda^\ast}$,
    $(H^0_{\lambda^\ast})^{S^1}=\{x\in H^0\,|\, g\cdot x=x\;\forall g\in S^1\}$, is $\{0\}$,
    which implies $\dim H^0_{\lambda^\ast}$ to be an even more than one),
             then  one of the following alternatives holds:
\begin{enumerate}
\item[\rm (i)] $(\lambda^\ast, 0)$ is not an isolated solution of (\ref{e:Bi.2.7.3}) in
 $\{\lambda^\ast\}\times U$;

\item[\rm (ii)] there exist left and right  neighborhoods $\Lambda^-$ and $\Lambda^+$ of $\lambda^\ast$ in $\mathbb{R}$
and integers $n^+, n^-\ge 0$, such that $n^++n^-\ge\dim H^0_{\lambda^\ast}$ (resp. $\frac{1}{2}\dim H^0_{\lambda^\ast}$),
and that for $\lambda\in\Lambda^-\setminus\{\lambda^\ast\}$ (resp. $\lambda\in\Lambda^+\setminus\{\lambda^\ast\}$)
(\ref{e:Bi.2.7.3}) has at least $n^-$ (resp. $n^+$) distinct critical
$G$-orbits different from $0$, which converge to
 $0$ as $\lambda\to\lambda^\ast$.
\end{enumerate}
In particular,  $(\lambda^\ast, 0)\in\mathbb{R}\times U$ is a bifurcation point  for the equation
(\ref{e:Bi.2.7.3}).
\end{theorem}
\begin{proof}
The first claim follows from Theorem~\ref{th:Bi.2.4}.
For others, by the arguments above Step 1 in the proof of Theorem~\ref{th:Bi.2.4}
the problem is reduced to finding the $G$-critical orbits of $\mathcal{L}^\circ_\lambda$
near $0\in H^0_{\lambda^\ast}$ for $\lambda$ near $\lambda^\ast$,
where $\mathcal{L}^\circ_\lambda: B_{H}(0, \epsilon)\cap H^0_{\lambda^\ast}\to\mathbb{R}$ is given by (\ref{e:Bi.2.15}).
Suppose that (i) does not hold.  By shrinking $\delta>0$, we obtain
that  for each  $\lambda\in [\lambda^\ast-\delta, \lambda^\ast+\delta]$, $0\in H$ is an isolated critical point of
$\mathcal{L}_\lambda$ and thus $0\in H^0_{\lambda^\ast}$ is an isolated critical point of
$\mathcal{L}^\circ_\lambda$. Then we have
(\ref{e:Bi.2.18}) and (\ref{e:Bi.2.19}), and so (ii) follows from
Theorem~\ref{th:Bi.2.1E}.
 \end{proof}

By Corollaries~\ref{cor:Bi.2.4.1} and \ref{cor:Bi.2.4.2} we get

\begin{corollary}\label{cor:Bi.3.2.1}
If ``Under the assumptions of Theorem~\ref{th:Bi.2.4}" in Theorem~\ref{th:Bi.3.2}
is replaced by ``Under the assumptions of one of Corollaries~\ref{cor:Bi.2.4.1} and \ref{cor:Bi.2.4.2}",
then the conclusions of Theorem~\ref{th:Bi.3.2} hold true.
\end{corollary}

Now we consider generalizations of
Fadell--Rabinowitz theorems in \cite{FaRa1, FaRa2} in the setting of \cite{Lu1,Lu2}.
Combing Theorem~\ref{th:Bi.2.1E} with Theorem~\ref{th:Bif.2.2.0} can naturally lead to one.
Instead of this we adopt another way.
Recall that Fadell--Rabinowitz bifurcation theorems \cite{FaRa1, FaRa2}
 were generalized  to the case of arbitrary compact Lie groups by
Bartsch and Clapp \cite{BaCl}, and Bartsch \cite{Ba1}.
Bartsch and Clapp \cite[\S4]{BaCl} proved the following theorem
(according to our notations).

\begin{theorem}\label{th:Bif.2.2.2+}
 Let $Z$ be a finite dimensional Hilbert space equipped with a linear isometric action of a compact Lie group $G$
 with $Z^G=\{0\}$, let $\delta>0$, $\epsilon>0$, $\lambda^\ast\in\mathbb{R}$ and
for every $\lambda\in [\lambda^\ast-\delta, \lambda^\ast+\delta]$ let
$f_\lambda:B_Z(0,\epsilon)\to\mathbb{R}$ be a function of class $C^2$.
Assume that
\begin{enumerate}
\item[\rm (a)] the functions $\{(\lambda,u)\to f_\lambda(u)\}$ and
$\{(\lambda,u)\to f'_\lambda(u)\}$  are continuous on
$[\lambda^\ast-\delta, \lambda^\ast+\delta]\times B_Z(0,\epsilon)$;
\item[\rm (b)] $u=0$ is a critical point of each $f_\lambda$, $L_{\lambda}:=D(\nabla f_{\lambda})(0)\in\mathscr{L}_s(Z)$ is an isomorphism
for each $\lambda\in (\lambda^\ast-\delta, \lambda^\ast+\delta)\setminus\{\lambda^\ast\}$, and  $L_{\lambda^\ast}=0$;
\item[\rm (c)] the eigenspaces $Z_\lambda$ and $Z_\mu$ belonging to
  $\sigma(L_{\lambda})\cap\mathbb{R}^-$ and $\sigma(L_{\mu})\cap\mathbb{R}^-$,
  respectively, are isomorphic if $(\lambda-\lambda^\ast)(\mu-\lambda^\ast)>0$ for any
  $\lambda,\mu\in (\lambda^\ast-\delta, \lambda^\ast+\delta)\setminus\{\lambda^\ast\}$, and thus
  the number
\begin{equation}\label{e:Bi.3.9.2*}
d:=\ell(SZ)-\min\{\ell(SZ_{\lambda})+\ell(SZ_{\mu}^\bot),
 \ell(SZ_{\lambda}^\bot)+\ell(SZ_{\mu})\}
\end{equation}
  is independent of $\lambda\in (\lambda^\ast-\delta, \lambda^\ast)$ and $\mu\in(\lambda^\ast, \lambda^\ast+\delta)$,
  where $Z_\lambda^\bot$ is the orthogonal complementary of $Z_\lambda$ in $Z$ and
  $\ell$ denote the $(\mathscr{G}, h^\ast)$-length  used in \cite{BaCl}.
\end{enumerate}
Then one at least of the following assertions holds if $d>0$:
\begin{enumerate}
\item[\rm (i)] $u=0$ is not an isolated critical point of $f_{\lambda^\ast}$;
\item[\rm (ii)] there exist left and right  neighborhoods $\Lambda^-$ and $\Lambda^+$ of $\lambda^\ast$ in $\mathbb{R}$
and integers $n^+, n^-\ge 0$, such that $n^++n^-\ge d$
and for $\lambda\in\Lambda^-\setminus\{\lambda^\ast\}$ (resp. $\lambda\in\Lambda^+\setminus\{\lambda^\ast\}$),
$f_\lambda$ has at least $n^-$ (resp. $n^+$) distinct critical
$G$-orbits different from $0$, which converge to
 $0$ as $\lambda\to\lambda^\ast$.
\end{enumerate}
In particular,  $(\lambda^\ast, 0)\in [\lambda^\ast-\delta, \lambda^\ast+\delta]\times B_Z(0,\epsilon)$
is a bifurcation point of $f'_\lambda(u)=0$.
\end{theorem}

 \begin{theorem}\label{th:Bif.2.2.4-}
Under the assumptions of Theorem~\ref{th:Bif.2.2.0}
suppose that  a compact Lie group $G$  acts on $H$ orthogonally, which induces  $C^1$ isometric actions on $X$,
 and that both $U$ and $\mathcal{L}_\lambda$ are $G$-invariant (and hence $H^0_\lambda$, $H^\pm_\lambda$
are $G$-invariant subspaces).
If the fixed point set of the induced $G$-action on $H^0_{\lambda^\ast}$ is $\{0\}$ then
 one of the following alternatives occurs:
\begin{enumerate}
\item[\rm (i)] $(\lambda^\ast,0)$ is not an isolated solution  in  $\{\lambda^\ast\}\times U^X$ of the equation (\ref{e:Bif.2.2.1*});
\item[\rm (ii)] there exist left and right  neighborhoods $\Lambda^-$ and $\Lambda^+$ of $\lambda^\ast$ in $\mathbb{R}$
and integers $n^+, n^-\ge 0$, such that $n^++n^-\ge \ell(SH^0_{\lambda^\ast})$
and for $\lambda\in\Lambda^-\setminus\{\lambda^\ast\}$ (resp. $\lambda\in\Lambda^+\setminus\{\lambda^\ast\}$),
$\mathcal{L}_\lambda$ has at least $n^-$ (resp. $n^+$) distinct critical
$G$-orbits different from $0$, which converge to
 $0$ as $\lambda\to\lambda^\ast$.
\end{enumerate}
In particular,  $(\lambda^\ast, 0)\in [\lambda^\ast-\delta, \lambda^\ast+\delta]\times U^X$
is a bifurcation point of (\ref{e:Bif.2.2.1*}).
\end{theorem}

\begin{proof}
Follow the notations in the proof of Theorem~\ref{th:Bif.2.2.0}. In the present situation,
for each $\lambda\in [\lambda^\ast-\delta, \lambda^\ast+\delta]$,
the maps $\psi(\lambda, \cdot)$  and $\Phi_{\lambda}(\cdot,\cdot)$  are
  $G$-equivariant, and $C^2$ functional $\mathcal{L}^\circ_{\lambda}$ given by (\ref{e:NSpl.2.3}) is $G$-invariant.
As in the proof of Theorem~\ref{th:Bi.2.4} we obtain either
\begin{eqnarray}\label{e:Bif.2.2.3}
0\in H^0_{\lambda^\ast}\;\hbox{is a strict local}\left\{
\begin{array}{ll}
\hbox{minimizer of}\;\mathcal{L}^\circ_{\lambda},&\quad
\forall \lambda\in [\lambda^\ast-\delta, \lambda^\ast),\\
\hbox{maximizer of}\;\mathcal{L}^\circ_{\lambda},&\quad
\forall \lambda\in (\lambda^\ast, \lambda^\ast+\delta]
\end{array}\right.
\end{eqnarray}
or
\begin{eqnarray}\label{e:Bif.2.2.4}
0\in H^0_{\lambda^\ast}\;\hbox{is a strict local}\left\{
\begin{array}{ll}
\hbox{maximizer of}\;\mathcal{L}^\circ_{\lambda},&\quad
\forall \lambda\in [\lambda^\ast-\delta, \lambda^\ast),\\
\hbox{minimizer of}\;\mathcal{L}^\circ_{\lambda},&\quad
\forall \lambda\in (\lambda^\ast, \lambda^\ast+\delta].
\end{array}\right.
\end{eqnarray}
Note that for each $\lambda\in [\lambda^\ast-\delta, \lambda^\ast+\delta]\setminus\{\lambda^\ast\}$,
 $0\in H^0_{\lambda^\ast}$ is a nondegenerate critical point of
the functional $\mathcal{L}^\circ_\lambda$ by Remark~\ref{rm:Spl.2.4}(ii).
Let $L_\lambda:=D(\nabla\mathcal{L}_{\lambda}^\circ)(0)$. Then $L_{\lambda^\ast}=0$ by
(\ref{e:Spli.2.5}). Hence this and (\ref{e:Spli.2.3})-(\ref{e:Spli.2.4}) show that functionals
$f_\lambda:=\mathcal{L}_{\lambda}^\circ$ on $B_Z(0,\epsilon)$ with $Z:=H^0_{\lambda^\ast}$,
 $\lambda\in [\lambda^\ast-\delta, \lambda^\ast+\delta]$,
satisfy the conditions (a)-(b) in Theorem~\ref{th:Bif.2.2.2+}.
(In fact, $\lambda\mapsto L_\lambda\in\mathscr{L}_s(Z)$ is also continuous by (\ref{e:Spli.2.5}).)

It remains to prove that they also satisfy the condition (c) in Theorem~\ref{th:Bif.2.2.2+}.

Firstly, we assume (\ref{e:Bif.2.2.3}) holding. Then it implies that
$L_\lambda$ is positive definite for any $\lambda\in (\lambda^\ast-\delta, \lambda^\ast)$,
and negative definite for any $\lambda\in (\lambda^\ast, \lambda^\ast+\delta)$.
It follows that $\sigma(L_{\lambda})\cap\mathbb{R}^-=\emptyset$ and so $Z_\lambda=\emptyset$
for any $\lambda\in (\lambda^\ast-\delta, \lambda^\ast)$, and
$\sigma(L_{\lambda})\cap\mathbb{R}^-=\sigma(L_{\lambda})$ and so $Z_\lambda=Z$
for any $\lambda\in (\lambda^\ast, \lambda^\ast+\delta)$.
Hence the condition (c) in Theorem~\ref{th:Bif.2.2.2+} is satisfied. Note also that in
the present situation we have
\begin{eqnarray*}
d&=&\ell(SZ)-\min\{\ell(SZ_{\lambda})+\ell(SZ_{\mu}^\bot),
 \ell(SZ_{\lambda}^\bot)+\ell(SZ_{\mu})\}\\
 &=&\ell(SZ)-\min\{0,  \ell(SZ)+\ell(SZ)\}=\ell(SZ)
\end{eqnarray*}
 for any of $\lambda\in (\lambda^\ast-\delta, \lambda^\ast)$ and $\mu\in(\lambda^\ast, \lambda^\ast+\delta)$
 because the $(\mathscr{G}, h^\ast)$-length of a $G$-space,  $\ell(X)$, is equal to zero if and only if $X=\emptyset$ by
  \cite[Proposition~1.5]{BaCl}.

Next, let (\ref{e:Bif.2.2.4}) be satisfied. Then we have
$\sigma(L_{\lambda})\cap\mathbb{R}^-=\sigma(L_{\lambda})$ and so $Z_\lambda=Z$
for any $\lambda\in (\lambda^\ast-\delta, \lambda^\ast)$, and
 $\sigma(L_{\lambda})\cap\mathbb{R}^-=\emptyset$ and so $Z_\lambda=\emptyset$
for any $\lambda\in (\lambda^\ast, \lambda^\ast+\delta)$.
These  lead to
\begin{eqnarray*}
d=\ell(SZ)-\min\{\ell(SZ_{\lambda})+\ell(SZ_{\mu}^\bot),
 \ell(SZ_{\lambda}^\bot)+\ell(SZ_{\mu})\}=\ell(SZ)
\end{eqnarray*}
 for any of $\lambda\in (\lambda^\ast-\delta, \lambda^\ast)$ and $\mu\in(\lambda^\ast, \lambda^\ast+\delta)$.

 Note that we have always $\ell(SZ)>0$ by
 \cite[Proposition~1.5]{BaCl}. Hence Theorem~\ref{th:Bif.2.2.2+} gives the desired conclusions.
 \end{proof}

When $G=\mathbb{Z}_2=\{-id_H, id_H\}$  (resp. $G=S^1$)
we have $d=\dim H^0_{\lambda^\ast}$ (resp. $d=\frac{1}{2}\dim H^0_{\lambda^\ast}$) as showed in
Remark~\ref{rem:Bif.3.4} below.

 \begin{corollary}\label{cor:Bif.3.3}
Under the assumptions of one of Corollaries~\ref{cor:Bif.2.2.1},~\ref{cor:Bif.2.2.2}
suppose that $H$ is equipped with an orthogonal action of  a compact Lie group $G$
  which induces a $C^1$ isometric action on $X$, and that $U$, $\mathcal{L}$ and $\widehat{\mathcal{L}}$ are $G$-invariant.
 Then the conclusions of Theorem~\ref{th:Bif.2.2.4-} hold true.
\end{corollary}

 \begin{remark}\label{rem:Bif.3.4}
{\rm For a finite-dimensional real $G$-module $M$ with $M^G=\{0\}$ and $h^\ast=H^\ast_G$ Borel cohomology it was proved in \cite[Propositions~2.4,2.6]{BaCl}:
\begin{enumerate}
\item[(i)] if $G=(\mathbb{Z}/p)^r$, where $r>0$ and $p$ is a prime, then $(\mathscr{G}, H^\ast_G)$-{\rm length}(SM) is equal
to $\dim M$ for $p=2$, and to $\frac{1}{2}\dim M$ for $p>2$;
\item[(ii)] if $G=(S^1)^r$,  $r>0$, then $(\mathscr{G}, H^\ast_G)$-{\rm length}(SM) is equal
to  $\frac{1}{2}\dim M$;
\item[(iii)] if $G=S^1\times\Gamma$, $\Gamma$ is finite, and such that the fixed point set of $S^1\equiv S^1\times\{e\}$
is trivial,  then $(\mathscr{G}, H^\ast_G)$-{\rm length}(SM) is equal
to  $\frac{1}{2}\dim M$.
\end{enumerate}
Here the coefficients in the Borel cohomology $H^\ast_G$ are $G=(\mathbb{Z}/p)^r$ in (i), and $\mathbb{Q}$ in (ii) and (iii).
}
\end{remark}

Let us point out that it is possible to generalize the above results with methods in  \cite[\S7.5]{Ba1}.

As remarked below Corollary~\ref{cor:Bif.2.2.2}, Theorem~\ref{th:Bif.2.2.4-}
and Corollary~\ref{cor:Bif.3.3} are powerful to applications in Lagrange systems and geodesics on Finsler manifolds.

 \begin{remark}\label{rem:Bif.3.5}
{\rm Under the assumptions of any one of Theorems~\ref{th:Bi.2.6},\ref{th:Bi.2.5}
and Corollaries~\ref{cor:Bi.2.6},\ref{cor:Bi.2.7},
if $H$ is equipped with an orthogonal action of  a compact Lie group $G$
 such that $U$ and $\mathcal{L}_\lambda$  are $G$-invariant and that
 the fixed point set of the induced $G$-action on $H^0_{\lambda^\ast}$ is $\{0\}$,
  then the corresponding conclusions with Theorem~\ref{th:Bi.3.2}  (for $G=\mathbb{Z}_2=\{{\rm id}, -{\rm id}\}$ or $S^1$)
 and those with Theorem~\ref{th:Bif.2.2.4-} with $X$ replaced by $H$ hold true.
Compare the latter with \cite[Therrem~3.1]{BaCl} by Bartsch and Clapp. }
\end{remark}

\subsection{Bifurcations starting a nontrivial critical orbit}\label{sec:B.3.2}

\begin{hypothesis}[\hbox{\cite[Hypothesis~2.20]{Lu7}}]\label{hyp:S.6.2}
{\rm {\bf (i)} Let $G$ be a compact
Lie group, and  $\mathcal{H}$  a $C^3$ Hilbert-Riemannian $G$-space
(that is, ${\mathcal{H}}$ is a $C^3$ $G$-Hilbert manifold with a Riemannian
metric $(\!(\cdot,\cdot)\!)$ such that $T{\mathcal{H}}$ is a $C^2$ Riemannian $G$-vector bundle (see \cite{Was})).\\
 {\bf (ii)} The $C^1$ functional $\mathcal{ L}:\mathcal{H}\to\mathbb{R}$ is $G$-invariant, the gradient
 $\nabla\mathcal{L}:\mathcal{H}\to T\mathcal{H}$ is G\^ateaux differentiable
 (i.e., under any $C^3$ local chart the functional $\mathcal{L}$
 has a G\^ateaux differentiable gradient map), and $\mathcal{ O}$ is an isolated
 critical orbit which is a $C^3$ critical submanifold with  Morse index $\mu_\mathcal{O}$.}
\end{hypothesis}

Under the above assumptions
let $\pi:N\mathcal{ O}\to \mathcal{ O}$ denote the normal bundle of it. The
bundle is a $C^2$-Hilbert vector bundle over $\mathcal{ O}$, and can
be considered as a subbundle of $T_\mathcal{ O}{\mathcal{H}}$ via the
Riemannian metric $(\!(\cdot, \cdot)\!)$. The metric $(\!(\cdot, \cdot)\!)$
induces a natural $C^2$ orthogonal bundle
projection ${\bf \Pi}:T_{\mathcal{O}}\mathcal{H}\to N\mathcal{O}$. For $\varepsilon>0$,
the so-called normal disk bundle of radius $\varepsilon$ is denoted by
$N\mathcal{ O}(\varepsilon):=\{(x,v)\in N\mathcal{O}\,|\,\|v\|_{x}<\varepsilon\}$.
 If $\varepsilon>0$ is small enough  the exponential map $\exp$ gives a $C^2$-diffeomorphism
 $\digamma$ from  $N\mathcal{ O}(\varepsilon)$ onto an open
neighborhood of $\mathcal{ O}$ in ${\mathcal{H}}$, $\mathcal{
N}(\mathcal{ O},\varepsilon)$.
For $x\in\mathcal{ O}$, let  $\mathscr{L}_s(N\mathcal{O}_x)$ denote the space
of those operators $S\in \mathscr{L}(N\mathcal{ O}_x)$ which are self-adjoint
with respect to the inner product $(\!(\cdot, \cdot)\!)_x$, i.e.
$(\!(S_xu, v)\!)_x=(\!(u, S_xv)\!)_x$ for all $u, v\in N\mathcal{
O}_x$. Then we have a $C^1$ vector bundle $\mathscr{L}_s(N\mathcal{ O})\to
\mathcal{O}$ whose fiber at $x\in\mathcal{ O}$ is given by
$\mathscr{L}_s(N\mathcal{ O}_x)$.

\begin{hypothesis}\label{hyp:Bi.3.19}
{\rm Under Hypothesis~\ref{hyp:S.6.2}, let for some $x_0\in\mathcal{ O}$ the pair
$(\mathcal{L}\circ\exp|_{N\mathcal{O}(\varepsilon)_{x_0}},  N\mathcal{O}(\varepsilon)_{x_0})$
satisfy the corresponding conditions with Hypothesis~\ref{hyp:1.1} with $X=H=N\mathcal{O}(\varepsilon)_{x_0}$.
(For this goal we only need require that the pair $(\mathcal{L}\circ\exp_{x_0},  B_{T_{x_0}\mathcal{H}}(0,\varepsilon))$
satisfy the corresponding conditions with Hypothesis~\ref{hyp:1.1} with $X=H=T_{x_0}\mathcal{H}$
by \cite[Lemma~2.8]{Lu7}.) Let $\widehat{\mathcal{L}}\in C^1(\mathcal{H},\mathbb{R})$  be  $G$-invariant, have
a critical orbit $\mathcal{O}$, and also satisfy:
\begin{enumerate}
\item[(i)] the gradient $\nabla(\widehat{\mathcal{L}}\circ\exp|_{B_{T_{x_0}\mathcal{H}}(0,\varepsilon)})$ is G\^ateaux differentiable, and its derivative
at any $u\in B_{T_{x_0}\mathcal{H}}(0,\varepsilon)$,
$$
d^2(\widehat{\mathcal{L}}\circ\exp|_{B_{T_{x_0}\mathcal{H}}(0,\varepsilon)})(u)\in\mathscr{L}_s(T_{x_0}\mathcal{H}),
$$
    is also a compact linear operator.
\item[(ii)]  $B_{T_{x_0}\mathcal{H}}(0,\varepsilon)\to \mathscr{L}_s(T_{x_0}\mathcal{H}),\;u\mapsto d^2(\widehat{\mathcal{L}}\circ\exp|_{B_{T_{x_0}\mathcal{H}}(0,\varepsilon)})(u)$
is continuous at $0\in T_{x_0}\mathcal{H}$.
(Thus the assumptions on $\mathcal{G}$ assure that the functionals
$\mathcal{L}_{\lambda}:=\mathcal{L}-\lambda\widehat{\mathcal{L}}$,
$\lambda\in\mathbb{R}$, also satisfy the conditions of
\cite[Theorems~2.21 and 2.22]{Lu7}.)
\end{enumerate}}
\end{hypothesis}

Under Hypothesis~\ref{hyp:Bi.3.19}, we say $\mathcal{O}$ to be a {\it bifurcation $G$-orbit
with parameter $\lambda^\ast$} of  the equation
\begin{equation}\label{e:Bi.3.11}
\mathcal{L}'(u)=\lambda\widehat{\mathcal{L}}'(u),\quad u\in \mathcal{H}
\end{equation}
if for any $\varepsilon>0$ and for any neighborhood $\mathscr{U}$ of $\mathcal{O}$ in $\mathcal{H}$
there exists a solution $G$-orbit $\mathcal{O}'\ne \mathcal{O}$ in $\mathscr{U}$ of
(\ref{e:Bi.3.11}) with some $\lambda\in (\lambda^\ast-\varepsilon, \lambda^\ast+\varepsilon)$.
Equivalently, for some (and so any) fixed $x_0\in\mathcal{O}$ there exists a sequence $(\lambda_n, u_n)\subset (\lambda^\ast-\varepsilon, \lambda^\ast+\varepsilon)\times \mathcal{H}$ such that
\begin{equation}\label{e:Bi.3.11.1}
(\lambda_n, u_n)\to(\lambda^\ast, x_0),\quad \mathcal{L}'(u_n)=\lambda_n\mathcal{G}'(u_n)
\quad\hbox{and}\quad u_n\notin \mathcal{O}\quad\forall n.
\end{equation}

For any $x_0\in\mathcal{O}$, $\mathcal{S}_{x_0}:=\exp_{x_0}({N\mathcal{O}(\varepsilon)_{x_0}})$
is a $C^2$ {\bf slice} for the action of $G$ on $\mathcal{H}$
in the following sense:
\begin{enumerate}
\item[(i)]  $\mathcal{S}_{x_0}$ is a $C^2$ submanifold of $\mathcal{H}$ containing $x_0$, and the tangent space
$T_{x_0}\mathcal{S}_{x_0}\subset T_{x_0}\mathcal{H}$ is a closed complement to $T_{x_0}\mathcal{O}$;
\item[(ii)] $G\cdot\mathcal{S}_{x_0}$  is a neighborhood of the orbit $\mathcal{O}=G\cdot x_0$, i.e., the orbit of every
$x\in\mathcal{H}$ sufficiently close to $x_0$ must intersect $\mathcal{S}_{x_0}$ and $(G\cdot x)\pitchfork\mathcal{S}_{x_0}$;
\item[(iii)] $\mathcal{S}_{x_0}$ is invariant under the isotropy group $G_{x_0}$.
\end{enumerate}
(The $C^1$ $G$-action is sufficient for (ii).)
Clearly, a point $u\in {N\mathcal{O}(\varepsilon)_{x_0}}$ near $0_{x_0}\in {N\mathcal{O}(\varepsilon)_{x_0}}$
is a critical point of $\mathcal{L}_\lambda\circ\exp|_{N\mathcal{O}(\varepsilon)_{x_0}}$
if and only if $x:=\exp_{x_0}(u)$ is a critical point of  $\mathcal{L}_\lambda|_{\mathcal{S}_{x_0}}$.
Since $d\mathcal{L}_\lambda(x)[\xi]=0\;\forall \xi\in T_{x}(G\cdot x)$ and
$T_x\mathcal{H}=T_{x}(G\cdot x)\oplus T_x\mathcal{S}_{x_0}$, we deduce
\begin{equation}\label{e:Bi.3.11.2}
d\mathcal{L}_\lambda(x)=0\quad\Longleftrightarrow\quad d(\mathcal{L}_\lambda\circ\exp|_{N\mathcal{O}(\varepsilon)_{x_0}})(u)=0.
\end{equation}

Note that for any $x_0\in\mathcal{O}$ the orthogonal complementary of $T_{x_0}\mathcal{O}$ in $T_{x_0}\mathcal{H}$, $N\mathcal{O}_{x_0}$,
 is an invariant subspace of
  $$
 \mathcal{L}''_\lambda(x_0):=d^2(\mathcal{L}\circ\exp|_{B_{T_{x_0}\mathcal{H}}(0,\varepsilon)})(0)\quad\forall\lambda\in\mathbb{R}
 $$
 (in particular{\small
 $\mathcal{L}''(x_0):=d^2(\mathcal{L}\circ\exp|_{B_{T_{x_0}\mathcal{H}}(0,\varepsilon)})(0)$ and $\widehat{\mathcal{L}}''(x_0):=d^2(\widehat{\mathcal{L}}\circ\exp|_{B_{T_{x_0}\mathcal{H}}(0,\varepsilon)})(0)$)}.
 Let $\mathcal{L}''_\lambda(x_0)^\bot$  (resp.  $\mathcal{L}''(x_0)^\bot$,  $\widehat{\mathcal{L}}''(x_0)^\bot$) denote
the restriction self-adjoint operator of $\mathcal{L}''_\lambda(x_0)$ (resp. $\mathcal{L}''(x_0)$, $\widehat{\mathcal{L}}''(x_0)$)
from  $N\mathcal{O}_{x_0}$ to itself. Then $\mathcal{L}''_\lambda(x_0)^\bot=d^2(\mathcal{L}_\lambda\circ\exp|_{N\mathcal{O}(\varepsilon)_{x_0}})(0)$ and
$$
\mathcal{L}''(x_0)^\bot=d^2(\mathcal{L}\circ\exp|_{N\mathcal{O}(\varepsilon)_{x_0}})(0),\quad
\widehat{\mathcal{L}}''(x_0)^\bot=d^2(\widehat{\mathcal{L}}\circ\exp|_{N\mathcal{O}(\varepsilon)_{x_0}})(0).
$$
Applying Corollary~\ref{cor:Bi.2.4.1} to
$(\mathcal{L}\circ\exp|_{N\mathcal{O}(\varepsilon)_{x_0}}, \widehat{\mathcal{L}}\circ\exp|_{N\mathcal{O}(\varepsilon)_{x_0}},
N\mathcal{O}(\varepsilon)_{x_0})$ and using (\ref{e:Bi.3.11.2}) we obtain:

\begin{theorem}\label{th:Bi.3.20.1}
Under Hypothesis~\ref{hyp:Bi.3.19},
suppose that  $\lambda^\ast\in\mathbb{R}$  an isolated eigenvalue of
\begin{equation}\label{e:Bi.3.12}
\mathcal{L}''(x_0)^\bot v-\lambda\widehat{\mathcal{L}}''(x_0)^\bot v=0,\quad v\in N\mathcal{O}_{x_0},
\end{equation}
and  that $\widehat{\mathcal{L}}''(x_0)^\bot$ is either semi-positive or semi-negative.
 Then $\mathcal{O}$ is a  bifurcation $G$-orbit
with parameter $\lambda^\ast$ of  the equation
(\ref{e:Bi.3.11})
 and  one of the following alternatives occurs:
\begin{enumerate}
\item[\rm (i)] $\mathcal{O}$ is not an isolated critical orbit of $\mathcal{L}_{\lambda^\ast}$;

\item[\rm (ii)]  for every $\lambda\in\mathbb{R}$ near $\lambda^\ast$ there is a critical point
$u_\lambda\notin\mathcal{O}$ of $\mathcal{L}_{\lambda}$ converging to $x_0$ as $\lambda\to\lambda^\ast$;

\item[\rm (iii)] there is an one-sided  neighborhood $\Lambda$ of $\lambda^\ast$ such that
for any $\lambda\in\Lambda\setminus\{\lambda^\ast\}$,
$\mathcal{L}_{\lambda}$  has at least two critical points which sit in $\mathcal{S}_{x_0}\setminus\mathcal{O}$ converging to
$x_0$ as $\lambda\to\lambda^\ast$.
\end{enumerate}
\end{theorem}

Suppose that $\mathcal{L}''(x_0)^\bot$ is invertible, or equivalently
 ${\rm Ker}(\mathcal{L}''(x_0))=T_{x_0}\mathcal{O}$. Then
$0$ is not an eigenvalue  of (\ref{e:Bi.3.12})
and $\lambda\in\mathbb{R}\setminus\{0\}$
is an eigenvalue  of (\ref{e:Bi.3.12}) if and only if $1/\lambda$
is an eigenvalue  of compact linear self-adjoint operator
$$
L_{x_0}:=[\mathcal{L}''(x_0)^\bot ]^{-1}\widehat{\mathcal{L}}''(x_0)^\bot\in\mathscr{L}_s(N\mathcal{O}_{x_0}).
$$
Hence $\sigma(L_{x_0})\setminus\{0\}=\{1/\lambda_k\}_{k=1}^\infty\subset\mathbb{R}$
with $|\lambda_k|\to \infty$, and each $1/\lambda_k$ has finite multiplicity.
Let $N\mathcal{O}_{x_0}^k$ be the eigenspace corresponding to $1/\lambda_k$ for $k\in\mathbb{N}$.
Then
\begin{eqnarray}\label{e:Bi.3.13}
&&N\mathcal{O}_{x_0}^0:={\rm Ker}(L_{x_0})={\rm Ker}(\widehat{\mathcal{L}}''(x_0)^\bot),\nonumber\\
&&N\mathcal{O}_{x_0}^k={\rm Ker}(I/\lambda_k-L_{x_0})={\rm Ker}(\mathcal{L}''(x_0)^\bot-\lambda_k\widehat{\mathcal{L}}''(x_0)^\bot),\;
\forall k\in\mathbb{N},
\end{eqnarray}
and $N\mathcal{O}_{x_0}=\oplus^\infty_{k=0}N\mathcal{O}_{x_0}^k$.
Since $T_{x_0}\mathcal{O}\subset{\rm Ker}(\mathcal{L}''(x_0))\cap{\rm Ker}(\widehat{\mathcal{L}}''(x_0))$,
for $\lambda\in\mathbb{R}$, $\mathcal{O}$ is a nondegenerate critical orbit
of $\mathcal{L}_\lambda$ if and only if $\lambda$ is not an eigenvalue  of (\ref{e:Bi.3.12}).
Applying Corollary~\ref{cor:Bi.2.4.2} to
$(\mathcal{L}\circ\exp|_{N\mathcal{O}(\varepsilon)_{x_0}}, \widehat{\mathcal{L}}\circ\exp|_{N\mathcal{O}(\varepsilon)_{x_0}},
N\mathcal{O}(\varepsilon)_{x_0})$ and using (\ref{e:Bi.3.11.2}) we obtain:

\begin{theorem}\label{th:Bi.3.20.2}
Under Hypothesis~\ref{hyp:Bi.3.19},
suppose that  $\mathcal{L}''(x_0)^\bot$ is invertible
and that $\lambda^\ast=\lambda_{k_0}$ is an eigenvalue of (\ref{e:Bi.3.12}) as above.
Then the conclusions of Theorem~\ref{th:Bi.3.20.1} hold true if
  one of the following two conditions is also satisfied:
 \begin{enumerate}
\item[\rm (a)] $\mathcal{L}''(x_0)^\bot$ is positive;
 \item[\rm (b)] each $N\mathcal{O}_{x_0}^k$ in (\ref{e:Bi.3.13}) with $L=[\mathcal{L}''(0)]^{-1}\mathcal{G}''(0)$
 is an invariant subspace of $\mathcal{L}''(x_0)^\bot$  (e.g. these are true if $\mathcal{L}''(x_0)^\bot$
 commutes with $\widehat{\mathcal{L}}''(x_0)^\bot$), and $\mathcal{L}''(x_0)^\bot$
 is either positive definite or negative one on $N\mathcal{O}_{x_0}^{k_0}$.
 \end{enumerate}
\end{theorem}

Of course, corresponding to Theorem~\ref{th:Bi.2.4} there is a more general result.
Moreover, if either $G_{x_0}=\mathbb{Z}_2$ acts on $N\mathcal{O}(\varepsilon)_{x_0}$ via the antipodal map
or $G_{x_0}=S^1$ acts orthogonally on $N\mathcal{O}_{x_0}$ and
the fixed point set of the induced $S^1$-action on $N\mathcal{O}_{x_0}^0$ is $\{0\}$,
we may apply Theorem~\ref{th:Bi.3.2} to
$(\mathcal{L}\circ\exp|_{N\mathcal{O}(\varepsilon)_{x_0}}, \widehat{\mathcal{L}}\circ\exp|_{N\mathcal{O}(\varepsilon)_{x_0}}, N\mathcal{O}(\varepsilon)_{x_0})$
near $0=0_{x_0}\in N\mathcal{O}(\varepsilon)_{x_0}$ to get a stronger result.

We can also apply to Corollaries~\ref{cor:Bif.2.2.1},~\ref{cor:Bif.2.2.2},
 Theorem~\ref{th:Bif.2.2.4-} and Corollary~\ref{cor:Bif.3.3}
to suitably chosen slices in the setting of \cite{Lu4} to get some  corresponding results.
It is better for us using this idea in concrete applications.
For Theorems~\ref{th:Bi.3.20.1}, \ref{th:Bi.3.20.2} there also exists a remark similar to Remark~\ref{rem:Bif.3.5}.

In applications, for example, to closed geodesics on Finsler or Riemannian manifolds,
Hypothesis~\ref{hyp:S.6.2} cannot be satisfied since the natural $S^1$ action on Hilbert manifold
$H^1(S^1, M)$ is only continuous. However, we can still use the above ideas to complete our aims.

\section{Bifurcations for potential operators of Banach-Hilbert regular functionals}\label{sec:BBH}
\setcounter{equation}{0}

The aim of this section is to generalize partial results in last three sections to
potential operators of Banach-Hilbert regular functionals on Banach spaces.
Different from bifurcation results in previous sections, these bifurcation theorems
are on Banach (rather than Hilbert) spaces.
There exist a few splitting theorems for Banach-Hilbert regular functionals
in the literature (though they were expressed with different terminologies),
for example, \cite[Theorem~1.2]{BoBu}, \cite[Theorem~2.5]{JM} and \cite[page 436, Theorem~1]{DiHiTr}
(and \cite[Theorem~1.4]{Tr1}
for the special nondegenerate case).
See Remarks~\ref{rem:BB.1},~\ref{rm:BB.4+} for some comments and relations among them.
We here choose the setting of the parameterized versions of
\cite[Theorem~1.2]{BoBu} given in Appendix~\ref{app:B}.
Actually, it is not hard to deduce a parameterized version
of \cite[Theorem~2.5]{JM} after  the corresponding form of nondegenerate case of it
is proved by slightly more troublesome arguments.

In this section we always assume that  $H$ is a Hilbert space with inner product $(\cdot,\cdot)_H$
and the induced norm $\|\cdot\|$,  $X$ is a Banach space with
norm $\|\cdot\|_X$, and that they satisfy the condition (S) in Appendix~\ref{app:B}, i.e.,
$X\subset H$ is dense in $H$ and  the inclusion $X\hookrightarrow H$ is continuous.

\begin{theorem}\label{th:BBHH.1}
Let $H$ and $X$ be as above,  $\Lambda\subset\mathbb{R}$
an open interval  containing  $\lambda^\ast$. For each $\lambda\in\Lambda$ let
 $\mathscr{L}_\lambda:B_X(0, \delta)\to\mathbb{R}$ be a $C^2$-functional having critical point $0\in B_X(0,\delta)$ and satisfying
  for some constant $M_\lambda>0$:
\begin{enumerate}
\item[\rm (1)] $|d\mathscr{L}_\lambda(x)[u]|\le M_\lambda\|u\|,\;\forall x\in B_X(0, \delta)$, $\forall u\in X$.
\item[\rm (2)] $|d^2\mathscr{L}_\lambda(x)[u,v]|\le M_\lambda\|u\|\cdot\|v\|,\;\forall x\in B_X(0, \delta)$, $\forall u,v\in X$; and therefore
there exists a map $B_\lambda:B_X(0, \delta)\to \mathscr{L}_s(H)$  defined by
\begin{eqnarray*}
d^2\mathscr{L}_\lambda(x)[u,v]=(B_\lambda(x)u,v)_H,\quad \forall x\in B_X(0, \delta),\;\forall u,v\in X.
\end{eqnarray*}
\item[\rm (3)] For each $\lambda\in\Lambda$ there is a $C^2$ map $A_\lambda:B_X(0, \delta)\to X$  such that
$d\mathscr{L}_\lambda(x)[u]=(A_\lambda(x),u)_H$ for all $(x,u)\in B_X(0, \delta)\times X$.
(That is, the condition (c') in Theorem~\ref{th:BB.3} holds.)
\end{enumerate}
   Suppose also that the following conditions are satisfied:
 \begin{enumerate}
\item[\rm (a)]  The condition (sd') (or (f1) and (sf2)-(sf3)) in Theorem~\ref{th:BB.5} is satisfied.
(These naturally hold if $\Lambda\times B_X(0, \delta)\ni (\lambda, u)\mapsto A(\lambda, u):=A_\lambda(u)\in X$ is $C^3$.)
\item[\rm (b)] $\sigma(B_{\lambda^\ast}(0)|_X)\setminus\{0\}$
is bounded away from the imaginary axis, and $H^0_{\lambda^\ast}:={\rm Ker}({B}_{\lambda^\ast}(0))$
is a nontrivial subspace of finite dimension contained in $X$. (Thus $0<\nu_{\lambda^\ast}<\infty$.)
\item[\rm (c)] For each $\lambda\in\Lambda\setminus\{\lambda^\ast\}$,
$\sigma(B_{\lambda}(0)|_X)$ is bounded away from the imaginary axis,
and $H^0_{\lambda}:={\rm Ker}({B}_\lambda(0))=\{0\}$.
\item[\rm (d)] The negative definite space  $H^-_{\lambda}$
of each ${B}_\lambda(0)$ is of finite dimension (and hence is contained in $X$ by  Lemma~\ref{lem:BB.8}).
\item[\rm (e)] The Morse indexes of $\mathscr{L}_\lambda$
at $0\in H$ 
 take, respectively, values $\mu_{\lambda^\ast}$ and $\mu_{\lambda^\ast}+\nu_{\lambda^\ast}$
 as $\lambda\in\mathbb{R}$ varies in
 two deleted half neighborhoods  of $\lambda^\ast$,
where $\mu_{\lambda}$ and $\nu_{\lambda}$ are the Morse index and the nullity of  $\mathscr{L}_{\lambda}$
at $0$ (as defined below Lemma~\ref{lem:BB.2}), respectively.
 \end{enumerate}
   Then  one of the following alternatives occurs:
\begin{enumerate}
\item[\rm (i)] $0\in X$ is not an isolated critical point of $\mathscr{L}_{\lambda^\ast}$.

\item[\rm (ii)]  For every $\lambda\in\Lambda$ near $\lambda^\ast$ there is a nonzero critical point
$x_\lambda$ of $\mathscr{L}_{\lambda}$ converging to $0$ as $\lambda\to\lambda^\ast$.

\item[\rm (iii)] There is an one-sided  neighborhood $\Lambda$ of $\lambda^\ast$ such that
for any $\lambda\in\Lambda\setminus\{\lambda^\ast\}$,
$\mathscr{L}_{\lambda}$ has at least two nonzero critical points converging to
zero as $\lambda\to\lambda^\ast$.
\end{enumerate}
In particular,  $(\lambda^\ast, 0)\in\mathbb{R}\times X$ is a bifurcation point  for the equation
\begin{equation}\label{e:BBHH.1}
d\mathscr{L}_\lambda(0)=0,\; x\in B_X(0, \delta).
\end{equation}
Conversely, suppose that the above parameter space $\Lambda$ is a metric space
and that the corresponding assumptions (1)-(3),  the first condition in (b), and (d') in Theorem~\ref{th:BB.3}  are satisfied.
 Then $\nu_{\lambda^\ast}>0$ provided that $(\lambda^\ast, 0)\in\Lambda\times X$ is a bifurcation point of (\ref{e:BBHH.1}).
\end{theorem}

\begin{proof}
{\bf Step 1} ({\it Prove the first claim}).
Since either $\sigma(B_{\lambda}(0)|_X)$ or $\sigma(B_{\lambda}(0)|_X)\setminus\{0\}$ is bounded away from the imaginary axis,
there exists a direct sum decomposition of Banach spaces  $X=X_0^{\lambda}\oplus X_+^{\lambda}\oplus X_-^{\lambda}$,
 which corresponds to the spectral sets $\{0\}$, $\sigma_+(B_{\lambda}(0)|_X)$
 and $\sigma_-(B_{\lambda}(0)|_X)$.  Denote by $P_0^{\lambda}$, $P_+^{\lambda}$ and $P_-^{\lambda}$ the corresponding
  projections to this decomposition.

Let $H^+_\lambda$, $H^-_\lambda$ and $H^0_\lambda$ be the positive definite, negative definite and zero spaces of
${B}_\lambda(0)$.  Denote by $P^\ast_\lambda$  the orthogonal projections onto $H^\ast_\lambda$,
 $\ast=+,-,0$.

By assumptions (b)-(d) and Lemma~\ref{lem:BB.8}, we have:
\begin{enumerate}
\item[1)] $X_0^{\lambda^\ast}=H^0_{\lambda^\ast}$, $X_-^{\lambda^\ast}=H^-_{\lambda^\ast}$
and $X_+^{\lambda^\ast}$ is a dense subspace in $H^+_{\lambda^\ast}$.

\item[2)] For each $\lambda\in\Lambda\setminus\{\lambda^\ast\}$,
$X_0^{\lambda}=H^0_{\lambda}=\{0\}$, $X_-^{\lambda}=H^-_{\lambda}$
and $X_+^{\lambda}$ is a dense subspace in $H^+_{\lambda}$.
\end{enumerate}
Hence the nullity $\nu_\lambda:=\dim X^\lambda_0$ and  Morse index $\mu_\lambda:=\dim X^\lambda_-$
of the functional $\mathscr{L}_{\lambda}$ at $0$ are equal to $\dim H_\lambda^0$  and
$\dim H^-_\lambda$ (that is, the nullity and Morse index of the quadratic form $({B}_\lambda(0) u,u)_H$ on $H$),
respectively, and all of them are finite numbers.

By Theorem~\ref{th:BB.5}  there exist small numbers $0<\epsilon<\delta$ and $\rho>0$ with
$[\lambda^\ast-\rho,\lambda^\ast+\rho]\subset\Lambda$,
 a (unique) $C^1$ map
$\mathfrak{h}:[\lambda^\ast-\rho,\lambda^\ast+\rho]\times B_{X}(0,\epsilon)\cap H_{\lambda^\ast}^0\to
X_0^{\lambda^\ast}\oplus H^-_{\lambda^\ast}$
satisfying
\begin{eqnarray}\label{e:BBHH.2}
\mathfrak{h}(\lambda, 0)=0\quad\hbox{and}\quad (id_X-P_0^{\lambda^\ast})A_{\lambda}(z+ \mathfrak{h}(\lambda,z))=0
\end{eqnarray}
for all $(\lambda,z)\in [\lambda^\ast-\rho,\lambda^\ast+\rho]\times B_{X}(0,\epsilon)\cap H_{\lambda^\ast}^0$,
and another $C^1$ map
$$
[\lambda^\ast- \rho,\lambda^\ast+\rho]\times B_X(0,\epsilon)
\to  X,\;({\lambda}, x)\mapsto \Phi_{{\lambda}}(x)
$$
such that
 $\Phi_{\lambda}$ is  a $C^1$  origin-preserving diffeomorphism from
$B_X(0,\epsilon)$ onto an open neighborhood  of $0$ in $X$ containing $B_X(0,\epsilon/4)$ and that
\begin{eqnarray}\label{e:BBHH.3}
\mathscr{L}_{\lambda}\circ\Phi_{\lambda}(x)=\|P^{\lambda^\ast}_+x\|_H^2-
\|P^{\lambda^\ast}_-x\|_H^2+ \mathscr{L}_{{\lambda}}(P^{\lambda^\ast}_0x+ \mathfrak{h}({\lambda}, P^{\lambda^\ast}_0x)),\,\forall x\in B_X(0,\epsilon).
\end{eqnarray}
 Moreover, the functional
\begin{eqnarray*}
\mathscr{L}_{\lambda}^\circ: B_{X}(0,\epsilon)\cap H_{\lambda^\ast}^0\to \mathbb{R},\;
z\mapsto\mathscr{L}_{\lambda}(z+ \mathfrak{h}({\lambda}, z))
\end{eqnarray*}
 is of class $C^{2}$,  has the first-order derivative at $z_0\in
B_{X}(0,\epsilon)\cap H_{\lambda^\ast}^0$,
 \begin{eqnarray*}
d\mathscr{L}^\circ_\lambda(z_0)[z]=\bigl(A_\lambda(z_0+ \mathfrak{h}(\lambda, z_0)), z\bigr)_H,\quad\forall z\in H_{\lambda^\ast}^0,
\end{eqnarray*}
and the map $[\lambda^\ast- \rho,\lambda^\ast+\rho]\times B_{X}(0,\epsilon)\cap H_{\lambda^\ast}^0\to \mathbb{R}\ni (\lambda,u)\mapsto \mathscr{L}_{\lambda}^\circ(u)\in\mathbb{R}$ is $C^1$.
Since the map $z\mapsto z+ \mathfrak{h}({\lambda}, z)$
induces an one-to-one correspondence
 between the critical points of  $\mathscr{L}_{\lambda}^\circ$ near $0\in H_{\lambda^\ast}^0$
and those of $\mathscr{L}_{\lambda}$ near $0\in X$, $0\in H_{\lambda^\ast}^0$
is an isolated critical point of  $\mathscr{L}_{\lambda}^\circ$ if and only if
$0\in X$ is such a critical point of $\mathscr{L}_{\lambda}$
(after shrinking $\rho>0$ if necessary).
By the assumption (c) and Claim~\ref{cl:BB.6+}  we get

\begin{claim}\label{cl:BBHH.5.2}
For each $\lambda\in[\lambda^\ast- \rho,\lambda^\ast+\rho]\setminus\{\lambda^\ast\}$,
$0\in B_{X}(0,\epsilon)\cap H_{\lambda^\ast}^0$ is a nondegenerate critical point of $\mathscr{L}_{\lambda}^\circ$.
\end{claim}

(Note that we have not used the condition (e)!)
By assumptions (b)-(e)  we obtain
\begin{eqnarray}\label{e:BBHH.8}
&&\nu_\lambda=0,\quad\forall\lambda\in \Lambda\setminus\{\lambda^\ast\},\\
&&\hbox{either}\; \mu_\lambda=\mu_{\lambda^\ast}\;\forall\lambda<\lambda^\ast,\;
\hbox{and}\; \mu_\lambda=\mu_{\lambda^\ast}+\nu_{\lambda^\ast}\;\forall\lambda>\lambda^\ast,\label{e:BBHH.9}\\
&&\hbox{or}\; \mu_\lambda=\mu_{\lambda^\ast}+\nu_{\lambda^\ast}\;\forall\lambda<\lambda^\ast,\;
\hbox{and}\; \mu_\lambda=\mu_{\lambda^\ast}\;\forall\lambda>\lambda^\ast. \label{e:BBHH.10}
\end{eqnarray}

For each $\lambda\in \Lambda\setminus\{\lambda^\ast\}$, (\ref{e:BBHH.8}) implies that
$0\in X$ is a nondegenerate (and hence isolated) critical point of $\mathscr{L}_\lambda$.
It follows from this and Theorem~\ref{th:BB.3} that
\begin{equation}\label{e:BBHH.11}
C_q(\mathscr{L}_\lambda, 0;{\bf K})=\delta_{q\mu_\lambda}{\bf K}\quad\forall
q=0, 1,\cdots.
\end{equation}
By Claim~\ref{cl:BBHH.5.2}, (\ref{e:BBHH.3}) and Corollary~\ref{cor:BB.6} we obtain that for any Abel group ${\bf K}$,
\begin{equation}\label{e:BBHH.12}
C_q(\mathscr{L}_\lambda, 0;{\bf K})\cong
C_{q-\mu_{\lambda^\ast}}(\mathscr{L}^{\circ}_\lambda, 0;{\bf K})\quad\forall
q=0, 1,\cdots.
\end{equation}
Then (\ref{e:BBHH.11}) and (\ref{e:BBHH.12}) lead to
\begin{eqnarray}\label{e:BBHH.13}
C_{j}(\mathscr{L}^\circ_{\lambda},0;{\bf K})=
\left\{\begin{array}{ll}
\delta_{j0}{\bf K},\quad\forall \lambda<\lambda^\ast,\\
\delta_{j\nu_{\lambda^\ast}}{\bf K},\quad\forall \lambda>\lambda^\ast
\end{array}\right.
\end{eqnarray}
for any $j\in\mathbb{N}_0$ if (\ref{e:BBHH.9}) holds, and
\begin{eqnarray}\label{e:BBHH.14}
C_{j}(\mathscr{L}^\circ_{\lambda},0;{\bf K})=
\left\{\begin{array}{ll}
\delta_{j0}{\bf K},\quad\forall \lambda>\lambda^\ast,\\
\delta_{j\nu_{\lambda^\ast}}{\bf K},\quad\forall \lambda<\lambda^\ast
\end{array}\right.
\end{eqnarray}
for any $j\in\mathbb{N}_0$ if (\ref{e:BBHH.10}) holds. These two groups of equalities and (\ref{e:Bi.2.17}) lead, respectively, to
\begin{eqnarray}\label{e:BBHH.14.1}
0\in H^0_{\lambda^\ast}\;\hbox{is a strict local}\left\{
\begin{array}{ll}
\hbox{minimizer of}\;\mathscr{L}^\circ_{\lambda},&\quad
\forall \lambda<\lambda^\ast,\\
\hbox{maximizer of}\;\mathscr{L}^\circ_{\lambda},&\quad
\forall \lambda>\lambda^\ast
\end{array}\right.
\end{eqnarray}
if (\ref{e:BBHH.9}) holds, and
\begin{eqnarray}\label{e:BBHH.15}
0\in H^0_{\lambda^\ast}\;\hbox{is a strict local}\left\{
\begin{array}{ll}
\hbox{maximizer of}\;\mathscr{L}^\circ_{\lambda},&\quad
\forall \lambda<\lambda^\ast,\\
\hbox{minimizer of}\;\mathscr{L}^\circ_{\lambda},&\quad
\forall \lambda>\lambda^\ast
\end{array}\right.
\end{eqnarray}
if (\ref{e:BBHH.10}) holds. The desired conclusions follow from
 Theorem~\ref{th:Bi.2.1} directly.

\vspace{4pt}\noindent
 {\bf Step 2} ({\it Prove the second claim}).
Under the assumptions of the converse part, we have  sequences
$(\lambda_n)\subset \Lambda$ and $(x_n)\subset X\setminus\{0\}$ such that
$\lambda_n\to\lambda^\ast$, $x_n\to 0$ and $d\mathscr{L}_{\lambda_n}(x_n)=0$ for all $n\in\N$.
Now by contradiction let $\nu_{\lambda^\ast}=0$.
Then we can use Theorem~\ref{th:BB.3} to get a neighborhood $\Lambda_0$ of $\lambda^\ast$ in $\Lambda$,
  $\epsilon\in (0, \delta)$ and  a $C^0$ map $\Lambda_0\times B_X(0,\epsilon)\ni(\lambda,u)\mapsto\varphi_\lambda(u)\in X$ such that
  each  $\varphi_\lambda$ is an origin-preserving $C^1$ diffeomorphism
   from $B_X(0,\epsilon)$ onto an open neighborhood of zero in $X$ containing $B_X(0,\epsilon/4)$
     and satisfies (\ref{e:BB.-1}). Clearly, we can assume that $(\lambda_n, x_n)\in \Lambda_0\times B_X(0,\epsilon/4)$
     for all $n\in\N$. Note that $y_n:=(\varphi_{\lambda_n})^{-1}(x_n)\in B_X(0,\epsilon)\setminus\{0\}$
is a critical point of $\mathscr{L}_{\lambda_n}$ for each $n\in\N$. By (\ref{e:BB.-1}) we deduce
$$
0=d\mathscr{L}_{\lambda_n}(x_n)\circ d\varphi_{\lambda_n}(y_n)=d(\mathscr{L}_{\lambda_n}\circ\varphi_{\lambda_n})(y_n)=B_{\lambda^\ast}(0)y_n
$$
and therefore $y_n=0$ and $x_n=0$ for all $n\in\N$. It contradicts the choices of $(x_n)$.
({\it Note}: It seems that the claim cannot be proved with the implicit function theorem because (d') in Theorem~\ref{th:BB.3}
does not imply the continuity of the map $(\lambda,u)\mapsto A_\lambda(u)$ near $(\lambda^\ast,0)\in\Lambda\times X$.
Of course, if the condition (d') is replaced by the first two in (sd') we can directly use the implicit function theorem
to get the conclusion.)
\end{proof}

Now, we give analogues of Corollaries~\ref{cor:Bi.2.4.1} and \ref{cor:Bi.2.4.2} and \ref{cor:Bi.3.2.1} under the following hypothesis.
(Related notations may be found in Appendix~\ref{app:B}. )

\begin{hypothesis}\label{hyp:BBH.1}
{\rm
Let $H$ and $X$ be as in Theorem~\ref{th:BBHH.1}, and  $\mathscr{L}, \widehat{\mathscr{L}}:B_X(0, \delta)\to\mathbb{R}$  two $C^2$ functionals
having critical point $0\in X$ and satisfying the conditions (a) and (b) above (\ref{e:BB.-2}). Denote by $B$ and $\widehat{B}$ the
corresponding operators defined as in (\ref{e:BB.-2}). Suppose that  there exist $C^2$ maps $A, \widehat{A}:B_X(0, \delta)\to X$
 verifying
$$
d\mathscr{L}(x)[u]=(A(x),u)_H\quad\hbox{and}\quad
d\widehat{\mathscr{L}}(x)[u]=(\widehat{A}(x),u)_H
$$
 for all $(x,u)\in B_X(0, \delta)\times X$ and that each of $A''$ and  $\widehat{A}''$  is either $C^1$ or uniformly continuous
  in some neighborhood of any line segment in $B_X(0, \delta)$.
Assume that  $\lambda^\ast\in\mathbb{R}$ is an isolated eigenvalue of
\begin{equation}\label{e:BBH.0}
B(0)v-\lambda \widehat{B}(0)v=0,\quad v\in H,
\end{equation}
and that for each $\lambda$ near $\lambda^\ast$ the operator $\mathfrak{B}_{\lambda}:=B(0)-\lambda \widehat{B}(0)$
satisfies the following properties:
\begin{enumerate}
\item[(i)] Either $\sigma(\mathfrak{B}_\lambda|_X)$ or $\sigma(\mathfrak{B}_\lambda|_X)\setminus\{0\}$ is bounded away from the imaginary axis.
\item[(ii)] $H^0_\lambda:={\rm Ker}(\mathfrak{B}_\lambda)\subset X$, and the negative definite space
 $H^-_{\lambda}$ of $\mathfrak{B}_\lambda$ is of finite dimension (and hence is contained in $X$ by  Lemma~\ref{lem:BB.8}).
\end{enumerate}}
\end{hypothesis}

Consider the bifurcation problem
near $(\lambda^\ast, 0)\in\mathbb{R}\times X$ for the equation
\begin{equation}\label{e:BBHH.8*}
d\mathscr{L}(u)=\lambda d\widehat{\mathscr{L}}(u),
\quad u\in B_X(0, \delta).
\end{equation}

The conditions (1)-(3) and (a)-(b) in Theorem~\ref{th:BBHH.1}
assure that Theorem~\ref{th:BB.5} can be used in the proof of Theorem~\ref{th:BBHH.1}.
Under Hypothesis~\ref{hyp:BBH.1}  Theorem~\ref{th:BB.5} is still true by Remark~\ref{rm:BB.5}.

\begin{corollary}\label{cor:BBH.2}
Under Hypothesis~\ref{hyp:BBH.1}, suppose:
\begin{enumerate}
\item[\rm (a)]  ${B}(0)\in \mathscr{L}_s(H)$ has a decomposition $P(0)+Q(0)$, where $P(0)\in\mathscr{L}_s(H)$ is positive definite and $Q(0)\in\mathscr{L}_s(H)$ is compact.
\item[\rm (b)] $\widehat{B}(0)\in \mathscr{L}_s(H)$ is compact, and either semi-positive or semi-negative.
\end{enumerate}
  Then  one of the following alternatives occurs:
 \begin{enumerate}
\item[\rm (i)] $(\lambda^\ast,0)$ is not an isolated solution  in  $\{\lambda^\ast\}\times B_X(0, \delta)$ of the equation (\ref{e:BBHH.8*}).
\item[\rm (ii)]  For every $\lambda\in\mathbb{R}$ near $\lambda^\ast$ there is a nontrivial solution $u_\lambda$ of (\ref{e:BBHH.8*}) in $B_X(0, \delta)$,
which  converges to $0$  as $\lambda\to\lambda^\ast$.
\item[\rm (iii)] There is an one-sided  neighborhood $\Lambda$ of $\lambda^\ast$ such that
for any $\lambda\in\Lambda\setminus\{\lambda^\ast\}$, (\ref{e:BBHH.8*}) has at least two nontrivial solutions in
$B_X(0, \delta)$, which  converge to $0$  as $\lambda\to\lambda^\ast$.
\end{enumerate}
In particular, $(\lambda^\ast, 0)\in\mathbb{R}\times X$ is a bifurcation point  for (\ref{e:BBHH.8*}).
 \end{corollary}

\begin{proof}
Let $m_\lambda$ and $n_\lambda$ denote the Morse index and the nullity of
the $C^\infty$ functional  $H\ni u\mapsto\mathfrak{L}_\lambda(u):=([B(0)-\lambda \widehat{B}(0)]u,u)_H$
at $0\in H$. Since $\lambda^\ast$ is an isolated eigenvalue of (\ref{e:BBH.0}),
we may take a number $\rho>0$ so small that
$n_{\lambda}=0$ for each $\lambda\in [\lambda^\ast-\rho,\lambda^\ast+\rho]\setminus\{\lambda^\ast\}$.
Using (a) and the first condition in (b) we derive from Lemma~\ref{lem:BB.7}  that  $n_{\lambda^\ast}<\infty$.
 By (i) and (ii) in Hypothesis~\ref{hyp:BBH.1} we can shrink $\rho>0$ such that
\begin{equation}\label{e:BBH.9.0}
m_\lambda=\mu_\lambda<\infty\quad\hbox{and}\quad n_\lambda=\nu_\lambda<\infty,\quad\forall \lambda\in[\lambda^\ast- \rho,\lambda^\ast+\rho].
\end{equation}
Now all conditions of Theorem~\ref{th:BBHH.1}, except (e), are satisfied for
 family $\{\mathscr{L}_\lambda=\mathscr{L}_1-\lambda \mathscr{L}_2\,|\,\lambda\in [\lambda^\ast- \rho,\lambda^\ast+\rho]\}$.

By (a) and (b), $B(0)-\lambda \widehat{B}(0)=P(0)+ (Q(0)-\lambda \widehat{B}(0))$,
where $(Q(0)-\lambda \widehat{B}(0))$ is compact. It follows that the
 family $\{\mathfrak{L}_\lambda\,|\,\lambda\in [\lambda^\ast- \rho,\lambda^\ast+\rho]\}$
satisfies the (PS) condition on any closed ball.
 As in the proof of Corollary~\ref{cor:Bi.2.4.1}, we can replace
 $\mathcal{L}''(0)$ and  $\widehat{\mathcal{L}}''(0)$ by
 $B(0)$ and  $\widehat{B}(0)$, respectively,  to obtain
\begin{equation}\label{e:BBH.9.1}
m_\lambda=\left\{\begin{array}{ll}
m_{\lambda^\ast},&\quad\forall \lambda\in [\lambda^\ast-\rho,\lambda^\ast),\\
m_{\lambda^\ast}+ n_{\lambda^\ast},&\quad\forall\lambda\in (\lambda^\ast,\lambda^\ast+\rho]
\end{array}\right.
\end{equation}
if $\widehat{B}(0)\ge 0$, and
\begin{equation}\label{e:BBH.9.2}
m_\lambda=\left\{\begin{array}{ll}
m_{\lambda^\ast},&\quad\forall \lambda\in (\lambda^\ast,\lambda^\ast+\rho],\\
m_{\lambda^\ast}+n_{\lambda^\ast},&\quad\forall\lambda\in [\lambda^\ast-\rho,\lambda^\ast)
\end{array}\right.
\end{equation}
if $\widehat{B}(0)\le 0$. Hence (e) of Theorem~\ref{th:BBHH.1} follows from (\ref{e:BBH.9.0}) and (\ref{e:BBH.9.1})-(\ref{e:BBH.9.2}).
\end{proof}

\begin{corollary}\label{cor:BBH.3}
Under Hypothesis~\ref{hyp:BBH.1}, suppose that the following two conditions are satisfied:
 \begin{enumerate}
\item[\rm (a)] $B(0)$ is invertible, and $\widehat{B}(0)$ is compact.
 \item[\rm (b)] $B(0)\widehat{B}(0)=\widehat{B}(0)B(0)$,   and $B(0)$
 is either positive  or negative on ${\rm Ker}(B(0)-\lambda^\ast\widehat{B}(0))$.
  \end{enumerate}
Then the conclusions of Corollary~\ref{cor:BBH.2} hold true.
Moreover, if $B(0)$ is positive definite, the condition (b) is unnecessary.
\end{corollary}
\begin{proof}

The first paragraph in the proof of Corollary~\ref{cor:BBH.2} is still valid
after changing ``Using (a) and the first condition in (b)'' into ``Using (a)''.
Next, for the part
``\textsf{Let us choose $\delta>0$......, we get
 either} (\ref{e:Bi.2.7.5}) or (\ref{e:Bi.2.19-})."
 in the proof of Corollary~\ref{cor:Bi.2.4.2},  replacing
$\mathcal{L}''(0)$ and  $\widehat{\mathcal{L}}''(0)$ by  $B(0)$ and  $\widehat{B}(0)$, respectively,
we obtain that either (\ref{e:BBH.9.1}) or (\ref{e:BBH.9.2}) holds.
Then for $\mathscr{L}_\lambda:=\mathscr{L}-\lambda\widehat{\mathscr{L}}$,
 as in the proof of Theorem~\ref{th:BBHH.1} we may use (\ref{e:BBH.9.1})-(\ref{e:BBH.9.2}) and (\ref{e:BBH.9.0})
to derive the corresponding results with (\ref{e:Bi.2.7.6}) and (\ref{e:Bi.2.18}),
(\ref{e:BBHH.15}) and (\ref{e:BBHH.14.1}).
\end{proof}

 By Theorem~\ref{th:Bi.2.1E} and the proof of Theorem~\ref{th:BBHH.1}
we may directly get corresponding results with  Theorem~\ref{th:Bi.3.2} and Corollary~\ref{cor:Bi.3.2.1}.
Instead of these, as Theorem~\ref{th:Bif.2.2.4-} and Corollary~\ref{cor:Bif.3.3} we use  Theorem~\ref{th:Bif.2.2.2+} by Bartsch and Clapp \cite[\S4]{BaCl}
to obtain:

\begin{theorem}\label{th:BBH.6}
Under the assumptions of Theorem~\ref{th:BBHH.1},
let $G$ be a compact Lie group acting on $H$ orthogonally,
which induces a $C^1$ isometric action on $X$. Suppose that each $\mathscr{L}_\lambda$ is $G$-invariant and that
 $A_\lambda, B_\lambda$  are equivariant, and that
$H^0_{\lambda^\ast}:={\rm Ker}({B}_{\lambda^\ast}(0))$ only intersects at zero with the fixed point set $H^G$.
Then one of the following alternatives occurs:
\begin{enumerate}
\item[\rm (i)] $0\in X$ is not an isolated critical point of $\mathscr{L}_{\lambda^\ast}$;
\item[\rm (ii)] there exist left and right  neighborhoods $\Lambda^-$ and $\Lambda^+$ of $\lambda^\ast$ in $\mathbb{R}$
and integers $n^+, n^-\ge 0$, such that $n^++n^-\ge \ell(SH^0_{\lambda^\ast})$
and for $\lambda\in\Lambda^-\setminus\{\lambda^\ast\}$ (resp. $\lambda\in\Lambda^+\setminus\{\lambda^\ast\}$),
$\mathscr{L}_\lambda$ has at least $n^-$ (resp. $n^+$) distinct critical
$G$-orbits different from $0$, which converge to
 $0$ as $\lambda\to\lambda^\ast$.
 \end{enumerate}
In particular,  $(\lambda^\ast, 0)\in \Lambda\times X$
is a bifurcation point of (\ref{e:BBHH.1}).
\end{theorem}

By Remark~\ref{rem:Bif.3.4}, $\ell(SH^0_{\lambda^\ast})=\dim H^0_{\lambda^\ast}$ (resp. $\frac{1}{2}\dim H^0_{\lambda^\ast}$)
if the Lie group  $G$ is equal to $\mathbb{Z}_2=\{{\rm id}, -{\rm id}\}$ (resp. $S^1$).

 \begin{corollary}\label{th:BBH.7}
 Under the assumptions of one of Corollaries~\ref{cor:BBH.2} and \ref{cor:BBH.3}
let $G$ be a compact Lie group acting on $H$ orthogonally, which induces a $C^1$ isometric action on $X$.
Suppose that $\mathscr{L}, \widehat{\mathscr{L}}$ are $G$-invariant and that
 $A, B$ and $\widehat{A}, \widehat{B}$ are equivariant.
    Then the conclusions of Theorem~\ref{th:BBH.6} hold.
  \end{corollary}

With methods in  \cite[\S7.5]{Ba1}
it is possible to make further generalizations for the above results.

\appendix

\section{Appendix:\quad
 Parameterized splitting theorems by the author}\label{app:A}\setcounter{equation}{0}

Under Hypothesis~\ref{hyp:1.1} or Hypothesis~\ref{hyp:1.3}
let $H=H^+\oplus H^0\oplus H^-$ be the orthogonal decomposition
according to the positive definite, null and negative definite spaces of $B(0)$.
Denote by $P^\ast$ the orthogonal projections onto $H^\ast$, $\ast=+,0,-$.
$\nu:=\dim H^0$ and $\mu:=\dim H^-$ are called the {\it Morse index} and
{\it nullity} of the critical point $0$. In particular, if $\nu=0$ the critical point $0$ is
said to be {\it nondegenerate}.
Such a critical point is isolated by \cite[Theorem~2.13]{Lu7} (resp.
by Claim~2(i) and the arguments in Step 3 in the proof of \cite[Theorem~1.1]{Lu1}])
for the case of Hypothesis~\ref{hyp:1.1} (resp. Hypothesis~\ref{hyp:1.3}).
We have the following parameterized Morse-Palais lemma.

\begin{theorem}[\hbox{\cite[Theorem~2.9]{Lu7}}]\label{th:A.1}
Let  $\mathcal{L}\in C^1(U,\mathbb{R})$ satisfy Hypothesis~\ref{hyp:1.1}, and
let $\widehat{\mathcal{L}}\in C^1(U,\mathbb{R})$ satisfy Hypothesis~\ref{hyp:1.2}
without requirement that each $\widehat{\mathcal{L}}''(u)\in\mathscr{L}_s(H)$ is compact.
Suppose that the critical point  $0$  of $\mathcal{L}$ is a nondegenerate.
 Then there exist $\rho>0$, $\epsilon>0$, a family of open neighborhoods of $0$ in
$H$, $\{W_\lambda\,|\, |\lambda|\le\rho\}$
and a family of origin-preserving homeomorphisms, $\phi_\lambda: B_{H^+}(0,\epsilon) +
B_{H^-}(0,\epsilon)\to W_\lambda$, $|\lambda|\le\rho$,
 such that
$$
(\mathcal{L}+\lambda\widehat{\mathcal{L}})\circ\phi_\lambda(u^++ u^-)=\|u^+\|^2-\|u^-\|^2,
\quad\forall (u^+, u^-)\in B_{H^+}(0,\epsilon)\times
B_{H^-}(0,\epsilon).
$$
Moreover, $[-\rho,\rho]\times (B_{H^+}(0,\epsilon) +
B_{H^-}(0,\epsilon))\ni (\lambda, u)\mapsto \phi_\lambda(u)\in H$
is continuous, and $0$ is an isolated critical point of each $\mathcal{L}+\lambda\widehat{\mathcal{L}}$.
\end{theorem}

There also exist the parameterized Morse-Palais lemma around critical orbits (\cite[Theorem~2.21]{Lu7}),
 the parameterized splitting theorem around critical orbits (\cite[Theorem~2.22]{Lu7}) and the corresponding
parameterized shifting theorem (\cite[Corollary~2.27]{Lu7}).

Under Hypothesis~\ref{hyp:1.3}, it easily follows from the proof of \cite[Theorem~2.9]{Lu7}
and \cite[Remark~2.2]{Lu2} that Theorem~\ref{th:A.1} has the following corresponding version.

\begin{theorem}\label{th:A.4}
Let  $\mathcal{L}\in C^1(U,\mathbb{R})$ satisfy Hypothesis~\ref{hyp:1.3}, and
let $\widehat{\mathcal{L}}\in C^1(U,\mathbb{R})$ satisfy Hypothesis~\ref{hyp:1.4}
without requirement that each $\widehat{B}(x)\in\mathscr{L}_s(H)$ is compact.
Suppose that the critical point  $0$  of $\mathcal{L}$ is a nondegenerate.
Then the conclusions of Theorem~\ref{th:A.1} still hold.
\end{theorem}

We can also prove a more general version of this result if the condition
``${\rm Ker}(B_{\lambda^\ast}(0))\ne\{0\}$" in the assumptions of the following
theorem is changed into ``${\rm Ker}(B_{\lambda^\ast}(0))=\{0\}$".

\begin{theorem}\label{th:A.5-}
Let $H$, $X$ and $U$ be as in Hypothesis~\ref{hyp:1.3},
 and $\Lambda$ a topological space.
Let $\mathcal{L}_\lambda\in C^1(U, \mathbb{R})$, $\lambda\in\Lambda$, be a continuous family of functionals
    satisfying $\mathcal{L}'_\lambda(0)=0\;\forall\lambda$.
 For each $\lambda\in\Lambda$, assume that
there exists a map $A_\lambda\in C^1(U^X, X)$
such that $\Lambda\times U^X\ni (\lambda, x)\to A_\lambda(x)\in X$ is continuous,  and
 that for all $x\in U\cap X$  and $u, v\in X$,
 $$
 D\mathcal{L}_\lambda(x)[u]=(A_\lambda(x), u)_H\quad\hbox{and}\quad
(DA_\lambda(x)[u], v)_H=(B_\lambda(x)u, v)_H,
$$
 and  that  $B_\lambda$ has a decomposition
$B_\lambda=P_\lambda+Q_\lambda$, where for each $x\in U\cap X$,
 $P_\lambda(x)\in\mathscr{L}_s(H)$ is  positive definitive and
$Q_\lambda(x)\in\mathscr{L}_s(H)$ is compact.
   Let $0\in H$ be a degenerate critical point of
  some $\mathcal{L}_{\lambda^\ast}$, i.e.,  ${\rm Ker}(B_{\lambda^\ast}(0))\ne\{0\}$.
Suppose also that $P_\lambda$ and $Q_\lambda$ satisfy the following conditions:
    \begin{enumerate}
\item[\rm (i)]  For each $h\in H$, it holds that $\|P_{\lambda}(x)h-P_{\lambda^\ast}(0)h\|\to 0$
as $x\in U\cap X$ approaches to $0$ in $H$ and $\lambda\in\Lambda$ converges to $\lambda^\ast$.

 \item[\rm (ii)]  For some small $\delta>0$, there exist positive constants $c_0>0$ such that
$$
(P_\lambda(x)u, u)\ge c_0\|u\|^2\quad\forall u\in H,\;\forall x\in
\bar{B}_H(0,\delta)\cap X,\quad\forall\lambda\in \Lambda.
$$
 \item[\rm (iii)]  $Q_\lambda: U\cap X\to \mathscr{L}_s(H)$ is uniformly continuous at $0$  with respect to $\lambda\in \Lambda$.
  \item[\rm (iv)]  If $\lambda\in \Lambda$ converges to $\lambda^\ast$ then
  $\|Q_{\lambda}(0)-Q_{\lambda^\ast}(0)\|\to 0$.
   \item[\rm (v)] $(\mathcal{L}_{\lambda^\ast}, H, X, U, A_{\lambda^\ast}, B_{\lambda^\ast}=P_{\lambda^\ast}+ Q_{\lambda^\ast})$ satisfies Hypothesis~\ref{hyp:1.3}.
    \end{enumerate}
 Let $H^+_\lambda$, $H^-_\lambda$ and $H^0_\lambda$ be the positive definite, negative definite and zero spaces of
${B}_\lambda(0)$.  Denote by $P^0_\lambda$ and $P^\pm_\lambda$ the orthogonal projections onto $H^0_\lambda$ and $H^\pm_\lambda=H^+_\lambda\oplus H^-_\lambda$,
and by $X^\star_\lambda=X\cap H^\star_\lambda$ for $\star=+,-$, and by  $X^\pm_\lambda=P^\pm_\lambda(X)$.
 Then there exists a neighborhood $\Lambda_0$ of $\lambda^\ast$ in $\Lambda$,
$\epsilon>0$, a (unique) $C^0$ map
\begin{equation}\label{e:Spli.2.1.1}
\psi:\Lambda_0\times B_{H^0_{\lambda^\ast}}(0,\epsilon)\to X^\pm_{\lambda^\ast}
\end{equation}
which is $C^1$ in the second variable and
satisfies $\psi(\lambda, 0)=0\;\forall\lambda\in \Lambda_0$ and
\begin{equation}\label{e:Spli.2.1.2}
 P^\pm_{\lambda^\ast}A_\lambda(z+ \psi(\lambda,z))=0\quad\forall (\lambda,z)\in \Lambda_0
 \times B_{H^0_{\lambda^\ast}}(0,\epsilon),
 \end{equation}
an open neighborhood $W$ of $0$ in $H$ and a homeomorphism
\begin{eqnarray}\label{e:Spli.2.1.3}
&&\Lambda_0\times B_{H^0_{\lambda^\ast}}(0,\epsilon)\times
\left(B_{H^+_{\lambda^\ast}}(0, \epsilon) + B_{H^-_{\lambda^\ast}}(0, \epsilon)\right)\to \Lambda_0\times W,\nonumber\\
&&\hspace{20mm}({\lambda}, z, u^++u^-)\mapsto ({\lambda},\Phi_{{\lambda}}(z, u^++u^-))
\end{eqnarray}
satisfying $\Phi_{{\lambda}}(0)=0$, such that for each $\lambda\in \Lambda_0$,
\begin{eqnarray}\label{e:Spli.2.2}
&&\mathcal{L}_{\lambda}\circ\Phi_{\lambda}(z, u^++ u^-)=\|u^+\|^2-\|u^-\|^2+ \mathcal{
L}_{{\lambda}}(z+ \psi({\lambda}, z))\\
&& \quad\quad \forall (z, u^+ + u^-)\in  B_{H^0_{\lambda^\ast}}(0,\epsilon)\times
\left(B_{H^+_{\lambda^\ast}}(0, \epsilon) + B_{H^-_{\lambda^\ast}}(0, \epsilon)\right).\nonumber
\end{eqnarray}
 Moreover, there also hold: {\bf (i)}
 $$
d_z\psi(\lambda,z)=-[P^\pm_{\lambda^\ast}\circ({B}_{\lambda}(z+\psi(\lambda,z))|_{X^\pm_{\lambda^\ast}})]^{-1}
\circ(P^\pm_{\lambda^\ast}\circ({B}_\lambda(z+\psi(\lambda,z))|_{H^0_{\lambda^\ast}})).
$$

\noindent{\bf (ii)} The functional
\begin{equation}\label{e:Spli.2.3}
\mathcal{L}_{\lambda}^\circ: B_{H^0_{\lambda^\ast}}(0,\epsilon)\to \mathbb{R},\;
z\mapsto\mathcal{L}_{\lambda}(z+ \psi({\lambda}, z))
\end{equation}
 is of class $C^{2}$, its first-order and second-order differentials  at $z\in
B_{H^0}(0, \epsilon)$ are given by
 \begin{eqnarray}\label{e:Spli.2.4}
&& d\mathcal{L}^\circ_\lambda(z)[\zeta]=\bigl(A_\lambda(z+ \psi(\lambda, z)), \zeta\bigr)_H\quad\forall \zeta\in H^0,\\
  &&d^2\mathcal{L}^\circ_\lambda(0)[z,z']=\left(P^0_{\lambda^\ast}\bigr[{B}_\lambda(0)-
 {B}_\lambda(0)(P^\pm_{\lambda^\ast}{B}_{\lambda}(0)|_{X^\pm_{\lambda^\ast}})^{-1}
 (P^\pm_{\lambda^\ast}{B}_\lambda(0))\bigr]z, z'\right)_H,\nonumber\\
&& \hspace{40mm} \forall z,z'\in H^0.\label{e:Spli.2.5}
 \end{eqnarray}

\noindent{\bf (iii)}  If a compact Lie group $G$  acts on $H$ orthogonally, which induces  $C^1$ isometric actions on $X$,
 both $U$ and $\mathcal{L}_\lambda$ are $G$-invariant (and hence $H^0_\lambda$, $H^\pm_\lambda$
are $G$-invariant subspaces), then for each $\lambda\in \Lambda$,
the above maps $\psi(\lambda, \cdot)$  and $\Phi_{\lambda}(\cdot,\cdot)$  are
  $G$-equivariant, and $\mathcal{L}^\circ_{\lambda}$ is $G$-invariant.
  If for some $p\in\mathbb{N}$,  $\Lambda$ is a $C^p$ manifold and
$\Lambda\times U^X\ni (\lambda,x)\mapsto A(\lambda, x):=A_\lambda(x)\in X$ is $C^p$, then so is $\psi$.
\end{theorem}

\begin{proof}
Take $\eta>0$ so small that $B_{H^0_{\lambda^\ast}}(0,\eta)\oplus B_{X^\pm_{\lambda^\ast}}(0,\eta)\subset U^X$.
Since $P^\pm_{\lambda^\ast}\circ(B(0)|_{X^\pm_{\lambda^\ast}})$ is a Banach space isomorphism from $X^\pm_{\lambda^\ast}$ onto itself, applying
the implicit function theorem to the map
$$
\Lambda\times B_{H^0_{\lambda^\ast}}(0,\eta)\oplus B_{X^\pm_{\lambda^\ast}}(0,\eta)\to X^\pm_{\lambda^\ast},\;(\lambda,
z, x)\mapsto P^\pm_{\lambda^\ast}(A_\lambda(z+ x))
$$
near $(\lambda^\ast, 0)$ we can get conclusions (\ref{e:Spli.2.1.1})-(\ref{e:Spli.2.1.2}) and (i)-(ii).

By (v), $X^\star_{\lambda^\ast}=H^\star_{\lambda^\ast}$, $\star=0,-$, have finite dimensions.
Let $e_1,\cdots,e_m$ be an unit orthogonal basis of $H^0_{\lambda^\ast}\oplus H^-_{\lambda^\ast}$.
By the proof of \cite[Theorem~1.1]{Lu1} (or
the proof of \cite[Lemma~3.3]{Lu2}), we have
\begin{eqnarray}\label{e:Spli.2.5.1}
|(B_\lambda(x)u,v)_H-(B_{\lambda^\ast}(0)u,v)_H|\le \omega_\lambda(x)\|u\|\|v\|
\end{eqnarray}
for any $x\in U\cap X$, $u\in H^0_{\lambda^\ast}\oplus H^-_{\lambda^\ast}$ and $v\in H$, where
$$
\omega_\lambda(x)=\left(\sum^m_{i=1}\|P_\lambda(x)e_i-P_{\lambda^\ast}(0)e_i\|^2\right)^{1/2}+ \sqrt{m}\|Q_\lambda(x)-Q_{\lambda^\ast}(0)\|.
$$
Clearly, assumptions (i), (iii) and (iv) imply that $\omega_\lambda(x)\to 0$
as $x\in U\cap X$ approaches to $0$ in $H$ and $\lambda\in\Lambda$ converges to $\lambda^\ast$.
It is  clear that (\ref{e:Spli.2.5.1}) leads to
\begin{eqnarray}\label{e:Spli.2.5.2}
|(B_\lambda(x)u,v)_H|\le \omega_\lambda(x)\|u\|\|v\|,\quad\forall u\in H^+_{\lambda^\ast}, \;\forall
v\in H^0_{\lambda^\ast}\oplus H^-_{\lambda^\ast}
\end{eqnarray}
for any $x\in U\cap X$.
Moreover,  there exists $a_0>0$ such that $(B_{\lambda^\ast}(0)u, u)_H\ge 2a_0\|u\|^2\;\forall u\in H^+_{\lambda^\ast}$.
Take a neighborhood of $0\in H$, $V\subset U$,  and shrink $\Lambda_0\subset\Lambda$, such that
 $\omega_\lambda(x)<a_0$ for all $x\in V\cap X$ and $\lambda\in\Lambda_0$.
    It follows from this and (\ref{e:Spli.2.5.1}) that
\begin{eqnarray}\label{e:Spli.2.5.3}
(B_\lambda(x)v,v)_H\le -2a_0\|v\|^2+ \omega_\lambda(x)\|v\|^2\le -a_0\|v\|^2\quad\forall v\in H^-_{\lambda^\ast}
\end{eqnarray}
for all $x\in V\cap X$ and $\lambda\in\Lambda_0$. As in the proof of \cite[Lemma~3.4]{Lu2}(i), (see below),
by shrinking $V$ and $\Lambda_0$ we can get $a_1>0$ such that
for all $x\in V\cap X$ and $\lambda\in\Lambda_0$,
\begin{eqnarray}\label{e:Spli.2.5.4}
(B_\lambda(x)u,u)_H\ge a_1\|u\|^2\quad\forall u\in H^+_{\lambda^\ast}.
\end{eqnarray}
Since $\psi(\lambda,0)=0$, we can choose $\epsilon\in (0,\eta)$ such that
$z+\psi(\lambda,z)+ u^++u^-\in V$ for all $\lambda\in\Lambda_0$,
$z\in\bar{B}_{H^0_{\lambda^\ast}}(0,\epsilon)$ and $u^\star\in \bar{B}_{H^\star_{\lambda^\ast}}(0,\epsilon)$, $\star=+,-$.
For each $\lambda\in \Lambda_0$, we define a functional
\begin{eqnarray*}
{\bf F}_\lambda: B_{H^0_{\lambda^\ast}}(0,\epsilon)\oplus B_{X^\pm_{\lambda^\ast}}(0,\epsilon)\to \mathbb{R},
(z,u)\mapsto \mathcal{L}_{\lambda}(z+\psi({\lambda}, z)+u)- \mathcal{L}_{{\lambda}}(z+ \psi({\lambda}, z)).
\end{eqnarray*}
Then following the proof ideas of \cite[Lemma~3.5]{Lu2} and shrinking $\Lambda_0$ and $\epsilon>0$
(if necessary) we can obtain positive constants $\mathfrak{a}_1$ and $\mathfrak{a}_2$ such that
\begin{eqnarray*}
 &&(D_2{\bf F}_{{\lambda}}(z, u^+ + u^-_2)-D_2{\bf F}_{{\lambda}}(z, u^++ u^-_1))[u^-_2-u^-_1]\le
-\mathfrak{a}_1\|u^-_2-u^-_1\|^2,\\
&&D_2{\bf F}_{{\lambda}}(z, u^++u^-)[u^+-u^-]\ge  \mathfrak{a}_2(\|u^+\|^2+ \|u^-\|^2)
\end{eqnarray*}
for all $\lambda\in \Lambda_0$,
$z\in B_{H^0_{\lambda^\ast}}(0,\epsilon)$ and $u^+\in B_{X^+_{\lambda^\ast}}(0,\epsilon)$, $u^-\in B_{X^-_{\lambda^\ast}}(0,\epsilon)$.
Using \cite[Theorem~A.2]{Lu2} leads to the desired results.

\textsf{Finally, for completeness let us prove (\ref{e:Spli.2.5.4})  by  contradiction.}
 Then there exist
sequences $(x_n)\subset V\cap X$ with $\|x_n\|\to 0$, $(\lambda_n)\subset\Lambda$ with $\lambda_n\to\lambda^\ast$
and
$(u_n)\subset SH^+_{\lambda^\ast}$, such that
$(B_{\lambda_n}(x_n)u_n, u_n)_H<1/n$ for any $n=1,2,\cdots$.
Passing a subsequence, we can assume that
\begin{equation}\label{e:2.24}
(B_{\lambda_n}(x_n)u_n, u_n)_H\to\beta\le 0\;\hbox{as}\;n\to\infty,
\end{equation}
and that $u_n\rightharpoonup u_0$ in $H$. We claim: $u_0\ne
\theta$. In fact, by (ii) we have
\begin{eqnarray}\label{e:2.25}
(B_{\lambda_n}(x_n)u_n, u_n)_H&=&(P_{\lambda_n}(x_n)u_n, u_n)_H + (Q_{\lambda_n}(x_n)u_n, u_n)_H\nonumber\\
&\ge & c_0+ (Q_{\lambda_n}(x_n)u_n, u_n)_H\quad\forall n>n_0.
\end{eqnarray}
 Moreover, a direct computation gives
\begin{eqnarray}\label{e:2.26}
&&\!\!\!\!\!\quad |(Q_{\lambda_n}(x_n)u_n, u_n)_H-(Q_{\lambda^\ast}(0)u_0, u_0)_H|\\
&&\!\!\!\!\!=|((Q_{\lambda_n}(x_n)-Q_{\lambda^\ast}(0))u_n, u_n)_H+ (Q_{\lambda^\ast}(0)u_n, u_n)_H-(Q_{\lambda^\ast}(0)u_0, u_n)_H\nonumber\\
&&\hspace{70mm}+ (Q_{\lambda^\ast}(0)u_0, u_n-u_0)_H|\nonumber\\
&&\!\!\!\!\!\le \|Q_{\lambda_n}(x_n)-Q_{\lambda^\ast}(0)\|\cdot\|u_n\|^2+
\|Q_{\lambda^\ast}(0)u_n-Q_{\lambda^\ast}(0)u_0\|\cdot\|u_n\|\nonumber\\
&&\hspace{40mm}+
|(Q_{\lambda^\ast}(0)u_0, u_n-u_0)_H|\nonumber\\
&&\!\!\!\!\!\le \|Q_{\lambda_n}(x_n)-Q_{\lambda^\ast}(0)\|+ \|Q_{\lambda^\ast}(0)u_n-Q_{\lambda^\ast}(0)u_0\|+
|(Q_{\lambda^\ast}(0)u_0, u_n-u_0)_H|.\nonumber
\end{eqnarray}
Since $u_n\rightharpoonup u_0$ in $H$,
$\lim_{n\to\infty}|(Q_{\lambda^\ast}(0)u_0, u_n-u_0)_H|=0$. We have also
\begin{equation}\label{e:2.27}
\lim_{n\to\infty}\|Q_{\lambda^\ast}(0)u_n-Q_{\lambda^\ast}(0)u_0\|=0
\end{equation}
by the  compactness  of $Q_{\lambda^\ast}(0)$, and
\begin{equation}\label{e:2.28}
\lim_{n\to\infty}\|Q_{\lambda_n}(x_n)-Q_{\lambda^\ast}(0)\|\le
\lim_{n\to\infty}\|Q_{\lambda_n}(x_n)-Q_{\lambda_n}(0)\|+\lim_{n\to\infty}\|Q_{\lambda_n}(0)-Q_{\lambda^\ast}(0)\|=0
\end{equation}
 by the conditions (iii)-(iv).
Hence  (\ref{e:2.26})-(\ref{e:2.28}) give
\begin{equation}\label{e:2.29}
\lim_{n\to\infty}(Q_{\lambda_n}(x_n)u_n, u_n)_H=(Q_{\lambda^\ast}(0)u_0, u_0)_H.
\end{equation}
This and (\ref{e:2.24})-(\ref{e:2.25}) yield
$$
0\ge \beta=\lim_{n\to\infty}(B_{\lambda_n}(x_n)u_n, u_n)_H\ge c_0 +
(Q_{\lambda^\ast}(0)u_0,u_0)_H,
$$
which implies $u_0\ne\theta$. Note that $u_0$ also sits in $H^+_{\lambda^\ast}$.

As above, using (\ref{e:2.28}) we derive
\begin{eqnarray}\label{e:2.30}
&&|(Q_{\lambda_n}(x_n)u_0, u_n)_H- (Q_{\lambda^\ast}(0)u_0, u_0)_H|\\
&\le& |(Q_{\lambda_n}(x_n)u_0, u_n)_H- (Q_{\lambda^\ast}(0)u_0, u_n)_H|\notag\\&&+ |(Q_{\lambda^\ast}(0)u_0,
u_n)_H- (Q_{\lambda^\ast}(0)u_0,
u_0)_H|\nonumber\\
&\le&  \|Q_{\lambda_n}(x_n)-Q_{\lambda^\ast}(0)\|\cdot\|u_0\|+ |(Q_{\lambda^\ast}(0)u_0, u_n-u_0)_H
|\to 0.\nonumber
\end{eqnarray}
 Note that
\begin{eqnarray*}
&&(B_{\lambda_n}(x_n)(u_n-u_0), u_n-u_0)_H\\
&=&(P_{\lambda_n}(x_n)(u_n-u_0), u_n-u_0)_H + (Q_{\lambda_n}(x_n)(u_n-u_0), u_n-u_0)_H\\
&\ge& c_0\|u_n-u_0\|^2+ (Q_{\lambda_n}(x_n)(u_n-u_0), u_n-u_0)_H\\
&\ge& (Q_{\lambda_n}(x_n)u_n, u_n)_H-2(Q_{\lambda_n}(x_n)u_0, u_n)_H+ (Q_{\lambda^\ast}(0)u_0, u_0)_H.
\end{eqnarray*}
It follows from this and (\ref{e:2.29})-(\ref{e:2.30}) that
\begin{eqnarray}\label{e:2.31}
&&\liminf_{n\to\infty}(B_{\lambda_n}(x_n)(u_n-u_0),
u_n-u_0)_H\notag\\&&\ge\lim_{n\to\infty}(Q_{\lambda_n}(x_n)(u_n-u_0), u_n-u_0)_H = 0.
\end{eqnarray}
Note that $u_n\rightharpoonup u_0$ implies that $(P_{\lambda^\ast}(0)u_0,
u_n-u_0)_H\to 0$. By (D2) and (\ref{e:2.30}) we get
\begin{eqnarray*}
&&|(B_{\lambda_n}(x_n)u_0, u_n)_H-(B_{\lambda^\ast}(0)u_0, u_0)_H|\\
&=&|(P_{\lambda_n}(x_n)u_0, u_n)_H+ (Q_{\lambda_n}(x_n)u_0, u_n)_H\\&&- (P_{\lambda^\ast}(0)u_0, u_0)_H-
(Q_{\lambda^\ast}(0)u_0, u_0)_H|\\
&\le & |(P_{\lambda_n}(x_n)u_0, u_n)_H-(P_{\lambda^\ast}(0)u_0, u_0)_H|\\&&+ |(Q_{\lambda^\ast}(x_n)u_0,
u_n)_H-(Q_{\lambda^\ast}(0)u_0, u_0)_H| \\
&\le & |(P_{\lambda_n}(x_n)u_0, u_n)_H-(P_{\lambda^\ast}(0)u_0, u_n)_H|+|(P_{\lambda^\ast}(0)u_0,
u_n)_H-(P_{\lambda^\ast}(0)u_0, u_0)_H| \\
&&\quad + |(Q_{\lambda_n}(x_n)u_0, u_n)_H-(Q_{\lambda^\ast}(0)u_0, u_0)_H|
\\
&\le & \|P_{\lambda_n}(x_n)u_0- P_{\lambda^\ast}(0)u_0\| + |(P_{\lambda^\ast}(0)u_0, u_n-u_0)_H| \\
&&\quad + |(Q_{\lambda_n}(x_n)u_0, u_n)_H-(Q_{\lambda^\ast}(0)u_0, u_0)_H| \to 0.
\end{eqnarray*}
 Similarly, we have
$\lim_{n\to\infty}(B_{\lambda_n}(x_n)u_0, u_0)_H=(B_{\lambda^\ast}(0)u_0, u_0)_H$.
From these, (\ref{e:2.24}) and (\ref{e:2.31}) it follows that
\begin{eqnarray*}
0&\le& \liminf_{n\to\infty}(B_{\lambda_n}(x_n)(u_n-u_0), u_n-u_0)_H\\
&=& \liminf_{n\to\infty}[(B_{\lambda_n}(x_n)u_n,u_n)_H-2(B_{\lambda_n}(x_n)u_0, u_n)_H+
(B_{\lambda_n}(x_n)u_0,
u_0)_H]\\
&=&\lim_{n\to\infty}(B_{\lambda_n}(x_n)u_n,u_n)_H- (B_{\lambda^\ast}(0)u_0, u_0)_H
=\beta- (B_{\lambda^\ast}(0)u_0, u_0)_H.
\end{eqnarray*}
Namely, $(B_{\lambda^\ast}(0)u_0,u_0)_H\le\beta\le 0$. It contradicts to
the fact that $(B_{\lambda^\ast}(0)u, u)_H\ge 2a_0\|u\|^2\;\forall u\in H^+_{\lambda^\ast}$
 because $u_0\in H^+_{\lambda^\ast}\setminus\{0\}$.

 As usual, conclusions in (iii) can be easily obtained from the above proof.
 \end{proof}

\begin{remark}\label{rm:Spl.2.4}
{\rm
{\bf (i)} As in the proof of \cite[Claim~2.17]{Lu7} we may show:
if $0\in H$ is a nondegenerate critical point of $\mathcal{L}_\lambda$,
i.e., ${\rm Ker}(B_\lambda(0))=\{0\}$,
then $0\in H^0_{\lambda^\ast}$ is such a critical point of $\mathcal{L}^\circ_\lambda$ too.\\
{\bf (ii)}   Every critical point $z$ of $\mathcal{L}_{\lambda}^\circ$ in
$B_{H^0_{\lambda^\ast}}(0,\epsilon)$ yields a critical point $z+ \psi({\lambda}, z)$
of $\mathcal{L}_{\lambda}$, and  $z+ \psi({\lambda}, z)$ sits in $X$.
Moreover, if $z\to 0$ in $H^0_{\lambda^\ast}$, then $z+ \psi({\lambda}, z)\to 0$ in $X$.
Conversely, every critical point of $\mathcal{L}_{\lambda}|_X$ near $0\in X$
has the form $z+\psi({\lambda}, z)$, where $z\in B_{H^0_{\lambda^\ast}}(0,\epsilon)$ is a
critical point of $\mathcal{L}_{\lambda}^\circ$.
Clearly, if $0\in H$ is an isolated critical point of $\mathcal{L}_{\lambda}$,
then $0\in H^0_{\lambda^\ast}$ is also an isolated critical point of $\mathcal{L}_{\lambda}^\circ$.
Conversely, it might not be true, which is  different from \cite[Theorem~2.2]{Lu7}.}
\end{remark}

Let $\nu_\lambda$ and $\mu_\lambda$ denote the nullity and  Morse index
of the functional $\mathcal{L}_{\lambda}$ at $0$. They are finite numbers and are
equal to $\dim H_\lambda^0$  and $\dim H^-_\lambda$ (that is, the nullity and Morse index
of the quadratic form $({B}_\lambda u,u)$ on $H$), respectively.
 As in \cite[Corollary~2.6]{Lu2} we have

\begin{corollary}[Shifting]\label{cor:A.6}
Under the assumptions of Theorem~\ref{th:A.5-}, suppose
 for some $\lambda\in [\lambda^\ast-\delta, \lambda^\ast+\delta]$ that $0\in H$ is an isolated
 critical point of $\mathcal{L}_{{\lambda}}$ (thus $0\in H^0_{\lambda^\ast}$ is that of $\mathcal{L}_{{\lambda}}^\circ$). Then
\begin{eqnarray*}
C_q(\mathcal{L}_{{\lambda}}, 0;{\bf K})=C_{q-\mu_{\lambda^\ast}}(\mathcal{L}^\circ_{{\lambda}}, 0;{\bf K}),\quad\forall
q\in\mathbb{N}\cup\{0\}.
\end{eqnarray*}
\end{corollary}

\begin{corollary}\label{cor:A.5}
Under Hypothesis~\ref{hyp:Bif.2.2.0} the conclusions of Theorem~\ref{th:A.5-}
hold with family $\{\mathcal{L}_\lambda:=\mathcal{L}-\lambda\widehat{\mathcal{L}}\,|\,
\lambda\in\Lambda=\mathbb{R}\}$.
\end{corollary}

\begin{remark}\label{rm:Spl.2.4+}
{\rm If  Hypothesis~\ref{hyp:Bif.2.2.0} in Corollary~\ref{cor:A.5} is replaced by
Hypothesis~\ref{hyp:Bif.2.2.0+}, that is, we  only assume that $A: U^X\to X$  is
 G\^ateaux differentiable, and strictly Fr\'{e}chet differentiable at $0\in U^X$,
then we can still prove Corollary~\ref{cor:A.5} with weaker conclusions
that $\psi$ and $\mathcal{L}_{\lambda}^\circ$ are $C^{1-0}$ and $C^{2-0}$, respectively,
and that all $\psi(\lambda,\cdot)$ and $d\mathcal{L}_{\lambda}^\circ$ are strictly Fr\'echet differentiable
at $0\in H^0_{\lambda^\ast}$.}
\end{remark}

\begin{remark}\label{rm:Spl.2.5}
{\rm From the proofs of  \cite[Theorem 5.1.13]{Ch1} and \cite[Theorem 8.3]{MaWi} it is easily seen that
the classical splitting lemma for $C^2$ functionals has also the parameterized version as Theorem~\ref{th:A.5-}.
Let $U\subset H$ be as in Theorem~\ref{th:A.5-}, $\Lambda$ a  topological space,
and let $\mathcal{L}_\lambda\in C^2(U, \mathbb{R})$, $\lambda\in\Lambda$, be a  family of functionals
    satisfying $\mathcal{L}'_\lambda(0)=0\;\forall\lambda$, and be such that
  $\Lambda\times U\ni (\lambda,u)\mapsto\nabla\mathcal{L}_\lambda(u)\in H$ is  continuous.
 For some $\lambda^\ast\in\Lambda$, assume that $\mathcal{L}^{\prime\prime}_{\lambda^\ast}(0)$ is a Fredholm operator with nontrivial kernel.
 Then the conclusions of Theorem~\ref{th:A.5-} hold as long as we take $X=H$, $A_\lambda=\nabla \mathcal{L}_\lambda$ and
 $B_\lambda=\mathcal{L}^{\prime\prime}_{\lambda}$, and  replace the paragraph
 `` an open neighborhood $W$ of $0$ in $H$ and a homeomorphism ......
satisfying $\Phi_{{\lambda}}(0)=0$" by
 `` a $C^0$ map
\begin{eqnarray*}
\Lambda_0\times B_{H^0_{\lambda^\ast}}(0,\epsilon)\oplus
B_{H^+_{\lambda^\ast}}(0, \epsilon)\oplus B_{H^-_{\lambda^\ast}}(0, \epsilon)
\to  H,\quad ({\lambda}, u)\mapsto \Phi_{{\lambda}}(u)
\end{eqnarray*}
 such that for each $\lambda\in \Lambda_0$ the map
 $\Phi_{\lambda}$ is an  origin-preserving homeomorphism from
$B_{H^0_{\lambda^\ast}}(0,\epsilon)\oplus
B_{H^+_{\lambda^\ast}}(0, \epsilon)\oplus B_{H^-_{\lambda^\ast}}(0, \epsilon)$ onto
an open neighborhood $W_\lambda$ of $0$ in $H$".

If the kernel of $\mathcal{L}^{\prime\prime}_{\lambda^\ast}(0)$ is trivial,
we have a corresponding  parameterized Morse-Palais lemma.

Moreover, there also exist corresponding corollaries of Theorem~\ref{th:A.4},~\ref{th:A.5-} as above.

As noted in \cite[Remark~2.4]{Lu2}, these can directly follow from
 Theorem~\ref{th:A.4},~\ref{th:A.5-} and their corollaries if $\mathcal{L}^{\prime\prime}_{\lambda^\ast}(0)$
has also a finite dimensional negative definite space.}
\end{remark}

\section{Appendix:\quad
 Parameterized Bobylev-Burman splitting  lemmas}\label{app:B}\setcounter{equation}{0}

We here give parameterized versions of splitting  lemmas in \cite{BoBu}
in consistent notations with those of this paper.
Let $H$ be a Hilbert space with inner product $(\cdot,\cdot)_H$
and the induced norm $\|\cdot\|$,  $X$  a Banach space with
norm $\|\cdot\|_X$, such that
\begin{enumerate}
\item[(S)]   $X\subset H$ is dense in $H$ and
 the inclusion $X\hookrightarrow H$ is continuous.
\end{enumerate}

Suppose that a $C^2$ functional $\mathscr{L}:B_X(0, \delta)\to\mathbb{R}$
satisfies  for some constant $M>0$:
\begin{enumerate}
\item[(a)] $|d\mathscr{L}(x)[u]|\le M\|u\|,\;\forall x\in B_X(0, \delta)$, $\forall u\in X$.
\item[(b)] $|d^2\mathscr{L}(x)[u,v]|\le M\|u\|\cdot\|v\|,\;\forall x\in B_X(0, \delta)$, $\forall u,v\in X$.
\end{enumerate}

Then there is a map $B:B_X(0, \delta)\to \mathscr{L}_s(H)$ such that
\begin{eqnarray}\label{e:BB.-2}
d^2\mathscr{L}(x)[u,v]=(B(x)u,v)_H
\end{eqnarray}
for all $x\in B_X(0, \delta)$ and  $u,v\in X$.
 Suppose also:
 \begin{enumerate}
\item[(c)] There is a uniformly continuously differentiable map $A:B_X(0, \delta)\to X$ such that
$d\mathscr{L}(x)[u]=(A(x),u)_H$ for all $x\in B_X(0, \delta)$ and for all $u\in X$.
 (Hence for any $x\in B_X(0, \delta)$,
            (\ref{e:BB.-2}) implies that $B(x)(X)\subset X$
 and $B(x)y=dA(x)[y]\;\forall y\in X$. As an element in $\mathscr{L}(X)$ we write
$B(x)|_X=dA(x)$ without confusions.)
\item[(d)]  $B(\cdot)|_X: B_X(0, \delta)\to \mathscr{L}(X)$ is continuously differentiable,
           and also uniformly  differentiable in a neighborhood of the line segment
           $[0,1]u:=\{tu\,|\,0\le t\le 1\}$ for each $u\in B_X(0, \delta)$. (The latter is satisfied
           if $A$ is $C^3$!)
\item[(e)]  $B:B_X(0, \delta)\to \mathscr{L}_s(H)$ is uniformly continuous.
\end{enumerate}

A functional $\mathscr{L}$ satisfying the above conditions (a)-(e) is called {\bf $(B_X(0, \delta),
H)$-regular} in \cite{BoBu}.

\begin{remark}\label{rem:BB.1}
{\rm  The second condition in (d) did not appear in \cite{BoBu}. We consider that it is needed for the proof of
$C^1$-smoothness of the operator
$$
 B_X(0, \delta)\ni u\mapsto\mathcal{B}(u):=\int^1_0(1-t)A'(tu) dt=\int^1_0(1-t)B(tu)|_X dt\in \mathscr{L}(X),
 $$
which corresponds to the operator $A(\cdot):B\to S(E)$ in \cite[Lemma~1.2]{BoBu}. In fact,
(c) implies that  $B(\cdot)|_X=dA:B_X(0, \delta)\to\mathscr{L}(X)$ is
 uniformly continuous and hence that $\mathcal{B}: B_X(0, \delta)\to\mathscr{L}(X)$ is continuous.
(Actually, the latter claim may follow from the differentiability of the map $A:B_X(0, \delta)\to X$
 and the first condition in (d). Indeed, the former can still lead to $B(x)|_X=dA(x)$ for any
 $x\in B_X(0, \delta)$. Then the first condition in (d) implies that $A$ is $C^2$.
 Therefore for a fixed $u_0\in B_X(0, \delta)$ there are $M>0$ and a convex neighborhood $\mathscr{N}([0,1]u_0)$ of the compact
 subset $[0,1]u_0\subset B_X(0, \delta)$ such that $\|A''(x)\|_{\mathscr{L}(X,\mathscr{L}(X))}\le M$ for all $x\in\mathscr{N}([0,1]u_0)$.
 Take $\rho>0$ so small that $B_X(u_0, \rho)\subset B_X(0, \delta)$ and $[0,1]u\subset\mathscr{N}([0,1]u_0)$
 for all $u\in B_X(u_0, \rho)$. It follows from the mean value theorem that for $u_i\in B_X(u_0, \rho)$, $i=1,2$,
  \begin{eqnarray*}
  \|\mathcal{B}(u_2)-\mathcal{B}(u_1)\|_{\mathscr{L}(X)}\le
  \int^1_0(1-t)\|A'(tu_2)-A'(tu_1)\|_{\mathscr{L}(X)} dt\le M\|u_2-u_1\|_X.
  \end{eqnarray*}
 That is, $\mathcal{B}$ is locally Lipschitz.)

By the differential method with parameter integral we can only derive from the first assumption in (d) that
$\mathcal{B}$ has a G\^ateaux derivative at $u\in B_X(0, \delta)$,
$$
D\mathcal{B}(u)=\int^1_0(1-t)tA''(tu) dt\in \mathscr{L}(X,\mathscr{L}(X)).
$$
When $\dim X<\infty$, for any $0<\epsilon<\delta$, since the closure of $B_X(0, \delta-\epsilon)$
is a compact subset contained in $B_X(0, \delta)$,
$A''$ is uniformly continuous on $Cl(B_X(0, \delta-\epsilon))$,
it follows that $u\mapsto D\mathcal{B}(u)$ is continuous and so $\mathcal{B}$ is $C^1$.

If $\dim X=\infty$, without the second condition in (d) it is difficult for us to prove that
$D\mathcal{B}$ is continuous at $u$ except $u=0\in  B_X(0, \delta)$.
For  $u, u_0\in B_X(0, \delta)$, since
\begin{eqnarray*}
\|D\mathcal{B}(u)-D\mathcal{B}(u_0)\|_{\mathscr{L}(X,\mathscr{L}(X))}\le
\int^1_0(1-t)t\|A''(tu)-A''(tu_0)\|_{\mathscr{L}(X,\mathscr{L}(X))} dt
\end{eqnarray*}
and the second condition in (d) implies that $A''$ is uniformly  continuous in a neighborhood of the line segment
$[0,1]u_0$, we conclude that  $D\mathcal{B}(u)\to D\mathcal{B}(u_0)$ as $u\to u_0$.

When $A$ is $C^3$, $A'''$ is bounded in  a convex neighborhood of the line segment
$[0,1]u_0$. It follows from the mean value theorem that
$$
\sup_{t\in[0,1]}\|A''(tu)-A''(tu_0)\|_{\mathscr{L}(X,\mathscr{L}(X))}\to 0
\quad\hbox{as $u\to u_0$.}
$$
Hence $\mathcal{B}$ is $C^1$.

In summary,  \textsf{all assumptions of differentiability on $A$ and $B$ in (c)-(d) can be replaced by one of the following two conditions:}
\begin{enumerate}
 \item[(i)] $A$ is $C^3$.
 \item[(ii)]  $A$ is $C^2$ and
$A''$ is uniformly continuous in some neighborhood of the line segment
$[0,1]u:=\{tu\,|\,0\le t\le 1\}$ for each $u\in B_X(0, \delta)$.
\end{enumerate}

The above arguments suggest that  proofs as in
\cite[Theorem~6.5.4]{Ber}, \cite[Theorem~1.4]{Tr1}, \cite[page 436, Theorem~1]{DiHiTr} and
\cite[Theorem~1.1]{BoBu} are defective.

Finally, the condition (e) seems not to be needed, see proofs of Theorems~\ref{th:BB.3},~\ref{th:BB.4},~\ref{th:BB.5}.}
\end{remark}

In this paper, \textsf{by the spectrum of a linear operator on a real Banach space
we always mean one of its natural complex linear extension on the complexification of
the real Banach space} (\cite[page 14]{DaKr}). Similarly, if  others cannot be clearly explained in the real world,
we consider their complexification and then take invariant parts under complex adjoint on the complexification
spaces.

Let $\mathscr{S}(X):=\{L\in\mathscr{L}(X)\,|\,  (Lu,v)_H=(u,Lv)_H\;\forall u,v\in X\}$.
Recall the following Riesz lemma (cf. \cite[Lemma~1.3]{BoBu} and \cite[page 19]{DaKr}, \cite[page 54, Lemma~1]{Tr} and
 \cite[page 59, Propositions~2,3]{Tr}). 

\begin{lemma}\label{lem:BB.2}
 Under the assumption {\bf (S)}, let $D\in\mathscr{S}(X)$ and  its spectrum have a decomposition
 $$
 \sigma(D)=\{0\}\cup\sigma_+(D)\cup\sigma_-(D),
 $$
 where $\sigma_+(D)$ and $\sigma_-(D)$ are closed subsets of $\sigma(D)$ contained in
 the interiors of the left and right half planes, respectively.
  Then the space $X$ can be decomposed into a direct sum of
spaces $X_0$, $X_+$ and $X_-$, which are closed in $X$, invariant with respect to $D$, orthogonal in $H$, and
$\sigma(D_\ast)=\sigma_\ast(D)$ for $D_\ast=D|_{X_\ast}$, $\ast=+,-$.  The projections $P_0$, $P_+$ and $P_-$,
corresponding to this decomposition, are symmetric operators with respect to the inner product in $H$, and
satisfy
$$
(P_\ast)^2=P_\ast,\quad \ast=0,+,-,\quad P_0+P_++P_-=id_X,
$$
and any two of three projections have zero compositions,
and moreover $P_\ast$, $\ast=+,-,0$,  are expressible as a limit of power series in $D$.
In addition, square roots $S_-$ and $S_+$ to $-D_-$ and $D_+$
may be defined with the functional calculus, and are expressible in power
series in $D_-$ and $D_+$, and also symmetric with respect to the inner product in $H$.
\end{lemma}

 For the spectrum $\sigma(B(0)|_X)$ of the operator  $B(0)|_X\in\mathscr{L}(X)$ in (d) suppose:
\textsl{$\sigma(B(0)|_X)\setminus\{0\}$ is bounded away from the imaginary axis.}
By Lemma~\ref{lem:BB.2}, corresponding to the spectral sets $\{0\}$,
 $\sigma_+(B(0)|_X)$ and $\sigma_-(B(0)|_X)$
 we have a direct sum decomposition of Banach spaces,
 $X=X_0\oplus X_+\oplus X_-$.   Denote by  $P_0$, $P_+$ and $P_-$
  the projections corresponding to the decomposition.
 When $\sigma(B(0)|_X)$ does not intersect the imaginary axis,
 namely the operator $B(0)|_X\in\mathscr{L}(X)$  is  {\bf hyperbolic} (\cite{Uh0}),
  (in particular, $\dim X_0=0$),
  we say that the critical point $0$ is  {\bf nondegenerate}.
  (Such nondegenerate critical points are isolated by Morse lemma \cite[Theorem~1.1]{BoBu}.)
  Call $\nu:=\dim X_0$ and $\mu:=\dim X_-$
  the {\bf nullity} and the {\bf Morse index} of $0$, respectively.
(Using Lemma~\ref{lem:BB.2} we may prove
that $\dim X_-$ is the supremum of the dimensions of the vector subspaces
of $X$ on which $(B(0)u,u)_H$ is negative definite).
The following is a parametrized version of \cite[Theorem~1.1]{BoBu}.

\begin{theorem}\label{th:BB.3}
Let $\Lambda$ be a metric space, and
 $\mathscr{L}_\lambda:B_X(0, \delta)\to\mathbb{R}$, $\lambda\in\Lambda$,
 a family of $C^2$-functionals satisfying the conditions (a)-(b) and $d\mathscr{L}_\lambda(0)=0\;\forall\lambda$.
 Let $B_\lambda:B_X(0, \delta)\to \mathscr{L}_s(H)$ be defined by
\begin{eqnarray*}
d^2\mathscr{L}_\lambda(x)[u,v]=(B_\lambda(x)u,v)_H,\quad \forall x\in B_X(0, \delta),\;\forall u,v\in X.
\end{eqnarray*}
 Suppose:
  \begin{enumerate}
  \item[\rm (c')] For each $\lambda\in\Lambda$ there is a $C^2$ map $A_\lambda:B_X(0, \delta)\to X$  such that
$$
d\mathscr{L}_\lambda(x)[u]=(A_\lambda(x),u)_H,\quad\forall x\in B_X(0, \delta),
\quad\forall u\in X.
$$
(Therefore each $\mathscr{L}_\lambda$ is $C^3$.)

  \item[\rm (d')] For each point $(\lambda_0, u_0)\in\Lambda\times B_X(0, \delta)$ there exist
  neighborhoods $\mathscr{N}(\lambda_0)$ of $\lambda_0$ in $\Lambda$, and $\mathscr{N}([0,1]u_0)$ of
  $[0,1]u_0$ in $B_X(0, \delta)$ such that
  \begin{eqnarray*}
 &&\mathscr{N}(\lambda_0)\times\mathscr{N}([0,1]u_0)
 \ni (\lambda, u)\mapsto dA_\lambda(u)\in\mathscr{L}(X)\quad\hbox{and}\\ &&\mathscr{N}(\lambda_0)\times\mathscr{N}([0,1]u_0)\ni (\lambda, u)\mapsto d^2A_\lambda(u)\in\mathscr{L}(X, \mathscr{L}(X))
 \end{eqnarray*}
  are uniform continuous.
  \item[\rm (h)] For some $\lambda^\ast\in\Lambda$  the operator $B_{\lambda^\ast}(0)|_X\in\mathscr{L}(X)$
is hyperbolic (i.e., $0$ is a nondegenerate critical point of $\mathscr{L}_{\lambda^\ast}$).
    \end{enumerate}
 \noindent{Then} there exist a neighborhood $\Lambda_0$ of $\lambda^\ast$ in $\Lambda$,
  $\epsilon_0\in (0, \delta)$,  and  a family of  origin-preserving  $C^1$ maps
  $\varphi_\lambda:B_X(0,\epsilon_0)\to X$, $\lambda\in\Lambda_0$, such that
  both $(\lambda,u)\mapsto\varphi_\lambda(u)$ and $(\lambda,u)\mapsto\varphi'_\lambda(u)$ are of class $C^0$,
   and that each  $\varphi_\lambda$ is a $C^1$ diffeomorphism
   from $B_X(0,\epsilon_0)$ onto an open neighborhood of zero in $X$ containing $B_X(0,\epsilon_0/4)$
  and satisfies
  \begin{eqnarray}\label{e:BB.-1}
\mathscr{L}_\lambda\circ\varphi_\lambda(x)=\frac{1}{2}(B_{\lambda^\ast}(0)x, x)_H+ \mathscr{L}_\lambda(0),\quad\forall x\in B_X(0,\epsilon_0),\quad\forall\lambda\in\Lambda_0.
\end{eqnarray}
Moreover, if $\Lambda$ is a nonempty open subset of a Banach space $Z$,
then $(\lambda,u)\mapsto\varphi_\lambda(u)$ is of class $C^1$ provided that (d') is replaced by the following three conditions:
  \begin{enumerate}
  \item[\rm (f1)] $\Lambda\times B_X(0, \delta)\ni (\lambda, u)\mapsto A(\lambda, u):=A_\lambda(u)\in X$
  and
  $$
  \Lambda\times B_X(0, \delta)\ni (\lambda, u)\mapsto D_2A(\lambda,u)\in \mathscr{L}(X)
  $$
  are  $C^1$. Here $D_1$ and $D_2$ are differentials of $A$ with respect to $\lambda$ and $u$, respectively.
  \item[\rm (f2)] Either $\Lambda\times B_X(0, \delta)\ni (\lambda, u)\mapsto D^2_2A(\lambda,u)\in \mathscr{L}(X,  \mathscr{L}(X))$
  is $C^1$, or  for each point $(\lambda_0, u_0)\in\Lambda\times B_X(0, \delta)$ there exist
  neighborhoods $\mathscr{N}(\lambda_0)$ of $\lambda_0$ in $\Lambda$, and $\mathscr{N}([0,1]u_0)$ of
  $[0,1]u_0$ in $B_X(0, \delta)$ such that
  $$
  \mathscr{N}(\lambda_0)\times\mathscr{N}([0,1]u_0)\ni (\lambda, u)\mapsto D^2_2A(\lambda,u)\in \mathscr{L}(X,  \mathscr{L}(X))
  $$
   is uniform continuous.
  \item[\rm (f3)] Either $\Lambda\times B_X(0, \delta)\ni (\lambda, u)\mapsto D_1D_2A(\lambda,u)\in \mathscr{L}(Z,  \mathscr{L}(X))$
  is $C^1$, or  for each point $(\lambda_0, u_0)\in\Lambda\times B_X(0, \delta)$ there exist
  neighborhoods $\mathscr{N}(\lambda_0)$ of $\lambda_0$ in $\Lambda$, and $\mathscr{N}([0,1]u_0)$ of
  $[0,1]u_0$ in $B_X(0, \delta)$ such that
  $$
  \mathscr{N}(\lambda_0)\times\mathscr{N}([0,1]u_0)\ni (\lambda, u)\mapsto D_1D_2A(\lambda,u)\in \mathscr{L}(Z,  \mathscr{L}(X))
  $$
   is uniform continuous.
      \end{enumerate}
    (It is clear that (d') and these three conditions are satisfied if
    $$\Lambda\times B_X(0, \delta)\ni (\lambda, u)\mapsto A(\lambda, u)\in X$$
     is $C^3$.
    Remark~\ref{rm:BB4}(iii) also give other substitution conditions.)
   \end{theorem}

Clearly, (\ref{e:BB.-1}) implies that  the nondegenerate critical point $0$ of $\mathscr{L}_{\lambda^\ast}$
is isolated.

Let $X=X^{\lambda^\ast}_+\oplus X_-^{\lambda^\ast}$ be a direct sum decomposition of Banach spaces
corresponding to the spectral sets  $\sigma_+(B_{\lambda^\ast}(0)|_X)$
 and $\sigma_-(B_{\lambda^\ast}(0)|_X)$, and $P_+^{\lambda^\ast}$ and $P_-^{\lambda^\ast}$
 denote the projections corresponding to the decomposition.
Let $S^-$ and $S^+$ be  square roots  to $-B_{\lambda^\ast}(0)|_{X_-}$ and $B_{\lambda^\ast}(0)|_{X_+}$,
respectively. Then
$$
\psi:X\to X,\;x\mapsto S_+P_+^{\lambda^\ast}x+ S_-P_-^{\lambda^\ast}x,
$$
is a Banach space isomorphism, and
\begin{eqnarray*}
&&(B_{\lambda^\ast}(0)x, x)_H=\|S_+P_+^{\lambda^\ast}x\|_H^2-\|S_-P_-^{\lambda^\ast}x\|_H^2\\&=&
\|P_+^{\lambda^\ast}S_+P_+^{\lambda^\ast}x\|_H^2-\|P_-^{\lambda^\ast}S_-P_-^{\lambda^\ast}x\|_H^2\\
&=&\|P_+^{\lambda^\ast}\psi(x)\|_H^2-\|P_-^{\lambda^\ast}\psi(x)\|_H^2
\end{eqnarray*}
Take $\epsilon_0'\in (0,\epsilon_0)$ such that $\psi^{-1}(B_X(0,\epsilon'_0))\subset B_X(0,\epsilon_0)$.
By (\ref{e:BB.-1})
\begin{eqnarray}\label{e:BB.-0.5}
\mathscr{L}_\lambda\circ(\varphi_\lambda\circ\psi^{-1})(x)&=&\|P_+^{\lambda^\ast}x\|_H^2-\|P_-^{\lambda^\ast}x\|_H^2+ \mathscr{L}_\lambda(0),\notag\\&&\quad\forall x\in B_X(0,\epsilon'_0),\quad\forall\lambda\in\Lambda_0.
\end{eqnarray}
If $\nu:=\min\{\epsilon_0, \inf\{\|x\|\,|\, \|\psi(x)\|_X=\epsilon_0'\}\}$ then $B_X(0,\nu)\subset\psi^{-1}(B(0,\epsilon_0'))$
and so
$$
B_X(0,\nu/4)\subset\varphi_\lambda(B_X(0,\nu))
\subset(\varphi_\lambda\circ\psi^{-1})(B(0,\epsilon_0'))
\quad\forall \lambda\in\Lambda_0
$$
by the final claim in the following Lemma~\ref{lem:Sim}.

 For the proofs of Theorems~\ref{th:BB.3}, \ref{th:BB.5} we need the following version of inverse function theorems,
  which may be derived from \cite[Lemma~A.3.2]{McSa12} directly.

\begin{lemma}\label{lem:Sim}
  Let $X$ and $Y$ be Banach spaces, $B_X(x_0, r)\subset X$
   an open ball with radius $r$ and center $x_0$,
   and $f\in C^1(B_X(x_0, r), Y)$. Suppose that $df(x_0)\in\mathscr{L}(X,Y)$ is invertible
   and that there exists a constant $\alpha>0$ such that
   $$
   \|[df(x_0)]^{-1}\|\le \alpha\quad\hbox{and}\quad \|df(x)-df(x_0)\|\le\frac{1}{2\alpha}\;\forall x\in B_X(x_0, r).
   $$
 Then $f(B_X(x_0, r))\subset Y$ is an open subset containing $B_{Y}(f(x_0), \frac{r}{2\alpha})$
 and $f$ is a $C^1$ diffeomorphism from  $B_X(x_0, r)$ to $f(B_X(x_0, r))$.
 Moreover, if $f$ is $C^k$ for some $k\in\N$ then so is $f^{-1}:f(B_X(x_0, r))\to B_X(x_0, r)$.
 {\rm (}Clearly, $r$ in the conclusions can be replaced by any $r_0\in (0, r]$.{\rm )}
 \end{lemma}

For the sake of clarity we also give the proof of Theorem~\ref{th:BB.3}
though it is a slight modification of that of \cite[Theorem~1.1]{BoBu}.
In a note after the proof we shall show that (d') can be slightly weakened.

\begin{proof}[\it Proof of Theorem~\ref{th:BB.3}]
 Since $A_\lambda:B_X(0, \delta)\to X$ is $C^2$,
 for any $u\in B_X(0, \delta)$,
 \begin{eqnarray}\label{e:S.3.1}
 \mathcal{B}_\lambda(u):=\int^1_0(1-t)A'_\lambda(tu) dt=\int^1_0(1-t)B_\lambda(tu)|_X dt
 \end{eqnarray}
 is a well-defined operator
 $\mathscr{L}(X)$ and symmetric with respect to $(\cdot, \cdot)_H$.
 It was proved in \cite[Lemma~1.2]{BoBu} that
  $\mathscr{L}_\lambda(u)=(\mathcal{B}_\lambda(u)u,u)_H$ for all $u\in B_X(0,\delta)$.
By Lemma~\ref{lem:BB.9},
$\Lambda\times B_X(0,\delta)\ni (\lambda, u)\mapsto \mathcal{B}_\lambda(u)\in\mathscr{L}(X)$ is continuous,  $C^1$ in $u$, and
$\Lambda\times B_X(0,\delta)\ni (\lambda, u)\mapsto \mathcal{B}'_\lambda(u)\in\mathscr{L}(X)$ is also continuous.
If the conditions in the ``Moreover" part hold the map
$\Lambda\times B_X(0,\delta)\ni (\lambda, u)\mapsto \mathcal{B}_\lambda(u)\in\mathscr{L}(X)$ is  $C^1$.

  Since (c') and (d') imply that
  $$
  \Lambda\times B_X(0,\delta)\ni (\lambda, u)\mapsto B_{\lambda}(0)|_X=dA_\lambda(0)\in\mathscr{L}(X)
  $$
  is continuous,    and the spectrum $\sigma(B_{\lambda^\ast}(0)|_X)$
  does not intersect the imaginary axis,
  by shrinking $\Lambda$ we can assume that for some $r>0$,
\begin{eqnarray}\label{e:S.3.2}
\sigma(B_{\lambda}(0)|_X)\cap\{z=x+iy\,|\, |x|\le r\}=\emptyset,\quad\forall \lambda\in\Lambda,
\end{eqnarray}
and that for any $\lambda\in\Lambda$, $\sigma(B_{\lambda}(0)|_X)$ is equal to the union
$\sigma(B_{\lambda}(0)|_X)_+\cup\sigma(B_{\lambda}(0)|_X)_-$,  where
$$
\sigma(B_{\lambda}(0)|_X)_+=\{\mu\in \sigma(B_{\lambda}(0)|_X)\,|\, {\rm Re}\mu>0\},\quad$$
$$\sigma(B_{\lambda}(0)|_X)_-=\{\mu\in \sigma(B_{\lambda}(0)|_X)\,|\, {\rm Re}\mu<0\}.
$$
By Lemma~\ref{lem:BB.2} with $D=B_{\lambda}(0)|_X$, corresponding to the
spectral sets $\sigma(B_{\lambda}(0)|_X)_+$ and $\sigma(B_{\lambda}(0)|_X)_-$ we have a
decomposition of spaces, $X=X_+^{\lambda}\oplus X_-^{\lambda}$,
 and  projections $P_\ast^\lambda:X\to X_\ast^\lambda$ ($\ast=+,-$), which
 belong to $\mathscr{S}(X)$ and are continuous in $\lambda$ by shrinking $\Lambda$ (if necessary).

\vspace{4pt}\noindent
\noindent{\bf Step 1} ({\it The case of $\sigma(B_{\lambda^\ast}(0)|_X)\cap(-\sigma(B_{\lambda^\ast}(0)|_X))=\emptyset$}).
Consider the  map
$$
\Phi:\Lambda\times B_X(0,\delta)\times\mathscr{S}(X)\to \mathscr{S}(X),\;(\lambda,u,Y)\mapsto
Y\mathcal{B}_\lambda(Yu)Y- \mathcal{B}_{\lambda^\ast}(0).
$$
 By the conclusion in the first paragraph it is not hard to check that $\Phi$ is continuous, and $C^1$ in $(u,Y)$, and satisfies $D_Y\Phi(\lambda^\ast,0,id_X)=q_{\mathcal{B}_{\lambda^\ast}(0)}$,
where the operator
$$
q_{\mathcal{B}_{\lambda^\ast}(0)}:\mathscr{S}(X)\to\mathscr{S}(X),\; Z\mapsto \mathcal{B}_{\lambda^\ast}(0)Z+ Z\mathcal{B}_{\lambda^\ast}(0).
$$
Since $\mathcal{B}_{\lambda^\ast}(0)=\frac{1}{2}B_{\lambda^\ast}(0)|_X$,
by \cite[Lemma~1.4]{BoBu} the operator $q_{\mathcal{B}_{\lambda^\ast}(0)}$
has a linear bounded inverse. With the classical implicit function theorem
we obtain a neighborhood $\Lambda_0$ of $\lambda^\ast$ in $\Lambda$,
  $\epsilon\in (0, \delta)$ and  a continuous map $\psi$ from
$\Lambda_0\times B_X(0,\epsilon)$ to $\mathscr{S}(X)$, which is also $C^1$
 with respect to the second variable, such that $(\lambda,u)\mapsto\psi_u'(\lambda,u)$ is continuous and
 $$
 \psi(\lambda^\ast, 0)=id_X\quad\hbox{and}\quad
  \Phi(\lambda, u, \psi(\lambda,u))\equiv 0\quad\forall (\lambda,u)\in\Lambda_0\times B_X(0,\epsilon).
 $$
 (Clearly, $\Phi$ and so $\psi$ is $C^1$ if the conditions in the ``Moreover" part hold.)
 Shrinking $\Lambda_0$ and $\epsilon>0$ such that
 $\|\psi(\lambda,u)-\psi(\lambda^\ast, 0)\|<1/2$ for all $(\lambda, u)\in \Lambda_0\times B_X(0,\epsilon)$,
 we get that each operator $\psi(\lambda,u):X\to X$
has a linear bounded inverse.  For each $\lambda\in\Lambda_0$, define
$$
\varphi_\lambda: B_X(0,\epsilon)\to X,\;u\mapsto\psi(\lambda,u)u.
$$
Then $\varphi_\lambda(0)=0$ and $d\varphi_\lambda(0)=\psi(\lambda,0)$. Clearly, both $(\lambda,u)\mapsto\varphi_\lambda(u)$
 and $(\lambda,u)\mapsto\varphi'_\lambda(u)$ are  $C^0$.
 It is easily checked that
$\varphi_\lambda$ satisfies (\ref{e:BB.-1}) (see \cite{BoBu}).

Since $\varphi_\lambda'(u)[\xi]=(\psi'_u(\lambda,u)[\xi])u+\psi(\lambda,u)\xi$ for $\xi\in X$, we derive
\begin{eqnarray*}
\|[\varphi_\lambda'(0)]^{-1}\|_X= \|[\psi(\lambda,0)]^{-1}\|_X=\|[id_X+ \psi(\lambda,0)-\psi(\lambda^\ast, 0)]^{-1}\|_X\le 2\quad\forall\lambda\in\Lambda_0
\end{eqnarray*}
and
\begin{eqnarray*}
\|\varphi_\lambda'(u)-\varphi_\lambda'(0)\|_X&\le& \|\psi(\lambda,u)-\psi(\lambda,0)\|_X+
\|\psi'_u(\lambda,u)\|_{\in\mathscr{L}(X, \mathscr{L}(X))}\|u\|_X\\
&&\hspace{-1.5cm}\le \sup_{t\in [0,1]}\|\psi'_u(t\lambda,u)\|_{\in\mathscr{L}(X, \mathscr{L}(X))}\|u\|_X+ \|\psi'_u(\lambda,u)\|_{\in\mathscr{L}(X, \mathscr{L}(X))}\|u\|_X
\end{eqnarray*}
for all $(\lambda,u)\in\Lambda_0\times B_X(0,\epsilon)$, where the second inequality comes from the mean value theorem.
Because $\psi'_u(\cdot,\cdot)$ is continuous,
we can shrink the neighborhood $\Lambda_0$ of $\lambda^\ast$ and $\epsilon>0$ such that
\begin{eqnarray*}
\|\varphi_\lambda'(u)-\varphi_\lambda'(0)\|_X\le \frac{1}{4}\quad\hbox{for all $(\lambda,u)\in\Lambda_0\times B_X(0,\epsilon)$.}
\end{eqnarray*}
By Lemma~\ref{lem:Sim}, for each $\lambda\in\Lambda_0$,
$\varphi_\lambda(B_X(0, \epsilon))\subset X$ is an open subset containing $B_{X}(0, \frac{\epsilon}{4})$
 and $\varphi_\lambda$ is a $C^1$ diffeomorphism from  $B_X(0, \epsilon)$ to $\varphi_\lambda(B_X(0, \epsilon))$.
It also holds that $B_X(0, \epsilon_0/4)\subset\varphi_\lambda(B_X(0, \epsilon_0)))$ for any $\epsilon_0\in (0,\epsilon]$.

({\it Note}: Using the usual inverse mapping theorem we can only prove that $\varphi_\lambda$ restricts to a $C^1$ diffeomorphism
   from some $B_X(0,\epsilon_\lambda)\subset B_X(0,\epsilon)$ onto an open neighborhood of $0$ in $X$.)

\vspace{4pt}\noindent
\noindent{\bf Step 2} ({\it The case of $\sigma(B_{\lambda^\ast}(0)|_X)\cap(-\sigma(B_{\lambda^\ast}(0)|_X))\ne\emptyset$}).
 We can take $M>0$ such that
\begin{eqnarray}\label{e:S.3.2.1}
\sigma(\mathcal{B}_{\lambda^\ast}(0))_+\cap(-M^2\sigma(\mathcal{B}_{\lambda^\ast}(0))_-)=\emptyset.
\end{eqnarray}
Set $R_{\lambda^\ast} := P_+^{\lambda^\ast}+M P_-^{\lambda^\ast}$ and $\widehat{\mathscr{L}}_\lambda:=\mathscr{L}_\lambda\circ R_{\lambda^\ast}$.
By \cite[Lemma~1.2]{BoBu},
  $$
\widehat{\mathscr{L}}_\lambda(u)=(\widehat{\mathcal{B}}_\lambda(u)u,u)_H,\quad\forall u\in B_X(0,\delta'),
$$
where $\delta'>0$ is such that $R_{\lambda^\ast}(B_X(0,\delta'))\subset B_X(0,\delta)$ and
$\widehat{\mathcal{B}}_\lambda(u)=R_{\lambda^\ast} {\mathcal{B}}_\lambda(R_{\lambda^\ast}u)R_{\lambda^\ast}$.
Then $\widehat{\mathcal{B}}_{\lambda^\ast}(0)=R_{\lambda^\ast} {\mathcal{B}}_{\lambda^\ast}(0)R_{\lambda^\ast}$ and so
$$
\sigma(\widehat{\mathcal{B}}_{\lambda^\ast}(0))=\sigma(\mathcal{B}_{\lambda^\ast}(0))_+\cup(M^2\sigma(\mathcal{B}_{\lambda^\ast}(0))_-).
$$
It follows from (\ref{e:S.3.2.1})  that $\sigma(\widehat{\mathcal{B}}_{\lambda^\ast}(0))$ is disjoint with
$-\sigma(\widehat{\mathcal{B}}_{\lambda^\ast}(0))$. By the first step,
there exists a neighborhood $\Lambda_0$ of $\lambda^\ast$ in $\Lambda$,
  $\epsilon\in (0, \delta)$ and  a continuous map $\widehat\varphi$ from
$\Lambda_0\times B_X(0,\epsilon)$ to $X$,
such that   each  $\widehat\varphi_{\lambda}(\cdot):=\widehat\varphi(\lambda,\cdot)$ is an origin-preserving  $C^1$-diffeomorphism
 from $B_X(0,\epsilon)$  onto an open neighborhood of zero of $X$ containing $B_{X}(0, \frac{\epsilon}{4})$,
 $(\lambda,u)\mapsto\widehat\varphi'_\lambda(u)$ is  $C^0$, and  that
 $$
 \widehat{\mathscr{L}}_\lambda(\widehat\varphi(\lambda,u))=(\widehat{\mathcal{B}}_{\lambda^\ast}(0)u,u)_H+
 \widehat{\mathscr{L}}_\lambda(0)=\frac{1}{2}(R_{\lambda^\ast}\circ(B_{\lambda^\ast}(0)|_X)\circ R_{\lambda^\ast} u,u)_H+
 {\mathscr{L}}_\lambda(0)
 $$
for  $(\lambda,u)\in\Lambda_0\times B_X(0,\epsilon)$. Actually, we have
\begin{eqnarray*}
\widehat\varphi_{\lambda^\ast}'(0)=id_X,\quad \|[\widehat\varphi_\lambda'(0)]^{-1}\|_X\le 2\quad\hbox{and}\quad
\|\widehat\varphi_\lambda'(u)-\widehat\varphi_\lambda'(0)\|_X\le \frac{1}{4}
\end{eqnarray*}
for all $(\lambda,u)\in\Lambda_0\times B_X(0,\epsilon)$.
Take $\epsilon_0\in (0,\epsilon]$ such that $(R_{\lambda^\ast})^{-1}(B_X(0,\epsilon_0))\subset B_X(0,\epsilon)$.
 Set $\widetilde\varphi_\lambda(u):=R_{\lambda^\ast}\widehat\varphi(\lambda,(R_{\lambda^\ast})^{-1}u)$,
 where $(R_{\lambda^\ast})^{-1}=P_+^{\lambda^\ast}+M^{-1} P_-^{\lambda^\ast}$ is the inverse  of $R_{\lambda^\ast}$. Then
 \begin{eqnarray*}
 {\mathscr{L}}_\lambda(\widetilde\varphi_{\lambda}(u))&=&{\mathscr{L}}_\lambda(R_{\lambda^\ast}\widehat\varphi(\lambda,(R_{\lambda^\ast})^{-1}u)
 )= \widehat{\mathscr{L}}_\lambda(\widehat\varphi(\lambda,(R_{\lambda^\ast})^{-1}u) )\\
 &=& \frac{1}{2}(B_{\lambda^\ast}(0)|_X u,u)_H+ {\mathscr{L}}_\lambda(0),\quad\forall (\lambda,u)\in\Lambda_0\times B_X(0,\epsilon_0).
 \end{eqnarray*}
 Moreover, $\widetilde\varphi'_\lambda(u)=R_{\lambda^\ast}\circ\widehat\varphi'_\lambda(u)\circ R_{\lambda^\ast}^{-1}$. It follows that
 $\widetilde\varphi'_\lambda(0)=id_X$ and
 $$
 \|\widetilde\varphi'_\lambda(u)-\widetilde\varphi'_\lambda(0)\|\le
 \|R_{\lambda^\ast}\|\|\widehat\varphi'_\lambda(u)-\widehat\varphi'_\lambda(0)\|\|(R_{\lambda^\ast})^{-1}\|.
 $$
Since $(\lambda,u)\mapsto\widehat\varphi'_\lambda(u)$ is  $C^0$,
$\|\widehat\varphi'_\lambda(u)-\widehat\varphi'_\lambda(0)\|\to 0$ as $(\lambda,u)\to (\lambda^\ast,0)$.
Further shrinking $\Lambda_0$ and $\epsilon_0>0$ we may get that
$\|[\widetilde\varphi_\lambda'(0)]^{-1}\|_X\le 2$ and $\|\widetilde\varphi_\lambda'(u)-\widetilde\varphi_\lambda'(0)\|_X\le \frac{1}{4}$
for all $(\lambda,u)\in\Lambda_0\times B_X(0,\epsilon_0)$. By Lemma~\ref{lem:Sim},
$\widetilde\varphi_\lambda(B_X(0, \epsilon_0))$  contains $B_{X}(0, \frac{\epsilon_0}{4})$
for each $\lambda\in\Lambda_0$.

Summarizing Steps 1, 2 we get the conclusions in the first part.

\vspace{4pt}\noindent
 \noindent{\bf Step 3}. For the ``Moreover" part,  by the final part in Lemma~\ref{lem:BB.9} the map
$\Lambda\times B_X(0,\delta)\ni (\lambda, u)\mapsto \mathcal{B}_\lambda(u)\in\mathscr{L}(X)$ is $C^1$, so is
$(\lambda, u)\mapsto\widehat{\mathcal{B}}_\lambda(u)=R_{\lambda^\ast} {\mathcal{B}}_\lambda(R_{\lambda^\ast}u)R_{\lambda^\ast}$.
Then $\Phi$ and so $\psi$ in Step 1 is $C^1$.  The desired claim is easily seen from the above proof.
\end{proof}

\begin{remark}\label{rm:BB4}
{\rm
 The assumptions (f1)-(f2) and (f3) in Theorem~\ref{th:BB.3} are used to guarantee that the map
 $\Lambda\times B_X(0,\delta)\ni (\lambda, u)\mapsto \mathcal{B}_\lambda(u)\in\mathscr{L}(X)$
 defined by (\ref{e:S.3.1})  is  $C^1$.
When $\Lambda=\mathbb{R}$ and $A_\lambda=A-\lambda\widehat{A}$, where
 $A, \widehat{A}:B_X(0, \delta)\to X$ are $C^2$ maps corresponding with
 $C^3$ functionals $\mathscr{L}, \widehat{\mathscr{L}}:B_X(0, \delta)\to\mathbb{R}$
satisfying (a) and (b) above (\ref{e:BB.-2}),
 the assumptions (f1)-(f3) can be replaced by condition that  each of $A''$ and  $\widehat{A}''$
 is either $C^1$ or uniformly continuous in some neighborhood of the line segment
$[0,1]u:=\{tu\,|\,0\le t\le 1\}$ for each $u\in B_X(0, \delta)$.
     See arguments in Remark~\ref{rem:BB.1} and Lemma~\ref{lem:BB.9}.}
 \end{remark}

By slightly strengthening assumptions  of {\cite[Theorem~1.2]{BoBu}} we can prove:

\begin{theorem}\label{th:BB.4}
Let $\mathscr{L}:B_X(0, \delta)\to\mathbb{R}$ be  a $C^2$-functional satisfying
the conditions (a)-(b) above (\ref{e:BB.-2}) and $d\mathscr{L}(0)=0$, and let $B:B_X(0, \delta)\to \mathscr{L}_s(H)$ be defined by (\ref{e:BB.-2}).
 Suppose that the following two conditions are satisfied:
  \begin{enumerate}
  \item[\rm (i)] There is a $C^2$ map $A:B_X(0, \delta)\to X$ such that
$d\mathscr{L}(x)[u]=(A(x),u)_H$ for all $x\in B_X(0, \delta)$ and  $u\in X$, and that $A''$ is either $C^1$ or uniformly continuous
in some neighborhood of any line segment in $B_X(0, \delta)$.
  \item[\rm (ii)] $0\in\sigma(B(0)|_X)$ and  $\sigma(B(0)|_X)\setminus\{0\}$ is bounded away from the imaginary axis.
    \end{enumerate}
\noindent{Let} $X=X_0\oplus X_+\oplus X_-$ be a direct sum decomposition of Banach spaces,
 which corresponds to the spectral sets $\{0\}$,
  $\sigma_+(B(0)|_X)$ and $\sigma_-(B(0)|_X)$ as above. Denote by  $P_0$, $P_+$ and $P_-$ the
corresponding  projections as in Lemma~\ref{lem:BB.2}.
 Then there exists $\epsilon\in (0, \delta)$,  a $C^2$ map $h:B_X(0,\epsilon)\cap X_0\to X_+\oplus X_-$ with $h(0)=0$ (which is also $C^3$ if $A$ is $C^3$), and
 a $C^1$ origin-preserving  diffeomorphism $\varphi$ from
$B_X(0,\epsilon)$ onto an open neighborhood of $0$ in $X$  such
that
\begin{eqnarray}\label{e:BB.0}
\mathscr{L}\circ\varphi(x)=\frac{1}{2}(B(0)x, x)_H+ \mathscr{L}^\circ(P_0x),\quad\forall x\in B_X(0,\epsilon).
\end{eqnarray}
where the functional
$\mathscr{L}^\circ: B_{X_0}(0, \epsilon)\to \mathbb{R}$, defined by
$\mathscr{L}^\circ(z)=\mathscr{L}(z+ h(z))$,
 is of class $C^{2}$, and has the first-order  derivative at $z_0\in
B_{X_0}(0, \epsilon)$  given by
 \begin{eqnarray}\label{e:BB.0+}
d\mathscr{L}^\circ(z_0)[z]=\bigl(A(z_0+ h(z_0)), z\bigr)_H,\quad\forall z\in X_0,
\end{eqnarray}
and the second-order derivative at $0\in B_{X_0}(0, \epsilon)$, $d^2\mathscr{L}^\circ(0)=0$.
Moreover, after suitably shrinking $\epsilon>0$ the diffeomorphism  $\varphi$ can be chosen to satisfy
$$
\mathscr{L}\circ\varphi(x)=\|P_+x\|_H^2-\|P_-x\|_H^2+ \mathscr{L}(h(P_0x)+P_0x),\quad\forall x\in B_X(0,\epsilon).
$$
\end{theorem}

(\ref{e:BB.0}) and (\ref{e:BB.0+}) imply  that $0\in X_0$ is an isolated critical point of $\mathscr{L}^\circ$
if and only if $0\in X$ is such a critical point of of $\mathscr{L}$.

\begin{proof}[\it Proof of Theorem~\ref{th:BB.4}]
For conveniences let $P_\pm:=P_++P_-$, $X_\pm:=X_++X_-$ and $H_\pm$ be the closure of $X_\pm$ in $H$.
A standard implicit function theorem argument yields a
$\epsilon\in (0, \delta)$ and  a $C^2$ map $h:B_X(0,\epsilon)\cap X_0\to X_\pm$ with $h(0)=0$ (which is also $C^3$ if $A$ is $C^3$),
such that
\begin{eqnarray*}
(id_X-P_0)A(z+ {h}(z))=0,\quad\forall z\in  B_{X_0}(0,\epsilon).
\end{eqnarray*}
For each $z\in \Lambda_0\times B_{X_0}(0,\epsilon)$, define
\begin{eqnarray*}
&&{\bf L}_{z}:B_{X_\pm}(0,\delta/2)\to\mathbb{R},\;u\mapsto \mathscr{L}(z+ {h}(z)+u)- \mathscr{L}(z+ {h}(z)),\\
&&{\bf A}_{z}: B_{X_\pm}(0, \delta/2)\to X_\pm,\;u\mapsto P_\pm A(z+{h}(z)+u),\\
&&{\bf B}_{z}: B_{X_\pm}(0, \delta/2)\to \mathscr{L}_s(H_\pm),\;u\mapsto P_\pm\circ (B(z+{h}(z)+u)|_{H_\pm}).
\end{eqnarray*}
Then  $d{\bf L}_{z}(0)=0$, $d{\bf L}_{z}(x)[u]=({\bf A}_{z}(x),u)_H$ and $d^2{\bf L}_{z}(x)[u,v]=({\bf B}_{z}(x)u,v)_H$
for all $x\in B_{X_\pm}(0, \delta/2)$ and  $u,v\in X_\pm$.
By (ii), ${\bf B}_{0}(0)|_{X_\pm}=P_\pm\circ(B(0)|_{X_\pm})\in\mathscr{L}(X_\pm)$
is  hyperbolic and its spectrum is bounded away from the imaginary axis.
Since ${\bf A}_z$ is $C^2$, we may define $\mathcal{B}:B_{X_0}(0,\epsilon)\times B_{X_0}(0,\epsilon)\to X_\pm$ by
 $$
 \mathcal{B}(z, u):=\int^1_0(1-t){\bf A}'_{z}(tu) dt=\int^1_0(1-t)P_\pm A'(z+{h}(z)+tu)|_{X_\pm} dt.
 $$
 Then ${\bf L}_{z}(u)=(\mathcal{B}(z, u)u,u)_H$.
For any $(z, u), (z_0,u_0)\in B_{X_0}(0,\epsilon)\times B_{X_\pm}(0, \delta/2)$ we have
$$
\mathcal{B}(z, u)-\mathcal{B}(z_0, u_0)=\int^1_0(1-t)P_\pm[A'(z+{h}(z)+tu)|_{X_\pm}-A'(z_0+{h}(z_0)+tu_0)|_{X_\pm}] dt.
$$
By (i),  $A$ is $C^2$, $A''$ is bounded in a convex neighborhood of the line segment
  $z_0+h(z_0)+[0,1]u_0$. It follows from the mean value theorem that
$\mathcal{B}(z, u)\to\mathcal{B}(z_0, u_0)$ as $(z,u)\to(z_0,u_0)$.

As before, since $A$ is $C^2$,  by the differential method with parameter integral we can get
 G\^ateaux derivatives $D_u\mathcal{B}(z,u)$ and $D_z\mathcal{B}(z,u)$,
\begin{eqnarray*}
&&D_u\mathcal{B}(z,u)[u']=\int^1_0(1-t)tP_\pm A''(z+{h}(z)+tu)[u'] dt,\quad u'\in {X_\pm},\\
&&D_z\mathcal{B}(z,u)[z']=\int^1_0(1-t)P_\pm A''(z+{h}(z)+tu)[z'+h'(z)[z']] dt,\quad z'\in {X_0}.
\end{eqnarray*}
Let $C_0:=\|P_\pm\|_{\mathscr{L}(X, X_\pm)}$. For any $(z, u), (z_0,u_0)\in B_{X_0}(0,\epsilon)\times B_{X_\pm}(0, \delta/2)$ it easily follows that
\begin{eqnarray*}
&&\|D_u\mathcal{B}(z,u)-D_u\mathcal{B}(z_0,u_0)\|_{\mathscr{L}(X_\pm,\mathscr{L}(X_\pm))}\\
&\le&C_0\int^1_0\| A''(z+{h}(z)+tu)-  A''(z_0+{h}(z_0)+tu_0)\|_{\mathscr{L}(X,\mathscr{L}(X))}dt
\end{eqnarray*}
and
\begin{eqnarray*}
&&\|D_z\mathcal{B}(z,u)-D_z\mathcal{B}(z_0,u_0)\|_{\mathscr{L}(X_0,\mathscr{L}(X_\pm))}\\
&\le&C_0\int^1_0\| A''(z+{h}(z)+tu)-  A''(z_0+{h}(z_0)+tu_0)\|_{\mathscr{L}(X,\mathscr{L}(X))}dt\\
&+&C_0\int^1_0\| A''(z+{h}(z)+tu)\circ h'(z)\\&&-  A''(z_0+{h}(z_0)+tu_0)\circ h'(z_0)\|_{\mathscr{L}(X_0,\mathscr{L}(X))}dt.
\end{eqnarray*}
Observe that
\begin{eqnarray*}
&&\| A''(z+{h}(z)+tu)\circ h'(z)-  A''(z_0+{h}(z_0)+tu_0)\circ h'(z_0)\|_{\mathscr{L}(X_0,\mathscr{L}(X))}\\
&\le&\| A''(z+{h}(z)+tu)\circ h'(z)-  A''(z_0+{h}(z_0)+tu_0)\circ h'(z)\|_{\mathscr{L}(X_0,\mathscr{L}(X))}\\
&+&\| A''(z_0+{h}(z_0)+tu_0)\circ h'(z)-  A''(z_0+{h}(z_0)+tu_0)\circ h'(z_0)\|_{\mathscr{L}(X_0,\mathscr{L}(X))}\\
&\le&\| A''(z+{h}(z)+tu)-  A''(z_0+{h}(z_0)+tu_0)\|_{\mathscr{L}(X,\mathscr{L}(X))}\| h'(z)\|_{\mathscr{L}(X_0, X)}\\
&+&\| A''(z_0+{h}(z_0)+tu_0)\|_{\mathscr{L}(X,\mathscr{L}(X))}\|h'(z)- h'(z_0)\|_{\mathscr{L}(X_0, X)}.
\end{eqnarray*}
As above, by (i) we can derive from these inequalities that
 both $D_u\mathcal{B}$ and $D_z\mathcal{B}$
are continuous. The other arguments are standard, see the following proof of Theorem~\ref{th:BB.5}.
\end{proof}

\begin{remark}\label{rm:BB.4+}
{\rm {\bf (I)}.  (\ref{e:BB.0}) was proved in \cite[page 436, Theorem~1]{DiHiTr}
when $\mathscr{L}$ is of class $C^r$, $A:B_X(0, \delta)\to X$ is $C^{r-1}$, and
 the operator $B(0)|_X=dA(0)\in\mathscr{L}(X)$ restricts an isomorphism on $X_++X_-$.
Correspondingly, (\ref{e:BB.-1})  with $\Lambda=\{\lambda^\ast\}$ was proved in
\cite[page 425, Theorem~2]{DiHiTr} (\cite[Theorem~1.4]{Tr1} with $r=3$)
under similar conditions. As above we cannot prove that the $C^{r-2}$-smoothness of
the map $A$ in the proof of \cite[page 425, Theorem~2]{DiHiTr} if $r=3$. There exists the same problem in the proof
 \cite[page 436, Theorem~1]{DiHiTr} as well.

\noindent{\bf (II)}. When $\varphi$ is only required
to be a homeomorphism,  using a different method (thus not involving the above problem)
Ming Jiang \cite[Theorem~2.5]{JM} proved (\ref{e:BB.0})
if the conditions {\bf (i)}--{\bf (ii)}
are replaced by the following:
 \begin{enumerate}
\item[1)] $A:B_X(0, \delta)\to X$ is $C^1$,
\item[2)]  $B:B_X(0, \delta)\to \mathscr{L}_s(H)$ is continuous,
\item[3)] C2) in Hypothesis~\ref{hyp:1.3}, and either $0\notin\sigma(B(0))$ or
$0$ is an isolated point of $\sigma(B(0))$.
 \end{enumerate}}
\end{remark}

Here is a  parametric version of Theorem~\ref{th:BB.4}.

\begin{theorem}\label{th:BB.5}
 Suppose that the conditions (d') and (h) in Theorem~\ref{th:BB.3} are respectively replaced by the following:
  \begin{enumerate}
   \item[\rm (sd')] For each point $\lambda_0\in\Lambda$ and any line segment $\ell\subset B_X(0, \delta)$ there exist
  neighborhoods $\mathscr{N}(\lambda_0)$ of $\lambda_0$ in $\Lambda$,
  and $\mathscr{N}(\ell)$ of $\ell$ in $B_X(0, \delta)$
such that
$$
\mathscr{N}(\lambda_0)\times \mathscr{N}(\ell)\ni (\lambda, u)\mapsto A_\lambda(u)\in X
$$
is continuous and that
\begin{eqnarray*}
&&\mathscr{N}(\lambda_0)\times \mathscr{N}(\ell)\ni (\lambda, u)\mapsto dA_\lambda(u)\in\mathscr{L}(X)\quad\hbox{and}\\
 &&\mathscr{N}(\lambda_0)\times \mathscr{N}(\ell)\ni (\lambda, u)\mapsto d^2A_\lambda(u)\in\mathscr{L}(X, \mathscr{L}(X))
 \end{eqnarray*}
  are uniform continuous.
  \item[\rm (wh)] For some $\lambda^\ast\in\Lambda$,   $\sigma(B_{\lambda^\ast}(0)|_X)\setminus\{0\}$
is bounded away from the imaginary axis.
    \end{enumerate}
 \noindent{Let} $X=X_0^{\lambda^\ast}\oplus X_+^{\lambda^\ast}\oplus X_-^{\lambda^\ast}$ be a direct sum decomposition of Banach spaces,
 which corresponds to the spectral sets $\{0\}$, $\sigma_+(B_{\lambda^\ast}(0)|_X)$
 and $\sigma_-(B_{\lambda^\ast}(0)|_X)$, and let $P_0^{\lambda^\ast}$, $P_+^{\lambda^\ast}$ and $P_-^{\lambda^\ast}$ denote the corresponding  projections.
    Then  there exists  a neighborhood $\Lambda_0$ of $\lambda^\ast$ in $\Lambda$,
$\epsilon>0$, and
\begin{enumerate}
\item[{\bf (i)}] a (unique) $C^0$ map
\begin{eqnarray}\label{e:BB.1-}
\mathfrak{h}:\Lambda_0\times B_{X_0^{\lambda^\ast}}(0,\epsilon)\to X_+^{\lambda^\ast}\oplus X_-^{\lambda^\ast}
\end{eqnarray}
which is $C^2$ in the second variable and
satisfies $\mathfrak{h}(\lambda, 0)=0\;\forall\lambda\in \Lambda_0$ and
\begin{eqnarray}\label{e:BB.1}
(id_X-P_0^{\lambda^\ast})A_\lambda(z+ \mathfrak{h}(\lambda,z))=0,\quad\forall (\lambda,z)\in \Lambda_0\times B_{X_0^{\lambda^\ast}}(0,\epsilon),
\end{eqnarray}
\item[{\bf (ii)}] a $C^0$ map
\begin{eqnarray}\label{e:BB.2}
\Lambda_0\times B_X(0,\epsilon)
\to  X,\quad
({\lambda}, x)\mapsto \Phi_{{\lambda}}(x)
\end{eqnarray}
 such that for each $\lambda\in \Lambda_0$ the map
 $\Phi_\lambda$ is an origin-preserving  $C^1$-differeomor-\\phism
 from $B_X(0,\epsilon)$ onto an open neighborhood of zero of $X$ containing $B_X\\(0, \epsilon/4)$
  and  satisfies
\begin{eqnarray}\label{e:BB.3}
 \mathscr{L}_{\lambda}(\Phi_\lambda(x))=\frac{1}{2}(B_{\lambda^\ast}(0)P_\pm^{\lambda^\ast}x, P_\pm^{\lambda^\ast}x)_H+
\mathscr{L}_{\lambda}(P_0^{\lambda^\ast}x+ \mathfrak{h}(\lambda,P_0^{\lambda^\ast}x))
\end{eqnarray}
for any $x\in B_X(0,\epsilon)$; moreover $({\lambda}, x)\mapsto \Phi'_{{\lambda}}(x)$ is continuous.
\end{enumerate}
In addition, with {\bf notations} $X_\pm^{\lambda^\ast}:=X_+^{\lambda^\ast}\oplus X_-^{\lambda^\ast}$ and
$P_\pm^{\lambda^\ast}:=P_+^{\lambda^\ast}+P_-^{\lambda^\ast}$,
we also have:
\begin{enumerate}
\item[{\bf (iii)}] $(\lambda,z)\mapsto d_z\mathfrak{h}(\lambda,z)$ is continuous and
\begin{equation}\label{e:BB.3+}
d_z\mathfrak{h}(\lambda,z)=-[P_\pm^{\lambda^\ast}\circ({B}_{\lambda}(z+\mathfrak{h}(\lambda,z))|_{X_\pm^{\lambda^\ast}})]^{-1}
\circ(P_\pm^{\lambda^\ast}\circ({B}_\lambda(z+\mathfrak{h}(\lambda,z))|_{X_0^{\lambda^\ast}})).
\end{equation}
\item[{\bf (iv)}]  The functional
\begin{equation}\label{e:BB.4}
\mathscr{L}_{\lambda}^\circ: B_{X_0^{\lambda^\ast}}(0, \epsilon)\to \mathbb{R},\;
z\mapsto\mathscr{L}_{\lambda}(z+ \mathfrak{h}({\lambda}, z))
\end{equation}
 is of class $C^{2}$, and has the first-order  derivative at $z_0\in
B_{X_0^{\lambda^\ast}}(0, \epsilon)$  given by
 \begin{eqnarray}\label{e:BB.5}
d\mathscr{L}^\circ_\lambda(z_0)[z]=\bigl(A_\lambda(z_0+ \mathfrak{h}(\lambda, z_0)), z\bigr)_H,\quad\forall z\in X_0^{\lambda^\ast},
\end{eqnarray}
and the second-order derivative at $0\in
B_{X_0^{\lambda^\ast}}(0, \epsilon)$
 \begin{eqnarray}\label{e:BB.6}
  &&d^2\mathscr{L}^\circ_\lambda(0)[z,z']=\left(P_0^{\lambda^\ast}\bigr[{B}_\lambda(0)-
 {B}_\lambda(0)(P_\pm^{\lambda^\ast}{B}_{\lambda}
 (0)|_{X_\pm^{\lambda^\ast}})^{-1}(P_\pm^{\lambda^\ast}{B}_\lambda(0))\bigr]z, z'\right)_H,\nonumber\\
&& \hspace{40mm} \forall z,z'\in X_0^{\lambda^\ast}.
 \end{eqnarray}
\item[{\bf (v)}] The map $z\mapsto z+ \mathfrak{h}({\lambda}, z))$ induces an one-to-one correspondence
 between the critical points of  $\mathscr{L}_{\lambda}^\circ$ near $0\in X_0^{\lambda^\ast}$
and those of $\mathscr{L}_{\lambda}$ near $0\in X$.
\end{enumerate}

Let $G$ be a compact Lie group acting on $H$ orthogonally,
which induces a $C^1$ isometric action on $X$. Suppose that each $\mathscr{L}_\lambda$ is $G$-invariant and that
 $A_\lambda, B_\lambda$  are equivariant. Then for each $\lambda\in\Lambda_0$,
 both $\mathfrak{h}(\lambda,\cdot)$ and $\Phi_\lambda$ are  equivariant, and
 $\mathscr{L}_{\lambda}^\circ$ is $G$-invariant.

Moreover, if $\Lambda$ is a nonempty open subset of a Banach space $Z$,
then the maps in (\ref{e:BB.1-}) and (\ref{e:BB.2})
are of class $C^1$ provided that (sd') is replaced by the following three conditions:
  \begin{enumerate}
  \item[\rm (f1)] $\Lambda\times B_X(0, \delta)\ni (\lambda, u)\mapsto A(\lambda, u):=A_\lambda(u)\in X$
  and
  $$
  \Lambda\times B_X(0, \delta)\ni (\lambda, u)\mapsto D_2A(\lambda,u)\in \mathscr{L}(X)
  $$
  are  $C^1$. Here $D_1$ and $D_2$ are differentials of $A$ with respect to $\lambda$ and $u$, respectively.
  \item[\rm (sf2)] Either $\Lambda\times B_X(0, \delta)\ni (\lambda, u)\mapsto D^2_2A(\lambda,u)\in \mathscr{L}(X,  \mathscr{L}(X))$
 is $C^1$, or for each point $\lambda_0\in\Lambda$ and any line segment $\ell\subset B_X(0, \delta)$ there exist
  neighborhoods $\mathscr{N}(\lambda_0)$ of $\lambda_0$ in $\Lambda$,
  and $\mathscr{N}(\ell)$ of $\ell$ in $B_X(0, \delta)$
such that
$$
\mathscr{N}(\lambda_0)\times \mathscr{N}(\ell)\ni (\lambda, u)\mapsto D^2_2A(\lambda,u)\in \mathscr{L}(X,  \mathscr{L}(X))
$$
is uniform continuous.

\item[\rm (sf3)] Either $\Lambda\times B_X(0, \delta)\ni (\lambda, u)\mapsto D_1D_2A(\lambda,u)\in \mathscr{L}(Z,  \mathscr{L}(X))$
 is $C^1$, or for each point $\lambda_0\in\Lambda$ and any line segment $\ell\subset B_X(0, \delta)$ there exist
  neighborhoods $\mathscr{N}(\lambda_0)$ of $\lambda_0$ in $\Lambda$,
  and $\mathscr{N}(\ell)$ of $\ell$ in $B_X(0, \delta)$
such that
$$
\mathscr{N}(\lambda_0)\times \mathscr{N}(\ell)\ni (\lambda, u)\mapsto D_1D_2A(\lambda,u)\in \mathscr{L}(Z,  \mathscr{L}(X))
$$
is uniform continuous.
    \end{enumerate}
    Clearly, (sd') and these three assumptions hold if $\Lambda\times B_X(0, \delta)\ni (\lambda, u)\mapsto A(\lambda, u)\in X$ is $C^3$.
\end{theorem}

\begin{proof}
\noindent{\bf Step 1}. By the first two conditions in (sd'), and (wh),
a standard implicit function theorem argument yields (i) and (iii).
As in Step 2 of the proof of \cite[Lemma~3.1]{Lu2} we can get (\ref{e:BB.5}).
Then it and (\ref{e:BB.3+}) lead to (\ref{e:BB.6}).

By shrinking $\Lambda_0$ and $\epsilon$ we can assume that
$z+ \mathfrak{h}(\lambda,z))\in B_X(0, \delta/2)$ for all $(\lambda,z)\in \Lambda_0\times B_{X_0^{\lambda^\ast}}(0,\epsilon)$.
Let  $H_\pm^{\lambda^\ast}$ be the closure of $X_\pm^{\lambda^\ast}$ in $H$.
For each $(\lambda,z)\in \Lambda_0\times B_{X_0^{\lambda^\ast}}(0,\epsilon)$, define maps
\begin{eqnarray*}
&&{\bf L}_{(\lambda,z)}:B_{X_\pm^{\lambda^\ast}}(0,\delta/2)\to\mathbb{R},\;u\mapsto \mathscr{L}_{\lambda}(z+ \mathfrak{h}(\lambda,z)+u)-
\mathscr{L}_{\lambda}(z+ \mathfrak{h}(\lambda,z)),\\
&&{\bf A}_{(\lambda,z)}: B_{X_\pm^{\lambda^\ast}}(0, \delta/2)\to X_\pm^{\lambda^\ast},\;u\mapsto P_\pm^{\lambda^\ast}A_\lambda(z+\mathfrak{h}(\lambda,z)+u),\\
&&{\bf B}_{(\lambda,z)}: B_{X_\pm^{\lambda^\ast}}(0, \delta/2)\to \mathscr{L}_s(H_\pm^{\lambda^\ast}),\;u\mapsto P_+^{\lambda^\ast}
\circ (B_\lambda(z+\mathfrak{h}(\lambda,z)+u)|_{H_\pm^{\lambda^\ast}}).
\end{eqnarray*}
Then $d{\bf L}_{(\lambda,z)}(0)=0$,
$d{\bf L}_{(\lambda,z)}(x)[u]=({\bf A}_{(\lambda,z)}(x),u)_H$ and $d^2{\bf L}_{(\lambda,z)}(x)[u,v]=({\bf B}_{(\lambda,z)}(x)u,v)_H$
for all $x\in B_{X_\pm^{\lambda^\ast}}(0, \delta/2)$ and for all $u,v\in X_\pm^{\lambda^\ast}$.
Moreover, the condition (wh) implies that
$$
{\bf B}_{(\lambda^\ast, 0)}(0)|_{X_\pm^{\lambda^\ast}}=P_\pm^{\lambda^\ast}\circ
(B_{\lambda^\ast}(0)|_{X_\pm^{\lambda^\ast}})\in\mathscr{L}(X_\pm^{\lambda^\ast})
$$
is  hyperbolic and its spectrum is bounded away from the imaginary axis.
Since ${\bf A}_{(\lambda,z)}$ is $C^2$, we can define a map
$\mathcal{B}:\Lambda_0\times B_{X_0^{\lambda^\ast}}(0,\epsilon)\times B_{X_\pm^{\lambda^\ast}}(0, \delta/2)\to\mathscr{L}(X_\pm^{\lambda^\ast})$ by
 \begin{eqnarray}\label{e:BB.6.0.0}
 \mathcal{B}(\lambda,z, u):=\int^1_0(1-t){\bf A}'_{(\lambda,z)}(tu) dt=\int^1_0(1-t){\bf B}_{(\lambda,z)}(tu)|_{X_\pm^{\lambda^\ast}} dt
 \end{eqnarray}
for any $(\lambda,z, u)\in \Lambda_0\times B_{X_0^{\lambda^\ast}}(0,\epsilon)\times B_{X_\pm^{\lambda^\ast}}(0, \delta/2)$.
Then ${\bf L}_{(\lambda,z)}(u)=(\mathcal{B}(\lambda,z, u)u,u)_H$.
As in Step 1 of the proof of Lemma~\ref{lem:BB.9}, we can use the first condition in (sd') to show that $\mathcal{B}$  is continuous.

 By the proof of Theorem~\ref{th:BB.4} it is easily seen that $\mathcal{B}(\lambda,z, u)$
 has G\^ateaux derivatives $D_u\mathcal{B}(\lambda, z,u)$ in $u$ and $D_z\mathcal{B}(\lambda, z,u)$ in $z$,
\begin{eqnarray*}
&&D_u\mathcal{B}(\lambda, z,u)[u']=\int^1_0(1-t)tP_\pm^{\lambda^\ast}A''_\lambda(z+\mathfrak{h}(\lambda,z)+tu)[u'] dt,\quad u'\in X_\pm^{\lambda^\ast},\\
&&D_z\mathcal{B}(\lambda, z,u)[z']\\&&=\int^1_0(1-t)P_\pm^{\lambda^\ast}A''_\lambda(z+\mathfrak{h}(\lambda,z)+tu)[z'+
D_z\mathfrak{h}(\lambda,z)[z']] dt,\quad z'\in X_0^{\lambda^\ast}.
\end{eqnarray*}

 Let $C:=\|P_\pm^{\lambda^\ast}\|_{\mathscr{L}(X, X_\pm^{\lambda^\ast})}$. For any $(\lambda, z, u), (\lambda_0, z_0,u_0)\in
 \Lambda_0\times B_{X_0^{\lambda^\ast}}(0,\epsilon)\times B_{X_\pm^{\lambda^\ast}}(0, \delta/2)$
  it easily follows from the first equation that
\begin{eqnarray*}
&&\|D_u\mathcal{B}(\lambda, z,u)-D_u\mathcal{B}(\lambda_0, z_0,u_0)\|_{\mathscr{L}(X_\pm^{\lambda^\ast},\mathscr{L}(X_\pm^{\lambda^\ast}))}\\
&\le&C\int^1_0\|A''_\lambda(z+\mathfrak{h}(\lambda,z)+tu) - A''_{\lambda_0}(z_0+\mathfrak{h}(\lambda_0, z_0)+tu_0) \|_{\mathscr{L}(X,\mathscr{L}(X))}dt.
\end{eqnarray*}
 As in Step 3 of the proof of Lemma~\ref{lem:BB.9}, we can use the second condition in (sd') to show that $D_u\mathcal{B}$  is continuous
 at $(\lambda_0, z_0,u_0)$.

 Similarly,  as in the proof of Theorem~\ref{th:BB.4} we derive
 \begin{eqnarray*}
&\|D_z\mathcal{B}(\lambda, z,u)-D_z\mathcal{B}(\lambda_0, z_0,u_0)\|_{\mathscr{L}(X_0^{\lambda^\ast},\mathscr{L}(X_\pm^{\lambda^\ast}))}\\
&\le C\int^1_0\|A''_\lambda(z+\mathfrak{h}(\lambda,z)+tu) - A''_{\lambda_0}(z_0+\mathfrak{h}(\lambda_0, z_0)+tu_0) \|_{\mathscr{L}(X,\mathscr{L}(X))}dt\\
&+C\int^1_0\|A''_\lambda(z+\mathfrak{h}(\lambda,z)+tu)\circ \mathfrak{h}'_z(\lambda,z)\\& -
A''_{\lambda_0}(z_0+\mathfrak{h}(\lambda_0, z_0)+tu_0)\circ \mathfrak{h}'_z(\lambda_0,z_0) \|_{\mathscr{L}(X,\mathscr{L}(X))}dt
\end{eqnarray*}
and
\begin{eqnarray*}
&&\|A''_\lambda(z+\mathfrak{h}(\lambda,z)+tu)\circ \mathfrak{h}'_z(\lambda,z)\\&& -
A''_{\lambda_0}(z_0+\mathfrak{h}(\lambda_0, z_0)+tu_0)\circ \mathfrak{h}'_z(\lambda_0,z_0) \|_{\mathscr{L}(X,\mathscr{L}(X))}\\
&\le&\| A''_\lambda(z+\mathfrak{h}(\lambda, z)+tu)\\&&-  A''_{\lambda_0}(z_0+\mathfrak{h}(\lambda_0, z_0)+tu_0)\|_{\mathscr{L}(X,\mathscr{L}(X))}\| \mathfrak{h}'_z(\lambda, z)\|_{\mathscr{L}(X_0, X)}\\
&+&\| A''_\lambda(z_0+\mathfrak{h}(\lambda_0, z_0)+tu_0)\|_{\mathscr{L}(X,\mathscr{L}(X))}\|\mathfrak{h}'_z(\lambda,
z)- \mathfrak{h}'_z(\lambda_0, z_0)\|_{\mathscr{L}(X_0, X)}.
\end{eqnarray*}
By (\ref{e:BB.3+}), $(\lambda,z)\mapsto \mathfrak{h}'_z(\lambda,z)$ is continuous.
   Using the second condition in (sd') again we deduce that
  $$
  \sup_{t\in [0,1]}\| A''_\lambda(z+\mathfrak{h}(\lambda, z)+tu)-  A''_{\lambda_0}(z_0+\mathfrak{h}(\lambda_0, z_0)+tu_0)\|_{\mathscr{L}(X,\mathscr{L}(X))}\to 0
$$
as $\lambda\to\lambda_0$ and $u\to u_0$, and that
$\| A''_\lambda(z+\mathfrak{h}(\lambda, z)+tu)\|_{\mathscr{L}(X,\mathscr{L}(X))}$
uniformly converges to \linebreak $\| A''_{\lambda_0}(z_0+\mathfrak{h}(\lambda_0, z_0)+tu_0)\|_{\mathscr{L}(X,\mathscr{L}(X))}$
with respect to $t\in [0,1]$. Hence  $D_z\mathcal{B}$  is continuous
 at $(\lambda_0, z_0,u_0)$.

 In summary, we have proved that $(\lambda, z, u)\mapsto\mathcal{B}(\lambda,z,u)$
  is continuous, $C^1$ in each of $z$ and $u$, and both
 $(\lambda, z, u)\mapsto\mathcal{B}'_z(\lambda,z,u)$
 and $(\lambda, z, u)\mapsto\mathcal{B}'_u(\lambda,z,u)$ are continuous.

\vspace{4pt}\noindent
 \noindent{\bf Step 2}. Since $\mathcal{B}(\lambda^\ast,0, 0)=\frac{1}{2}{\bf B}_{(\lambda^\ast, 0)}(0)|_{X_\pm^{\lambda^\ast}}$ we can take $M>0$ such that
\begin{eqnarray}\label{e:BB.6.0.1}
\sigma(\mathcal{B}({\lambda^\ast},0,0))_+\cap(-M^2\sigma(\mathcal{B}({\lambda^\ast},0,0))_-)
=\emptyset.
\end{eqnarray}
Set $R_{\lambda^\ast} := P_+^{\lambda^\ast}|_{X_\pm^{\lambda^\ast}}+M P_-^{\lambda^\ast}|_{X_\pm^{\lambda^\ast}}$ and
$\widehat{\bf L}_{(\lambda,z)}:={\bf L}_{(\lambda,z)}\circ R_{\lambda^\ast}$.
As before we have 
  \begin{eqnarray}\label{e:BB.6.0.2}
\widehat{\bf L}_{(\lambda,z)}(u)=(\widehat{\mathcal{B}}(\lambda,z,u)u,u)_H,\quad\forall u\in B_{X_\pm^{\lambda^\ast}}(0, \delta_1),
\end{eqnarray}
where $\delta_1>0$ is such that $R_{\lambda^\ast}\left(B_{X_\pm^{\lambda^\ast}}(0, \delta_1)\right)\subset B_{X_\pm^{\lambda^\ast}}(0, \delta/2)$, and
\begin{eqnarray}\label{e:BB.6.0.3}
\widehat{\mathcal{B}}(\lambda,z,u):=R_{\lambda^\ast} {\mathcal{B}}(\lambda,z, R_{\lambda^\ast}u)R_{\lambda^\ast}.
\end{eqnarray}
The latter implies
$\sigma(\widehat{\mathcal{B}}({\lambda^\ast},0,0))=\sigma(\mathcal{B}({\lambda^\ast},0,0))_+\cup(M^2\sigma(\mathcal{B}({\lambda^\ast},0,0))_-)$ and hence
\begin{eqnarray}\label{e:BB.6.0.4}
\sigma(\widehat{\mathcal{B}}({\lambda^\ast},0,0))\cap (-\sigma(\widehat{\mathcal{B}}({\lambda^\ast},0,0)))=\emptyset
\end{eqnarray}
by (\ref{e:BB.6.0.1}).
Consider the  map
$$
\Psi:\Lambda_0\times B_{X_0^{\lambda^\ast}}(0,\epsilon)\times B_{X_\pm^{\lambda^\ast}}(0, \delta_1)\times\mathscr{S}(X_\pm^{\lambda^\ast})\to \mathscr{S}(X_\pm^{\lambda^\ast})
$$
defined by $\Psi(\lambda,z, u,Y)=Y\widehat{\mathcal{B}}(\lambda, z, Yu)Y- \widehat{\mathcal{B}}(\lambda^\ast, 0, 0)$.
Then $\Psi$  is continuous, $C^1$ in $(z, u,Y)$, and
 $\Psi'_z(\lambda,z,u,Y)$, $\Psi'_u(\lambda,z,u,Y)$ and $\Psi'_Y(\lambda,z,u,Y)$ continuously depend on $(\lambda,z,u,Y)$.
Since
$$
D_Y\Psi(\lambda^\ast,0,0, id_X):\mathscr{S}(X_\pm^{\lambda^\ast})\to\mathscr{S}(X_\pm^{\lambda^\ast}),\; Z\mapsto \widehat{\mathcal{B}}({\lambda^\ast},0,0)Z+ Z\widehat{\mathcal{B}}({\lambda^\ast},0,0)
$$
is a Banach space isomorphism by (\ref{e:BB.6.0.4}) and \cite[Lemma~1.4]{BoBu},
as in the proof of Theorem~\ref{th:BB.3} using the implicit function theorem
we can, by shrinking the neighborhood $\Lambda_0$ of $\lambda^\ast$ and $\epsilon>0$, get
  $\delta_2\in (0, \delta_1)$ and  a continuous map
  $$
  \psi:\Lambda_0\times B_{X_0^{\lambda^\ast}}(0,\epsilon)\times B_{X_\pm^{\lambda^\ast}}(0, \delta_2)\to\mathscr{S}(X_\pm^{\lambda^\ast})
  $$
  such that\\
  $\bullet$  for all $(\lambda, z, u)\in \Lambda_0\times B_{X_0^{\lambda^\ast}}(0,\epsilon)\times B_{X_\pm^{\lambda^\ast}}(0, \delta_2)$,
 \begin{equation}\label{e:BB.6.0.4+}
 \left.\begin{array}{ll}
 &\psi(\lambda^\ast,0, 0)=id_{X_\pm^{\lambda^\ast}},\quad\Psi(\lambda, z, u, \psi(\lambda,z, u))\equiv 0\quad\hbox{and}\quad \\ 
& \|\psi(\lambda,z, u)-\psi(\lambda^\ast, 0, 0)\|<1/2,
\end{array}\right\}
 \end{equation}
$\bullet$ $\psi(\lambda,z,u)$ is $C^1$ in $(z,u)$,  both  $\psi'_z(\lambda, z,u)$ and $\psi'_u(\lambda,z,u)$  continuously depend on $(\lambda,z,u)$.
Clearly, (\ref{e:BB.6.0.4+}) implies that each operator $\psi(\lambda,z, u):X_\pm^{\lambda^\ast}\to X_\pm^{\lambda^\ast}$
has a linear bounded inverse.
 Now for any $(\lambda, z, u)\in \Lambda_0\times B_{X_0^{\lambda^\ast}}(0,\epsilon)\times B_{X_\pm^{\lambda^\ast}}(0, \delta_2)$
 it  holds that
 \begin{eqnarray}\label{e:BB.6.0.4++}
&& \psi(\lambda,z, u)\widehat{\mathcal{B}}(\lambda, z, \psi(\lambda,z, u)u)\psi(\lambda,z, u)\equiv \widehat{\mathcal{B}}(\lambda^\ast, 0, 0)\notag\\&&=
 \frac{1}{2}R_{\lambda^\ast}\left({\bf B}_{(\lambda^\ast, 0)}(0)|_{X_\pm^{\lambda^\ast}}\right) R_{\lambda^\ast}.
 \end{eqnarray}
 Moreover,  from (\ref{e:BB.6.0.2}) we derive
   \begin{eqnarray}\label{e:BB.6.0.5}
  && \mathscr{L}_{\lambda}(z+ \mathfrak{h}(\lambda,z)+R_{\lambda^\ast}u)-
\mathscr{L}_{\lambda}(z+ \mathfrak{h}(\lambda,z)),\nonumber\\
  &=&{\bf L}_{(\lambda,z)}(R_{\lambda^\ast}u)=(\widehat{\mathcal{B}}(\lambda,z,u)u,u)_H,\quad\forall u\in B_{X_\pm^{\lambda^\ast}}(0, \delta_1).
\end{eqnarray}
By further shrinking the neighborhood $\Lambda_0$ of $\lambda^\ast$ and $\epsilon>0$, we can get
  $\delta_3\in (0, \delta_2)$ such that
 $$
 \psi(\lambda,z, u)u\in B_{X_\pm^{\lambda^\ast}}(0, \delta_1),\quad
 \forall (\lambda, z, u)\in \Lambda_0\times B_{X_0^{\lambda^\ast}}(0,\epsilon)\times B_{X_\pm^{\lambda^\ast}}(0, \delta_3)
 $$
 because of (\ref{e:BB.6.0.4+}). Then (\ref{e:BB.6.0.4++}) and
 (\ref{e:BB.6.0.5}) lead to
  \begin{eqnarray}\label{e:BB.6.0.6}
  && \mathscr{L}_{\lambda}(z+ \mathfrak{h}(\lambda,z)+R_{\lambda^\ast}\psi(\lambda,z, u)u)-
\mathscr{L}_{\lambda}(z+ \mathfrak{h}(\lambda,z)),\nonumber\\
  &=&(\widehat{\mathcal{B}}(\lambda,z,\psi(\lambda,z, u)u)\psi(\lambda,z, u)u,\psi(\lambda,z, u)u)_H\nonumber\\
  &=&(\psi(\lambda,z, u)\widehat{\mathcal{B}}(\lambda,z,\psi(\lambda,z, u)u)\psi(\lambda,z, u)u,u)_H\nonumber\\
  &=&\frac{1}{2}(({\bf B}_{(\lambda^\ast, 0)}(0)|_{X_\pm^{\lambda^\ast}}) R_{\lambda^\ast}u, R_{\lambda^\ast}u)_H.
\end{eqnarray}
 Take $\delta_4\in (0, \delta_3)$ such that $R_{\lambda^\ast}\left(B_{X_\pm^{\lambda^\ast}}(0, \delta_4)\right)\subset B_{X_\pm^{\lambda^\ast}}(0, \delta_3)$. Then
 (\ref{e:BB.6.0.6}) yields
   \begin{eqnarray}\label{e:BB.6.0.7}
  && \mathscr{L}_{\lambda}\big(z+ \mathfrak{h}(\lambda,z)+ R_{\lambda^\ast}\psi(\lambda,z, (R_{\lambda^\ast})^{-1}u)(R_{\lambda^\ast})^{-1}u\big)-
\mathscr{L}_{\lambda}(z+ \mathfrak{h}(\lambda,z)),\nonumber\\
    &=&\frac{1}{2}(({\bf B}_{(\lambda^\ast, 0)}(0)|_{X_\pm^{\lambda^\ast}}) u, u)_H=\frac{1}{2}(B_{\lambda^\ast}(0)u, u)_H
\end{eqnarray}
 for all $(\lambda, z, u)\in \Lambda_0\times B_{X_0^{\lambda^\ast}}(0,\epsilon)\times B_{X_\pm^{\lambda^\ast}}(0, \delta_4)$,
where $(R_{\lambda^\ast})^{-1}= P_+^{\lambda^\ast}|_{X_\pm^{\lambda^\ast}}+M^{-1} P_-^{\lambda^\ast}|_{X_\pm^{\lambda^\ast}}$ is the inverse
of $R_{\lambda^\ast}$. 

Take $\delta_5=\min\{\epsilon,\delta_4\}$. Then for $x\in B_X(0,\delta_5)$ we have
$z=P_0^{\lambda^\ast}x\in B_{X_0^{\lambda^\ast}}(0,\epsilon)$ and $u=P_\pm^{\lambda^\ast} x\in B_{X_\pm^{\lambda^\ast}}(0, \delta_4)$.
For each $\lambda\in\Lambda_0$, define $\Phi_\lambda: B_X(0,\delta_5)\to  X$ by
\begin{eqnarray}\label{e:BB.6.2-}
\Phi_\lambda(x)=z+\mathfrak{h}(\lambda,z)+ R_{\lambda^\ast}\psi(\lambda,z, (R_{\lambda^\ast})^{-1}u)(R_{\lambda^\ast})^{-1}u,
\end{eqnarray}
 where $z=P_0^{\lambda^\ast}x$ and $u=P_\pm^{\lambda^\ast} x$. Then
  $\Phi_\lambda(0)=0$, $\Phi_\lambda$ is $C^1$,
 and both $(\lambda,x)\mapsto\Phi_\lambda(x)$ and $(\lambda,x)\mapsto\Phi'_\lambda(x)$ are continuous.
 For $z'\in X_0^{\lambda^\ast}$ and $y'\in X_\pm^{\lambda^\ast}$ a direct computation
leads to
 \begin{eqnarray}\label{e:BB.6.2}
 d\Phi_\lambda(0)[z'+y']&=&z'+ d_z\mathfrak{h}(\lambda,0)[z']+ R_{\lambda^\ast}\psi(\lambda,0, 0)(R_{\lambda^\ast})^{-1}y'
 \end{eqnarray}
 and so $d\Phi_{\lambda^\ast}(0)[z'+y']=z'+y'$ because $d_z\mathfrak{h}(\lambda^\ast, 0)=0$ by (\ref{e:BB.3+}).
 From these and (\ref{e:BB.6.0.4+}) we get
  \begin{eqnarray*}
 d\Phi_\lambda(0)[z'+y']-z'-y=d_z\mathfrak{h}(\lambda,0)[z']+ R_{\lambda^\ast}(\psi(\lambda,0, 0)-\psi(\lambda^\ast,0, 0))(R_{\lambda^\ast})^{-1}y'
 \end{eqnarray*}
 and therefore we may shrink $\Lambda_0$ so that $\|d\Phi_\lambda(0)-id_{X_\pm^{\lambda^\ast}}\|<\frac{1}{2}$ for all $\lambda\in\Lambda_0$.
 Then  we have
 $$
 \|[d\Phi_\lambda(0)]^{-1}\|\le 2\quad\forall\lambda\in\Lambda_0
 $$
 as before.
  Since $\psi'_z(\lambda, z,u)$ and $\psi'_u(\lambda,z,u)$  continuously depend on $(\lambda,z,u)$,
 shrinking $\Lambda_0$ and $\delta_5>0$ again, it follows from (\ref{e:BB.6.2-}) and (\ref{e:BB.6.2})  that
 $$
 \|d\Phi_\lambda(x)-d\Phi_{\lambda}(0)\|\le\frac{1}{4}\quad\forall (\lambda,x)\in\Lambda_0\times B_X(\theta,\delta_5).
 $$
 By Lemma~\ref{lem:Sim}, for each $\lambda\in\Lambda_0$,
$\Phi_\lambda(B_X(0, \delta_5))\subset X$ is an open neighborhood of zero in $X$  containing $B_{X}(0, \frac{\delta_5}{4})$
 and $\Phi_\lambda$ is a $C^1$ diffeomorphism from  $B_X(0, \delta_5)$ to $\Phi_\lambda(B_X(0, \delta_5))$.
  Moreover, (\ref{e:BB.6.0.7}) yields (\ref{e:BB.3}) with $\epsilon=\delta_5$.

 It is clear that  the positive number $\epsilon$ can be replaced by $\delta_5$ in all conclusions of Theorem~\ref{th:BB.5} which have been proved.

 Claims in (iv) and (v) easily follows as in \cite{Lu2, Lu7}.

  Under the assumptions of  the ``Moreover" part,  by the above arguments and the final part of proof of Lemma~\ref{lem:BB.9}
 it is not hard to check that the maps
$\mathcal{B}:\Lambda_0\times B_{X_0^{\lambda^\ast}}(0,\epsilon)\times B_{X_\pm^{\lambda^\ast}}(0, \delta/2)\to\mathscr{L}(X_\pm^{\lambda^\ast})$
and $\widehat{\mathcal{B}}$ in (\ref{e:BB.6.0.3}) are $C^1$.
Therefore the desired claims are easily seen from the above proof.
\end{proof}

\begin{remark}\label{rm:BB.5}
{\rm The assumptions (f1) and (sf2)-(sf3) in Theorem~\ref{th:BB.5} guarantee that the map
$\mathcal{B}$ in (\ref{e:BB.6.0.0}) (and hence $\widehat{\mathcal{B}}$ in (\ref{e:BB.6.0.3}))
is $C^1$. When $\Lambda=\mathbb{R}$ and $A_\lambda=A-\lambda\widehat{A}$, where
 $A, \widehat{A}:B_X(0, \delta)\to X$ are $C^2$ maps corresponding with
 $C^3$ functionals $\mathscr{L}, \widehat{\mathscr{L}}:B_X(0, \delta)\to\mathbb{R}$
satisfying (a) and (b) above (\ref{e:BB.-2}),
if each of $A''$ and  $\widehat{A}''$ is either $C^1$ or uniformly continuous in some neighborhood of
any segment in $B_X(0, \delta)$
    it is easily seen from the above proof that the corresponding maps $\mathcal{B}$ and
    $\widehat{\mathcal{B}}$ are also $C^1$.  Hence Theorem~\ref{th:BB.5} is still true in this situation.}
\end{remark}

Since $0\in X_0^{\lambda^\ast}$ is an isolated critical point of  $\mathscr{L}_{\lambda}^\circ$ if and only if
$0\in X$ is such a critical point of $\mathscr{L}_{\lambda}$,
as usual Theorem~\ref{th:BB.5} lead to:

\begin{corollary}[Shifting]\label{cor:BB.6}
Under the assumptions of Theorem~\ref{th:BB.5},
let $\nu_{\lambda^\ast}:=\dim X_0^{\lambda^\ast}<\infty$ and $\mu_{\lambda^\ast}:=\dim X_-^{\lambda^\ast}<\infty$.
For any Abel group ${\bf K}$, and for any given $\lambda\in\Lambda_0$, if $0\in X$ is an
isolated critical point of $\mathscr{L}$ then
$C_q(\mathscr{L}_\lambda, 0;{\bf K})\cong
C_{q-\mu_{\lambda^\ast}}(\mathscr{L}^{\circ}_\lambda, 0;{\bf K})$ for all
$q\in\mathbb{N}\cup\{0\}$.
\end{corollary}

\begin{claim}\label{cl:BB.6+}
In Theorem~\ref{th:BB.5}, if  for some $\lambda\in\Lambda_0\setminus\{\lambda^\ast\}$,
$0\in X$ is a nondegenerate critical point of $\mathscr{L}_\lambda$,
i.e.,  the operator $B_\lambda(0)|_X\in\mathscr{L}(X)$  is   hyperbolic (\cite{Uh0}),
then $0\in X_0^{\lambda^\ast}$ is such a critical point of $\mathscr{L}^\circ_\lambda$ too.
\end{claim}

\begin{proof}
Let $z\in X_0^{\lambda^\ast}$ be such that  $d^2\mathscr{L}^\circ_\lambda(0)[z,z']=0\;\forall z'\in X_0^{\lambda^\ast}$.
By (\ref{e:BB.6}) we deduce
 \begin{eqnarray}\label{e:BB.7}
  P_0^{\lambda^\ast}\bigr[{B}_\lambda(0)-
 {B}_\lambda(0)(P_\pm^{\lambda^\ast}{B}_{\lambda^\ast}
 (0)|_{X_\pm^{\lambda^\ast}})^{-1}(P_\pm^{\lambda^\ast}{B}_\lambda(0))\bigr]z=0.
 \end{eqnarray}
Moreover, differentiating the equality in (\ref{e:BB.1}) at $z=0$ and using (\ref{e:BB.3+}) we get that
\begin{eqnarray*}
0&=&(id_X-P_0^{\lambda^\ast})B_\lambda(0)|_X(z+ d_z\mathfrak{h}(\lambda,0)z)\\
&=&(id_X-P_0^{\lambda^\ast})B_\lambda(0)|_X(z-[P_\pm^{\lambda^\ast}{B}_{\lambda^\ast}(0)|_{X_\pm^{\lambda^\ast}}]^{-1}
(P_\pm^{\lambda^\ast}{B}_\lambda(0))|_{X_0^{\lambda^\ast}}z)\\
&=&(id_X-P_0^{\lambda^\ast})[B_\lambda(0)-B_\lambda(0)(P_\pm^{\lambda^\ast}{B}_{\lambda^\ast}(0)|_{X_\pm^{\lambda^\ast}})^{-1}
(P_\pm^{\lambda^\ast}{B}_\lambda(0))|_{X_0^{\lambda^\ast}})]z\\
&=&[B_\lambda(0)-B_\lambda(0)(P_\pm^{\lambda^\ast}{B}_{\lambda^\ast}(0)|_{X_\pm^{\lambda^\ast}})^{-1}
(P_\pm^{\lambda^\ast}{B}_\lambda(0))|_{X_0^{\lambda^\ast}})]z\\
&-&P_0^{\lambda^\ast}[B_\lambda(0)-B_\lambda(0)(P_\pm^{\lambda^\ast}{B}_{\lambda^\ast}(0)|_{X_\pm^{\lambda^\ast}})^{-1}
(P_\pm^{\lambda^\ast}{B}_\lambda(0))|_{X_0^{\lambda^\ast}})]z\\
&=&[B_\lambda(0)-B_\lambda(0)(P_\pm^{\lambda^\ast}{B}_{\lambda^\ast}(0)|_{X_\pm^{\lambda^\ast}})^{-1}
(P_\pm^{\lambda^\ast}{B}_\lambda(0))|_{X_0^{\lambda^\ast}})]z,
\end{eqnarray*}
where the final equality comes from   (\ref{e:BB.7}). Since $B_\lambda(0)$ is invertible, we deduce
$$
z=(P_\pm^{\lambda^\ast}{B}_{\lambda^\ast}(0)|_{X_\pm^{\lambda^\ast}})^{-1}
(P_\pm^{\lambda^\ast}{B}_\lambda(0))|_{X_0^{\lambda^\ast}})z.
$$
Note that $z\in X_0^{\lambda^\ast}$ and that the right side belongs to $X_\pm^{\lambda^\ast}$.
Hence $z=0$.
\end{proof}

In general, it is not easy to judge whether $\sigma(B(0)|_X)\setminus\{0\}$ is bounded away from the imaginary axis.
However, it has the following equivalent version:
\begin{enumerate}
\item[(*)] $0$ is at most an isolated point of
$\sigma(B(0)|_X)$ and $B(0)|_X$ induces a hyperbolic  operator on the quotient space $X/X_0$,
where $X_0={\rm Ker}(B(0)|_X)$.
\end{enumerate}
In fact, it is clear if $0\notin\sigma(B(0)|_X)$. Assume that $\sigma_0:=\{0\}\subset\sigma(B(0)|_X)$.
Since $\sigma_1:=\sigma(B(0)|_X)\setminus\{0\}$ is bounded away from the imaginary axis, $0$ is  an isolated point of $\sigma(B(0)|_X)$.
By Lemma~\ref{lem:BB.2} there exists a direct sum decomposition of the Banach space $X$, $X=X_0+X_1$,
 where $X_0={\rm Ker}(B(0)|_X)$ and $X_1$ is an invariant subspace on which $B(0)|_X$ has the spectrum $\sigma_1$.
Note that the quotient space $X/X_0$ and the induced operator of $B(0)|_X$ on it can be identified with
$X_1$ and $B(0)|_{X_1}$, respectively. Hence (*) is satisfied. Conversely, since $0$ is  an isolated point of
$\sigma(B(0)|_X)$ we have a decomposition $X=X_0+X_1$ as above, and $B(0)|_{X_1}$ is a hyperbolic  operator on $X_1\equiv X/X_0$.
It follows that $\sigma(B(0)|_X)\setminus\{0\}$ is bounded away from the imaginary axis.

For a class of operators used in this paper the following lemma determines when $0$
is at most an isolated point of $\sigma(B(0))$.

\begin{lemma}\label{lem:BB.7}
 Let a self-adjoined operator $\mathfrak{B}\in\mathscr{L}_s(H)$ be a sum $\mathfrak{B} = \mathfrak{P} +
 \mathfrak{Q}$, where
$\mathfrak{P}\in\mathscr{L}(H)$ is invertible, and $\mathfrak{Q}\in\mathscr{L}(H)$ is compact. Let
 $0\in\sigma(\mathfrak{B})$. Then $0$ is an
isolated point of $\sigma(\mathfrak{B})$ and an eigenvalue of $\mathfrak{B}$ of the finite multiplicity. ({\cite[Lemma~2.2]{BoBu}})

 In addition, if $\mathfrak{P}$ is also positive definite, then every $\lambda<\inf\{(\mathfrak{P}u,u)_H\,|\,\|u\|_H=1\}$
 is either a regular value of $\mathfrak{B}$ or an isolated point of
$\sigma(\mathfrak{B})$, which is also an eigenvalue of finite multiplicity. (\cite[Proposition~B.2]{Lu2})
\end{lemma}

 As before let $H^0$, $H^+$ and $H^-$ be null, positive and negative definite spaces of $B(0)$,
and let $P^\ast:H\to H^\ast$ for $\ast=0,+,-$,  be the orthogonal projections.
Clearly, $X_0\subset H^0$. The following lemma shows how the Morse index $\mu$
is computed.

\begin{lemma}\label{lem:BB.8}
 If $X_0=H^0$ and
either $\dim H^-<\infty$ or $\dim H^+<\infty$, then
$X_\ast$ are dense subspaces of $H^\ast$ for $\ast=+,-$, moreover
$X_-=H^-$ (resp. $X_+=H^+$) as
 $\dim H^-<\infty$ (resp. $\dim H^+<\infty$).
(These imply $X_\ast=H^\ast\cap X$ and $P_\ast=P^\ast|_{X}$
for $\ast=0,+,-$, and hence
 $X_\pm=X\cap H^\pm=(P^++P^-)(X)$, where $H^\pm:=H^++H^-$.)
\end{lemma}
\begin{proof}
Since $X$ is dense in $H$, it is easily proved that $X_\pm:=X_+\oplus X_-$ is dense in
$H^\pm$. Let $X_+^H$ and $X_-^H$ be the closure of $X_+$ and $X_-$ in $H$, respectively.
They are orthogonal in $H$, invariant with respect to $B(0)$ and $H^\pm=X_+^H+ X_-^H$.
Since $B(0)$ is semi-positive (resp. semi-negative) on $X_+^H$ (resp. $X_-^H$)
by Lemma~\ref{lem:BB.2}, and $B(0)$ has no nontrivial kernel on $H^\pm$, we deduce that
$B(0)$ must be positive (resp. negative) on $X_+^H$ (resp. $X_-^H$).
Let $\dim H^-<\infty$. If $X^H_-\cap H^-\ne\{0\}$, it is an invariant subspace of $B(0)$.
Denote by $X^H_{--}$ and $H^{--}$ the orthogonal complements of $X^H_-\cap H^-$ in $X^H_-$ and $H^-$,
respectively. Then $X^H_{--}+ H^{--}+X^H_-\cap H^-$ is a negative definite subspace of $B(0)$.
But $\dim H^-$ is the maximal dimension of negative definite subspaces of $B(0)$.
Hence $X^H_{--}=\{0\}$,  and thus $X^H_{-}=X^H_-\cap H^-$. The latter means $X^H_{-}\subset H^-$
and so $X^H_{+}\supseteq H^+$. For any $x\in H^-$, since
$X_\pm=X_+\oplus X_-$ is dense in $H^\pm=H^++H^-$, we have a sequence $(x_n^++x_n^-)\subset X_+\oplus X_-$
such that $x_n^++x_n^-\to x$. Note that
$\|x_n^++x_n^--x_m^+-x_m^-\|^2=\|x_n^+-x_m^+\|^2+\|x_n^--x_m^-\|^2$.
We get $x_n^+\to x^+\in X^H_+$ and $x_n^-\to x^-\in X^H_-$. Hence $x=x^++x^-$.
But $(x,x^+)_H=0$. So $x^+=0$ and $x=x^-$. This shows that $X_-$ is dense in $H^-$
and thus $X_-=X_-^H=H^-$ (since $\dim H^-<\infty$). The latter implies $X_+^H=H^+$.
 Similarly, we can deal with the case that $\dim H^+<\infty$.
\end{proof}

\begin{lemma}\label{lem:BB.9}
Let $\Lambda$ be a metric space, and
 $\mathscr{L}_\lambda:B_X(0, \delta)\to\mathbb{R}$, $\lambda\in\Lambda$, be
 a family of $C^2$-functionals satisfying the conditions (a)-(b) and $d\mathscr{L}_\lambda(0)=0\;\forall\lambda$.
 Suppose that the corresponding operators $A_\lambda$
 satisfies the assumptions (c')-(d') in Theorem~\ref{th:BB.3}.
Then $\Lambda\times B_X(0,\delta)\ni (\lambda, u)\mapsto \mathcal{B}_\lambda(u)\in\mathscr{L}(X)$ is continuous,  $C^1$ in $u$, and
$\Lambda\times B_X(0,\delta)\ni (\lambda, u)\mapsto \mathcal{B}'_\lambda(u)\in\mathscr{L}(X, \mathscr{L}(X))$ is also continuous.
Moreover, if $\Lambda$ is a nonempty open subset of a Banach space $Z$ and the assumptions (c') and (f1)-(f3) in Theorem~\ref{th:BB.3}
are satisfied, then $\Lambda\times B_X(0,\delta)\ni (\lambda, u)\mapsto \mathcal{B}_\lambda(u)\in\mathscr{L}(X)$ is $C^1$.
In particular, for $C^2$ maps $A, \widehat{A}:B_X(0, \delta)\to X$,  if  each of $A''$ and  $\widehat{A}''$  is either $C^1$ or
  uniformly continuous in some neighborhood of the line segment $[0,1]u:=\{tu\,|\,0\le t\le 1\}$ for each $u\in B_X(0, \delta)$,
 then
$$
\mathbb{R}\times B_X(0,\delta)\to \mathscr{L}(X),\; (\lambda, u)\mapsto\mathcal{B}(\lambda,u)=\int^1_0(1-t)A'(tu)dt-\lambda
\int^1_0(1-t)\widehat{A}'(tu)dt
$$
is $C^1$.
 \end{lemma}
\begin{proof}
 {\bf Step 1} ({\it Prove that $\Lambda\times B_X(0,\delta)\ni (\lambda, u)\mapsto \mathcal{B}_\lambda(u)\in\mathscr{L}(X)$
  is continuous}). Fix a point $(\lambda_0, u_0)\in \Lambda\times B_X(0,\delta)$. Then
  for any $(\lambda, u)\in \Lambda\times B_X(0,\delta)$ it holds that
  \begin{eqnarray}\label{e:BB.8}
  \|\mathcal{B}_\lambda(u)-\mathcal{B}_{\lambda_0}(u_0)\|_{\mathscr{L}(X)}\le
  \int^1_0(1-t)\|A'_\lambda(tu)-A'_{\lambda_0}(tu_0)\|_{\mathscr{L}(X)} dt.
  \end{eqnarray}
  By (d') there exist neighborhoods $\mathscr{N}(\lambda_0)$ of $\lambda_0$ in $\Lambda$, and $\mathscr{N}([0,1]u_0)$ of
  $[0,1]u_0$ in $B_X(0, \delta)$ such that
  $\mathscr{N}(\lambda_0)\times\mathscr{N}([0,1]u_0)\ni (\lambda, u)\mapsto A'_\lambda(x)\in\mathscr{L}(X)$
  is uniform continuous, that is, for any $\varepsilon>0$, we have $\nu>0$ such that
  $\|A'_{\lambda_1}(u_1)-A'_{\lambda_2}(u_2)\|_{\mathscr{L}(X)}<\varepsilon$
 provided that  $(\lambda_i, u_i)\in\mathscr{N}(\lambda_0)\times\mathscr{N}([0,1]u_0)$, $i=1,2$,
  satisfies $d(\lambda_1,\lambda_2)<\nu$ and $\|u_1-u_2\|_X<\nu$.
  Thus if $(\lambda, u)\in \Lambda\times B_X(\theta,\delta)$ satisfies $d(\lambda,\lambda_0)<\nu$ and $\|u-u_0\|_X<\nu$
  then $\|A'_\lambda(tu)-A'_{\lambda_0}(tu_0)\|_{\mathscr{L}(X)}<\varepsilon$
  for all $t\in [0,1]$. It follows that
  $$
  \|\mathcal{B}_\lambda(u)-\mathcal{B}_{\lambda_0}(u_0)\|_{\mathscr{L}(X)}\le
  \varepsilon\int^1_0(1-t) dt=\frac{1}{2}\varepsilon.
  $$

\noindent
   {\bf Step 2} ({\it Prove that $\mathcal{B}_\lambda:B_X(0,\delta)\to\mathscr{L}(X)$
  is G\^ateaux differentiable}).
For fixed $u\in B_X(0, \delta)$ and $h\in X$ we have $\rho>0$ such that $u+sh\in B_X(0, \delta)$
for all $s\in [-\rho,\rho]$. Then the function $[-\rho,\rho]\times [0,1]\ni (s,t)\mapsto
(1-t)A'_\lambda(t(u+sh))\in\mathscr{L}(X)$ is continuous and has a continuous partial derivative in $s$
$$
[-\rho,\rho]\times [0,1]\ni (s,t)\mapsto\frac{\partial}{\partial s}((1-t)A'_\lambda(t(u+sh)))$$ $$=(1-t)tA''_\lambda(t(u+sh))h\in
\mathscr{L}(X,\mathscr{L}(X)).
$$
By the differential method with parameter integral we get
$$
\frac{\partial}{\partial s}\mathcal{B}_\lambda(u+sh)=\int^1_0(1-t)t A''_\lambda(t(u+sh))h dt.
$$
In particular, $\mathcal{B}_\lambda$ has a G\^ateaux derivative $D\mathcal{B}_\lambda(u)\in \mathscr{L}(X,\mathscr{L}(X))$ at $u$ given by
$$
D\mathcal{B}_\lambda(u)[h]=\int^1_0(1-t)tA''_\lambda(tu)h dt=\left(\int^1_0(1-t)tA''_\lambda(tu) dt\right)h,\quad\forall h\in X.
$$

\noindent
 {\bf Step 3} ({\it Prove that $\Lambda\times B_X(0,\delta)\ni (\lambda,u)\mapsto \mathcal{B}'_\lambda(u)\in\mathscr{L}(X,\mathscr{L}(X))$
is continuous}). For a fixed point $(\lambda_0, u_0)\in \Lambda\times B_X(0,\delta)$,
 by (d') we have neighborhoods $\mathscr{N}(\lambda_0)$ of $\lambda_0$ in $\Lambda$, and $\mathscr{N}([0,1]u_0)$ of
  $[0,1]u_0$ in $B_X(0, \delta)$ such that
  $$
  \mathscr{N}(\lambda_0)\times\mathscr{N}([0,1]u_0)\ni (\lambda, u)\mapsto A''_\lambda(x)\in\mathscr{L}(X,\mathscr{L}(X))
  $$
  is uniform continuous. Then there exists $\nu>0$ such that when $(\lambda, u)\in \Lambda\times B_X(0,\delta)$ satisfies $d(\lambda,\lambda_0)<\nu$ and $\|u-u_0\|_X<\nu$
  the following holds:
  $$
  \|A''_\lambda(tu)-A''_{\lambda_0}(tu_0)\|_{\mathscr{L}(X,\mathscr{L}(X))}<\varepsilon,
  \quad\forall t\in [0,1].
  $$
   Hence for such $(\lambda,u)$ we deduce
 \begin{eqnarray}\label{e:BB.9}
\|D\mathcal{B}_\lambda(u)-D\mathcal{B}_{\lambda_0}(u_0)\|_{\mathscr{L}(X,\mathscr{L}(X))}&\le&
\int^1_0(1-t)t\|A''_\lambda(tu)-A''_{\lambda_0}(tu_0)\|_{\mathscr{L}(X,\mathscr{L}(X))} dt\nonumber\\
&\le&\frac{1}{6}\varepsilon.
\end{eqnarray}
That is, $\Lambda\times B_X(0,\delta)\ni (\lambda,u)\mapsto D\mathcal{B}_\lambda(u)\in\mathscr{L}(X,\mathscr{L}(X))$
is continuous. In particular, $B_X(0,\delta)\ni u\mapsto D\mathcal{B}_\lambda(u)\in\mathscr{L}(X,\mathscr{L}(X))$
is continuous, and so $\mathcal{B}_\lambda$ has a Fr\'echet derivative $\mathcal{B}'_\lambda(u)=D\mathcal{B}_\lambda(u)$ at $u$.
Then $\mathcal{B}_\lambda$ is $C^1$.

\vspace{4pt}\noindent
 {\bf Step 4} ({\it Prove the conclusion in the ``Moreover" part}). Write
  \begin{eqnarray*}
  \mathcal{B}(\lambda, u)=  \int^1_0(1-t)D_2A(\lambda, tu) dt.
  \end{eqnarray*}
 Since (f1) and (f2) imply (c') and (d'), $\mathcal{B}$ is continuous, $C^1$ in $u$, and
$$
\Lambda\times B_X(0,\delta)\ni (\lambda, u)\mapsto D_2\mathcal{B}(\lambda,u)\in\mathscr{L}(X, \mathscr{L}(X))
$$
is also continuous.
 Similarly, from the second condition in (f1) we may derive that
 $\mathcal{B}(\lambda, u)$ has a G\^ateaux derivative in $\lambda$ given by
 $$
  D_1\mathcal{B}(\lambda, u)=  \int^1_0(1-t)D_1D_2A(\lambda, tu) dt.
 $$
 As above using (f3) we get that  $\Lambda\times B_X(0, \delta)\ni (\lambda, u)\mapsto D_1\mathcal{B}(\lambda,u)\in \mathscr{L}(Z,  \mathscr{L}(X))$
 is continuous. Hence $\mathcal{B}$ is $C^1$.

 The final claim is easily seen from the above arguments.
\end{proof}

\noindent{\bf Acknowledgments}.
The author is deeply grateful to Professor Xiaochun Rong for his invitation. 
I also deeply thank the anonymous referees for useful remarks.

\renewcommand{\refname}{REFERENCES}

%
\medskip
\begin{tabular}{l}
 School of Mathematical Sciences, Beijing Normal University\\
 Laboratory of Mathematics and Complex Systems, Ministry of Education\\
 Beijing 100875, The People's Republic of China\\
 E-mail address: gclu@bnu.edu.cn\\
\end{tabular}

\end{document}